\documentclass[11pt,reqno]{amsproc}

\usepackage{amsmath,amsfonts,amssymb,amsthm}
\usepackage[abbrev,lite,nobysame]{amsrefs}
\usepackage{mathrsfs}
\usepackage{graphics,graphicx}
\usepackage[usenames,dvipsnames]{color}
\usepackage{times}
\usepackage{bbm}
\usepackage[margin=1in]{geometry}
\usepackage[colorlinks=true, pdfstartview=FitV, linkcolor=blue, citecolor=blue, urlcolor=blue]{hyperref}
\usepackage{subcaption}

\usepackage{enumerate}
\usepackage{enumitem}

\makeatletter

\renewcommand\subsection{\@startsection{subsection}{2}%
\normalparindent{.5\linespacing\@plus.7\linespacing}{-.5em}
{\normalfont\bfseries}}

\renewcommand\subsubsection{\@startsection{subsubsection}{3}%
\normalparindent{.5\linespacing\@plus.7\linespacing}{-.5em}
{\normalfont\bfseries}}

\newcommand{\bullpar}[1]{\vspace{.5em}\noindent $\bullet$ \normalfont {\bfseries #1.}}

\newcommand{\bullsubpar}[1]{\vspace{.5em} $\bullet$ \normalfont {\bfseries #1.}}

\newcommand{\diampar}[1]{\vspace{.5em}\noindent $\diamond$ \normalfont {\itshape #1.}}

\def\@tocline#1#2#3#4#5#6#7{\relax
  \ifnum #1>\c@tocdepth 
  \else
    \par \addpenalty\@secpenalty\addvspace{#2}%
    \begingroup \hyphenpenalty\@M
    \@ifempty{#4}{%
      \@tempdima\csname r@tocindent\number#1\endcsname\relax
    }{%
      \@tempdima#4\relax
    }%
    \parindent\z@ \leftskip#3\relax \advance\leftskip\@tempdima\relax
    \rightskip\@pnumwidth plus4em \parfillskip-\@pnumwidth
    #5\leavevmode\hskip-\@tempdima
      \ifcase #1
       \or\or \hskip 1em \or \hskip 2em \else \hskip 3em \fi%
      #6\nobreak\relax
    \dotfill\hbox to\@pnumwidth{\@tocpagenum{#7}}\par
    \nobreak
    \endgroup
  \fi}
\makeatother

\newcounter{bigthm}

\newtheorem{bigtheorem}[bigthm]{Theorem}
\newtheorem{proposition}{Proposition}[section]
\newtheorem{theorem}[proposition]{Theorem}
\newtheorem{lemma}[proposition]{Lemma}
\newtheorem{corollary}[proposition]{Corollary}

\theoremstyle{definition}

\newtheorem{remark}[proposition]{Remark}

\numberwithin{equation}{section}

\newcommand{\B}{{\mathcal{B}}}
\newcommand{\C}{{\mathbb C}}
\newcommand{\E}{{\mathcal E}}
\newcommand{\scE}{{\textsc{E}}}
\newcommand{\F}{{\mathcal F}}

\newcommand{\cI}{{\mathcal I}}
\renewcommand{\L}{{\mathcal{L}}}
\newcommand{\rL}{{\mathrm{L}}}
\newcommand{\G}{{\mathcal{G}}}
\newcommand{\cJ}{{\mathcal{J}}}

\newcommand{\R}{{\mathbb R}}
\newcommand{\Z}{{\mathbb Z}}
\newcommand{\Q}{{\mathcal{Q}}}
\newcommand{\D}{{\Delta}}
\newcommand{\sr}{{\mathsf{r}}}
\newcommand{\s}{{\mathsf{s}}}
\newcommand{\su}{{\mathsf{u}}}
\renewcommand{\sl}{{\mathsf{l}}}
\renewcommand{\sp}{{\text{supp }}}
\newcommand{\T}{{\mathbb T}}
\newcommand{\cT}{{\mathcal T}}
\newcommand{\cS}{{\mathcal S}}
\newcommand{\cV}{{\mathcal{V}}}
\newcommand{\uv}{{0}}
\newcommand{\ov}{{2}}
\newcommand{\W}{{\mathcal{W}}}

\renewcommand{\o}{{\omega}}

\renewcommand{\d}{\mathrm{d}}
\newcommand{\ep}{\varepsilon}
\newcommand\e{{\rm e}}

\renewcommand{\v}{\boldsymbol{v}}

\renewcommand{\b}{\beta}
\newcommand{\g}{{\mathfrak{g}}}

\renewcommand{\P}{{\mathcal{P}}}
\newcommand{\p}{\partial}

\renewcommand{\l}{\left}

\renewcommand{\r}{\right}

\def\Re{{\rm Re}}
\def\Im{{\rm Im}}

\begin{document}

\title{Linear inviscid damping for stably stratified Boussinesq flows} 

\author[A. Enciso]{Alberto Enciso}
\address{Instituto de Ciencias Matem\'{a}ticas, Consejo Superior de Investigaciones Cient\'{i}ficas,\newline
28049 Madrid, Spain}
\email{a.enciso@icmat.es}

\author[M. Nualart]{Marc Nualart}
\address{Instituto de Ciencias Matem\'{a}ticas, Consejo Superior de Investigaciones Cient\'{i}ficas,\newline
28049 Madrid, Spain}
\email{marc.nualart@icmat.es}

\subjclass[2020]{35Q31, 76B70, 35P05, 76E05}

\keywords{Inviscid damping, limiting absorption principle, Boussinesq equations}

\begin{abstract}
We study the linear asymptotic stability of stably stratified monotone shear flows for the Boussinesq equations in the periodic channel. By means of the limiting absorption principle, we obtain a precise description of the inviscid damping experienced by the perturbed velocity field and density, with time-decay rates that depend on the local Richardson number $\cJ(y)$ and split into four stratification regimes (non-stratified, weak, mild, and strong) reflecting qualitative changes in the structure of the Green's function at the critical thresholds $\cJ(y)=0$ and $\cJ(y) = \frac14$. The velocity and density decay estimates are later used to prove quantitative sub-linear growth of the vorticity and gradient of density. As a byproduct of our analysis, we show that, under mild hypotheses on the underlying shear-type equilibrium, the spectrum of the linearised Boussinesq operator is purely continuous.
\end{abstract}

\maketitle

\setcounter{tocdepth}{2}
\tableofcontents

\section{Introduction}\label{inviscid}

Our objective in this paper is to prove linear inviscid damping for the two-dimensional Boussinesq equations on a periodic channel, linearised around a monotone, stably stratified shear flow $\bar{\v}=(v(y),0)$, $\bar{\rho}=\bar{\rho}(y)$. Intuitively, the monotonicity condition ($v'(y)>0$) serves to ensure sufficiently strong mixing, while the stable stratification condition ($\bar\rho'(y)\leq0$, so that  lighter fluid lies above heavier fluid) ensures that the potential energy resists vertical displacement.

We provide decay estimates for the perturbations of the velocity $\v=(v_x,v_y)$ and density $\bar\rho$, with the key property that the time-decay rates depend on the local Richardson number $\cJ(y)=-\mathfrak{g} \bar\rho'(y)/[v’(y)]^2$, where $\mathfrak{g}>0$ is the gravity constant. Specifically, in the non-stratified region ($\cJ(y)=0$), the system reduces to the 2D Euler equations, and we recover the classical decay rates $\|v_x\|_{L^2_x}+t\|v_y\|_{L^2_x}\lesssim t^{-1}$. In the weakly stratified regions ($0<\cJ(y)<\frac14$), the decay is slower: $\|v_x\|_{L^2_x}+t\|v_y\|_{L^2_x}\lesssim t^{-\frac12+[{\frac14-\cJ(y)}]^{1/2}}$, and logarithmic losses appear near the critical threshold $\cJ(y)\approx \frac14$. In strongly stratified regions ($\cJ(y)>\frac14$), one obtains the rate $\Vert v_x\Vert_{L^2_x}+t \Vert v_y\Vert_{L^2_x}\lesssim t^{-1/2}$. These four regimes (which recover the expected rates provided by a back-of-the-envelope heuristic calculation) thus provide a unified quantitative description of how stable stratification modulates inviscid damping at the linearized level. Perturbations of the density $\bar\rho$ also vanish in time, with the decay rates of $\Vert v_x \Vert_{L^2_x}$.

\subsection{The Boussinesq equations}

The Boussinesq equations model the evolution of an incompressible fluid with small density variations. On the periodic channel $\T \times[0,2]$, the equations read as
\begin{equation}\label{eq:EBintro}
\begin{aligned}
(\p_t+\tilde\v\cdot\nabla)\tilde\o  &= -\mathfrak{g}\p_x\tilde\rho, \\
(\p_t+\tilde\v\cdot\nabla)\tilde\rho&=0,
\end{aligned}
\end{equation}
where $\mathfrak{g}>0$ is the gravity constant.
The unknowns are the velocity field of the fluid, $\tilde\v :[0,\infty)\times \T\times[0,2] \rightarrow \R^2$ and its density $\tilde\rho : [0,\infty) \times \T\times [0,2]$, which appears in the equations as an active scalar that influences the velocity field through a buoyancy force. As the fluid is incompressible, we write the velocity field as $\tilde\v=\nabla^\perp\Delta^{-1}\tilde\omega$, where $\tilde\o=\nabla^\perp \cdot \tilde\v:[0,\infty) \times \T\times [0,2]$ is the vorticity. {Here and in what follows, $\Delta^{-1}$~denotes the inverse of the Laplacian on $\T\times[0,2]$ with Dirichlet boundary conditions and $\nabla^\perp=(-\partial_y,\partial_x)$.}

Let us now consider the equilibrium solution 
\begin{equation}\label{eq:Strat}
\bar{\v}=(v(y),0), \qquad \bar{\rho}=\bar{\rho}(y), 
\end{equation}
with $(x,y)\in[0,2]\times\T$, which describes a shear flow velocity field transporting a density stratified in the vertical direction. To understand the long-time dynamics of the equation near this equilibrium, we introduce the perturbed velocity 
$\tilde{\v}=\bar{\v}+\v$ and density profile $\tilde{\rho}=\bar{\rho}+\rho$, and define the corresponding vorticity perturbation $\o=\nabla^\perp\cdot \v$. 

The linearised Euler--Boussinesq system \eqref{eq:EBintro}  around the equilibrium \eqref{eq:Strat} can then be written as
\begin{equation}\label{eq:linEulerBouss}
\begin{cases}
\partial_t\omega + v(y)\partial_x\omega-v''(y)\partial_x\psi=-\mathfrak{g}\partial_x\rho \\
\partial_t\rho + v(y)\partial_x\rho =-\partial_y\bar{\rho}\p_x\psi,\\
\D\psi=\omega,
\end{cases}
\end{equation}
where $\psi= \Delta^{-1}\nabla^\perp \cdot\v $ is the stream-function of $\v$, which satisfies $\psi(x,0)=\psi(x,2)=0$. The linearised system can be further expressed in the more compact formulation
\begin{equation}\label{eq:compactLinBoussinesq}
\partial_t \begin{pmatrix}
\omega \\ \rho
\end{pmatrix} + \L \begin{pmatrix}
\omega \\ \rho
\end{pmatrix} = 0, \quad \begin{pmatrix}
\omega \\ \rho
\end{pmatrix} \Big|_{t=0} = \begin{pmatrix}
\omega ^0\\ \rho^0
\end{pmatrix}
\end{equation}
in terms of the matrix linear operator
\begin{equation}\label{eq:LinSSC}
\L=\begin{pmatrix}
\l(v(y) - v''(y)\D^{-1}\r)\p_x & \mathfrak{g}\p_x \\
\p_y\bar{\rho}\D^{-1}\p_x & v(y)\p_x
\end{pmatrix}.
\end{equation}

The asymptotic stability of the Boussinesq equations \eqref{eq:EBintro} near stably stratified monotonic shear flows began with the works of Howard \cite{Howard} and Miles \cite{Miles}, who addressed the spectral stability of the \emph{constantly stratified Couette flow}, which is given by
\begin{align}\label{eq:couette}
\overline{v}_{\mathrm C}= (y,0), \qquad \overline{\rho}_{\mathrm C} = 1 - \beta^2 y/\mathfrak g.
\end{align}
Their celebrated Miles--Howard stability criterion ensures the absence of unstable eigenvalues of the linearised operator $\L$ if $ \beta^2 > \frac14$ throughout the fluid domain, while such eigenvalues may appear~\cite{drazin1958stability} for $ \beta^2 < \frac14$. In the physical literature,  asymptotic stability was extensively studied in \cite{brown1980algebraic, case1960stability, chimonas1979algebraic, dikii1960stability, farrell1993transient, hoiland1953dynamic, kuo1963perturbations}, albeit with disparate conclusions. Later, Hartman \cite{Hartman} solved \eqref{eq:EBintro} in the periodic channel $\T\times \R$ nearby \eqref{eq:couette}, and conjectured that the vorticity should be unstable in~$L^2$, growing as $\sqrt{t}$ for $\beta^2>\frac14$.

Hartman's approach was mathematically justified in the recent paper \cite{YL18}, which proves linear inviscid damping of density and velocity fields close to \eqref{eq:couette}. There, the authors use Fourier analysis and hypergeometric functions to show that the time-decay rates obtained depend on the strength of the stratification $ \beta^2$: they are uniform for $ \beta^2>\frac14$,  experience a logarithmic loss for $\beta^2=\frac14$, and are an affine function of $\sqrt{\frac12- \b^2}$ for $ \beta^2\in (0,\frac14)$, degenerating as $ \beta^2\rightarrow 0$. 

Shortly afterwards, a variational–energetic method was introduced in \cite{BCZD22} to address the same linear problem in the Miles–Howard regime $\beta^{2} > \tfrac{1}{4}$ and to confirm Hartman’s conjectured growth of the vorticity. The method was later refined to treat the full nonlinear equations: in \cite{BBCZD21}, the authors proved the asymptotic stability of \eqref{eq:couette} over a finite but long time interval, during which the velocity field and density exhibit constant inviscid damping while the vorticity grows in accordance with Hartman's predictions. See also the recent review \cite{bianchini2022symmetrization}.

The techniques developed in all the above results rely on the unboundedness of the vertical variable $y \in \R$ and on the  Fourier-analytic tools available in that setting. As such, the case of the periodic channel \( \T \times [0,2] \) requires a completely different approach. In~\cite{CZN25chan}, the authors studied the linear asymptotic stability of the Couette flow~\eqref{eq:couette} with $\beta^{2} > 0$ by means of the \emph{limiting absorption principle}, which consists in capturing refined resolvent estimates near the spectrum of the linearized operator. 

In the periodic channel, the spectrum of the linearized operator \( \L \) exhibits several features absent in the periodic strip. For $\beta^{2} > \tfrac{1}{4}$, \cite{CZN25chan} proves the existence of two infinite sequences of neutral eigenvalues converging to the endpoints of the essential spectrum of \( \L \). For $\beta^{2} = \tfrac{1}{4}$, the spectrum of \( \L \) is purely continuous, with no discrete (stable, unstable, or neutral) eigenvalues. For $\beta^{2} < \tfrac{1}{4}$, there are no neutral eigenvalues, but the Miles--Howard criterion no longer applies, and unstable discrete eigenvalues may appear. Note that this spectral structure is not favorable for the nonlinear problem, since the presence of neutral or unstable eigenvalues would complicate the analysis. Indeed, linearly unstable eigenvalues may lead to nonlinear instabilities, while neutral eigenvalues, though they may not cause growth, might still prevent decay of the velocity field.

\subsection{Main results}
The goal of this article is to understand in depth the linear asymptotic stability of stratified monotone shear flows, beyond the Couette flow with a linear density stratification. 

The spectral properties of the linearised Boussinesq operator $\L$ at a general shear-type stationary solution need not, in general, be favourable for the stability of the shear flow. Indeed, if $\cJ(y)< \frac14$ in a non-empty subset of $[0,2]$,  the Miles--Howard stability criterion no longer applies, so it is not clear whether unstable eigenvalues exist or not. Likewise, if there are points where $\cJ(y) > \frac14$ the existence of neutral eigenvalues is not ruled out either.

Therefore, we will need several natural assumptions on the equilibrium solution~\eqref{eq:Strat} that we want to perturb. We will formulate these hypotheses in terms of the background shear flow~$v(y)$ and the functions
\begin{align*}
\mathrm{P}(y) := -\partial_y\overline{\rho}, \qquad \P(y) := \g \mathrm{P}(y),\qquad \cJ(y) := \frac{\P(y)}{v'(y)^2}.
\end{align*}
Physically, $\mathrm{P}$ describes the density stratification and $\cJ$ is the local Richardson number.

Specifically, we will need two sets of hypotheses. Firstly, we assume that 
\begin{itemize}
\item H$\P$: $\P(y) \in C^2([0,2])$ with $\P(y)=0$ for all $y\in[0,\vartheta_1]\cup [\vartheta_2,2]$ and  $\P(y)>0$  for all $y\in(\vartheta_1,\vartheta_2)$, for some $0 < \vartheta_1 < \vartheta_2 < 1$.
\item H$v$: $v(y)\in C^4([0,2])$ with $0<c_0\leq v'(y) \leq C_0$ for some $0<c_0<C_0$ and $\text{supp }v'' \subseteq (\vartheta_1, \vartheta_2)$.
\end{itemize}
Second, we also impose the following:
\begin{itemize}
\item[H1] There holds $c_0^{-1}\Vert v''' \Vert_{L^\infty(0,2)} + \frac12 c_0^{-2}\Vert \P'' \Vert_{L^\infty (0,2)}< 1$. 
\item[H2] There exists a unique $\tilde{y}\in [0,2]$ such that ${\cJ}(\tilde{y}) = \max_{y\in[0,2]}{\cJ}(y)$. Moreover,
\begin{itemize}
\item[H2.1] On $[0,\tilde{y}]$,  $v''(y) \leq 0$ and $\cJ'(y) \geq 0$.
\item[H2.2] On $[\tilde{y},2]$,  $v''(y) \geq 0$ and $\cJ'(y) \leq 0$.
\item[H2.3] The equation $\cJ(y) = \frac14$ has exactly two distinct solutions in $[0,2]$.
\end{itemize}
\item [H3] The linearised Euler operator $\L_E\omega:=v(y)\p_x\omega - v''(y)\D^{-1}\p_x\omega$ has no eigenvalue nor embedded eigenvalue.
\end{itemize}

\begin{remark}
{The condition $c_0^{-1}\Vert v''' \Vert_{L^\infty(0,2)}< 1$ suffices to show that $\L_E$ has no embedded eigenvalues, but it does not rule out the existence of discrete eigenvalues outside the continuous spectrum. Thus, \textbf{H3} is not superfluous next to \textbf{H1}}.
Assumption H2.3 ensures that there are two distinct regions where $\cJ\approx \frac14$ and exactly two roots of the equation $\cJ(y)=\frac14$. This is a simplification that we performed to keep the focus on the strong and weak regions. Indeed, if $\cJ(y) = \frac14$ on an interval we can use the ideas presented in \cite{CZN25chan} to show that Theorems~\ref{thm:mainID} and~\ref{thm:growth} below remain true there: the Green's function in such a region is still well-defined with analogous estimates to those presented in Section \ref{sec:GreensM} and the homogeneous solutions constructed for the spectral theory in Section \ref{subsec:mild} are defined so that they already contain the appropriate logarithmic corrections.
\end{remark}

The starting point of our analysis is the observation that, under these hypotheses, the linearised operator has purely continuous spectrum:

\begin{bigtheorem}\label{thm:spectrumL}
Assume that hypotheses H$\P$, H$v$ and H1--H3 hold. Then, the spectrum of $\L$ is  purely continuous; there are no eigenvalues or embedded eigenvalues. 
\end{bigtheorem}

   We emphasize that Theorem~\ref{thm:spectrumL} is of independent interest, apart from the stability estimates that we shall state next. Indeed, characterising the spectrum of the linearised Boussinesq equations near stably stratified monotone shear flows has attracted attention in the physics community for quite some time, and despite several partial results (see~\cites{Hirota,Miles,Howard} and~\cite[Section~3.2.2]{Yaglom}), no unifying description of the full spectrum had been available. Under assumptions~\textbf{H1}--\textbf{H3}, Theorem~\ref{thm:spectrumL} provides the first rigorous and complete characterization of the spectral properties of the linearized operator. On the other hand, while~\textbf{H1}--\textbf{H3} are sufficient to ensure the absence of discrete or embedded eigenvalues of \( \L \),  in Sections~\ref{sec:spectrum} and~\ref{sec:LAPstrat} we shall provide further insights on the main structural conditions required for Theorem~\ref{thm:mainID} to hold.

Once we have made sure that the spectrum of $\L$ might lead to the spectral dispersive effects that underlie inviscid damping, we turn to the analysis of the asymptotic stability of the linearised equations. For this, since it is already known from the analysis of the Couette flow~\cites{BCZD22, YL18, CZN25chan, CZN23strip} that the Richardson number controls the dynamics of the linearised equations, it is necessary to tag the positions where there is a qualitative change in the behavior of the local Richardson number.
Consequently, let us divide the interval $[0,2]$ according to the behavior of~$\cJ(y)$. For this, we fix some $0< \tilde\delta < \frac18$ and note that, as a consequence of hypothesis~H2, there exists $\varpi_1, \varpi_2\in (0,2)$ and $\varpi_{n,j}\in (0,2)$, for $n,j=1,2$ such that:
\begin{enumerate}
\item $\vartheta_1 < \varpi_{1,1} <  \varpi_{1} < \varpi_{1,2} < \varpi_{2,1} < \varpi_{2}  < \varpi_{2,2} < \vartheta_2$.
\item $\cJ(\varpi_1) = \cJ(\varpi_2) =\frac14$.
\item $\cJ(y) >\frac14$ for all $y\in (\varpi_{1}, \varpi_2)$.
\item $0 <\cJ(y) < \frac14 $ for all $y\in (\vartheta_1, \varpi_1) \cup (\varpi_2, \vartheta_2)$.
\item $|\cJ(y)-\frac14| \leq \tilde\delta$,  for all $y\in (\varpi_{1,1}, \varpi_{1,2}) \cup (\varpi_{2,1}, \varpi_{2,2})$.
\end{enumerate}
We will refer to the regions defined by these positions as {\em non-stratified} ($[0,\vartheta_1] \cup [\vartheta_2,2]$), {\em weakly stratified} ($(\vartheta_1, \varpi_{1,1}) \cup (\varpi_{2,2}, \vartheta_2)$), {\em mildly stratified} ($[\varpi_{1,1}, \varpi_{1,2}] \cup [\varpi_{2,1}, \varpi_{2,2}]$) and {\em strongly stratified} ($(\varpi_{1,2}, \varpi_{2,1})$).

Using this notation, we are ready to present our main result , which is the linear inviscid damping of the perturbed velocity and density, with time-decay rates that depend on the local Richardson number $\cJ(y)$. To this end, we define 
\begin{align*}
 \mu(y) := \Re \sqrt{\frac14 - \cJ(y)} .
\end{align*}

\begin{bigtheorem}\label{thm:mainID}
Assume that the initial data $(\o^0,\rho^0)$ is compactly supported inside $(\vartheta_1, \vartheta_2)$ and that hypotheses {H$\P$, H$v$} and H1--H3 hold. Writing $\rho^0(y) = \mathrm{P}(y)\varrho^0(y)$, we also assume that
\begin{equation}\label{eq:zeroxave}
\int_\T \o^0(x,y)\d x = \int_\T\varrho^0(x,y)\d x =0.
\end{equation}
Let $\v=(v^x,v^y)=\nabla^\perp\psi=(-\p_y\psi,\p_x\psi)$ be the corresponding velocity field. There exists a constant $C>0$ such that, for all $t\geq 1$, the solution decays as follows:
\begin{itemize}
\item {\em Non-stratified region:} for $y\in [0,\vartheta_1] \cup [\vartheta_2,2]$,
\begin{align}
\Vert v^x(t,x,y) \Vert_{L^2_x}    &\leq C t^{-1} \left(  \Vert \omega^0 \Vert_{H^{1/2}_x H^2_y} + \Vert \varrho^0 \Vert_{H^{1/2}_x H^3_y}\right), \label{eq:decayvxnonstrat} \\
\Vert v^y(t,x,y) \Vert_{L^2_x}  &\leq C t^{-2}\left(  \Vert \omega^0 \Vert_{H^{1/2}_x H^3_y} + \Vert \varrho^0 \Vert_{H^{1/2}_x H^4_y}\right),\label{eq:decayvynonstrat}\\
\Vert \rho(t,x,y) \Vert_{L^2_x} &= 0. \label{eq:decayrhononstrat} 
\end{align}

\item {\em Weakly stratified region:} for $y\in (\vartheta_1, \varpi_{1,1}) \cup (\varpi_{2,2}, \vartheta_2)$,
\begin{align}
\Vert v^x(t,x,y) \Vert_{L^2_x}  & \leq C t^{-\frac12+\mu(y)} \left(  \Vert \omega^0 \Vert_{H^{1/2}_x H^2_y} + \Vert \varrho^0 \Vert_{H^{1/2}_x H^3_y}\right), \label{eq:decayvxweak} \\
\Vert v^y(t,x,y) \Vert_{L^2_x}  &\leq C t^{-\frac32+\mu(y)}\left(  \Vert \omega^0 \Vert_{H^{1/2}_x H^3_y} + \Vert \varrho^0 \Vert_{H^{1/2}_x H^4_y}\right),\label{eq:decayvyweak}\\
\Vert \rho(t,x,y) \Vert_{L^2_x} &\leq C t^{-\frac12+\mu(y)} \left(  \Vert \omega^0 \Vert_{H^{1/2}_x H^2_y} + \Vert \varrho^0 \Vert_{H^{1/2}_x H^3_y}\right). \label{eq:decayrhoweak} 
\end{align}

\item {\em Mildly stratified region:} for $y\in [\varpi_{1,1}, \varpi_{1,2}] \cup [\varpi_{2,1}, \varpi_{2,2}]$,
\begin{align}
\Vert v^x(t,x,y) \Vert_{L^2_x}  & \leq C t^{-\frac12+\mu(y)}( 1 + \log t) \left(  \Vert \omega^0 \Vert_{H^{1/2}_x H^2_y} + \Vert \varrho^0 \Vert_{H^{1/2}_x H^3_y}\right), \label{eq:decayvxmild} \\
\Vert v^y(t,x,y) \Vert_{L^2_x}  &\leq C t^{-\frac32+\mu(y)}(1+\log t)  \left(  \Vert \omega^0 \Vert_{H^{1/2}_x H^3_y} + \Vert \varrho^0 \Vert_{H^{1/2}_x H^4_y}\right),\label{eq:decayvymild}\\
\Vert \rho(t,x,y) \Vert_{L^2_x} &\leq C t^{-\frac12+\mu(y)} ( 1 + \log t) \left(  \Vert \omega^0 \Vert_{H^{1/2}_x H^2_y} + \Vert \varrho^0 \Vert_{H^{1/2}_x H^3_y}\right). \label{eq:decayrhomild} 
\end{align}

\item {\em Strongly stratified region:} for $y\in (\varpi_{1,2}, \varpi_{2,1})$,
\begin{align}
\Vert v^x(t,x,y) \Vert_{L^2_x}  & \leq C t^{-\frac12}\left(  \Vert \omega^0 \Vert_{H^{1/2}_x H^2_y} + \Vert \varrho^0 \Vert_{H^{1/2}_x H^3_y}\right), \label{eq:decayvxstrong} \\
\Vert v^y(t,x,y) \Vert_{L^2_x}  &\leq C t^{-\frac32}\left(  \Vert \omega^0 \Vert_{H^{1/2}_x H^3_y} + \Vert \varrho^0 \Vert_{H^{1/2}_x H^4_y}\right),\label{eq:decayvystrong}\\
\Vert \rho(t,x,y) \Vert_{L^2_x} &\leq C t^{-\frac12} \left(  \Vert \omega^0 \Vert_{H^{1/2}_x H^2_y} + \Vert \varrho^0 \Vert_{H^{1/2}_x H^3_y}\right). \label{eq:decayrhostrong} 
\end{align}
\end{itemize}
\end{bigtheorem}

\begin{remark}
It is natural to wonder where the obtained time-decay rates come from, and whether they are optimal. In Subsection~\ref{subsec:heuristics} below, we have included a relatively short heuristic argument that, in a non-rigorous way, suggests that these should be the optimal exponents. It is therefore reasonable to expect that the decay estimates of Theorem~\ref{thm:mainID} (and the derivative bounds of Theorem~\ref{thm:growth} below) are likely sharp. Perhaps this could be easily seen in the case of  linearised equations on the periodic strip which admit explicit solutions~\cite{CZN23strip}; in fact, the quantity $\Vert \omega \Vert_{L^2} + \Vert \nabla\rho \Vert_{L^2}$ is known to grow like $t^\frac12$ for the full non-linear problem~\cite{BBCZD21}.
\end{remark}

Several comments are now in order. First, note that the zero average condition \eqref{eq:zeroxave} is standard in all linear inviscid damping results, since the averages of the initial data are otherwise preserved by the linear evolution of the equations. The inviscid damping decay rates \eqref{eq:decayvxnonstrat}--\eqref{eq:decayrhostrong} are pointwise in $y\in[0,2]$ and describe the asymptotic stability of the stably stratified shear flow $(\overline\v, \overline\rho)$. The velocity field and density at $y\in [0,2]$ decay in time with rates that crucially depend on the local Richardson number $\cJ(y)$ through $\mu(y)$. In the strongly stratified regime $\mu(y) = 0$ and the decay rates are faster than in the weakly stratified regime, where $\mu(y) >0$. The transition region between the strongly stratified region and weakly stratified region experiences a logarithmic loss. Because of the compact support assumptions, in the non-stratified region the equations locally reduce to the Euler equations of motion, and the usual inviscid-damping decay rates for the velocity field are recovered. Likewise, the compact support of the density perturbation is preserved: due to the support assumptions on $\P(y)$ and $\rho^0(y)$, no density is transported into the non-stratified regime.

Second, note that, as is standard in inviscid damping estimates, to quantify the decay of the solutions we need bounds on fairly high Sobolev norms of the initial datum. This is because a key underlying mechanism behind the decay estimates is the trade off of regularity of the initial data for decay in time, so that sufficiently smooth initial data is needed to reach the desired decay rates. For instance, the inviscid damping for $\Vert v^y(t,x,y) \Vert_{L^2_x}$ needs $\omega^0\in H^{1/2}_x H^3_y$ and $\varrho^0\in H^{1/2}_x H^4_y$. These effects, which will also be apparent in other ensuing theorems, will be discussed in Section \ref{sec:growth}.

Third, let us say a few words about the support condition, which is crucially used in the proofs. Motivated by the non-linear inviscid damping results for the Euler equations in the periodic channel, our goal is to ensure that the linear evolution of the vorticity and density remains compactly supported away from the physical boundaries.  Unlike the linearised Euler equations, we observe from \eqref{eq:EBintro} that having $v''(y)$ and $\omega^0$ compactly supported is not sufficient to ensure that $\omega$ remains compactly supported inside the periodic channel. Indeed, the vorticity perturbation is now transported by $v(y)\partial_x$ but also forced by $-\g \partial_x\rho$. In turn, the density $\rho$ is transported by $v(y)\partial_x$ and crucially forced by $-\partial_y \overline{\rho}\partial_x\psi$, which is a non-local operator and in general it will not be zero near the boundary. Therefore, for $\rho$ not to be forced near the boundary we need $-\partial_y \overline{\rho}$ to be compactly supported in the channel. This is the main reason behind the choice of our density profile as in H$\P$. Then, $\rho(t)$ is compactly supported in the channel if $\rho^0$ itself is so. Consequently,  $\omega(t)$ is also compactly supported once $\omega^0$ is so as well.

Once $\P(y)$ is supported in $[\vartheta_1, \vartheta_2]$, it is natural to assume that the perturbed density is in fact compactly supported inside $(\vartheta_1,\vartheta_2)$, since otherwise it would be the main leading term responsible for the buoyancy forces outside $[\vartheta_1,\vartheta_2]$ and the dynamics would be much more complicated. This rather heuristic justification becomes transparent in Section \ref{sec:sketch}, where $\rho^0$  is used to regularize the equation governing the evolution of the velocity field.

Since our main goal in this paper is to understand the effects of stratification, we likewise assume that $\text{supp } \omega^0 \subset (\vartheta_1,\vartheta_2)$, which is a simplification of the more essential condition $\text{supp } \omega^0 \subset (0,2)$. While the former ensures that the vorticity remains supported in the stratified region, which is enough for our purposes, the later allows vorticity in the non-stratified region, and away from the physical boundaries. As we shall see in Section \ref{sec:nonstrat}, the dynamics of the linearised equations outside $(\vartheta_1,\vartheta_2)$ are closely related to those of the linearised Euler equations nearby a monotone shear flow: the relevant inviscid damping \eqref{eq:decayvxnonstrat}-\eqref{eq:decayvynonstrat} and vorticity scattering can be obtained for initial vorticity supported outside the stratified region following the ideas of \cites{Jia20, JiaGev20}.

{As a minor remark, we mention that assumptions \textbf{H1} and \textbf{H2.3} cannot hold simultaneously if they are posed in the shorter yet more usual interval $[0,1]$, the main reason being that $\P$ does not grow sufficiently fast in $[0,1]$ so that $\cJ(y) < \frac14$ for all $y\in[0,1]$. Hence, there are no strongly nor mildly stratified regions, and the situation is less rich. Of course, this is not a limitation of the method of proof:  the analysis remains valid and the statements are essentially the same. In particular, Theorem~\ref{thm:spectrumL} remains true on $[0,1]$ and the estimates of Theorems~\ref{thm:mainID}-\ref{thm:growthlocalweak} now apply in the non-stratified region $[0,\vartheta_1]\cup[\vartheta_1,1]$ and the weakly stratified region $(\vartheta_1,\vartheta_2)$. The size of the channel and a weak Richardson number $\cJ < \frac14$ influence the stability properties of a family of background steady states \eqref{eq:Strat}, which has been recently addressed in \cite{bianchini2025ill}.}

{Therefore, for concreteness, we have opted to work in the broader channel $\T\times[0,2]$, which admits steady states \eqref{eq:Strat} for which \textbf{H1}--\textbf{H3} holds and the strongly stratified region is non-empty. Such an example would be the Couette flow $v_C(y) =y$ paired with a suitable compactly supported stratification $\P(y)\in C^2$ such that $\Vert \P''\Vert_{L^\infty(0,1)}\leq 2$, non-decreasing in $[0,1]$, symmetric with respect to $y=1$ and with $\P(1)>\frac14$. We conclude by noting that increasing the size of the channel from $[0,2]$ to $[0,\mathsf{H}]$ for some $\mathsf{H}>2$ is automatic: this enlarges the family of steady states \eqref{eq:Strat} for which \textbf{H1}--\textbf{H3} holds, the strongly stratified region ($\cJ>\frac14$) increases, and the conclusions of Theorems~\ref{thm:spectrumL}-\ref{thm:growthlocalweak} remain true.  }

\subsection{Refinements and further estimates} 

Here we present a few additional results that complement our main inviscid damping estimate, Theorem~\ref{thm:mainID}.

One should note that the decay rates \eqref{eq:decayvxnonstrat}--\eqref{eq:decayrhononstrat} are not the continuous limit as $y\rightarrow \vartheta_1$ or $y\rightarrow \vartheta_2$ of the weaker rates \eqref{eq:decayvxweak}-\eqref{eq:decayrhoweak}. Indeed, those rates degenerate into merely boundedness of $\v^x$ and $\rho$ and to $t^{-1}$ decay of $v^y$, thus representing a loss of a full $t^{-1}$ power. While this mismatch is difficult to reconcile, we are nevertheless able to locally sharpen~\eqref{eq:decayvxweak}–\eqref{eq:decayrhoweak}. This constitutes our second main result, which we state next. Estimates analogous to \eqref{eq:localdecayvxweak}-\eqref{eq:localdecayrhoweak} also hold for $y\in [0,\varpi_{1,1})$, with  $(y-\vartheta_1)_+$ replacing $(\vartheta_2-y)_+$. 

\begin{bigtheorem}\label{thm:IDlocalweak}
Under the assumptions of Theorem \ref{thm:mainID}, for all $t\geq1$ and $y\in  (\varpi_{2,2}, 2]$,  there holds
\begin{align}
\Vert v^x(t,x,y) \Vert_{L^2_x}  & \leq C \min \left( t^{-\frac12+\mu(y)}, t^{-1} + (\vartheta_2 - y)_+ \right) \left(  \Vert \omega^0 \Vert_{H^{3/2}_x H^3_y} + \Vert \varrho^0 \Vert_{H^{3/2}_x H^4_y}\right), \label{eq:localdecayvxweak} \\
\Vert v^y(t,x,y) \Vert_{L^2_x}  &\leq C \min \left( t^{-\frac32+\mu(y)}, t^{-2} + (\vartheta_2 - y)_+t^{-1} \right) \left(  \Vert \omega^0 \Vert_{H^{3/2}_x H^3_y} + \Vert \varrho^0 \Vert_{H^{3/2}_x H^4_y}\right),\label{eq:localdecayvyweak}\\
\Vert \rho(t,x,y) \Vert_{L^2_x} &\leq C  \min \left( t^{-\frac12+\mu(y)}, \mathrm{P}(y)(1+\log t) \right) \left(  \Vert \omega^0 \Vert_{H^{1/2}_x H^2_y} + \Vert \varrho^0 \Vert_{H^{1/2}_x H^3_y}\right). \label{eq:localdecayrhoweak} 
\end{align}
\end{bigtheorem}

Note that Theorem \ref{thm:IDlocalweak} shows that, for all $t\geq 1$, the decay estimates (though not the rates) \eqref{eq:localdecayvxweak}-\eqref{eq:localdecayrhoweak} are now continuous in $y\in [0,\varpi_{1,1})\cup (\varpi_{2,2},2]$. Namely, the velocity field can be decomposed into an Euler-type fast-decaying part, together with a slower contribution that vanishes in the non-stratified region. 

In addition to the bounds for the velocity and density presented so far, it is convenient to consider the evolution of the vorticity and of the density gradient.
This is because, in the Boussinesq equations \eqref{eq:EBintro}, the vorticity is not only transported by the velocity field, but also forced by $\partial_x\rho$. Thus, there is no reason for the $L^\infty$ norm of the vorticity to be preserved. In fact, for bounded in time forcing, the vorticity could grow up to linearly in time. The density perturbation is also transported by the velocity field $v(y)\partial_x$ and forced by $\mathrm{P}(y)v^y(t,x,y)$. Since $v(y)$ is smooth and strictly monotone, $\partial_y\rho$ should experience at most linear growth as the fluid is stretched by the mixing effects. 

Our next result asserts that such at most linear-in-time growth is damped by the stabilizing mechanism responsible for the decay of  $v(t,x,y)$ and $\rho(t,x,y)$, thereby  justifying the (upper bound in the) predictions made by Hartman~\cite{Hartman} in the periodic channel.

\begin{bigtheorem}\label{thm:growth}
Under the assumptions of Theorem \ref{thm:mainID}, there exists $C>0$ such that, for all $t\geq 1$, 
\begin{itemize}
\item {\em Non-stratified region:} for $y\in [0,\vartheta_1] \cup [\vartheta_2,2]$,
\begin{align}
\Vert \omega(t,x,y) \Vert_{L^2_x}  &=0, \label{eq:growthomeganonstrat} \\
\Vert \partial_y\rho(t,x,y) \Vert_{L^2_x} &= 0.\label{eq:growthpyrhononstrat} 
\end{align}

\item {\em Weakly stratified region:} for $y\in (\vartheta_1, \varpi_{1,1}) \cup (\varpi_{2,2}, \vartheta_2)$,
\begin{align}
\Vert \omega(t,x,y) \Vert_{L^2_x}  & \leq C t^{\frac12+\mu(y)} \left(  \Vert \omega^0 \Vert_{H^{3/2}_x H^3_y} + \Vert \varrho^0 \Vert_{H^{3/2}_x H^4_y}\right), \label{eq:growthomegaweak} \\
\Vert \partial_y\rho(t,x,y) \Vert_{L^2_x} &\leq C t^{\frac12+\mu(y)} \left(  \Vert \omega^0 \Vert_{H^{3/2}_x H^3_y} + \Vert \varrho^0 \Vert_{H^{3/2}_x H^4_y}\right). \label{eq:growthpyrhoweak} 
\end{align}

\item {\em Mildly stratified region:} for $y\in [\varpi_{1,1}, \varpi_{1,2}] \cup [\varpi_{2,1}, \varpi_{2,2}]$,
\begin{align}
\Vert \omega(t,x,y) \Vert_{L^2_x}  & \leq C t^{\frac12+\mu(y)}( 1 + \log t) \left(  \Vert \omega^0 \Vert_{H^{3/2}_x H^3_y} + \Vert \varrho^0 \Vert_{H^{3/2}_x H^4_y}\right), \label{eq:growthomegamild} \\
\Vert \partial_y\rho(t,x,y) \Vert_{L^2_x} &\leq C t^{\frac12+\mu(y)} ( 1 + \log t) \left(  \Vert \omega^0 \Vert_{H^{3/2}_x H^3_y} + \Vert \varrho^0 \Vert_{H^{3/2}_x H^4_y}\right). \label{eq:growthpyrhomild} 
\end{align}

\item {\em Strongly stratified region:} for $y\in (\varpi_{1,2}, \varpi_{2,1})$,
\begin{align}
\Vert \omega(t,x,y) \Vert_{L^2_x}  & \leq C t^{\frac12}\left(  \Vert \omega^0 \Vert_{H^{3/2}_x H^3_y} + \Vert \varrho^0 \Vert_{H^{3/2}_x H^4_y}\right), \label{eq:growthomegastrong} \\
\Vert \partial_y\rho(t,x,y) \Vert_{L^2_x} &\leq C t^{\frac12} \left(  \Vert \omega^0 \Vert_{H^{3/2}_x H^3_y} + \Vert \varrho^0 \Vert_{H^{3/2}_x H^4_y}\right). \label{eq:growthpyrhostrong} 
\end{align}
\end{itemize}
\end{bigtheorem}

The compact support of the initial data, of \( v''(y) \), and of \( \P(y) \) ensures that the vorticity remains compactly supported within the stratified region. The growth bounds~\eqref{eq:growthomegastrong}--\eqref{eq:growthpyrhostrong} are sublinear in time due to the inviscid damping experienced by \( v(t,x,y) \) and \( \rho(t,x,y) \). While in the strongly stratified regime the growth rate is at most \( t^{\frac{1}{2}} \), we observe that in the weakly stratified region the growth rate becomes linear as \( y \to \vartheta_{1} \) or \( y \to \vartheta_{2} \), since the stabilizing effects of inviscid damping weaken and ultimately vanish there. However, we can refine these bounds in the spirit of Theorem~\ref{thm:IDlocalweak}, so that they converge to those of the non-stratified setting:

\begin{bigtheorem}\label{thm:growthlocalweak}

Under the assumptions of Theorem \ref{thm:mainID}, for $y\in  (\varpi_{2,2}, 2]$ and all $t\geq1$, there holds
\begin{align}\label{eq:localgrowthomegaweak}
\Vert \omega(t,x,y) \Vert_{L^2_x}  & \leq C \min \left( t^{\frac12+\mu(y)}, 1 + (\vartheta_2 - y)_+t \right) \left(  \Vert \omega^0 \Vert_{H^{5/2}_x H^3_y} + \Vert \varrho^0 \Vert_{H^{5/2}_x H^4_y}\right),  \\
\label{eq:localgrowthpyrhoweak} 
\Vert \partial_y\rho(t,x,y) \Vert_{L^2_x} &\leq C \left( \mathrm{P}'(y)(1+\log t) +  t\min \left( t^{-\frac12+\mu(y)}, \mathrm{P}(y)(1+\log t) \right) \right)  \\
&\qquad\qquad\qquad \times \left(  \Vert \omega^0 \Vert_{H^{3/2}_x H^3_y} + \Vert \varrho^0 \Vert_{H^{3/2}_x H^4_y}\right). \notag
\end{align}
\end{bigtheorem}

In particular, for any fixed time \( t \geq 1 \), the estimates~\eqref{eq:localgrowthomegaweak} and~\eqref{eq:localgrowthpyrhoweak} capture the limiting behaviour as \( y \to \vartheta_{2} \) expected from~\eqref{eq:growthomeganonstrat} and~\eqref{eq:growthpyrhononstrat}, while also retaining the sublinear growth described by~\eqref{eq:growthomegaweak} and~\eqref{eq:growthpyrhoweak} in the long-time limit.

\subsection{A heuristic derivation of the decay exponents}\label{subsec:heuristics}

As anticipated above, we now present a non-rigorous derivation of the inviscid damping decay rates \eqref{eq:decayvxnonstrat}-\eqref{eq:decayrhostrong}. Formally, the solution $(\omega(t), \rho(t))$ to the linearised Boussinesq equations \eqref{eq:compactLinBoussinesq} can be written as
\begin{align*}
\begin{pmatrix}
\omega \\ \rho
\end{pmatrix}(t) = e^{-t\L}\begin{pmatrix}
\omega^0 \\ \rho^0
\end{pmatrix} = \frac{1}{2\pi i}\int_{\partial\Omega}e^{-\lambda t}(\lambda-\L)^{-1}\begin{pmatrix}
\omega^0 \\ \rho^0
\end{pmatrix} \d \lambda
\end{align*}
where $\Omega$ denotes an open set in the complex plane $\C$ containing the spectrum of $\L$. Since $\L$ decouples in the $x$-variable, for 
\begin{align*}
\omega(t,x,y) = \sum_{k\in \Z}\omega_k(t,y)e^{ikx}, \quad \rho(t,x,y) = \sum_{k\in \Z}\rho(t,y)e^{ikx}
\end{align*}
the evolution for each $k$ mode is given by
\begin{align*}
\begin{pmatrix}
\omega_k \\ \rho_k
\end{pmatrix}(t,y) = \frac{1}{2\pi i}\int_{\partial\Omega} e^{-\lambda t} (\lambda - ik \L_k)^{-1} \begin{pmatrix}
\omega_k^0 \\ \rho_k^0 
\end{pmatrix} \d \lambda, \quad \L_k = \begin{pmatrix}
v(y) - v''(y)\Delta_k^{-1} & \g \\
-\mathrm{P}(y)\Delta_k^{-1} & v(y)
\end{pmatrix}
\end{align*}
with $\Delta_k = \partial_y^2 - k^2$. If the spectrum of $\L_k$ consists only of a purely continuous part that coincides with the range of $v(y)$, then for $\lambda = ik v(y_0)$ with $y_0\in [0,2]$ we may write
\begin{align*}
\begin{pmatrix}
\omega_k \\ \rho_k
\end{pmatrix}(t,y) = \frac{1}{2\pi i}\int_0^2 e^{- ikv(y_0) t} (v(y_0) -\L_k)^{-1} \begin{pmatrix}
\omega_k^0 \\ \rho_k^0 
\end{pmatrix} v'(y_0) \d y_0.
\end{align*}
Since we are interested in the long-time dynamics of the velocity field, for $\psi_k(t,y) = \Delta_k^{-1}\omega_k(t,y)$ we now observe that
\begin{align}\label{eq:heuristicDunford}
\begin{pmatrix}
\psi_k \\ \rho_k
\end{pmatrix}(t,y) = \frac{1}{2\pi i}\int_0^2 e^{- ikv(y_0) t}  \begin{pmatrix}
\varphi_k(y,y_0) \\ \rho_k(y,y_0) 
\end{pmatrix} v'(y_0) \d y_0,
\end{align}
where the pair $(\Delta_k\varphi_k, \rho_k)$ is the unique solution to the resolvent equations
\begin{align*}
(v(y_0) - \L_k) \begin{pmatrix}
\Delta_k\varphi_k \\ \rho_k
\end{pmatrix} = \begin{pmatrix}
\omega^0 \\ \rho^0
\end{pmatrix},
\end{align*}
which are equivalent to
\begin{align}
\Delta_k \varphi_k(y,y_0) - \frac{v''(y)}{v(y) - v(y_0)}\varphi_k(y,y_0) + \frac{\P(y)}{(v(y) - v(y_0))^2}\varphi_k(y,y_0) &= \frac{\omega_k^0(y)}{v(y) - v(y_0)} - \frac{\g \rho_k^0(y)}{(v(y) - v(y_0)^2}, \label{eq:heuristicTG}\\
\rho_k(y,y_0) &= \frac{\rho_k^0 + \mathrm{P}(y) \varphi_k(y,y_0)}{v(y) - v(y_0)} \label{eq:heuristicrho}
\end{align}
together with $\varphi_k(0,y_0) = \varphi_k(2,y_0) = 0$. Already from \eqref{eq:heuristicDunford} we observe that time decay for $(\psi_k, \rho_k)$ may be obtained integrating by parts, if the resolvent functions $(\varphi_k,\rho_k)$ are sufficiently smooth. However, we also already see that equations \eqref{eq:heuristicTG} and \eqref{eq:heuristicrho} governing $(\varphi_k,\rho_k)$ exhibit a potential singularity for $y$ near $y_0$ and thus the regularity of the resolvent functions is not a priori straightforward.

Nevertheless, for all $y_0\in (\vartheta_1, \vartheta_2)$ we recall that $\P(y_0)>0$ and, heuristically, homogeneous solution $\phi_k(y,y_0)$ to \eqref{eq:heuristicTG} should locally satisfy
\begin{align}\label{eq:heursiticRTG}
\partial_y^2\phi_k(y,y_0) - k^2\phi_k(y,y_0) + \frac{\cJ(y_0)}{(y-y_0)^2}\phi_k(y,y_0) = 0.
\end{align}
since $\cJ(y_0) = \frac{\P(y_0)}{(v'(y_0))^2}$. Inspired by the Euler index formula for ordinary differential equations, we ansatz $\phi_k(y,y_0) = (2k(y-y_0))^\alpha \mathcal{E}_k(y,y_0)$, for some $\alpha>0$ and some smooth function $\mathcal{E}_k$ with $\mathcal{E}_k(y_0,y_0)=1$. Then, plugging the ansatz into \eqref{eq:heursiticRTG}, multiplying by $(2k(y-y_0))^{2-\alpha}$ and evaluating at $y=y_0$ we obtain
\begin{align*}
4k^2\alpha(\alpha-1) - 4k^2 \cJ(y_0) = 0,
\end{align*}
whose two solutions are
\begin{align*}
\alpha_1(y_0) = \frac12 +\gamma(y_0), \quad \alpha_2(y_0) = \frac12-\gamma(y_0),  \quad\gamma(y_0) = \sqrt{\frac14 - \cJ(y_0)}.
\end{align*}
Consequently, the most singular homogeneous solution $\psi_{\s,k}$ should behave like 
$$\phi_{\s,k}(y,y_0) \approx (2k(y-y_0))^{\frac12-\gamma(y_0)},$$
so that the regularity of solutions non-trivially depend on $\cJ(y_0)$, the local Richardson number. It is more regular for $y_0\in (\varpi_{1,2}, \varpi_{2,1})$ the strongly stratified regime since there $\cJ(y_0)>\frac14$ and thus $\gamma(y_0) = i\nu(y_0)$, while in the weakly stratified regime $\mu(y_0)\in (0,\frac12)$ for $y_0\in (\vartheta_1, \varpi_{1,1}) \cup (\varpi_{2,2},\vartheta_2)$, with $\mu(y)\rightarrow \frac12$ for $y_0\rightarrow \vartheta_1$ or $y_0\rightarrow \vartheta_2$, which corresponds to worse regularity for $\phi_{\s,k}$. We further remark here that for $y_0\in (0,2)$ with $\cJ(y_0)=\frac14$ two roots $\alpha_1(y_0)= \alpha_2(y_0)=\frac12$, the associated homogeneous solutions are then not linearly independent and a logarithmic correction needs to be introduced.

Coming back to \eqref{eq:heuristicDunford} and further supposing that the resolvent solutions locally behave like the homogeneous solutions to \eqref{eq:heuristicTG}, we have that
\begin{align*}
\psi_k(t,y) &\approx \frac{1}{2\pi i}\int_0^2 e^{-ikv(y_0)t} (2k(y-y_0))^{\frac12-\mu(y_0)} v'(y_0) \d y_0 \\
&\approx \frac{1}{2\pi i}\int_0^2 e^{-ikv(y_0)t} (2k(y-y_0))^{\frac12-\mu(y)}v'(y_0) \d y_0
\end{align*}
because $\text{exp}\left( \left|\mu(y) - \mu(y_0) \right| \left| \log (|2k(y-y_0)|) \right| \right) \lesssim 1$ whenever $2k|y-y_0|\lesssim 1$. Disregarding any boundary contributions, integrating by parts once,
\begin{align*}
\psi_k(t,y) &\approx \frac{1}{2\pi i}\frac{k(1-2\mu(y))}{ikt}\int_0^2 e^{-ikv(y_0)t} (2k(y-y_0))^{-\frac12-\mu(y)}\d y_0 \\
&\approx \frac{1}{2\pi i}\frac{k(1-2\mu(y))}{ikt}\int_{y-\delta}^{y+\delta} e^{-ikv(y_0)t} (2k(y-y_0))^{-\frac12-\mu(y)}\d y_0 \\
&\quad + \frac{1}{2\pi i}\frac{k(1-2\mu(y))}{ikt}\int_{(0,2)\setminus (y-\delta, y+\delta)} e^{-ikv(y_0)t} (2k(y-y_0))^{-\frac12-\mu(y)}\d y_0
\end{align*}
for some $\delta>0$. Integrating back, we have
\begin{align*}
\left| \frac{1}{2\pi i}\frac{k(1-2\mu(y))}{ikt}\int_{y-\delta}^{y+\delta} e^{-ikv(y_0)t} (2k(y-y_0))^{-\frac12-\mu(y)}\d y_0 \right| \lesssim k^{-\frac12}(k\delta)^{\frac12-\mu(y)}
\end{align*}
while integrating by parts once more, discarding any possible boundary contributions, taking the modulus and then integrating back also gives
\begin{align*}
&\left| \frac{1}{2\pi i}\frac{k(1-2\mu(y))}{ikt}\int_{(0,2)\setminus (y-\delta, y+\delta)} e^{-ikv(y_0)t} (2k(y-y_0))^{-\frac12-\mu(y)}\d y_0 \right| \\
&\qquad = \left| \frac{1}{2\pi i}\frac{k^2\cJ(y)}{(ikt)^2}\int_{(0,2)\setminus (y-\delta, y+\delta)} e^{-ikv(y_0)t} (v'(y_0))^{-1} (2k(y-y_0))^{-\frac32-\mu(y)} \d y_0 \right|  \\
&\lesssim k^{-\frac12}t^{-2}(k\delta)^{-\frac12-\mu(y)}
\end{align*}
Optimizing in $\delta>0$, we see that $k\delta = \frac{1}{t}$ gives the best scaling and produces $|\psi_k(t,y)| \lesssim k^{-\frac12}t^{-\frac32+\mu(y)}$, which is the claimed decay rate. The estimates for $v^x_k(t,y) = \partial_y\psi_k(t,y)$ and $\rho_k(t,y)$ follow similarly. Let us finish by observing that for $y_0\in [0,\vartheta_1]\cup [\vartheta_2,2]$ we have $v''(y_0) = \P(y_0) = 0$, so that \eqref{eq:heuristicTG} becomes the Rayleigh equation for the Couette shear flow, whose solutions enjoy better regularity. This is the main reason why the inviscid damping decay rates \eqref{eq:decayvxnonstrat} - \eqref{eq:decayrhononstrat} for $y\in [0,\vartheta_1]\cup [\vartheta_2,2]$ are those obtained for the Euler equations near the Couette flow.

\subsection{Literature review}

In the absence of gravity and for constant density, the Boussinesq equations \eqref{eq:EBintro} reduce to the well-known two-dimensional Euler equations, for which shear flows $\overline{\v}=(v(y),0)$ are steady-state solutions. The asymptotic stability of such shear flows in ideal fluids has been investigated since the work of Lord Kelvin~\cite{kelvin1887stability}. Later, Orr~\cite{orr1907stability} discovered that the linear transport of vorticity induced by the background shear produces a mixing mechanism, leading to weak convergence of the perturbed vorticity toward its average and to decay in time of the associated velocity field.

For for the full non-linear Euler equations, these effects were first captured  near the Couette flow in the periodic strip $\T\times\R$ in the ground-breaking work of Bedrossian and Masmoudi~\cite{BM15}. Since then, there has been much work on inviscid damping for other background flows and geometries. A broad collection of linear results has been obtained for various configurations, see for instance~\cites{lu2024singularity, ionescu2022linear, ren2025domain, ren2023linear, ren2025linear, BCZV19, CZZ19, GNRS20, Jia20, JiaGev20, WZZ18, WZZ19, WZZKolmo20, Zillinger16, Zillinger17, ZillingerCirc17, Zillinger21}. However, extending these results to the non-linear regime is notoriously difficult. The non-linear analysis in~\cite{BM15} required Gevrey regularity of the initial data to control the so-called ``non-linear echoes'' (transient growth of Fourier modes through non-linear interactions) as well as a delicate choice of coordinates. Non-linear inviscid damping for the Euler equations near monotone shear flows and point vortices was later established in~\cite{IJ22, IJ20, IJnon20, MZ20}, which remain the only known results for the full non-linear problem and also rely on Gevrey regularity. In contrast, initial data lacking such regularity can exhibit transient growth~\cite{deng2023long} or even fail to decay altogether~\cite{lin2011inviscid, sinambela2025transition}.

The studies~\cite{IJ22, IJnon20, MZ20} consider the periodic channel, where even at the linear level, the boundary behaviour of the perturbed vorticity is subtle, see~\cite{Zillinger16, lu2024singularity}. In this setting, the non-local terms can generate vorticity near the boundary that does not scatter with the background flow. To overcome this difficulty, it is typically assumed that both the perturbation $\omega^0$ and the background curvature $v''$ are compactly supported inside the channel, ensuring that the vorticity evolution reduces to pure transport near the boundaries and that compact support is preserved in time. We refer the reader to the surveys~\cite{ionescu2022nonlinear, wei2022hydro} for an in-depth discussion of inviscid damping phenomena in the Euler and equations. 

The limiting absorption principle, central to the present work, has also been implemented in the periodic setting in~\cite{CZN23strip}, where explicit solutions to the linearised equations yield an alternative proof of damping for the velocity and density perturbations.

A particular sub-class of steady states of the Boussinesq system of the form~\eqref{eq:Strat} are the so-called \emph{rest states}, corresponding to $v(y)\equiv0$. For these equilibria, the stabilising effects of vorticity mixing are absent. Nevertheless, the stable stratification of the background density $\overline\rho$ introduces dispersive effects that can extend the lifespan of solutions; see~\cite{elgindi2015sharp, jurja2024long, jurja2025effect} and references therein. When viscosity and thermal diffusivity are added to~\eqref{eq:EBintro}, the picture changes drastically: the equations become globally well-posed, and the vorticity mixing gives rise to an additional mechanism known as \emph{enhanced dissipation}, which further stabilises the dynamics. We refer to the recent works~\cites{arbon2025enhanced, coti2024stability, knobel2025suppression, masmoudi2022stability, masmoudi2023asymptotic, niu2024improved, zhang2023stability, zhai2023stability, zillinger2021enhanced, masmoudi2025linear} for significant progress in this direction.

The two-dimensional Boussinesq equations are also closely related to the three-dimensional axisymmetric Euler equations away from the axis of symmetry. In that setting, finite-time singularities are known to occur, as demonstrated in the seminal work of Elgindi~\cite{Elgindi21} and further developed in~\cites{ElgindiJeong19, ElgindiJeong20, ChenHou23}. Conversely, inviscid damping plays a key role in establishing global well-posedness for the three-dimensional Euler equations and for the inhomogeneous two-dimensional Euler equations near certain stationary states, see~\cite{GPW} and~\cite{CWZZ, Zhao23}, respectively. In the context of the 2D Boussinesq equations near the stably stratified Couette flow, long but finite time existence has been obtained in~\cite{BBCZD21} using inviscid damping estimates. Whether this can be extended to global-in-time well-posedness, or whether instabilities eventually drive the system out of the perturbative regime, remains an intriguing open question.

In this direction, it is natural to ask whether the linear asymptotic stability of the stably stratified monotone shear flows~\eqref{eq:Strat} established in Theorem~\ref{thm:mainID} persists for the full non-linear Boussinesq system, at least over long but finite times. A first step in this direction is to establish Gevrey regularity of the linear solutions. A second step is to determine whether the compact support of both the initial data and the background flow, together with the velocity estimates from Theorem~\ref{thm:mainID}, suffices to guarantee that the vorticity remains compactly supported. These issues will be investigated in future work.

Finally, concerning the axisymmetric 3D Euler equations, it would be interesting to study the asymptotic stability of radial, pipe-like steady states where the velocity has axial and swirl components depending on the radial coordinate. These configurations constitute the closest analogue to the stably stratified shear flows~\eqref{eq:Strat} considered here. For such radial profiles, one can define an appropriate Richardson number, and a Miles–Howard-type stability criterion is available~\cite{howard1962hydrodynamic}. A complete description of the long-time linear dynamics near these steady states remains open. Motivated by the parallels between the two systems, we plan to explore whether the analytical framework and techniques developed in this paper can be adapted to the Euler setting in future work.

\subsection{Organization of the paper}

In Section~2 we outline the strategy of the proof: we decouple the dynamics in Fourier modes, reduce the linearised problem to a family of (reduced) Taylor–Goldstein boundary value problems, and state the limiting absorption principle that underlies the time-decay effects. Sections~\ref{sec:Greens}-\ref{sec:GreensM} construct and analyse the Green’s function of the Reduced Taylor--Goldstein operator in the strong/weak (\S\ref{sec:GreensWS}) and mild (\S\ref{sec:GreensM}) stratification regimes, isolating the regular and singular behaviours that determine the decay exponents. Section~\ref{sec:solopstrat} develops mapping estimates for the associated solution and error operators in the tailored function spaces that capture the local structure near the critical layer. Section~\ref{sec:homTG} builds homogeneous solutions adapted to these singular structures, and Section~\ref{sec:spectrum} settles the spectral picture of the linearised operator, proving the absence of discrete spectrum and establishing that the spectrum is purely continuous, thereby justifying the contour representations used later.

Sections~\ref{sec:fragile}–\ref{sec:nonstrat} treat the regime-by-regime analysis required to pass from local resolvent control to global dynamics: we handle the ``fragile'' limit near the edges of stratification (Section~\ref{sec:fragile}), prove the limiting absorption principle in the stratified region (Section~\ref{sec:LAPstrat}), and recover the Euler-type behaviour in the non-stratified region (Section~\ref{sec:nonstrat}). Building on these ingredients, Section~\ref{sec:SobolevReg} derives Sobolev regularity for the spectral density and its good derivatives, which feeds into the oscillatory-phase arguments in Section~\ref{sec:IDestimates} to obtain the inviscid damping estimates for velocity and density. Finally, Section~\ref{sec:growth} quantifies the sublinear growth of vorticity and density gradients. To conclude, the appendices contain a number of technical results that are used throughout the paper: properties of Whittaker functions, logarithmic approximations, formulas for Green's functions, and higher-order operator bounds.

\section{Main ideas and sketch of the proof}\label{sec:sketch}
In this section we rigorously present the arguments in the derivation of the decay rates, dividing the proof into several steps that are then completed in the corresponding sections of the manuscript. Throughout, we assume that hypotheses H$\P$, H$v$, and H1--H3 hold.

\subsection{Fourier decomposition and spectral representation}
The setting of the periodic channel $\T\times[0,2]$ considered in this article prevents the direct use of Fourier methods in the vertical direction $y$. However, we can still decouple \eqref{eq:linEulerBouss}
 in Fourier modes in $x\in\T$, writing
\begin{align*}
\o=\sum_{k\in\Z}\o_k(t,y)\e^{ikx}, \quad \rho=\sum_{k\in\Z}\rho_k(t,y)\e^{ikx}, \qquad \psi=\sum_{k\in\Z}\psi_k(t,y)\e^{ikx},
\end{align*} 
so that 
\begin{equation}\label{eq:linEBomegarho}
\begin{split}
(\partial_t+ikv(y))\omega_k-ikv''(y)\psi_k&=-ik\mathfrak{g}\rho_k, \\ 
(\partial_t+ikv(y))\rho_k&=ik\mathrm{P}(y)\psi_k,
\end{split}
\end{equation}
for each $k\in\Z$, with 
\begin{equation*}
\begin{cases}
\D_k\psi_k =\o_k, \\
\psi_k|_{y=0,2}=0,
\end{cases}
\qquad  \D_k:= \p_y^2-k^2.
\end{equation*}  
The modes corresponding to the $x$-average, that is $k=0$, are preserved by the time-evolution and thus we will not consider them further (cf. \eqref{eq:zeroxave}). 
Moreover, as $\omega$ and $\rho$ are real-valued functions, we necessarily have that $\overline{\omega_{-k}}=\omega_k$ and $\overline{\rho_{-k}}=\rho_k$. Hence, without loss of generality, we take $k\geq 1$ throughout the manuscript.

For our purposes, it is more convenient to write \eqref{eq:linEulerBouss} in the compact  stream-function formulation 
\begin{equation*}
\partial_t \begin{pmatrix} \psi_k \\ \rho_k\end{pmatrix}+ikL_k\begin{pmatrix}
\psi_k \\ \rho_k
\end{pmatrix}=0,
\end{equation*}
and directly obtain its solution as
\begin{equation*}
\begin{pmatrix} \psi_k \\ \rho_k\end{pmatrix}=\e^{-ikL_kt}\begin{pmatrix}
\psi_k^0 \\ \rho_k^0
\end{pmatrix}
\end{equation*}
where $L_k$ is the linear operator defined by
\begin{equation}\label{eq:linOP}
L_k =\begin{pmatrix}
\D_k^{-1}(v(y)\D_k-v''(y))  & \mathfrak{g}\D_k^{-1} \\
-\mathrm{P}(y) & v(y)
\end{pmatrix}.
\end{equation}
Using Dunford's formula \cites{Engel-Nagel, Taylor-11}, we have that
\begin{equation}\label{eq:Dunford}
\begin{pmatrix}
\psi_k(t,y) \\ \rho_k(t,y)
\end{pmatrix} = \frac{1}{2\pi i} \int_{\p\Omega}\e^{-ikct} (c-L_k)^{-1}\begin{pmatrix}
\psi_k^0(y) \\ \rho_k^0(y)
\end{pmatrix} \,\d c, 
\end{equation}
where here $\Omega$ is any domain containing the spectrum $\sigma(L_k)$. Under the assumptions H1-H3, we have from Theorem \ref{thm:spectrumL}, see Theorem \ref{thm:spectrumlinop} for a more precise statement, that $\sigma(L_k) = [v(0), v(2)]$, the range of the velocity field. Hence, we can reduce the contour of integration to  
\begin{equation}\label{eq:psirholim}
\begin{pmatrix}
\psi_k(t,y) \\ \rho_k(t,y)
\end{pmatrix} 
=\frac{1}{2\pi i }\lim_{\ep\rightarrow 0}\int_0^2 \e^{-ikv(y_0)t}\left((-v(y_0)-i\ep+L_k)^{-1}-(-v(y_0)+i\ep+L_k)^{-1}\right)
\begin{pmatrix}
\psi_k^0 \\ \rho_k^0 
\end{pmatrix}\, v'(y_0) \d y_0.
\end{equation}
For $\ep>0$, we denote the resolvent pair
\begin{equation}\label{eq:geneigen}
\begin{pmatrix}
\psi^{\pm}_{k,\ep}(y,y_0) \\ \rho^\pm_{k,\ep}(y,y_0)
\end{pmatrix}:=\l( -v(y_0)\pm i\ep+L_k\r)^{-1}\begin{pmatrix}
\psi_k^0(y) \\ \rho_k^0(y)
\end{pmatrix}
\end{equation}
and obtain the coupled system of equations
\begin{equation*}
\begin{aligned}
\o_k^0(y)&=(v(y)-v(y_0)\pm i\ep)\D_k\psi^\pm_{k,\ep}(y,y_0) - v''(y)\psi_{m,\ep}^\pm+\mathfrak{g} \rho^\pm_{k,\ep}(y,y_0), \\
\rho_k^0(y)&=(v(y)-v(y_0)\pm i\ep)\rho^\pm_{k,\ep}(y,y_0) -\mathrm{P}(y)\psi^\pm_{k,\ep}(y,y_0).
\end{aligned}
\end{equation*}
We first solve
\begin{equation}\label{eq:rhomep}
\rho^\pm_{k,\ep}(y,y_0)=\frac{\rho_k^0(y) + \mathrm{P}(y)\psi^\pm_{k,\ep}(y,y_0)}{v(y)-v(y_0)\pm i\ep}
\end{equation}
and from there we obtain the following inhomogeneous \emph{Taylor-Goldstein equation} for $\psi^\pm_{k,\ep}$,
\begin{equation}\tag{TG}\label{eq:TG}
\begin{split}
&\left(\Delta_k-\frac{v''(y)}{v(y)-v(y_0)\pm i\ep}+\frac{\cJ(y)}{(v(y)-v(y_0)\pm i\ep)^2}\right)\psi_{k,\ep}^\pm(y,y_0) \\
&\qquad\qquad\qquad=\frac{\omega_k^0(y)}{v(y)-v(y_0)\pm i\ep}-\frac{\mathfrak{g}\rho_k^0(y)}{(v(y)-v(y_0)\pm i\ep)^2},
\end{split}
\end{equation}
along with homogeneous Dirichlet boundary conditions $\psi_{k,\ep}^\pm(0,y_0) = \psi_{k,\ep}^\pm(2,y_0) = 0$, for all $y_0\in [0,2]$. Note how the local Richardson number enters the equation as the strength of the factor that becomes critically singular in the limit $\ep\to0$.

Motivated by the heuristic justifications in Section  \ref{subsec:heuristics}, understanding the precise regularity properties of $\psi_{k,\ep}^\pm$ and $\rho_{k,\ep}^\pm$ is key for the application of oscillatory-phase arguments to \eqref{eq:psirholim}. In this direction, we aim to study a simplified version of \eqref{eq:TG} by reducing the operator into a  fixed-coefficient operator. 

\subsection{The Reduced Taylor-Goldstein operator and its Green's function}
We define the Reduced Taylor-Goldstein operator
\begin{align}\label{eq:defRTGoperator}
\text{RTG}_{k,\ep}^\pm :=\p_y^2 - k^2 + \frac{\cJ(y_0)}{(y-y_0 \pm i\ep_0)^2},  \quad \ep_0:=\frac{\ep}{v'(y_0)}
\end{align}
and we further set the error operator
\begin{align}\label{eq:deferroroperator}
\scE_{k,\ep}^\pm &:=  -\frac{v''(y)}{v(y) - v(y_0) \pm i\ep} + \frac{\P(y) - \P(y_0)}{(v(y) - v(y_0) \pm i\ep)^2} +\left( \frac{\P(y_0)}{(v(y) - v(y_0) \pm i\ep)^2} - \frac{\cJ(y_0)}{(y-y_0\pm i\ep_0)^2}\right)
\end{align}
so that \eqref{eq:TG} can be cast as
\begin{align*}
\text{RTG}_{k,\ep}^\pm\psi_{k,\ep}^\pm(y,y_0) + \scE_{k,\ep}^\pm(y,y_0)\psi_{k,\ep}^\pm(y,y_0) = \frac{\omega_k^0(y)}{v(y)-v(y_0)\pm i\ep}-\frac{\mathfrak{g}\rho_k^0(y)}{(v(y)-v(y_0)\pm i\ep)^2}.
\end{align*}
Let $\G_{k,\ep}^\pm(y,y_0,z)$ denote the Green's function of the Reduced Taylor-Goldstein operator with Dirichlett boundary conditions at $y=0$ and $y=2$. It is such that
\begin{align*}
\left( \p_y^2 - k^2 + \frac{\cJ(y_0)}{(y-y_0 \pm i\ep_0)^2}\right) \G_{k,\ep}^\pm(y,y_0,z) = \delta(y-z), \quad \G_{k,\ep}^\pm(0,y_0,z) = \G_{k,\ep}^\pm(2,y_0,z) = 0,
\end{align*}
for all $y_0,z\in[0,2]$.

To properly determine the regularity properties of $\psi_{k,\ep}^\pm$, a precise understanding of the Green's function $\G_{k,\ep}^\pm$ is most useful. We start by noting that for $y_0\in [0,\vartheta_1]\cup [\vartheta_2,2]$ we have $\cJ(y_0) = v''(y_0) = 0$ and the Reduced Taylor-Goldstein operator becomes the usual Laplacian operator, so that $\G_{k,\ep}^\pm(y,y_0,z)$ denotes the Green's function of the Dirichlet Laplacian, which is studied in Section \ref{sec:nonstrat}.

For the most singular setting $y_0\in (\vartheta_1,\vartheta_2)$, we have $\cJ(y_0)>0$ and we construct the Green's function with the standard method, first finding the homogeneous solutions to \eqref{eq:defRTGoperator}. We introduce here the regularity index
\begin{align}
    \gamma(y_0) = \sqrt{\frac14 - \cJ(y_0)}
\end{align}
for which we note that $\mu(y_0) = \Re(\gamma(y_0))$ and we define $\nu(y_0) = \Im( \gamma(y_0))$. Fortunately, 
\begin{align}\label{eq:defMkappagamma}
M_{0, \gamma(y_0)}((2k(y-y_0\pm i\ep_0)), \quad M_{0, -\gamma(y_0)}((2k(y-y_0\pm i\ep))
\end{align}
are two linearly independent homogeneous solutions to \eqref{eq:defRTGoperator} for all $y_0\in (\vartheta_1, \vartheta_2)$ with $\cJ(y_0)\neq \frac14$. Here, $M_{0,\gamma}$ denotes the modified Whittaker function~\cite{Whittaker03}, and constitutes a central piece in our analysis. While we properly define and state its main properties in Appendix \ref{app:Whittaker}, we record here that 
\begin{align}\label{eq:asymptoticMkappagamma}
M_{0,\gamma}(\zeta)= \zeta^{\frac12+\gamma}\mathcal{E}_{0,\gamma}(\zeta),
\end{align} 
where $\mathcal{E}_{0,\gamma}(\zeta)$ is an analytic function, with $\mathcal{E}_{0,\gamma}(0)=1$ and $ \mathcal{E}_{0,\gamma} ' (0) = 0$. Since $\G_{k,\ep}^\pm$ is a suitable linear combination of the two homogeneous solutions \eqref{eq:defMkappagamma}, it follows that the regularity properties of $\G_{k,\ep}^\pm$ are dictated by those of $M_{0,\gamma(y_0)}$ and $M_{0,-\gamma(y_0)}$. 

We next observe that for $y_0\in (0,2)$ with $\cJ(y_0)=\frac14$, the two homogeneous solutions are no longer linearly independent, and a second pair of fundamental solutions must be considered.  Without going into further details, which are relegated to the forthcoming sections, we further define
\begin{align*}
\widetilde\varpi_{1,1} &= \frac{\varpi_{1}-\varpi_{1,1}}{2}, \quad  \widetilde\varpi_{1,2} = \frac{\varpi_{1,2}-\varpi_{1}}{2} \\
\widetilde\varpi_{2,1} &= \frac{\varpi_{2}-\varpi_{2,1}}{2}, \quad  \widetilde\varpi_{2,2} = \frac{\varpi_{2,2}-\varpi_{2}}{2}
\end{align*}
and we divide the interval $[0,2]$ into four (intersecting) distinct regimes:
\begin{itemize}
\item The non-stratified regime $I_E := [0,\vartheta_1] \cup [\vartheta_2,2]$, where $\cJ(y)=0$ for all $y\in I_E$.
\item The weakly stratified regime $I_W := (\vartheta_1, \widetilde{\varpi}_{1,1}) \cup (\widetilde\varpi_{2,2}, \vartheta_2)$. There holds $0<\cJ< \frac14$ in $I_W$.
\item The mildly stratified regime $I_M :=  ({\varpi}_{1,1}, \varpi_{1,2}) \cup ({\varpi}_{2,1}, \varpi_{2,2})\setminus \lbrace \varpi_1, \varpi_2 \rbrace$. There holds $\left| \cJ -\frac14 \right| \leq \tilde\delta$ in $I_M$.
\item The strongly stratified regime $I_S := (\widetilde\varpi_{1,2}, \widetilde\varpi_{2,1})$. There holds $\cJ > \frac14$ in $I_S$.
\end{itemize}
Given the relevance of $\gamma(y_0)$, to keep the notation as simple as possible, we set
\begin{align*}
\gamma_0 := \gamma(y_0), \quad \mu_0 :=\mu(y_0), \quad \nu_0 := \nu(y_0)
\end{align*}
throughout the manuscript.

\subsection{Functional Spaces}
To motivate the choice of our spaces, we first note that the meaningful scaling that balances out the two potential terms $-k^2$ and $\frac{\cJ(y_0)}{(y-y_0\pm i\ep)^2}$ in the Reduced Taylor-Goldstein operator \eqref{eq:defRTGoperator} is roughly given by $k^2|y-y_0|^2\approx \cJ(y_0)$. Hence, for $\beta^2 := \Vert \cJ \Vert_{L^\infty}>0$ we define the \textit{local}, resp. \textit{non-local} sets
\begin{align*}
I_n(y_0) := \left\lbrace y\in [0,2] \, : \, |y-y_0|\leq \frac{n\beta}{k} \right\rbrace, \quad I_n^c(y_0) := [0,2]\setminus I_n(y_0)
\end{align*}
for $n\geq 1$. 

For the sake of clarity, we next argue for $y_0\in I_E\cup I_S$.  Given the asymptotic expansion \eqref{eq:asymptoticMkappagamma} for the homogeneous solutions to \eqref{eq:defRTGoperator}, we set 
\begin{align*}
\eta = 2k(y-y_0\pm i\ep_0)
\end{align*}
and for functions $\varphi(y,y_0)$ we define
\begin{align*}
\Vert \varphi \Vert_{X_{k,\ep,y_0}^j}:=\inf_{\substack{\varphi_\sr, \varphi_\s \in H^1(I_3(y_0)) \\  \varphi = \eta^{\frac12+\gamma_0}\varphi_\sr + \eta^{\frac12-\gamma_0}\varphi_\s }} \sum_{n=0}^j k^{-n}\left( \Vert \eta^{-\frac{n}{2}}\partial_y^n\varphi_\sr\Vert_{L^{\infty}(I_3(y_0))} + \Vert \eta^{-\frac{n}{2}} \partial_y^n\varphi_\s\Vert_{L^{\infty}(I_3(y_0))} \right).
\end{align*}
and
\begin{align*}
\Vert \varphi(\cdot, y_0) \Vert_{X_{k,\ep,y_0}} := \Vert \varphi \Vert_{X_k^1}  + k^\frac12\Vert \varphi(\cdot, y_0) \Vert_{H^1_k(I_3^c(y_0))}. 
\end{align*}
Intuitively, the space $X_{k,\ep,y_0}^1$ measures the size of the coefficients $\varphi_\sr$ and $\varphi_\s$ accompanying the \textit{regular} $\eta^{\frac12+\gamma_0}$ and \textit{singular} $\eta^{\frac12-\gamma_0}$ roots, respectively, in the scaling region $I_3(y_0)$ where the potential singularity is strongest. 

While the space $X_{k,\ep,y_0}$ captures the regularity structures of the homogeneous solutions \eqref{eq:defMkappagamma}, the prescribed pointwise vanishing of $\partial_y\varphi_\sigma$ is sometimes too restrictive, so  we need to introduce a weaker space that nonetheless retains the most basic regularity properties. Indeed, we next define 
\begin{align*}
\Vert \varphi(\cdot, y_0) \Vert_{Z_{k,\ep,y_0}} := \Vert \varphi \Vert_{Z_{k,\ep,y_0}^1}  + k^\frac12\Vert \varphi(\cdot, y_0) \Vert_{H^1_k(I_3^c(y_0))} 
\end{align*}
where now
\begin{align*}
\Vert \varphi \Vert_{Z_{k,\ep,y_0}^1}:=\inf_{\substack{\varphi_\sr, \varphi_\s \in H^1(I_3(y_0)) \\  \varphi = \eta^{\frac12+\gamma_0}\varphi_\sr + \eta^{\frac12-\gamma_0}\varphi_\s }} &\Big( \Vert \varphi_\sr\Vert_{L^{\infty}(I_3(y_0))} + \Vert  \varphi_\s\Vert_{L^{\infty}(I_3(y_0))} \\
&\quad + k^{-\frac12}\Vert \partial_y\varphi_\sr\Vert_{L^{2}(I_3(y_0))} + k^{-\frac12}\Vert  \partial_y\varphi_\s\Vert_{L^{2}(I_3(y_0))} \Big).
\end{align*}
As usual, we set
\begin{align*}
X_{k,\ep,y_0} := \lbrace \varphi	\in L^2(0,2) : \Vert \varphi \Vert_{X_{k,\ep,y_0}} < +\infty \rbrace, \quad Z_{k,\ep,y_0} := \lbrace \varphi	\in L^2(0,2) : \Vert \varphi \Vert_{Z_{k,\ep,y_0}} < +\infty \rbrace
\end{align*}
and we readily note that $X_{k,\ep,y_0}\subset Z_{k,\ep,y_0}$ since $\Vert \varphi \Vert_{Z_{k,\ep,y_0}} \leq C\Vert \varphi \Vert_{X_{k,\ep,y_0}}$, for some universal constant $C>0$ independent of $k\geq 1$, $y_0\in[0,2]$ and $\ep\in (0,1)$. To simplify notation, we shall omit the dependence on $\ep$ and $y_0$ of the spaces, and just denote
\begin{align*}
X_k := X_{k,\ep,y_0}, \quad Z_k := Z_{k,\ep,y_0}.
\end{align*}
\subsubsection{The mild regime}
When $y_0\in I_M$ the mildly stratified region, $\cJ(y_0)$ is close to $\frac14$, the associated $\gamma_0$ is close to  $0$ and the Wronskian $\W\lbrace\eta^{\frac12+\gamma_0}, \eta^{\frac12-\gamma_0}\rbrace = -4k\gamma_0$ is also close to $0$, so that in the limit as $\cJ(y_0)\rightarrow\frac14$ the two homogeneous solutions  \eqref{eq:defMkappagamma} lose their linear independence. To overcome this difficulty, we instead consider decompositions of the form
\begin{equation}\label{eq:decomlogvarphi}
\varphi(y,y_0) = \varphi_\sr(y,y_0)\eta^{\frac12 +\gamma_0} + \varphi_\s(y,y_0)\eta^{\frac12-\gamma_0}\log(\eta)\Q_{\gamma_0}(\eta),
\end{equation}
where 
\begin{equation}\label{eq:defQgamma}
\Q_{\gamma_0}(\eta) := \int_0^1 e^{2\gamma_0 s \log(\eta)} \d s.
\end{equation}
and for which $\mathcal{W}\lbrace \gamma^{\frac12+\gamma_0}, \eta^{\frac12-\gamma_0}\log(\eta)\mathcal{Q}_{\gamma}(\eta) \rbrace = -2k$, which is always non-zero. Already simplifying the notation, we thus define
\begin{align*}
\Vert \varphi(\cdot, y_0) \Vert_{LZ_k} := \Vert \varphi \Vert_{LZ_k^1}  + k^\frac12\Vert \varphi(\cdot, y_0) \Vert_{H^1_k(I_3^c(y_0))} 
\end{align*}
with
\begin{align*}
\Vert \varphi \Vert_{LZ_k^j}:=\inf_{\substack{\varphi_\sr, \varphi_\s \in H^1(I_3(y_0)) \\  \varphi = \eta^{\frac12+\gamma_0}\varphi_\sr + \eta^{\frac12-\gamma_0}\log(\eta) \Q_{\gamma_0}(\eta)\varphi_\s }} & \Big( \Vert \varphi_\sr\Vert_{L^{\infty}(I_3(y_0))} + \Vert  \varphi_\s\Vert_{L^{\infty}(I_3(y_0))}  \\
&\quad + \left. k^{-\frac12}\Vert \partial_y\varphi_\sr\Vert_{L^{2}(I_3(y_0))} + k^{-\frac12}\Vert  \partial_y\varphi_\s\Vert_{L^{2}(I_3(y_0))} \right)
\end{align*}
and similarly
\begin{align*}
\Vert \varphi(\cdot, y_0) \Vert_{LX_k} := \Vert \varphi \Vert_{LX_k^1}  + k^\frac12\Vert \varphi(\cdot, y_0) \Vert_{H^1_k(I_3^c(y_0))} 
\end{align*}
with
\begin{align*}
\Vert \varphi \Vert_{LX_k^j}:= \inf_{\substack{\varphi_\sr, \varphi_\s \in H^1(I_3(y_0)) \\  \varphi = \eta^{\frac12+\gamma_0}\varphi_\sr + \eta^{\frac12-\gamma_0}\log(\eta) \Q_{\gamma_0}(\eta)\varphi_\s }}\sum_{n=0}^j k^{-n} \left( \Vert \eta^{-\frac{n}{2}}\partial_y^n\varphi_\sr\Vert_{L^{\infty}(I_3(y_0))} + \Vert \eta^{-\frac{n}{2}}  \partial_y^n\varphi_\s\Vert_{L^{\infty}(I_3(y_0))} \right)
\end{align*}
As before, there holds $\Vert \varphi \Vert_{LZ_k} \lesssim \Vert \varphi \Vert_{LX_k}$ and thus $LX_k \subset LZ_k$.

\subsection{Regularity of the Green's function, local to global bounds and regularization}
As we previously indicated, we construct the Green's function $\G_{k,\ep}^\pm$ through the homogeneous solutions to \eqref{eq:defRTGoperator}. Since these homogeneous solutions \eqref{eq:defMkappagamma} are fairly explicit, we can obtain explicit expressions for $\G_{k,\ep}
^\pm$, that are recorded in Section \ref{sec:Greens} below. Crucially, the Green's function inherits the regularity properties of these homogeneous solutions in the local region. In the weak and strong regimes, they are:

\begin{theorem}\label{thm:XkregGSW}
Let $k\geq 1$, $\ep\in (0,1)$ and $y_0\in I_S \cup I_W$. There exists $\ep_*>0$ such that the Green's function admits the decomposition
\begin{align*}
\G_{k,\ep}^\pm(y,y_0,z) = \left( \G_{k,\ep}^\pm\right)_\sr(y,y_0,z) \eta^{\frac12 + \gamma_0} + \left( \G_{k,\ep}^\pm\right)_\s(y,y_0,z) \eta^{\frac12 - \gamma_0}
\end{align*}
where $\eta=2k(y-y_0\pm i\ep_0)$, 
\begin{align*}
\sup_{y\in I_3(y_0)} \left\Vert \left( \G_{k,\ep}^\pm\right)_\sigma(y,y_0,\cdot) \right\Vert_{X_k^1} \lesssim k^{-1}, \quad \sup_{y\in I_3(y_0)} \left\Vert \left( \G_{k,\ep}^\pm\right)_\sigma(y,y_0,\cdot) \right\Vert_{H_k^1(I_3^c(y_0))} \lesssim k^{-\frac32},
\end{align*}
and
\begin{align*}
\sup_{y\in I_3(y_0)} \left\Vert \partial_y\left( \G_{k,\ep}^\pm\right)_\sigma(y,y_0,\cdot) \right\Vert_{X_k^1} + \sup_{y\in I_3(y_0)}  k^\frac12 \left\Vert \partial_y\left( \G_{k,\ep}^\pm \right)_\sigma(y,y_0,\cdot) \right \Vert_{L^2(I_3^c(y_0))}\lesssim |\eta|
\end{align*}
uniformly for all $y_0\in I_S\cup I_W$, all $0<\ep<\ep_*$ and $\sigma\in \lbrace \sr, \s \rbrace$. 
\end{theorem}

Section \ref{sec:GreensWS} is devoted to the proof of the theorem for $y_0\in I_S\cup I_W$. An analogue of Theorem \ref{thm:XkregGSW} for $LX_k$ bounds on $\G_{k,\ep}^\pm(y,y_0,z)$ when $y_0\in I_M$ is obtained in Section \ref{sec:GreensM}.

\subsubsection{Local to global bounds}
The above theorem is most useful to describe in $I_3(y_0)$ the solution $\Phi_{k,\ep}^\pm$ to 
\begin{equation}\label{eq:RTGPhi}
\left(\D_k+\frac{\cJ(y_0)}{(y-y_0\pm i\ep_0)^2}\right){\Phi}_{k,\ep}^\pm(y,y_0) = F_{k,\ep}^\pm(y,y_0)
\end{equation}
with $\Phi_{k,\ep}^\pm(0,y_0) = \Phi_{k,\ep}^\pm(2,y_0) = 0$, since it is given by
\begin{align*}
{\Phi}_{k,\ep}^\pm = \int_0^2 \G_{k,\ep}^\pm(y,y_0,z) F_{k,\ep}^\pm(z,y_0) \d z
\end{align*}
However, for $y\in I_3^c(y_0)$, we cannot use Theorem \ref{thm:XkregGSW} and we need a different way to estimate $\Phi_{k,\ep}^\pm$. Nevertheless, we can upgrade local $L^2(I_2^c(y_0)\cap I_3(y_0))$ bounds on $\Phi_{k,\ep}^\pm$ to global $H_k^1(I_3^c(y_0))$ estimates on $\Phi_{k,\ep}^\pm$ thanks to the following inequality.

\begin{lemma}[Entanglement Inequality -- Lemma 7.1 in \cite{CZN25chan}]\label{lemma:entanglementRTG}
Let $\Phi_{k,\ep}^\pm$ be a solution to \eqref{eq:RTGPhi}, for some $F\in L^2(I_2^c(y_0))$. Then,
\begin{align*}
 \Vert \Phi_{k,\ep}^\pm \Vert_{H_k^1(I_3^c(y_0))} \lesssim \Vert \Phi_{k,\ep}^\pm \Vert_{L^2(I_2^c(y_0)\cap I_3(y_0))} + \frac{1}{k^2}\Vert F \Vert_{L^2(I_2^c(y_0))}
\end{align*}
uniformly for all $k\geq 1$, $y_0\in[0,2]$ and $\ep>0$. 
\end{lemma}

Similarly, while we may obtain pointwise bounds on $\Phi_{k,\ep}^\pm$ for $y\in I_3(y_0)$ through Theorem \ref{thm:XkregGSW} we also have the useful Sobolev-type inequality.

\begin{lemma}\label{lemma:LinfH1bound}
Let $f\in H^1_k(I_3(y_0))$. Then, $\Vert f \Vert_{L^\infty(I_3(y_0))} \lesssim k^{\frac12}\Vert f \Vert_{H_k^1(I_3(y_0))}$.
\end{lemma}

\begin{proof}
Let $z\in I_3(y_0)$. Then, $f(z) = f(y) + \int_y^z f'(s) \d s$, for all $y\in I_3(y_0)$. In particular, integrating over $y\in I_3(y_0)$, since $|I_3(y_0)|\approx k^{-1}$, we reach
\begin{align*}
|f(z)| k^{-1} &\lesssim \int_{I_3(y_0)} \left( |f(y)| + \left| \int_y^z f'(s) \d s \right| \right) \d y \\
&\lesssim k^{-\frac12}\Vert f\Vert_{L^2(I_3(y_0))} + \Vert f'\Vert_{L^2(I_3(y_0))}  \int_{I_3(y_0)} |z-y|^\frac12 \d y \\
&\lesssim k^{-\frac12} \Vert f \Vert_{H_k^1(I_3(y_0))}
\end{align*}
and the lemma follows.
\end{proof}

\subsubsection{Regularization of the source term}
With the Green's function $\G_{k,\ep}^\pm$ at hand we observe that 
\begin{align*}
 \int_0^2 \G_{k,\ep}^\pm(y,y_0,z) \left( \frac{\omega_k^0(z)}{v(z) - v(y_0) \pm i\ep} - \frac{\g \rho_k^0(y)}{(v(z) - v(y_0) \pm i\ep)^2} \right) \d z
\end{align*}
should be finite for $\psi_{k,\ep}^\pm$ to be well-defined. Thanks to Theorem \ref{thm:XkregGSW}, we know that $\G_{k,\ep}^\pm(y,y_0,z)$ has two components, locally behaving like $\xi^{\frac12+\gamma_0}$ and $\xi^{\frac12-\gamma_0}$, for $\xi = 2k(z-y_0\pm i\ep_0)$. Therefore, while one hope to be able to estimate the vorticity integral uniformly in $\ep>0$ because $\xi^{-\frac12\pm \gamma_0}\in L^1(I_3(y_0))$ uniformly for all $\ep>0$, one also notices that uniform boundedness of the density integral is much more subtle, since in general $\xi^{-\frac32\pm \gamma_0}\in L^1(I_3(y_0))$ for all $\ep>0$ but $\lim_{\ep\rightarrow 0}\xi^{-\frac32\pm \gamma_0}\not  \in L^1(I_3(y_0))$. 

In other words, the density term $\frac{\g \rho_k^0(y)}{(v(z) - v(y_0) \pm i\ep)^2}$ is a priori too singular to be handled by the Green's function. This is to be expected, since it is in fact critically singular with respect to the Reduced Taylor-Goldstein equation and thus cannot be treated in a perturbative fashion. Nevertheless, since $\rho_k^0(y) = \mathrm{P}(y)\varrho_k^0(y)$, we define the regularized generalized stream-function
\begin{align}\label{eq:defvarphi}
\varphi_{k,\ep}^\pm(y,y_0) :=  {\psi}_{k,\ep}^\pm(y,y_0) + \varrho_k^0(y) 
\end{align}
which is such that
\begin{equation}\label{eq:introTGvarphi}
\left(  \D_k-\frac{v''(y)}{v(y)-v(y_0)\pm i\ep}+\frac{\P(y)}{(v(y)-v(y_0)\pm i\ep)^2}\right) {\varphi}_{k,\ep}^\pm=\frac{w_k^0(y)}{v(y)-v(y_0)\pm i\ep} + q_k^0(y)
\end{equation}
with $\varphi_{k,\ep}^\pm(0,y_0) = \varphi_{k,\ep}^\pm(2,y_0) = 0$ where now
\begin{align*}
w_k^0(y) := \omega_k^0(y) - v''(y) \varrho_k^0(y), \quad q_k^0(y) := \Delta_k \varrho_k^0(y).
\end{align*}
Hence, we can now make sense of the action of $\G_{k,\ep}^\pm$ onto $q_k^0$ if $\varrho_k^0$ is sufficiently regular, a testament of the underlying motif in inviscid damping that decay costs regularity. That $\rho_k^0(y) = \mathrm{P}(y)\varrho_k^0(y)$ is to consider density perturbations to take place only in the stratified region. See \cite{Zhao23} for recent results on inviscid damping for the inhomogeneous Euler equations in the absence of buoyancy forces, namely $\g=0$. While we do not address it here, the study of density perturbations in the background homogeneous Euler region under the effect of gravity forces is a very interesting open problem.

\subsection{The limiting absorption principle}
With the help of the Green's function $\G_{k,\ep}^\pm$, the solution $\varphi_{k,\ep
}^\pm$ to \eqref{eq:introTGvarphi} is given by
\begin{equation}\label{eq:fixedpointvarphi}
\begin{split}
\varphi_{k,\ep}^\pm(y,y_0) + \int_0^2 \G_{k,\ep}^\pm(y,y_0,z) &\scE_{k,\ep}^\pm(z,y_0)\varphi_{k,\ep}^\pm(z,y_0) \d z  \\
&= \int_0^2 \G_{k,\ep}^\pm(y,y_0,z) \left( \frac{w_k^0(z)}{v(z) - v(y_0) \pm i\ep} + q_k^0(z) \right) \d z.
\end{split}
\end{equation}
To obtain estimates on the regularity structures of $\varphi_{k,\ep}^\pm$, we rely on the Limiting Absorption Principle, which consist on showing the existence of some $\kappa>0$ independent of $k\geq 1$, $y_0\in I_W\cup I_S$ and $\ep>0$ such that  
\begin{align}\label{eq:coercivitykappa}
\Vert \varphi_{k,\ep}^\pm \Vert_{Z_k} \leq \kappa \left \Vert \varphi_{k,\ep}^\pm(y,y_0) + \int_0^2 \G_{k,\ep}^\pm(y,y_0,z)  \scE_{k,\ep}^\pm(z,y_0)\varphi_{k,\ep}^\pm(z,y_0) \d z \right \Vert_{Z_k},
\end{align}
in order to use \eqref{eq:fixedpointvarphi} to deduce estimates for $\varphi_{k,\ep}^\pm$ in terms of the initial data $w_k^0$ and $q_k^0$. 

The existence of such $\kappa>0$ is shown in Proposition \ref{prop:LAPZk} in Section \ref{sec:LAPstrat} and involves a preliminary study on the mapping properties with respect to the $X_k$ and $Z_k$ spaces of the solution operators 
\begin{align}\label{eq:introdefsolopT}
T_{k,\ep}^\pm f(y) := \int_{0} ^{2} \G_{k,\ep}^\pm(y,y_0,z)\scE_{k,\ep}^\pm(y,y_0) f(z) \d z
\end{align}
and 
\begin{align}\label{eq:introdefsolopR}
(R_{m,k,\ep}^\pm f)(y,y_0) := \int_0^2\G_{k,\ep}^\pm(y,y_0,z)\frac{f(z)}{(v(z)-v(y_0)\pm i\ep)^m} \d z,
\end{align}
for $m\geq 0$, that are carried out in Section \ref{sec:solopstrat} and can be summarized to be
\begin{align}\label{eq:introXkestimatesTR}
\Vert T_{k,\ep}^\pm f \Vert_{X_k}\lesssim k^{-\frac12}\Vert f \Vert_{Z_k}, \quad \Vert R_{1,k,\ep}^\pm f \Vert_{X_k}\lesssim k^{-\frac12}\Vert f \Vert_{H_k^1}, \quad  \Vert R_{0,k,\ep}^\pm f \Vert_{X_k}\lesssim k^{-\frac32}\Vert f \Vert_{L^2}, \quad 
\end{align}
uniformly for all $0<\ep<\ep_*$ and all $y_0\in I_S\cup I_W$. 

The proof of Proposition \ref{prop:LAPZk} argues by contradiction, essentially showing that \eqref{eq:coercivitykappa} fails if we can find a sequence of parameters $k_j\geq 1$, $y_j\in I_S\cup I_W$ such that $y_j\rightarrow y_*\in \overline{I_S\cup I_W}$, $\ep_j\rightarrow 0^+$ and $f_j\in Z_{k_j}$, with $\Vert f_j\Vert_{Z_{k_j}}=1$ such that 
\begin{align}\label{eq:introfailureLAP}
\Vert f_j + T_{k_j,\ep_j}^\pm f_j(\cdot,y_j)\Vert_{Z_{k_j}} \rightarrow 0
\end{align}
as $j\rightarrow \infty$. The estimate $\Vert T_{k,\ep}^\pm f \Vert_{X_k}\lesssim k^{-\frac12}\Vert f \Vert_{Z_k}$ shows that $|k_j|\lesssim 1$ and hence $k_j\rightarrow k_*\in \mathbb{Z}\setminus \lbrace 0 \rbrace$ up to a subsequence, so that $k_j \equiv k_*$ for all $j$ sufficiently large. In what follows, we already consider $j$ large enough so that $k_j=k_*$. Moreover, for $g_j(y):=f_j(y) + T_{k_*,\ep_j}^\pm f_j(y,y_j)$ and $h_j(y) = f_j(y) - g_j(y)$ we have $\lim_{j\rightarrow \infty} \Vert h_j \Vert_{Z_{k_*}}=1$ and
\begin{align}\label{eq:introdefRj}
h_j(y) + T_{k_*,\ep_j}^\pm h_j(y,y_j) = -T_{k_*,\ep_j}^\pm g_j(y,y_j) 
\end{align}
with $h_j(0)=h_j(2)=0$. Equivalently, we obtain
\begin{equation}\label{eq:introTGhj}
\left( \D_{k_*} - \frac{v''(y)}{v(y) - v(y_j) \pm i\ep_j} + \frac{\P(y)}{(v(y) - v(y_j) \pm i\ep_j)^2}\right) h_j(y,y_j) = -\textsc{E}_{k_*,\ep_j}^\pm g_j(y,y_j).
\end{equation}
We then reach a contradiction with $\Vert h_j \Vert_{Z_{k_*}}\rightarrow 1$ if we can show that the solution $h_j$ to \eqref{eq:introTGhj} converges to zero since $\Vert g_j\Vert_{Z_{k_*}}\rightarrow 0$, namely there are no embedded eigenvalues. In order to show that $h_j$ does indeed vanish in $Z_{k_*}$, we need to solve the in-homogeneous Taylor-Goldstein equation \eqref{eq:introTGhj} for some $k_*>1$ fixed. To do so, we first study the \textit{homogeneous} Taylor-Goldstein equation in Section \ref{sec:homTG}, where we find a pair of homogeneous solutions of the form
\begin{equation}\label{eq:introdefphisigma}
\begin{split}
\phi_{\sr,k,\ep}^\pm(y,y_0) &:= (v(y) - v(y_0) \pm i\ep)^{\frac12+\gamma_0}\phi_{\sr,1,k,\ep}^\pm(y,y_0) , \\
\phi_{\s,k,\ep}^\pm(y,y_0) &:= (v(y) - v(y_0) \pm i\ep)^{\frac12-\gamma_0}\phi_{\s,1,k,\ep}^\pm(y,y_0)
\end{split}
\end{equation}
for some smooth functions $\phi_{\sr,1,k,\ep}^\pm$ and $\phi_{\s,1,k,\ep}^\pm$, which closely resemble \eqref{eq:defMkappagamma} that we later use to construct the in-homogeneous solution. These homogeneous solutions are also further used to study the spectral properties of the linearised operator $L_k$, see the next subsection for more details.

We observe here that since $y_*\in \overline{I_S\cup I_W}$ we may have $y_*=\vartheta_1$ or $y_*=\vartheta_2$, where there is change in the nature of the singularities of the Taylor-Goldstein operator, as $\P(y_j)\rightarrow0$ the quadratic singularity vanishes, but the regularity of $\phi_{\s,k,\ep}^\pm(y,y_0)$ may degenerate. Hence, in this \emph{fragile} regime in which $\P(y_j)\rightarrow 0$ a careful study of the in-homogeneous solution $h_j$ is carried out in Section \ref{sec:fragile}.

Likewise, we also note that $\P(y_j)\rightarrow \frac14$ is a critical limit, in that just like for \eqref{eq:defMkappagamma} the homogeneous solutions \eqref{eq:introdefphisigma} also lose their linear independence and we instead need to find homogeneous solutions answering to a local decomposition of the form \eqref{eq:decomlogvarphi}. The behaviour of the in-homogeneous solution $h_j$ to \eqref{eq:introTGhj} in this \emph{mild} regime where $\P(y_j)\rightarrow \frac14$  is considered in Section \ref{subsec:mild}.

We finally remark that one would actually like to have \eqref{eq:coercivitykappa} in the stronger $X_k$ space rather than $Z_k$. However, it is not straightforward to show that the in-homogeneous solution $h_j\in X_{k_*}$ and $\Vert h_j \Vert_{X_{k_*}}\rightarrow 0$ since, although similar to \eqref{eq:defMkappagamma}, the homogeneous solutions \eqref{eq:introdefphisigma} may not share the exact same properties near the critical layer $y\approx y_j$. Nevertheless, once $\varphi_{k,\ep}^\pm\in Z_k$ thanks to \eqref{eq:coercivitykappa} and \eqref{eq:introXkestimatesTR}, we can use \eqref{eq:fixedpointvarphi} and the regularizing properties of $T_{k,\ep}^\pm$ of \eqref{eq:introXkestimatesTR} to conclude that actually $\varphi_{k,\ep}^\pm\in X_k$.

\subsection{Spectral theory of the linearised operator}
A complete description of the spectrum of the linearised operator $L_k$ is needed to establish the integral reduction \eqref{eq:psirholim} and the validity of limiting absorption principle encoded in \eqref{eq:coercivitykappa}. A more precise version of Theorem \ref{thm:spectrumL} reads

\begin{theorem}\label{thm:spectrumlinop}
Suppose that the background shear flow $v(y)$ and stratified density $\overline{\rho}(y)$ satisfy {H$\P$, H$v$} and H1--H3. Then, $L_k$ does not have any eigenvalues (embedded or not) and $\sigma(L_k) = \sigma_{ess}(L_k)=[v(\vartheta_1), v(\vartheta_2)]$ .
\end{theorem}
Its proof is presented in Section \ref{sec:spectrum} and consists in showing that $(L_k - \lambda)$ is continuously invertible for all $\lambda\not \in [v(\vartheta_1), v(\vartheta_2)]$. To further understand the invertibility properties of $(L_k - \lambda)$, we now need to investigate, among others, the eigenvalue problem
\begin{align}
\left( \Delta_k - \frac{v''(y)}{v(y) - \lambda} + \frac{\P(y)}{(v(y) - \lambda)^2} \right)\psi &= 0\label{eq:homTGeigenvalue} \\
\psi|_{y=0,2} &=0,
\end{align}
with $\lambda\in \C$. To that purpose, we make use of the main properties of the homogeneous solutions \eqref{eq:introdefphisigma}  to the Taylor-Goldstein equation introduced in Section \ref{sec:homTG} in order to construct eigenfunction candidates. For instance, 
\begin{align*}
\varphi_k(y) = \phi_{\s,k,\ep}^\pm(2,y_0)\phi_{\sr,k,\ep}^\pm(y,y_0) - \phi_{\s,k,\ep}^\pm(y,y_0)\phi_{\sr,k,\ep}^\pm(2,y_0) 
\end{align*}
is an eigenfunction of \eqref{eq:homTGeigenvalue} of eigenvalue $\lambda = v(y_0) \mp i\ep$ if $\varphi_k(0)=0$. In Section \ref{sec:spectrum} we show that $\varphi_k(0)\neq 0$ if H2 holds and $\ep>0$ is small enough. Standard energy estimates and H1 further show that $\lambda\in \C$ with $\Re(\lambda)\not \in [v(\vartheta_1),v(\vartheta_2)]$ is in the resolvent of $L_k$. The last case of $\lambda\in \C$ with $\Im(\lambda)\neq 0$ and $\Re(\lambda)\in [v(\vartheta_1),v(\vartheta_2)]$ is treated with a connectedness argument that involves a limiting procedure towards the linearised Euler equations inspired by \cite{sinambela2025transition} and a detailed understanding of the homogeneous solutions \eqref{eq:introdefphisigma} in the fragile, weak, mild and strong regimes, which is carried out in Sections \ref{sec:homTG}, \ref{sec:spectrum} and \ref{sec:fragile}.

\subsection{Sobolev regularity of the generalized stream-function}
Once the coercive estimate \eqref{eq:coercivitykappa} is established, we now obtain $Z_k$ regularity of $\varphi_{k,\ep}^\pm$ in terms of regularity of $w_k^0$ and $q_k^0$ thanks to the mapping properties \eqref{eq:introXkestimatesTR} of the solutions operators $R_{m,k,\ep}^\pm$ for $m=0,1$. 

In view of \eqref{eq:psirholim}, we are also interested in regularity of $\partial_{y_0}\varphi_{k,\ep}^\pm$ in order to use oscillatory-phase arguments to obtain decay in time. Now,
\begin{align*}
&\left( \Delta_k - \frac{v''(y)}{v(y) - v(y_0)\pm i\ep} + \frac{\P(y)}{(v(y) - v(y_0) \pm i\ep)^2} \right) \partial_{y_0}\varphi_{k,\ep}^\pm \\
&\qquad\qquad= v'(y_0) \left( \frac{v''(y)}{(v(y) - v(y_0)\pm i\ep)^2} -2 \frac{\P(y)}{(v(y) - v(y_0)\pm i\ep)^3}\right) \varphi_{k,\ep}^\pm(y,y_0) + \frac{v'(y_0)w_k^0(y)}{(v(y) - v(y_0) \pm i\ep)^2}
\end{align*}
and we readily observe that $\varphi_{k,\ep}^\pm\in X_k$ is not enough to ensure that
\begin{align*}
\int_0^2 \G_{k,\ep}^\pm(y,y_0,z)  \left( \frac{v''(z)}{(v(z) - v(y_0)\pm i\ep)^2} -2 \frac{\P(z)}{(v(z) - v(y_0)\pm i\ep)^3}\right) \varphi_{k,\ep}^\pm(z,y_0) \d z
\end{align*}
is uniformly bounded, since $\G_{k,\ep}^\pm\varphi_{k,\ep}^\pm$ is not enough to compensate the cubic singularity. However, for 
\begin{align*}
\varphi_{1,k,\ep}^\pm(y,y_0) := \left( \partial_y + \partial_{y_0}\right)\varphi_{k,\ep}^\pm(y,y_0)
\end{align*}
we readily have that $\partial_{y_0}\varphi_{k,\ep}^\pm(y,y_0) = \varphi_{1,k,\ep}^\pm(y,y_0) - \partial_y\varphi_{k,\ep}^\pm$. Since we have a proper understanding of $\varphi_{k,\ep}^\pm$ in $Z_k$, we deduce from there the main regularity properties of $\partial_y\varphi_{k,\ep}^\pm$. More importantly, $\varphi_{1,k,\ep}^\pm$ is now the solution to 
\begin{equation}\label{eq:introTGvarphi1}
\begin{split}
\text{TG}_{k,\ep}^\pm\varphi_{1,k,\ep}^\pm &= \frac{v'''(y)}{v(y) - v(y_0) \pm i\ep}\varphi_{k,\ep}^\pm - \frac{\P'(y)}{(v(y) - v(y_0)\pm i\ep)^2}\varphi_{k,\ep}^\pm \\
&\quad + 2\P(y)\frac{v'(y) - v'(y_0)}{(v(y) - v(y_0) \pm i\ep)^3}\varphi_{k,\ep}^\pm - v''(y) \frac{v'(y)-v'(y_0)}{(v(y) - v(y_0)\pm i\ep)^2} \varphi_{k,\ep}^\pm \\
&\quad + \frac{\partial_y\omega_k^0(y)}{v(y) - v(y_0)\pm i\ep} - \omega_k^0(y)\frac{v'(y) - v'(y_0)}{(v(y) - v(y_0) \pm i\ep)^2} + \partial_y \varrho_k^0(y),
\end{split}
\end{equation}
with now
\begin{align*}
\varphi_{1,k,\ep}^\pm(0,y_0) = \partial_y\varphi_{k,\ep}^\pm(0,y_0), \quad  \varphi_{1,k,\ep}^\pm(0,y_0) = \partial_y\varphi_{k,\ep}^\pm(2,y_0).
\end{align*}

The limiting absorption principle then shows that the $X_k$ bounds of $\varphi_{1,k,\ep}^\pm$ are bounded by the $X_k$ bounds of the Green's function acting on the source term of \eqref{eq:introTGvarphi1}. A power law counting argument shows that this action may now be bounded: the regularity structures of $\G_{k,\ep}^\pm$ and $\varphi_{k,\ep}^\pm$, in fact that of the coefficients $( \G_{k,\ep}^\pm )_\sigma$ and  $( \varphi_{k,\ep}^\pm )_\sigma$, may compensate the most singular factors, which are morally at most quadratic singularities. The estimates on the regularity properties  of the source term are carried out in Section \ref{sec:highorderweakstrong} and \ref{sec:highordermild}. They are then used in Section \ref{sec:SobolevReg} to obtain the regularity of $\varphi_{1,k,\ep}^\pm$ and thus also of $\partial_{y_0}\varphi_{k,\ep}^\pm$. The same idea proves also useful to describe $\partial_{y_0}^2\varphi_{k,\ep}^\pm$, where now one must study $\varphi_{2,k,\ep}^\pm = (\partial_y + \partial_{y_0})\varphi_{1,k,\ep}^\pm$. 

We remark here that $\partial_y+\partial_{y_0}$ is sometimes referred to in the literature as the "good derivative" in that, while not fully commuting with the differential operator (the Taylor-Goldstein here and the Rayleigh operator in the Euler equations) it does not increase the singularities of the equation and one can hope to close the relevant estimates. We refer the interested reader to \cites{BCZV19, WZZ18, WZZ19, WZZKolmo20, IJ22, IIJ22} for other instances where the good derivative has been used.

\subsection{Linear inviscid damping of the perturbed density and velocity field}
The regularity structures for $\varphi_{k,\ep}^\pm, \varphi_{1,k,\ep}^\pm$ and $\varphi_{2,k,\ep}^\pm$ encoded in their $X_k$ estimates are next used to show that the perturbed velocity field and density given by \eqref{eq:psirholim} experience time-decay by means of oscillatory-phase methods: time decay is obtained at the expenses of integrability of $\partial_{y_0}$ derivatives. For the velocity field $v_k=(-\partial_y\psi_k,ik\psi_k)$ it is immediate to see from the $X_k$ definition that the regularities of $\varphi_{k,\ep}^\pm, \varphi_{1,k,\ep}^\pm$ and $\varphi_{2,k,\ep}^\pm$ depend on the local Richardson number $\cJ(y_0)$ and that, more importantly, $\partial_{y_0}^2\varphi_{k,\ep}^\pm$ and $\partial_{y,y_0}^2\varphi_{k,\ep}^\pm$ are in general not uniformly integrable as $\ep\rightarrow 0$. 

This has two related consequences when estimating the velocity field $v_k(t,y)$. First, the regularity of $\partial_{y_0}\varphi_{k,\ep}^\pm(y,y_0)$ and $\partial_y\varphi_{k,\ep}^\pm(y,y_0)$ is dictated by $\mu(y_0)$ near the critical layer. Precisely since $y_0\approx y$, we can actually assume that the regularity of $\partial_{y_0}\varphi(y,y_0)$ and $\partial_y\varphi_{k,\ep}^\pm(y,y_0)$ is in fact determined by $\mu(y)$ and thus essentially uniform near the critical layer: it is uniform in the strongly stratified regime, it experiences a logarithmic loss in the mild regime, becomes weaker in the weak regime and it eventually degenerates as the spectral parameter $y_0$ approaches the stratification boundaries $\vartheta_1$ and $\vartheta_2$.

Secondly, we need to consider the integrability of $\partial_{y_0}\varphi_{k,\ep}^\pm$ and $\partial_y\varphi_{k,\ep}^\pm$ both near the critical layer and away from it. In the critical layer they are locally integrable, and thus small (depending on $\mu(y)$) on a small subset of the critical layer. On the other hand, we can integrate by parts once more in the complement of the small subset to obtain an extra $t^{-1}$ decay, at the expenses of getting closer and closer to the critical layer, where $\partial_{y_0}^2\varphi_{k,\ep}^\pm$ and $\partial_{y,y_0}^2\varphi_{k,\ep}^\pm$ become singular (also depending on $\mu(y)$).  Balancing out the size of the subset gives the claimed decay estimates in Theorem \ref{thm:mainID} and is carried out in Section \ref{sec:IDestimates}.

The proof of Theorem \ref{thm:IDlocalweak} is also done in Section \ref{sec:IDestimates} and hinges on two observations. The first one is that for $y_0$ in the non-stratified regime, that is $\cJ(y_0)=0$, the Taylor-Goldstein operator becomes the Rayleigh operator for the linearised Euler equations, which is one power less singular than the Taylor-Goldstein operator. In particular, the solution $\varphi_{k,\ep}^\pm(y,y_0)$ to \eqref{eq:introTGvarphi} is now more regular, and we can argue as in \cite{Jia20, JiaGev20} to show that for $y\in[0,\vartheta_1]\cup [\vartheta_2,2]$, the non-stratified region, the velocity field $v_k(t,y)$ experiences Euler-type inviscid damping decay rates, which are comparatively faster than the Boussinesq rates.  Secondly, for $y$ close to $\vartheta_2$, say,  we are able to use the fundamental theorem of calculus to relate $\psi_k(t,y)$ with $\psi_k(t,\vartheta_2)$ and a slower time-decaying contribution which vanishes as $y\rightarrow \vartheta_2$. 

\subsection{Sub-linear growth of vorticity and gradient of density}
Once the decay estimates of Theorem~\ref{thm:mainID} are available, we can use them to obtain Theorem \ref{thm:growth}. More precisely, from \eqref{eq:psirholim} we now have
\begin{align}\label{eq:oscintomegak}
\omega_k(t,y) = \frac{1}{2 \pi i}\lim_{\ep\rightarrow 0}\int_0^2 e^{-ikv(y_0)t}\left( \Delta_k\varphi_{k,\ep}^-(y,y_0) - \Delta_k \varphi_{k,\ep}^+(y,y_0) \right) v'(y_0) \d y_0
\end{align} 
and
\begin{align}\label{eq:oscintpyrhok}
\partial_y\rho_k(t,y) = \frac{1}{2 \pi i}\lim_{\ep\rightarrow 0}\int_0^2 e^{-ikv(y_0)t}\partial_y\left( \frac{\varphi_{k,\ep}^-(y,y_0)}{v(y) - v(y_0) - i\ep} - \frac{\varphi_{k,\ep}^+(y,y_0)}{v(y) - v(y_0) + i\ep} \right) v'(y_0) \d y_0.
\end{align} 
According to the regularity properties of $\varphi_{k,\ep}^\pm$ and \eqref{eq:introTGvarphi}, we observe that $\Delta_k\varphi_{k,\ep}^\pm$ and $\partial_y\left( \frac{\varphi_{k,\ep}^-(y,y_0)}{v(y) - v(y_0) - i\ep} \right)$ are no longer uniformly integrable as $\ep\rightarrow 0$. To overcome this singularity, we now integrate backwards, noting here that $\partial_y = \partial_y + \partial_{y_0} - \partial_{y_0}$ and $\partial_y+\partial_{y_0}$ is in general more regular. This backwards integration produces powers of $t$ when $\partial_{y_0}$ acts on the oscillatory factors and, after suitable optimization on the size of the small set, is the main reason behind the sub-linear growth of $\omega(t,x,y)$ and $\partial_y\rho(t,x,y)$. Just as in Theorem \ref{thm:mainID}, the growth estimates obtained this way degenerate as $y\rightarrow\vartheta_1$ or $y\rightarrow \vartheta_2$. Nonetheless, we can use again the fundamental theorem of calculus and the behaviour of $\omega(t,x,y)$ and $\partial_y\rho(t,x,y)$ in the non-stratified regions to locally improve the growth bounds, which culminate in Theorem \ref{thm:growthlocalweak} after suitably regularizing the governing equations for $\omega_k(t,y)$ and $\partial_y\rho_k(t,y)$. The proofs of Theorem \ref{thm:growth} and \ref{thm:growthlocalweak} are carried out in Section \ref{sec:growth}.

We remark here that Theorem \ref{thm:growth} can also be proved integrating the linearised equations \eqref{eq:linEBomegarho} and using the decay estimates for $\psi_k(t,y)$, $\partial_y\psi_k(t,y)$ and $\rho_k(t,y)$ obtained in Section \ref{sec:IDestimates}. While we opted for studying \eqref{eq:oscintomegak} and \eqref{eq:oscintpyrhok} above since they are  later needed to improve the growth bounds to Theorem \ref{thm:growthlocalweak}, we give further details on how to use \eqref{eq:linEBomegarho} to prove the growth bounds in Remark \ref{rmk:growthlineq} below.

\section{Green's function of the RTG operator}\label{sec:Greens}
In this section we give the explicit formulas for $\G_{k,\ep}^\pm(y,y_0,z)$ that will be used in the manuscript. Central to the study of the Green's function are the Whittaker functions \cite{Whittaker03}, a class of  hyper-geometric functions that satisfy equations of the form
\begin{align*}
\partial_z^2 M_{\kappa,\gamma} + \left( - \frac14 + \frac{\kappa}{\zeta} + \frac{\tfrac14 - \gamma^2}{\zeta^2} \right) M_{\kappa,\gamma} = 0, \quad \zeta\in \C
\end{align*}
for $\kappa,\gamma\in \C$. Their regularity properties, reported in Appendix \ref{app:Whittaker}, are fundamental to explain the long-time behaviour of solutions to the linearised system and the precise form of the time-decay rates presented in Theorem \ref{thm:mainID}. In this manuscript we use the Whittaker functions with $\kappa=0$ and $\gamma=\mu+i\nu\in \C$ with $\mu\nu=0$. For $\zeta\in C$, we define $M_\sr(\xi) := M_{0, \gamma}\left( 2k\xi \right)$ and $M_\s(\xi) := M_{0,-\gamma}\left( 2k\xi \right)$, which solve the (re-scaled) Whittaker equation
\begin{align}\label{eq:WhittakerMW}
\partial_z^2 M_{\sigma} + \left( - \frac14 +  \frac{\tfrac14 - \gamma^2}{4k^2\zeta^2} \right) M_{\sigma} = 0, \quad\sigma\in \lbrace \sr, \s \rbrace.
\end{align}
When $\gamma=0$ then $M_\sr$ and $M_\s$ are no longer linearly independent. To overcome this difficulty, we define $W_\sr(\zeta):= W_{0,\gamma}(\zeta)$ to be the unique solution to \eqref{eq:WhittakerMW} such that
\begin{align*}
W_{0,\gamma}(\zeta) = e^{-\frac12\zeta}\zeta^{\frac12+\gamma}U(\tfrac12+\gamma, 1+2\gamma,\zeta),
\end{align*}
where $U(a,b,\zeta)$ denotes the Kummer function uniquely determined by the property that $U(a,b,\zeta) \approx \zeta^{-a}$ for $\zeta\rightarrow\infty$. While we shall give more precise details on these hyper-geometric functions in Appendix \ref{app:Whittaker}, we now state our main formulae for the Green's function of \eqref{eq:defRTGoperator}.

\begin{proposition}\label{prop: def Green's Function}
Let  $\ep\in(0,1)$. The Green's function $\G_{m,\ep}^\pm$ of  $\textsc{TG}_{m,\ep}^\pm$
is given by
\begin{equation}\label{eq:Greens TG}
\G_{k,\ep}^\pm(y,y_0,z)=
\begin{cases}
\G_{\su,k,\ep}^\pm(y,y_0,z), \quad &0\leq z\leq y\leq 2,\\
\G_{\sl,k,\ep}^\pm(y,y_0,z), \quad &0\leq y\leq z\leq 2,
\end{cases}
\end{equation}
with
\begin{align}\label{ed:defGuGl}
\G_{\su,k,\ep}^\pm(y,y_0,z) &= \frac{1}{\W_{k,\ep}^\pm(y_0)}\phi_{\su,k,\ep}^\pm(y,y_0)\phi_{\sl,k,\ep}^\pm(z,y_0), \\ \G_{\sl,k,\ep}^\pm(y,y_0,z) &= \frac{1}{\W_{k,\ep}^\pm(y_0)}\phi_{\sl,k,\ep}^\pm(y,y_0)\phi_{\su,k,\ep}^\pm(z,y_0)
\end{align}
where  $\phi_{\su,k,\ep}^\pm(\cdot,y_0)$ and $\phi_{\sl,k,\ep}^\pm(\cdot,y_0)$ are two homogeneous solutions to 
\begin{equation}\label{eq:homTGoperator}
\left( \p_y^2 - k^2 + \frac{\cJ(y_0)}{(y-y_0 \pm i\ep_0)^2}\right) \phi_{k,\ep}^\pm(y,y_0) = 0 
\end{equation}
 such that  
$\phi_{\sl,k,\ep}^\pm(0,y_0)=0$ and $\phi_{\su,k,\ep}^\pm(2,y_0)=0$, respectively, for all $y_0\in [0,2]$. 
\begin{itemize}
\item For $y_0\in I_S\cup I_W$, they are given by
\begin{equation}\label{eq:homoupss}
\phi_{\su,k,\ep}^\pm(y,y_0):=M_\sr(2 - y_0\pm i\ep_0)M_\s(y- y_0\pm i\ep_0) - M_\s(2 - y_0\pm i\ep_0)M_\sr(y-y_0\pm i\ep_0)
\end{equation}
and
\begin{equation}\label{eq:homolowss}
\phi_{\sl,k,\ep}^\pm(y,y_0):=M_\sr(- y_0\pm i\ep_0)M_\s(y- y_0\pm i\ep_0) - M_\s( - y_0\pm i\ep_0)M_\sr(y-y_0\pm i\ep_0)
\end{equation}
with Wronskian 
\begin{equation}\label{eq:def Wronskianss}
\W_{k,\ep}^\pm(y_0):= -4k\gamma_0\Big( M_\sr( - y_0\pm i\ep_0)M_\s(2- y_0\pm i\ep_0) - M_\s( - y_0\pm i\ep_0)M_\sr(2-y_0\pm i\ep_0) \Big).
\end{equation} 
\item For $y_0\in I_M$ in the mildly stratified region, they are given by 
\begin{equation}\label{eq:homoupms}
\phi_{\su,k,\ep}^\pm(y,y_0):=M_\sr(2 - y_0\pm i\ep_0)W_\sr(y- y_0\pm i\ep_0) - W_\sr(2 - y_0\pm i\ep_0)M_\sr(y-y_0\pm i\ep_0)
\end{equation}
and
\begin{equation}\label{eq:homolowms}
\phi_{\sl,k,\ep}^\pm(y,y_0):=M_\sr(- y_0\pm i\ep_0)W_\sr(y- y_0\pm i\ep_0) - W_\sr( - y_0\pm i\ep_0)M_\sr(y-y_0\pm i\ep_0)
\end{equation}
with Wronskian 
\begin{equation}\label{eq:def Wronskianms}
\W_{k,\ep}^\pm(y_0):= -2k\frac{\Gamma(1+2\gamma)}{\Gamma \left( \frac12 + \gamma \right)} \Big( M_\sr( - y_0\pm i\ep_0)W_\sr(2-y_0\pm i\ep_0) - W_\sr( - y_0\pm i\ep_0)M_\sr(2-y_0\pm i\ep_0) \Big).
\end{equation} 
\end{itemize}
Furthermore, we have the relation $\G_{k,\ep}^+(y,y_0,z)=\overline{\G_{k,\ep}^-(y,y_0,z)}$, for all $y,y_0,z\in[0,2]$.
\end{proposition}

The proof of the proposition is rather standard: The Green's functions is constructed by first solving the boundary value problems on each side of the interval using the homogeneous solutions and then fixing the coefficients so that $\G_{k,\ep}^\pm(y,y_0,z)$ is continuous at $y=z$ and  $\partial_y\G_{k,\ep}^\pm(y,y_0,z)$ has a jump discontinuity at $y=z$. More details can be found in  \cite{CZN25chan}.

\section{Green's function for weak and strong stratifications}\label{sec:GreensWS}
In this section we present the main regularity structures of the Green's function for $y_0\in I_S\cup I_W$ and we prove Theorem \ref{thm:XkregGSW} through Propositions \ref{prop:sobolevregdecomG}, \ref{prop:sobolevregpartialdecomG} and Corollary \ref{cor:sobolevregpartialdecomG} below.

\begin{proposition}\label{prop:sobolevregdecomG}
Let $k\geq 1$. Then, there exists $\ep_*>0$ small enough such for all $y_0\in I_S\cup I_W$ the Green's function admits the decomposition
\begin{align*}
\G_{k,\ep}^\pm(y,y_0,z) = \left( \G_{k,\ep}^\pm\right)_\sr(y,y_0,z) \eta^{\frac12 + \gamma_0} + \left( \G_{k,\ep}^\pm\right)_\s(y,y_0,z) \eta^{\frac12 - \gamma_0}
\end{align*}
where $\eta=2k(y-y_0\pm i\ep_0)$ and
\begin{align*}
\sup_{y_0\in I_S\cup I_W, \, y\in I_3(y_0)} \left\Vert \left( \G_{k,\ep}^\pm\right)_\sigma(y,y_0,\cdot) \right\Vert_{X_k^1} \lesssim k^{-1}, \quad \sup_{y_0\in I_S\cup I_W, \, y\in I_3(y_0)} \left\Vert \left( \G_{k,\ep}^\pm\right)_\sigma(y,y_0,\cdot) \right\Vert_{H_k^1(I_3^c(y_0))} \lesssim k^{-\frac32},
\end{align*}
for $\sigma\in \lbrace \sr, \s \rbrace$ and all $0<\ep< \ep_*$.
\end{proposition}

\begin{proof}
We argue for $\left( \G_{\su,k,\ep}^\pm\right)_\sr$ since the proof for $\left( \G_{\su, k,\ep}^\pm\right)_\s$ and for $\G_{\sl,k,\ep}^\pm$ is identical. For $z< y$, an inspection of the Green's function and Lemma \ref{lemma:asymptoticexpansionM} show that
\begin{align*}
\left( \G_{\su,k,\ep}^\pm\right)_\sr(y,y_0,z) &= -\frac{1}{\W_{k,\ep}^\pm(y_0)}\phi_{\sl,k,\ep}^\pm(z,y_0)M_\s(1 - y_0 \pm i\ep_0) \mathcal{E}_{0,\gamma_0}(\eta)
\end{align*}
Note that for $z\in I_3(y_0)$ we also have
\begin{align*}
\phi_{\sl,k,\ep}^\pm(z,y_0) &= M_\sr( - y_0\pm i\ep_0)\mathcal{E}_{0,-\gamma_0}(\xi)  \xi^{\frac12-\gamma_0} \\
&\quad- M_\s( - y_0\pm i\ep_0)\mathcal{E}_{0,\gamma_0}(\xi) \xi^{\frac12+\gamma_0},
\end{align*}
with $\xi=2k(z-y_0\pm i\ep_0)$. With this we decompose 
\begin{align*}
\left( \G_{\su,k,\ep}^\pm\right)_\sr(y,y_0,z) = \left( \G_{\su, k,\ep}^\pm\right)_{\sr,\sr}(y,y_0,z)\xi^{\frac12+\gamma_0} + \left( \G_{\su, k,\ep}^\pm\right)_{\sr,\s}(y,y_0,z)\xi^{\frac12-\gamma_0},
\end{align*}
Owing to \cite[Proposition 4.4]{CZN25chan} when $y_0\in I_S$ and to Lemma \ref{lemma:Wronskianlowerboundweakstrat} when $y_0\in I_W$, there holds
\begin{align*}
\sup_{y_0\in I_S\cup I_W } \left| \frac{M_\tau( - y_0 \pm i \ep)}{\W_{k,\ep}^\pm (y_0)} M_\sigma(2 - y_0 \pm i \ep)\right| \lesssim \frac{1}{|k|}.
\end{align*}
We remark here that for $y_0\in I_S$, we have that $\nu_0$ is uniformly bounded from above and from below away from 0, and an inspection of the proof of Proposition 4.4 in \cite{CZN25chan} shows that the bounds are uniform in $I_S$.  Together with $\Vert \mathcal{E}_{0,\pm \gamma_0}\Vert_{L^\infty(B_{10}(0))}\lesssim 1$,  we see that
\begin{align*}
\sup_{y_0\in I_S\cup I_W} \sup_{y\in I_3(y_0)} \sup_{z\in I_3(y_0)} \left| \left( \G_{\su, k,\ep}^\pm\right)_{\sigma,\tau}(y,y_0,z) \right| \lesssim k^{-1},
\end{align*}
for all $\sigma,\tau \in \lbrace \sr, \s \rbrace$. To prove the estimates for $\partial_z\left( \G_{\su, k,\ep}^\pm\right)_\sr(y,y_0,z)$, we note that 
\begin{align*}
\partial_z\phi_{\sl,k,\ep}^\pm(z,y_0) = 2k \left( M_\sr( - y_0\pm i\ep_0)M_\s'(z- y_0\pm i\ep_0) - M_\s( - y_0\pm i\ep_0)M_\sr'(z-y_0\pm i\ep_0) \right)
\end{align*}
and since 
\begin{align*}
M'_\sr(z-y_0\pm i\ep_0) &= 2k \mathcal{E}_{1,\gamma_0}(\xi) \xi^{-\frac12+\gamma_0}, 
\\
 M'_\s(z-y_0\pm i\ep_0) &= 2k \mathcal{E}_{1,-\gamma_0}(\xi) \xi^{-\frac12-\gamma_0},
\end{align*}
with $\Vert \mathcal{E}_{1,\pm \gamma_0}\Vert_{L^\infty(B_{10}(0))}\lesssim 1$, we conclude that 
\begin{align*}
\sup_{y_0\in I_S}\sup_{y\in I_3(y_0)} \left\Vert \left( \G_{\su, k,\ep}^\pm\right)_{\sigma,\tau}(y,y_0,\cdot) \right\Vert_{W^{1,\infty}_k(I_3(y_0))}\lesssim 1.
\end{align*}
Next, we address the $H^1_k(I_3^c(y_0))$ bounds of $\left( \G_{k,\ep}^\pm\right)_{\sigma}(y,y_0,\cdot)$. To do so, for $z < y$ must obtain $H^1_k$ bounds on $\phi_{\sl,k,\ep}^\pm(z,y_0)$. 
Using the entanglement inequality, if $y_0 < y$, we have that 
\begin{align*}
\Vert \p_z\phi_{\sl,k,\ep}^\pm \Vert_{L^2_z(I_3^c(y_0)\cap (0, y_0))}^2 + k^2 \Vert \phi_{\sl,k,\ep}^\pm \Vert_{L^2_z(I_3^c(y_0)\cap (0, y_0))}^2 \lesssim k^2 \Vert \phi_{\sl,k,\ep}^\pm \Vert_{L^2_z(I_2^c(y_0)\cap I_3(y_0)\cap (0, y_0))}^2,
\end{align*}
namely
\begin{equation}\label{eq:H1kphil}
\Vert \phi_{\sl,k,\ep}^\pm \Vert_{H^1_k(I_3^c(y_0)\cap (0, y_0))} \lesssim \Vert \phi_{\sl,k,\ep}^\pm \Vert_{L^2_z(I_2^c(y_0)\cap I_3(y_0)\cap (0, y_0))} \lesssim \frac{|M_\sr(-y_0 \pm i\ep_0)| + |M_\s( - y_0\pm i\ep_0)|}{k^\frac12},
\end{equation}
where we have used the local bounds of $\phi_{l,k,\ep}^\pm(z,y_0)$ when $z\in I_2^c(y_0)\cap I_3(y_0)$. Note further than since $y\in I_3(y_0)$, then $(y_0,y)\cap I_3^c(y_0)=\emptyset$ and thus $\Vert \phi_{\sl,k,\ep}^\pm \Vert_{H^1_k(I_3^c(y_0)\cap (0, y))} = \Vert \phi_{\sl,k,\ep}^\pm \Vert_{H^1_k(I_3^c(y_0)\cap (0, y_0))}$. On the other hand, for $y\leq y_0$, we observe that $(0,y) \subseteq (0,y_0)$ and we use the $H^1_k(I_3^c(y_0)\cap (0, y_0))$ bounds of $\phi_{\sl,k,\ep}^\pm$ that we already have.

Summing up, for $y_0\in I_S \cup I_W$, $y\in I_3(y_0)$ and $z < y$, appealing once again to Proposition 4.2 in \cite{CZN25chan} we reach
\begin{align*}
 \left\Vert \left( \G_{\su,k,\ep}^\pm\right)_{\sr}(y,y_0,\cdot) \right\Vert_{H^1_k(I_3^c(y_0)\cap(0, y))} &\lesssim \frac{|M_\sr(-y_0 \pm i\ep_0)| + |M_\s( - y_0\pm i\ep_0)|}{k^\frac12\W_{k,\ep}^\pm(y_0)}| M_\sr(2 - y_0 \pm i\ep_0) | \\
 &\lesssim k^{-\frac32}
\end{align*}
For $y < z$, we consider $\left( \G_{\sl,k,\ep}^\pm\right)_{\sr}$, one now needs to study $\phi_{\su,k,\ep}^\pm(z,y_0)$. Nevertheless, the bounds and strategies to prove them remain the same as for the case $z<y$, we omit the details. With this, we obtain
\begin{align*}
\sup_{y_0\in I_S\cup I_W}\sup_{y\in I_3(y_0)} \left\Vert \left( \G_{k,\ep}^\pm\right)_{\sr}(y,y_0,\cdot) \right\Vert_{H^1_k(I_3^c(y_0))} \lesssim k^{-\frac32}
\end{align*}
and the proof is finished.
\end{proof}

We need more refined properties for the $\partial_y$ derivatives of the regular and singular factors of the decomposition.

\begin{proposition}\label{prop:sobolevregpartialdecomG}
Let $k\geq 1$, $y_0\in I_S\cup I_W$ and $\ep\in (0,\ep_*)$. For $z,y\in I_3(y_0)$, and the decomposition
\begin{align*}
\G_{k,\ep}^\pm(y,y_0,z) = \left( \G_{k,\ep}^\pm\right)_\sr(y,y_0,z) \xi^{\frac12 + \gamma_0} + \left( \G_{k,\ep}^\pm\right)_\s(y,y_0,z) \xi^{\frac12 - \gamma_0},
\end{align*}
where $\xi=2k(z-y_0\pm i\ep_0)$, there holds, 
\begin{align*}
\left\Vert \partial_z\left( \G_{k,\ep}^\pm\right)_\sigma(\cdot,y_0,z) \right\Vert_{X_k^1} + k^\frac12 \left\Vert \partial_z\left( \G_{k,\ep}^\pm \right)_\sigma(\cdot,y_0,z) \right \Vert_{L^2(I_3(y_0))}\lesssim |\xi|
\end{align*}
uniformly for all $y_0\in I_S\cup I_W$ and $z\in I_3(y_0)$, for $\sigma\in \lbrace \sr, \s \rbrace$. Moreover, for $f\in C(I_3(y_0))$, $y\in I_3(y_0)$ and $\eta=2k(y-y_0\pm i\ep_0)$, we have
\begin{align*}
\partial_y \int_{I_3(y_0)} \left( \partial_z \left( \G_{k,\ep}^\pm \right)_\sigma  \right)_\tau (y,y_0,z) f(z) \d z &= f(y) \left( \left( \partial_z \left( \G_{\su,k,\ep}^\pm \right)_\sigma \right)_\tau - \left( \partial_z \left( \G_{\sl,k,\ep}^\pm \right)_\sigma \right)_\tau \right)(y,y_0,y)  \\
&\quad + \int_{(0,y)\cap I_3(y_0)} \partial_y \left( \partial_z \left( \G_{\su,k,\ep}^\pm \right)_\sigma  \right)_\tau (y,y_0,z) f(z) \d z \\
&\quad + \int_{(y,2)\cap I_3(y_0)} \partial_y \left( \partial_z \left( \G_{\sl,k,\ep}^\pm \right)_\sigma  \right)_\tau (y,y_0,z) f(z) \d z
\end{align*}
with
\begin{align*}
\left\vert \partial_y \left( \partial_z \left( \G_{\su,k,\ep}^\pm \right)_\sigma  \right)_\tau(y,y_0,z) \right\vert + \left\vert \partial_y \left( \partial_z \left( \G_{\sl,k,\ep}^\pm \right)_\sigma  \right)_\tau(y,y_0,z) \right\vert \lesssim k |\eta||\xi|
\end{align*}
uniformly for all $y_0\in I_S\cup I_W$, for $\sigma,\,\tau\in \lbrace \sr, \s \rbrace$.
\end{proposition}

\begin{proof}
We argue for $\sigma=\sr$. Assume that $z < y$ so that $\G_{k,\ep}^\pm(y,y_0,z) = \G_{\su,k,\ep}^\pm(y,y_0,z)$. Since
\begin{align*}
\left( \G_{\su,k,\ep}^\pm\right)_\sr(y,y_0,z) &= -\frac{1}{\W_{k,\ep}^\pm(y_0)}\phi_{\su,k,\ep}^\pm(y,y_0)M_\s( - y_0 \pm i\ep_0) \mathcal{E}_{0,\gamma_0}(\xi)
\end{align*}
and $\partial_z\xi=2k$, there holds 
\begin{align*}
\partial_z \left( \G_{k,\ep}^\pm\right)_\sr(y,y_0,z) = -2k\frac{1}{\W_{k,\ep}^\pm(y_0)}\phi_{\su,k,\ep}^\pm(y,y_0)M_\s(2 - y_0 \pm i\ep_0) \mathcal{E}_{0,\gamma_0}'(\xi)
\end{align*}
The $X_k^0$ bound on $\partial_z \left( \G_{k,\ep}^\pm\right)_\sr(y,y_0,\cdot)$ follow easily as in the proof of Proposition \ref{prop:sobolevregdecomG} once we observe, see Lemma \ref{lemma:asymptoticexpansionM}, that $\left| \mathcal{E}_{0,\gamma_0}'(\xi) \right| \lesssim |\xi|$. 
Regarding the $X_k^1$ bound, note that
\begin{align*}
\left( \partial_z \left( \G_{k,\ep}^\pm \right)_\sigma \right)_\tau = \begin{cases}
\left( \partial_z \left( \G_{\su,k,\ep}^\pm \right)_\sigma \right)_\tau & \quad z<y, \\
\left( \partial_z \left( \G_{\sl,k,\ep}^\pm \right)_\sigma \right)_\tau & \quad y<z, 
\end{cases} 
\end{align*}
for all $\sigma,\,\tau\in \lbrace \sr, \s \rbrace$. Moreover, an inspection of Proposition \ref{prop: def Green's Function} shows that
\begin{align*}
\partial_y\left(\partial_z \left( \G_{\su,k,\ep}^\pm\right)_\sr\right)_{\s}(y,y_0,z) &= -4k^2\frac{1}{\W_{k,\ep}^\pm(y_0)}M_\s(2 - y_0 \pm i\ep_0) \mathcal{E}_{0,\gamma_0}'(\xi)M_\sr( - y_0\pm i\ep_0)\mathcal{E}_{0,-\gamma_0}'(\eta)
\end{align*}
since $\partial_y\eta = 2k$. Therefore, appealing to Proposition \ref{lemma:asymptoticexpansionM}, we obtain $|\E_{0,\gamma}'(\xi)|\lesssim |\xi|$ while $|\E_{0,\gamma}'(\eta)|\lesssim |\eta|$. With this and the lower bounds on the Wronskian the local estimate for $\partial_y\left(\partial_z \left( \G_{\su,k,\ep}^\pm\right)_\sr\right)_{\s}$ follows. The bounds for the other combinations of $\sigma,\tau\in \lbrace \sr, \s \rbrace$ are the same, we omit the details.

We next address the $L^2(I_3(y_0))$ estimate. We argue for $\sigma = \sr$ and for $z<y$, the other choices follow similarly. Recall form Lemma \ref{lemma:asymptoticexpansionM} that
\begin{align*}
\partial_z \left( \G_{k,\ep}^\pm\right)_\sr(y,y_0,z) = -2k\frac{1}{\W_{k,\ep}^\pm(y_0)}\phi_{\su,k,\ep}^\pm(y,y_0)M_\s(2 - y_0 \pm i\ep_0) \mathcal{E}_{0,\gamma_0}'(\xi).
\end{align*}
with $\left| \mathcal{E}_{0,\gamma_0}'(\xi) \right| \lesssim |\xi|$, for $|\xi|\leq 4\beta$.
The entanglement inequality in $(y_0,2)$ gives
\begin{align*}
\Vert \p_y\phi_{\su,k,\ep}^\pm \Vert_{L^2_z(I_3^c(y_0)\cap (y_0,2))}^2 + k^2 \Vert \phi_{\su,k,\ep}^\pm \Vert_{L^2_z(I_3^c(y_0)\cap (y_0,2))}^2 \lesssim k^2 \Vert \phi_{\su,k,\ep}^\pm \Vert_{L^2(I_2^c(y_0)\cap I_3(y_0)\cap (y_0,2))}^2,
\end{align*}
namely
\begin{align*}
\Vert \phi_{\su,k,\ep}^\pm \Vert_{H^1_k(I_3^c(y_0)\cap (y_0,2))} \lesssim \Vert \phi_{\su,k,\ep}^\pm \Vert_{L^2_z(I_2^c(y_0)\cap I_3(y_0)\cap (y_0,2))} \lesssim \frac{|M_\sr(-y_0 \pm i\ep_0)| + |M_\s( - y_0\pm i\ep_0)|}{k^\frac12}.
\end{align*}
On the other hand, since $z\in I_3(y_0)$ then $(z,y_0)\subset I_3(y_0)$ and thus $I_3^c(y_0)\cap (z,y_0) = \emptyset$. Appealing once again to the Wrosnkian bounds, confer Proposition 4.2 in \cite{CZN25chan}, we reach
the desired bounds and the proof is finished.
\end{proof}

Exploiting the symmetry of the Green's function, we obtain the following.
\begin{corollary}\label{cor:sobolevregpartialdecomG}
Let $k\geq 1$, $y_0\in I_S\cup I_W$ and $\ep\in (0,\ep_*)$. For $z,y\in I_3(y_0)$, and the decomposition
\begin{align*}
\G_{k,\ep}^\pm(y,y_0,z) = \left( \G_{k,\ep}^\pm\right)_\sr(y,y_0,z) \eta^{\frac12 + \gamma_0} + \left( \G_{k,\ep}^\pm\right)_\s(y,y_0,z) \eta^{\frac12 - \gamma_0},
\end{align*}
where $\eta=2k(z-y_0\pm i\ep_0)$, there holds, 
\begin{align*}
\left\Vert \partial_y\left( \G_{k,\ep}^\pm\right)_\sigma(y,y_0,\cdot) \right\Vert_{X_k^1} + k^\frac12 \left\Vert \partial_y\left( \G_{k,\ep}^\pm \right)_\sigma(y,y_0,\cdot) \right \Vert_{L^2(I_3(y_0))}\lesssim |\eta|
\end{align*}
for $\sigma\in \lbrace \sr, \s \rbrace$.
\end{corollary}

\section{Green's function for mild stratifications}\label{sec:GreensM}
Next, we present the main regularity structures of the Green's function associated to the mild stratification regime, for which a logarithmic correction needs to be incorporated. We assume here that $\tilde\delta >0$ is small enough so that for all $y_0\in I_M$ we have $|\gamma(y_0)| \leq \gamma_*$, where $\gamma_*$ is given by Lemma \ref{lemma:Wronskianlowerboundmildstrat},  \ref{lemma:smallnulargearg} and \ref{lemma:smallnuboundedarg}.

\begin{proposition}\label{prop:logsobolevregdecomG}
Let $k\geq 1$. There exists $\ep_*>0$ small enough such that for all $y_0\in I_M$  the Green's function admits the decomposition
\begin{align*}
\G_{k,\ep}^\pm(y,y_0,z) &= \left( \G_{k,\ep}^\pm\right)_\sr(y,y_0,z) \eta^{\frac12 + \gamma_0} + \left( \G_{k,\ep}^\pm\right)_\s(y,y_0,z) \eta^{\frac12 -\gamma_0}\log(\eta) \mathcal{Q}_\gamma(\eta),
\end{align*}
where $\eta=2k(y-y_0\pm i\ep_0)$ and $\Q_{\gamma}(\eta) = \int_0^1 e^{2\gamma s \log(\eta)} \d s$. Furthermore,
\begin{align*}
\sup_{y_0\in I_M, \, y\in I_3(y_0)} \left\Vert \left( \G_{k,\ep}^\pm\right)_\sigma(y,y_0,\cdot) \right\Vert_{LX_k^1} \lesssim k^{-1}, \quad \sup_{y_0\in I_M, \, y\in I_3(y_0)} \left\Vert \left( \G_{k,\ep}^\pm\right)_\sigma(y,y_0,\cdot) \right\Vert_{H_k^1(I_3^c(y_0))} \lesssim k^{-\frac32}
\end{align*}
for $\sigma\in \lbrace \sr, \s \rbrace$ and all $0<\ep< \ep_*$.
\end{proposition}

\begin{proof}
The arguments are similar to those for Proposition \ref{prop:sobolevregdecomG}. The local bounds are obtained decomposing each homogeneous solution that constitutes the Green's function into the regularity structures of $M_\sr$ and $W_\sr$, see Lemma \ref{lemma:asymptoticexpansionM} and Lemma \ref{lemma:asymptoticexpansionW}, respectively. For any of the four possible combinations of pre-factors, namely
\begin{align*}
M_\sr(-y_0\pm i\ep_0)M_\sr(2-y_0\pm i\ep_0), \quad M_\sr(-y_0\pm i\ep_0)W_\sr(2-y_0\pm i\ep_0), \quad \\
W_\sr(-y_0\pm i\ep_0) M_\sr(2-y_0\pm i\ep_0), \quad W_\sr(-y_0\pm i\ep_0)W_\sr(2-y_0\pm i\ep_0),
\end{align*}
its quotient with the Wronskian $\W_{k,\ep}^\pm(y_0)$ is uniformly bounded by $k^{-1}$, confer Lemma \ref{lemma:Wronskianlowerboundmildstrat},  \ref{lemma:smallnulargearg} and \ref{lemma:smallnuboundedarg}.

Next, we address the $H^1_k(I_3^c(y_0))$ bounds of $\left( \G_{k,\ep}^\pm\right)_{\sigma}(y,y_0,\cdot)$ as in Proposition \ref{prop:sobolevregdecomG} They follow once we now note that
\begin{align*}
\Vert \phi_{l,k,\ep}^\pm \Vert_{L^2_z(I_2^c(y_0)\cap I_3(y_0)\cap (0, y_0))}^2 \lesssim |M_\sr(y_0\mp i\ep_0)| \int_{I_2^c(y_0)\cap I_3(y_0)\cap(0,y_0)}\left( |\xi|^{1+2\mu} + |\xi|^{1-2\mu} |\log(\xi)|^2|\mathcal{Q}_\gamma(\xi)|^2 \right)\d z,
\end{align*}
where $\xi=2k(z-y_0\pm i\ep_0)$, with $\beta \leq |\xi|\leq 4\beta$. For $\gamma=\mu+i\nu$, with $\mu,\nu\geq 0$, and $\gamma\in B_\frac14(0)$, we further observe that for $|\xi|$ bounded we have
\begin{equation}\label{eq:boundQgamma}
| Q_\gamma(\xi)| \leq \int_0^2 e^{ 2u\left( \mu \log(|\xi|-\nu \text{Arg}(\xi) \right)}\d u \lesssim  \int_0^2 e^{ 2u\mu \log(|\xi|)}\d u \lesssim 1.
\end{equation}
Since $|\log(\xi)| \lesssim 1 + |\log(|\xi|)|\lesssim 1 + |\log(2k|z-y_0|)|$, we then have
\begin{align*}
\Vert \phi_{l,k,\ep}^\pm \Vert_{L^2_z(I_2^c(y_0)\cap I_3(y_0)\cap (0, y_0))}^2 &\lesssim \frac{|M_\sr(y_0\mp i\ep_0)|}{k}\int_{2\beta}^{3\beta} \left( x^{1+2\mu} + x^{1-2\mu}( 1 + |\log^2(x)|) \right) \d x \\
&\lesssim \frac{|M_\sr(y_0\mp i\ep_0)|}{k}
\end{align*}
and the conclusion follows as before.
\end{proof}

We next state analogous regularity results to those for the weak and strong stratification regime adapted to the regularity structures of the mild stratification. We omit their proofs, as they are identical to the ones for the weak and strong stratification setting.

\begin{proposition}\label{prop:logsobolevregpartialdecomG}
Let $k\geq 1$, $y_0\in I_M$ and $\ep\in (0,\ep_*)$. For $z,y\in I_3(y_0)$, and the decomposition
\begin{align*}
\G_{k,\ep}^\pm(y,y_0,z) = \left( \G_{k,\ep}^\pm\right)_\sr(y,y_0,z) \xi^{\frac12 + \gamma_0} + \left( \G_{k,\ep}^\pm\right)_\s(y,y_0,z) \xi^{\frac12 - \gamma_0} \log(\xi)\Q_{\gamma_0}(\xi)
\end{align*}
where $\xi=2k(z-y_0\pm i\ep_0)$, there holds, 
\begin{align*}
\left\Vert \partial_z\left( \G_{k,\ep}^\pm\right)_\sigma(\cdot,y_0,z) \right\Vert_{X_k^1} + k^\frac12 \left\Vert \partial_z\left( \G_{k,\ep}^\pm \right)_\sigma(\cdot,y_0,z) \right \Vert_{L^2(I_3(y_0))}\lesssim |\xi|
\end{align*}
for $\sigma\in \lbrace \sr, \s \rbrace$. Moreover, for $f\in C(I_3(y_0))$, $y\in I_3=I_3(y_0)$ and $\eta=2k(y-y_0\pm i\ep_0)$, we have
\begin{align*}
\partial_y \int_{I_3(y_0)} \left( \partial_z \left( \G_{k,\ep}^\pm \right)_\sigma  \right)_\tau (y,y_0,z) f(z) \d z &= f(y) \left( \left( \partial_z \left( \G_{\su,k,\ep}^\pm \right)_\sigma \right)_\tau - \left( \partial_z \left( \G_{\sl,k,\ep}^\pm \right)_\sigma \right)_\tau \right)(y,y_0,y)  \\
&\quad + \int_{(0,y)\cap I_3} \partial_y \left( \partial_z \left( \G_{\su,k,\ep}^\pm \right)_\sigma  \right)_\tau (y,y_0,z) f(z) \d z \\
&\quad + \int_{(y,2)\cap I_3} \partial_y \left( \partial_z \left( \G_{\sl,k,\ep}^\pm \right)_\sigma  \right)_\tau (y,y_0,z) f(z) \d z
\end{align*}
with
\begin{align*}
\left\vert \partial_y \left( \partial_z \left( \G_{\su,k,\ep}^\pm \right)_\sigma  \right)_\tau(y,y_0,z) \right\vert + \left\vert \partial_y \left( \partial_z \left( \G_{\sl,k,\ep}^\pm \right)_\sigma  \right)_\tau(y,y_0,z) \right\vert \lesssim k |\eta||\xi|
\end{align*}
uniformly for all $y_0\in I_M$, for $\sigma,\,\tau\in \lbrace \sr, \s \rbrace$.
\end{proposition}

\begin{corollary}\label{cor:logsobolevregpartialdecomG}
Let $k\geq 1$, $y_0\in I_M$ and $\ep\in (0,\ep_*)$. For $z,y\in I_3(y_0)$, and the decomposition
\begin{align*}
\G_{k,\ep}^\pm(y,y_0,z) = \left( \G_{k,\ep}^\pm\right)_\sr(y,y_0,z) \eta^{\frac12 + \gamma_0} + \left( \G_{k,\ep}^\pm\right)_\s(y,y_0,z) \eta^{\frac12 - \gamma_0}\log(\eta)\Q_{\gamma_0}(\eta) ,
\end{align*}
where $\eta=2k(z-y_0\pm i\ep_0)$, there holds, 
\begin{align*}
\left\Vert \partial_y\left( \G_{k,\ep}^\pm\right)_\sigma(y,y_0,\cdot) \right\Vert_{X_k^1} + k^\frac12 \left\Vert \partial_y\left( \G_{k,\ep}^\pm \right)_\sigma(y,y_0,\cdot) \right \Vert_{L^2(I_3(y_0))}\lesssim |\eta|
\end{align*}
for $\sigma\in \lbrace \sr, \s \rbrace$.
\end{corollary}

\section{Solution operator estimates for the stratified regime}\label{sec:solopstrat}
Unless stated otherwise we assume from now on $\ep_*>0$ is given by Propositions \ref{prop:sobolevregdecomG} and \ref{prop:logsobolevregdecomG}. To pave the way for the limiting absorption principle, here we define the error operator
\begin{align}\label{eq:defsolopT}
T_{k,y_0,\ep}^\pm f(y) = \int_{0} ^{2} \G_{k,\ep}^\pm(y,y_0,z)\scE_{k,\ep}^\pm(y,y_0) f(z) \d z.
\end{align}
and the solution operators
\begin{align}\label{eq:defsolopR}
(R_{m,k,\ep}^\pm f)(y,y_0) = \int_0^2\G_{k,\ep}^\pm(y,y_0,z)\frac{f(z)}{(v(z)-v(y_0)\pm i\ep)^m} \d z,
\end{align}
for $m=0,1,2,3,4$. The above operators are such that
\begin{equation*}
\textsc{RTG}_{k,\ep}^\pm  (T_{k,y_0,\ep}^\pm f)(y) = \scE_{k,\ep}^\pm(y,y_0) f(y)
\end{equation*}
and
\begin{equation*}
\textsc{RTG}_{k,\ep}^\pm  (R_{m,k,\ep}^\pm f)(y,y_0) =  \frac{f(y)}{(v(y)-v(y_0)\pm i\ep)^m},
\end{equation*}
respectively, for $m=0,1,2,3,4$. The purpose of this section is to establish the necessary mapping properties of the operators $T_{k,\ep}^\pm$ and $R_{m,k,\ep}^\pm$ with respect to the spaces $X_k$ and $LX_k$ and uniform in $\ep\in (0,\ep_*)$ to apply the limiting absorption principle in Section \ref{sec:LAPstrat}.

\subsection{Operator estimates for strong and weak stratifications}
We first consider the space $X_k$ and $y_0\in I_S\cup I_W$. For $y_0\in I_S$ we recall that $\mu_0=0$ and $\nu_0>0$. Likewise,  for $y_0\in I_W$ we have that $\nu_0=0$ and $\mu_0\neq 0$ since $\mu_0 = \sqrt{\tfrac14 - \cJ(y_0)}$, with $\cJ(y_0)=\tfrac{\P(y_0)}{v'(y_0)^2}$. In particular, we record
\begin{align}\label{eq:upperboundPynot}
\P(y_0) = (1-2\mu_0)\frac{1+2\mu_0}{4}v'(y_0) ^2 \lesssim (1-2\mu_0)
\end{align}
which will prove useful in several estimates below.
\begin{lemma}\label{lemma:R0mapsL2toXk}
Let $k\geq 1$ and $f\in L^2(\uv, \ov)$. There holds
\begin{align*}
\Vert (R_{0,k,\ep}^\pm f)(\cdot,y_0) \Vert_{X_k} \lesssim k^{-\frac32}\Vert f \Vert_{L^2(0,2)},
\end{align*}
uniformly for all $0<\ep< \ep_*$ and all $y_0\in I_S\cup I_W$.
\end{lemma}

\begin{proof}
Let $g_{k,\ep}^\pm(y,y_0) := (R_{0,k,\ep}^\pm f)(y,y_0)$, we note that
\begin{align*}
g_{k,\ep}^\pm(y,y_0) &=\eta^{\frac12+\gamma_0} \int_\uv ^ \ov \left( \G_{k,\ep}^\pm\right)_\sr(y,y_0,z) f(z) \d z \\
&\quad+\eta^{\frac12-\gamma_0} \int_\uv ^ \ov \left( \G_{k,\ep}^\pm\right)_\s (y,y_0,z) f(z) \d z,
\end{align*}
where as usual $\eta=2k(y-y_0\pm i\ep_0)$ and we define $\left( g_{k,\ep}^\pm\right)_{\sigma}(y,y_0)$ accordingly. Since $\left( \G_{k,\ep}^\pm\right)_{\sigma}(y,y_0,z)\in X_k$ as a function of $z$ uniformly in $y_0\in I_S \cup I_W$ and in $y\in I_3(y_0)$, confer Proposition \ref{prop:sobolevregdecomG}, we have that 
\begin{align*}
\left( g_{k,\ep}^\pm\right)_{\sigma}(y,y_0) &= \int_{I_3(y_0)} \left( \G_{k,\ep}^\pm\right)_{\sigma}(y,y_0,z) f(z) \d z+ \int_{I_3^c(y_0)} \left( \G_{k,\ep}^\pm\right)_{\sigma}(y,y_0,z) f(z) \d z
\end{align*}
where we estimate
\begin{align*}
\sup_{y_0\in I_S \cup I_W}\sup_{y\in I_3(y_0)}\left| \int_{I_3(y_0)} \left( \G_{k,\ep}^\pm\right)_{\sigma}(y,y_0,z)f(z) \d z \right| \lesssim k^{-\frac32}\Vert f \Vert_{L^2}
\end{align*}
and
\begin{align*}
\sup_{y_0\in I_S \cup I_W}\sup_{y\in I_3(y_0)}\left| \int_{I_3^c(y_0)} \left( \G_{k,\ep}^\pm\right)_{\sigma}(y,y_0,z) f(z) \d z \right| \lesssim k^{-\frac32}\Vert f \Vert_{L^2}.
\end{align*}
Hence, we conclude that
\begin{align*}
\left\Vert \left( g_{k,\ep}^\pm\right)_{\sigma}(\cdot,y_0) \right\Vert_{L^\infty(I_3(y_0))}\lesssim k^{-\frac32}\Vert f \Vert_{L^2}.
\end{align*}
Similarly, since
\begin{align*}
\partial_y\left( g_{k,\ep}^\pm\right)_{\sigma}(y,y_0) &= \int_{I_3(y_0)} \partial_y\left( \G_{k,\ep}^\pm\right)_{\sigma}(y,y_0,z)\chi_{I_3(y_0)}(z) f(z) \d z + \int_{I_3^c(y_0)} \partial_y\left( \G_{k,\ep}^\pm\right)_{\sigma}(y,y_0,z) f(z) \d z,
\end{align*}
an application of Corollary \ref{cor:sobolevregpartialdecomG} shows that
\begin{align*}
 \left| \int_{I_3(y_0)} \partial_y\left( \G_{k,\ep}^\pm\right)_{\sigma}(y,y_0,z)f(z) \right| \d z \lesssim k^{-\frac12}|\eta|\Vert f \Vert_{L^2}
\end{align*}
 and
 \begin{align*}
\left| \int_{I_3^c(y_0)} \partial_y\left( \G_{k,\ep}^\pm\right)_{\sigma}(y,y_0,z) f(z) \right| \d z \lesssim k^{-\frac12}|\eta|\Vert f \Vert_{L^2}
\end{align*}
Therefore, we see that $\Vert g_{k,\ep}^\pm(\cdot,y_0) \Vert_{X_k^1}\lesssim k^{-\frac32}\Vert f \Vert_{L^2}$ uniformly in $y_0\in I_S \cup I_W$. Next, since $g_{k,\ep}^\pm$ satisfies
\begin{align*}
\left( \D_k - k^2 + \frac{\cJ(y_0)}{(y-y_0\pm i\ep_0)^2} \right) g_{k,\ep}^\pm(y,y_0) = f(y),
\end{align*}
we use the entanglement inequality to obtain
\begin{align*}
\Vert g_{k,\ep}^\pm \Vert_{H^1_k(I_3^c(y_0))} \lesssim \frac{1}{k^2}\Vert f \Vert_{L^2} + \Vert g_{k,\ep}^\pm \Vert_{L^2(I_2^c(y_0)\cap I_3(y_0))} \lesssim \frac{1}{k^2}\Vert f \Vert_{L^2}
\end{align*}
because $\Vert g_{k,\ep}^\pm \Vert_{L^2(I_2^c(y_0)\cap I_3(y_0))} \lesssim k^{-\frac32}\Vert f \Vert_{L^2} \Vert \eta^{\frac12-\gamma_0}\Vert_{L^2(I_2^c(y_0)\cap I_3(y_0))} \lesssim k^{-2}\Vert f \Vert_{L^2}$. With this, the proof is finished.
\end{proof}

\begin{lemma}\label{lemma:R1mapsPLinfL2toXk}
Let $k\geq 1$ and $y_0\in I_S\cup I_{W}$. Let $f\in L^\infty (I_3)\cap L^2(I_3^c)$. There holds
\begin{align*}
\P(y_0)\Vert (R_{1,k,\ep}^\pm f)(\cdot,y_0) \Vert_{X_k} \lesssim k^{-1} \Vert f \Vert_{L^\infty(I_3)} +  k^{-\frac12} \Vert f \Vert_{L^2(I_3^c)},
\end{align*}
uniformly both in $\ep\in (0, \ep_*)$, and in $y_0\in I_S\cup I_{W}$.
\end{lemma}

\begin{proof}
Let $g_{k,\ep}^\pm(y,y_0):= (R_{1,k,\ep}^\pm f)(y,y_0)$, we now have
\begin{align*}
\left( g_{k,\ep}^\pm \right)_{\sigma}(y,y_0) = \int_\uv ^ \ov \left( \G_{k,\ep}^\pm\right)_\sigma(y,y_0,z) \frac{f(z)}{v(z)-v(y_0)\pm i\ep} \d z. 
\end{align*}
Proposition \ref{prop:sobolevregdecomG} yields
\begin{align*}
\left| \int_{I_3(y_0)} \left( \G_{k,\ep}^\pm\right)_\sigma(y,y_0,z)  \frac{f(z)}{v(z)-v(y_0) \pm i\ep} \d z \right| \lesssim k^{-1}\frac{ \Vert f \Vert_{L^\infty(I_3)}}{1-2\mu_0},
\end{align*}
while there also holds
\begin{align*}
\left| \int_{I_3^c(y_0)} \left( \G_{k,\ep}^\pm\right)_\sigma(y,y_0,z) \frac{f(z)}{v(z)-v(y_0) \pm i\ep} \d z \right| &\lesssim k  \left\Vert \left( \G_{k,\ep}^\pm\right)_\sigma(y,y_0,z) \right\Vert_{L^2(I_3^c(y_0))} \Vert f \Vert_{L^2(I_3^c)} \\
&\lesssim k^{-\frac12}\Vert f \Vert_{L^2(I_3^c)}.
\end{align*}
On the other hand, from Corollary \ref{cor:sobolevregpartialdecomG}, we obtain
\begin{align*}
\left| \int_{I_3(y_0)} \partial_y\left( \G_{k,\ep}^\pm\right)_\sigma(y,y_0,z)) \frac{f(z)}{v(z)-v(y_0) \pm i\ep} \d z \right| \lesssim \frac{\Vert f \Vert_{L^\infty(I_3)}}{1-2\mu_0}|\eta|,
\end{align*}
and
\begin{align*}
\left| \int_{I_3^c(y_0)} \partial_y \left( \G_{k,\ep}^\pm\right)_\sigma(y,y_0,z) \frac{f(z)}{v(z)-v(y_0) \pm i\ep} \d z \right| &\lesssim k^\frac12\Vert f \Vert_{L^2(I_3^c)}|\eta|.
\end{align*}
As a result we get $\P(y_0)\left \Vert \left( g_{k,\ep}^\pm \right)_{\sigma}(\cdot,y_0)\right\Vert_{X_k^1}\lesssim k^{-1} \Vert f \Vert_{L^\infty(I_3)} + k^{-\frac12}\Vert f \Vert_{L^2(I_3^c)} $. Now we use the equation
\begin{align*}
\left( \D_k - k^2 + \frac{\cJ(y_0)}{(y-y_0\pm i\ep_0)^2} \right) g_{k,\ep}^\pm(y,y_0) = \frac{f(y)}{v(y)-v(y_0)\pm i\ep}
\end{align*}
and the entanglement inequality to obtain
\begin{align*}
k^2\Vert g_{k,\ep}^\pm \Vert_{H^1_k(I_3^c(y_0))}^2 &\lesssim \left| \int_{I_2^c(y_0)} \frac{f(y)g_{k,\ep}^\pm(y,y_0)}{v(y)-v(y_0)\pm i\ep} \d y \right| \\
&\lesssim \left| \int_{I_2^c(y_0)\cap I_3(y_0))} \frac{f(y)g_{k,\ep}^\pm(y,y_0)}{v(y)-v(y_0)\pm i\ep} \d y \right| + \left| \int_{I_3^c(y_0)} \frac{f(y)g_{k,\ep}^\pm(y,y_0)}{v(y)-v(y_0)\pm i\ep} \d y \right| \\
& \lesssim k^{-1}\Vert f \Vert_{L^\infty(I_3)}\Vert g_{k,\ep}^\pm(\cdot, y_0) \Vert_{X_k^0} + \delta \int_{I_3^c(y_0)}k^2 | g_{k,\ep}^\pm(y,y_0)|^2 \d y + C_\delta \Vert f \Vert_{L^2(I_3^c)}^2 
\end{align*}
for some $C_\delta>0$. Since $\P(y_0)\left \Vert \left( g_{k,\ep}^\pm \right)_{\sigma}(\cdot,y_0)\right\Vert_{X_k^1}\lesssim k^{-1} \Vert f \Vert_{L^\infty(I_3)} + k^{-\frac12}\Vert f \Vert_{L^2(I_3^c)} $, for $\delta>0$ small enough, there holds
\begin{align*}
\P(y_0)^2k^2\Vert g_{k,\ep}^\pm \Vert_{H^1_k(I_3^c(y_0))}^2 &\lesssim k^{-2}\Vert f \Vert_{L^\infty(I_3)}^2 + C_\delta\Vert f \Vert_{L^2(I_3^c)}^2 \lesssim k^{-2}\Vert f \Vert_{L^\infty(I_3)}^2 + \Vert f \Vert_{L^2(I_3^c)}^2
\end{align*}
from which we obtain the desired estimate. 
\end{proof}

In the sequel we shall also need a stronger version of Lemma \ref{lemma:R1mapsPLinfL2toXk}. We record it here.
\begin{lemma}\label{lemma:R1mapsXktoXk}
Let $k\geq 1$ and $y_0\in I_S\cup I_{W}$. Let $f\in Z_k$. There holds
\begin{align*}
\Vert (R_{1,k,\ep}^\pm f)(\cdot,y_0) \Vert_{X_k} \lesssim k^{-1}\Vert f \Vert_{Z_k}
\end{align*}
uniformly both in $\ep\in (0, \ep_*)$, and in $y_0\in I_S\cup I_{W}$.
\end{lemma}

We remark here that $R_{1,k,\ep}^\pm$ regularizes from $Z_k$ to $X_k$. This fact will be important later to upgrade the regularity of $\varphi_{k,\ep}^\pm$.

\begin{proof}
Recall that for $g_{k,\ep}^\pm(y,y_0):= (R_{1,k,\ep}^\pm f)(y,y_0)$ we have
\begin{align*}
\left( g_{k,\ep}^\pm \right)_{\sigma}(y,y_0) = \int_\uv ^ \ov \left( \G_{k,\ep}^\pm\right)_\sigma(y,y_0,z) \frac{f(z)}{v(z)-v(y_0)\pm i\ep} \d z. 
\end{align*}
We note that
\begin{align*}
\frac{1}{v(y) - v(y_0) \pm i\ep} = \frac{(v'(y_0))^{-1}}{y-y_0\pm i\ep_0} + \frac{1}{v(y) - v(y_0) \pm i\ep} - \frac{(v'(y_0))^{-1}}{y-y_0\pm i\ep_0} 
\end{align*}
with
\begin{align*}
\frac{(v'(y_0))^{-1}}{y-y_0\pm i\ep_0}  &- \frac{1}{v(y) - v(y_0) \pm i\ep} \\
&= \frac{(y-y_0)^2}{(v(y)-v(y_0)\pm i\ep)(y-y_0\pm i\ep_0)}\int_0^1\int_0^1 v''(y_0 + u_1 u_2(y-y_0)) \d u_1 \d u_2 \\
&=\mathrm{V}_{1,\ep}^\pm(y,y_0)\in L^\infty(0,2)
\end{align*}
uniformly for all $y_0\in [0,2]$ and all $0<\ep$. Hence, we have
\begin{align*}
g_{k,\ep}^\pm(y,y_0) &= \int_0^2 \G_{k,\ep}^\pm(y,y_0,z) \frac{f(z)}{z-y_0\pm i\ep_0}\d z - R_{0,k,\ep}^\pm\left[ \mathrm{V}_{1,\ep}^\pm f \right] \\
&=g_{1,k,\ep}^\pm(y,y_0) - g_{2,k,\ep}^\pm(y,y_0).
\end{align*}
Since $\mathrm{V}_{1,\ep}^\pm(\cdot,y_0)\in L^\infty(I_3)$ and $f\in Z_k$, from Lemma \ref{lemma:R0mapsL2toXk} we have that 
\begin{align*}
\Vert g_{2,k,\ep}\Vert_{X_k}\lesssim k^{-\frac32}\Vert f \Vert_{L^2}\lesssim k^{-1}\Vert f \Vert_{Z_k}.
\end{align*}
We further split
\begin{align*}
g_{1,k,\ep}^\pm(y,y_0) &= \int_{I_3(y_0)} \G_{k,\ep}^\pm(y,y_0,z) \frac{f(z)}{z-y_0\pm i\ep_0}\d z + \int_{I_3^c(y_0)} \G_{k,\ep}^\pm(y,y_0,z) \frac{f(z)}{z-y_0\pm i\ep_0}\d z \\
&=: g_{1,\ell,k,\ep}^\pm(y,y_0) + g_{1,n\ell,k,\ep}^\pm(y,y_0)
\end{align*}
so that from Proposition \ref{prop:sobolevregdecomG} we readily have $\Vert g_{1,n\ell,k,\ep}^\pm \Vert_{X_k^1}\lesssim k^{-1}\Vert f \Vert_{Z_k}$. On the other hand, using the regularity structures of $f\in Z_k$ and $\G_{k,\ep}^\pm$, we have
\begin{align*}
f(z,y_0) = f_\sr(z,y_0)\xi^{\frac12+\gamma_0} + f_\s(z,y_0) \xi^{\frac12-\gamma_0}
\end{align*}
and likewise
\begin{align*}
\G_{k,\ep}^\pm(y,y_0,z) = \left( \G_{k,\ep}^\pm\right)_{\sr}(y,y_0,z)\xi^{\frac12+\gamma_0} + \left( \G_{k,\ep}^\pm\right)_{\s}(y,y_0,z) \xi^{\frac12-\gamma_0}
\end{align*}
where we recall that $\xi=2k(z-y_0\pm i\ep_0)$. Then,
\begin{align*}
g_{1,\ell,k,\ep}^\pm(y,y_0) &= 2k\int_{I_3(y_0)} \left( \G_{k,\ep}^\pm\right)_{\sr}(y,y_0,z)f_\sr(z,y_0)(2k(z-y_0\pm i\ep))^{2\gamma_0} \d z \\
&\quad +2k \int_{I_3(y_0)} \left( \left( \G_{k,\ep}^\pm\right)_{\sr}(y,y_0,z) f_\s(z,y_0) + \left( \G_{k,\ep}^\pm\right)_{\s}(y,y_0,z) f_\sr(z,y_0)  \right) \d z \\
&\quad+ 2k\int_{I_3(y_0)} \left( \G_{k,\ep}^\pm\right)_{\s}(y,y_0,z)\frac{f_\s(z,y_0)}{(2k(z-y_0\pm i\ep))^{2\gamma_0}} \d z.
\end{align*}
We shall consider $\gamma_0 = \mu_0 \neq 0$, since for $\gamma_0=i\nu_0\neq0$ we can already use Proposition \ref{prop:sobolevregdecomG} to get the desired $X_k^1$ bounds for $g_{1,\ell,k,\ep}^\pm$. Now, for $\gamma_0=\mu_0\in (\mu_*,\tfrac12)$ we shall only consider the integral involving the most singular factor $\xi^{-2\mu_0}$, namely
\begin{align*}
2k&\int_{I_3(y_0)} \left( \G_{k,\ep}^\pm\right)_{\s}(y,y_0,z)\frac{f_\s(z,y_0)}{(2k(z-y_0\pm i\ep))^{2\mu_0}} \d z \\
&= \int_{I_3(y_0)} \left( \G_{k,\ep}^\pm\right)_{\s}(y,y_0,z)f_\s(z,y_0) \partial_z \left( \frac{(2k(z-y_0\pm i\ep))^{1-2\mu_0} - 1}{1-2\mu_0} \right) \d z \\
&= \left. \left( \G_{k,\ep}^\pm\right)_{\s}(y,y_0,z)f_\s(z,y_0)  \frac{(2k(z-y_0\pm i\ep))^{1-2\mu_0} - 1}{1-2\mu_0} \right|_{\partial I_3} \\
&\quad -  \int_{I_3(y_0)} \partial_z \left( \left( \G_{k,\ep}^\pm\right)_{\s}(y,y_0,z)f_\s(z,y_0) \right)  \frac{(2k(z-y_0\pm i\ep))^{1-2\mu_0} - 1}{1-2\mu_0}  \d z.
\end{align*}
 Further observing that 
\begin{align*}
\left |\zeta^{1-2\mu_0} - 1 \right| = (1-2\mu_0)\left| \log(\zeta) \int_0^2 e^{u(1-2\mu_0)\log(\zeta)} \d u \right| \lesssim (1-2\mu_0)|\log(\zeta)|
\end{align*}
uniformly in $\mu_0\in (\mu_*,\tfrac12)$ for all $\zeta\in \C$ with $|\zeta|$ bounded, and the fact that $-\int_0^2\log(x) \d x$ is finite, we obtain from Proposition \ref{prop:sobolevregdecomG} and Corollary \ref{cor:sobolevregpartialdecomG} that the solid term is bounded and thus, 
\begin{align*}
\left \Vert 2k \int_{I_3(y_0)} \left( \G_{k,\ep}^\pm\right)_{\s}(\cdot,y_0,z) \frac{f_\s(z,y_0)}{(2k(z-y_0\pm i\ep))^{2\mu_0}} \d z \right\Vert_{X_k^0}   \lesssim  k^{-1} \Vert f \Vert_{Z_k^1}.
\end{align*}
Concerning the $X_k^1$ bound, note that
\begin{align*}
\int_{I_3(y_0)} &\partial_z \left( \left( \G_{k,\ep}^\pm\right)_{\s}(y,y_0,z)f_\s(z,y_0) \right)  \frac{(2k(z-y_0\pm i\ep))^{1-2\gamma_0} - 1}{1-2\gamma_0}  \d z \\
&=\int_{I_3(y_0)} \partial_z  \left( \G_{k,\ep}^\pm\right)_{\s}(y,y_0,z)f_\s(z,y_0)  \frac{(2k(z-y_0\pm i\ep))^{1-2\gamma_0} - 1}{1-2\gamma_0}  \d z \\ 
&\quad + \int_{I_3(y_0)}  \left( \G_{k,\ep}^\pm\right)_{\s}(y,y_0,z)\partial_z \left( f_\s(z,y_0) \right)  \frac{(2k(z-y_0\pm i\ep))^{1-2\gamma_0} - 1}{1-2\gamma_0}  \d z.
\end{align*}
In particular, Proposition \ref{prop:sobolevregpartialdecomG} applies and 
\begin{align*}
\partial_y \int_{I_3(y_0)} & \left(\partial_z  \left( \G_{k,\ep}^\pm\right)_{\s} \right)_\tau (y,y_0,z)f_\s(z,y_0)  \frac{(2k(z-y_0\pm i\ep))^{1-2\gamma_0} - 1}{1-2\gamma_0}  \d z \\
&= \left( \left(\partial_z \left( \G_{\su,k,\ep}^\pm\right)_\s \right)_\tau  - \left( \partial_z \left( \G_{\sl,k,\ep}^\pm\right)_\s \right)_\tau \right)(y,y_0,y) f_\s(y,y_0)\frac{\eta^{1-2\gamma_0}-1}{1-2\gamma_0} \\
&\quad +\int_{(0,y)\cap I_3} \partial_y \left(\partial_z \left( \G_{\su,k,\ep}^\pm\right)_\s\right)_\tau (y,y_0,z)  f_\s(z,y_0)  \frac{(2k(z-y_0\pm i\ep))^{1-2\gamma_0} - 1}{1-2\gamma_0} \d z \\
&\quad + \int_{(y,2)\cap I_3} \partial_y \left( \partial_z \left( \G_{\sl,k,\ep}^\pm\right)_\s \right)_\tau (y,y_0,z) f_\s(z,y_0)  \frac{(2k(z-y_0\pm i\ep))^{1-2\gamma_0} - 1}{1-2\gamma_0}  \d z. 
\end{align*}
Further recalling, see Proposition \ref{prop:sobolevregpartialdecomG},  that 
\begin{align*}
\left|  \left( \partial_z \left( \G_{\su,k,\ep}^\pm\right)_{\s} \right)_\tau (y,y_0,y) \right| + \left|  \left( \partial_z \left( \G_{\sl,k,\ep}^\pm\right)_{\s}\right)_\tau (y,y_0,y) \right| \lesssim |\eta|,
\end{align*}
for $\eta=2k(y-y_0\pm i\ep_0)$ with $|\eta|$ bounded, together with
\begin{align*}
\left| \eta^{1-2\mu_0} -1 \right| \lesssim (1-2\mu_0)|\log(\eta)|
\end{align*}
uniformly in $y_0\in I_S\cup I_W$, and Proposition \ref{prop:sobolevregdecomG}  we conclude that 
\begin{align*}
\left \Vert  \eta^{-\frac12} \partial_y \int_{I_3(y_0)} \left( \partial_z  \left( \G_{k,\ep}^\pm\right)_{\s} \right)_\tau (y,y_0,z)f_\s(z,y_0)  \frac{(2k(z-y_0\pm i\ep))^{1-2\gamma_0} - 1}{1-2\gamma_0}  \d z \right\Vert_{L^\infty(I_3(y_0))} \lesssim \Vert f \Vert_{Z_k^1}.
\end{align*}
uniformly for all $y_0 \in I_S\cup I_W$, for $\tau\in \lbrace \sr, \s \rbrace$. Thus, the $\left \Vert \eta^{-\frac12}\partial_y \left( g_{k,\ep}^\pm \right)_\sigma  \right\Vert_{L^\infty(I_3(y_0))}$ bounds follow from the above, Proposition \ref{prop:sobolevregdecomG}, Corollary \ref{cor:sobolevregpartialdecomG} and the observation that
\begin{align*}
\partial_y \left( g_{k,\ep}^\pm \right)_\tau &=\int_{I_3(y_0)} \partial_y\left( \G_{k,\ep}^\pm\right)_\sigma(y,y_0,z) \frac{f(z,y_0)}{z-y_0\pm i\ep} \d z \\
&= 2k\int_{I_3(y_0)} \partial_y\left( \G_{k,\ep}^\pm\right)_{\sigma,\sr}(y,y_0,z)f_\sr(z,y_0)(2k(z-y_0\pm i\ep))^{2\gamma_0} \d z \\
&\quad +2k \int_{I_3(y_0)} \left( \partial_y\left( \G_{k,\ep}^\pm\right)_{\sigma,\sr}(y,y_0,z) f_\s(z,y_0) + \partial_y\left( \G_{k,\ep}^\pm\right)_{\sigma,\s}(y,y_0,z) f_\sr(z,y_0)  \right) \d z \\
&\quad+ 2k\partial_y\int_{I_3(y_0)}  \left( \G_{k,\ep}^\pm\right)_{\sigma,\s}(y,y_0,z)\frac{f_\s(z,y_0)}{(2k(z-y_0\pm i\ep))^{2\gamma_0}} \d z.
\end{align*}
Once the local $X_k^1$ bounds on $g_{1,k,\ep}^\pm$ are obtained, the entanglement inequality provides
\begin{align*}
\Vert g_{1,k,\ep}^\pm \Vert_{H_k^1(I_3^c)} \lesssim k^{-\frac12}\Vert g_{1,k,\ep}^\pm \Vert_{X_k^1} + k^{-1}\Vert f \Vert_{L^2(I_2^c(y_0))} \lesssim k^{-\frac32}\Vert f \Vert_{Z_k}
\end{align*}
and the second part of the Lemma follows.
\end{proof}

Combining Lemma \ref{lemma:R0mapsL2toXk} and Lemma \ref{lemma:R1mapsXktoXk} we obtain the following result.

\begin{corollary}\label{cor:R1mapsCXktoXk}
Let $k\geq 1$ and $y_0\in I_S\cup I_{W}$. Let $f\in Z_k$ and $h\in C^1$. There holds
\begin{align*}
\Vert (R_{1,k,\ep}^\pm h f) \Vert_{X_k} \lesssim k^{-1}\Vert f \Vert_{Z_k}
\end{align*}
uniformly both in $\ep\in (0,\ep_*)$, and in $y_0\in I_S\cup I_{W}$.
\end{corollary}

\begin{proof}
For $h_1(y,y_0) = h(y) - h(y_0)$ we write 
\begin{align*}
R_{1,k,\ep}^\pm hf(y,y_0) = h(y_0) R_{1,k,\ep}^\pm f(y,y_0) + \left( R_{0,k,\ep}^\pm \frac{h_1(\cdot,y_0)f(\cdot,y_0)}{v(\cdot) - v(y_0)\pm i\ep} \right)(y,y_0).
\end{align*}	 
Moreover, since 
\begin{align*}
\left \Vert \frac{h_1(\cdot,y_0)f(\cdot,y_0)}{v(\cdot) - v(y_0)\pm i\ep} \right\Vert_{L^2}\lesssim k^{-\frac12}\Vert f \Vert_{Z_k},
\end{align*}
the conclusion follows using Lemma \ref{lemma:R0mapsL2toXk} and Lemma \ref{lemma:R1mapsXktoXk}.
\end{proof}

The last lemma on $R_{1,k,\ep}^\pm$ concerns the case where the operator is applied to $H^1$ functions.

\begin{lemma}\label{lemma:R1mapsH1toXk}
Let $k\geq1$, $y_0\in I_S\cup I_W$ and $f\in H^1_k$. Then,
\begin{align*}
\Vert R_{1,k,\ep}^\pm f \Vert_{X_k}\lesssim k^{-\frac12}\Vert f \Vert_{H^1_k}
\end{align*}
uniformly for all $\ep\in (0,\ep_*)$ and all $y_0\in I_S\cup I_W$.
\end{lemma}

\begin{proof}
As usual, let $g_{k,\ep}^\pm(y,y_0) = R_{1,k,\ep}^\pm f(y,y_0)$, we have
\begin{align*}
g_{k,\ep}^\pm(y,y_0) &= 2k\int_0^2 \G_{k,\ep}^\pm(y,y_0,z) \frac{f(z)}{\xi}\d z - R_{0,k,\ep}^\pm \left[ V_{1,\ep}^\pm(\cdot, y_0) f \right](y,y_0)\\
&=:g_{1,k,\ep}^\pm(y,y_0) - g_{2,k,\ep}^\pm(y,y_0).
\end{align*}
While Lemma \ref{lemma:R0mapsL2toXk} gives $\Vert g_{2,k,\ep}^\pm \Vert_{X_k}\lesssim k^{-\frac32}\Vert f \Vert_{L^2}$, we note that
\begin{align*}
g_{1,k,\ep}^\pm &= 2k \int_{I_3^c(y_0)} \G_{k,\ep}^\pm(y,y_0,z) \frac{f(z)}{\xi}\d z  + 2k \int_{I_3^c(y_0)} \G_{k,\ep}^\pm(y,y_0,z) \frac{f(z)}{\xi}\d z \\
&=g_{1,\ell,k,\ep}^\pm(y,y_0) + g_{1,n\ell,k,\ep}^\pm(y,y_0)
\end{align*}
so that Proposition \ref{prop:sobolevregdecomG} and Corollary \ref{cor:sobolevregpartialdecomG} readily gives $\Vert g_{1,n\ell,k,\ep}^\pm\Vert_{X_k}\lesssim k^{-\frac12}\Vert f \Vert_{L^2}$. As for the local contribution, we have 
\begin{align*}
g_{1,\ell,k,\ep}^\pm(y,y_0) &= \int_{I_3(y_0)} \left( \left(\G_{k,\ep}^\pm \right)_\sr (y,y_0,z) \partial_z \left( \frac{\xi^{\frac12+\gamma_0}}{\frac12+\gamma_0} \right) + \left(\G_{k,\ep}^\pm \right)_\s (y,y_0,z) \partial_z \left( \frac{\xi^{\frac12-\gamma_0}}{\frac12-\gamma_0}-1 \right) \right) f(z) \d z \\
&=\left( \left(\G_{k,\ep}^\pm \right)_\sr (y,y_0,z)  \frac{\xi^{\frac12+\gamma_0}}{\frac12+\gamma_0} + \left(\G_{k,\ep}^\pm \right)_\s (y,y_0,z)  \frac{\xi^{\frac12-\gamma_0}-1}{\frac12-\gamma_0}  \right) f(z) \Big|_{\partial I_3(y_0)} \\
&\quad - \frac{2}{1+2\gamma_0}\int_{I_3(y_0)} \partial_z \left ( \left(\G_{k,\ep}^\pm \right)_\sr (y,y_0,z) f(z) \right)\frac{\xi^{\frac12+\gamma_0}}{\frac12+\gamma_0} \d z \\
&\quad - \int_{I_3(y_0)} \partial_z \left ( \left(\G_{k,\ep}^\pm \right)_\s (y,y_0,z) f(z) \right)\log(\xi)\mathcal{Q}_{\frac14-\frac{\gamma_0}{2}}(\xi) \d z.
\end{align*}
Then,  $\Vert g_{1,\ell,k ,\ep}^\pm \Vert_{X_k^1}\lesssim k^{-\frac12}\Vert f\Vert_{H_k^1} $ follows from Lemma \ref{lemma:LinfH1bound}, Propositions \ref{prop:sobolevregdecomG},  \ref{prop:sobolevregpartialdecomG} and Corollary \ref{cor:sobolevregpartialdecomG} and the observations that $\log(\xi)$ is squared-integrable, bounded for $|\xi|$ bounded uniformly from below and above and $|Q_{\frac14-\frac{\gamma_0}{2}}(\xi)|$ is uniformly bounded for $|\xi|$ bounded. The $H_k^1(I_3^c(y_0))$ bounds of $g_{1,k,\ep}^\pm$ follow from the local $X_k^1$ bounds and the entanglement inequality, we omit the details.
\end{proof}

We now consider the most singular operator $R_{2,k,\ep}^\pm$ that we will need to use the limiting absorption principle. First, we record a basic identity satisfied for functions $f(\cdot,y_0)\in X_k$ that will prove useful:
\begin{equation}\label{eq:partialzGf}
\begin{split}
\partial_z \left( \G_{k,\ep}^\pm(y,y_0,\cdot) f(z,y_0) \right) &= \partial_z \left( \left( \G_{k,\ep}^\pm \right)_\sr(y,y_0,z) f_\sr(z, y_0) \right) \xi^{1+2\gamma_0} \\
&\quad + 2k(1+2\gamma_0)  \left( \G_{k,\ep}^\pm \right)_\sr(y,y_0,z) f_\sr(z, y_0) \xi^{2\gamma_0} \\
&\quad + \partial_z \left( \left( \G_{k,\ep}^\pm \right)_\sr(y,y_0,z) f_\s(z, y_0) + \left( \G_{k,\ep}^\pm \right)_\s(y,y_0,z) f_\sr(z, y_0) \right) \xi \\
&\quad + 2k \left( \left( \G_{k,\ep}^\pm \right)_\sr(y,y_0,z) f_\s(z, y_0) + \left( \G_{k,\ep}^\pm \right)_\s(y,y_0,z) f_\sr(z, y_0) \right)  \\
&\quad + \partial_z \left( \left( \G_{k,\ep}^\pm \right)_\s(y,y_0,z) f_\s(z, y_0) \right) \xi^{1-2\gamma_0} \\
&\quad + 2k(1-2\gamma_0)  \left( \G_{k,\ep}^\pm \right)_\s(y,y_0,z) f_\s(z, y_0) \xi^{-2\gamma_0},
\end{split}
\end{equation}
where we recall that $\xi=2k(z-y_0\pm i\ep_0)$. 

\begin{lemma}\label{lemma:R2mapsC1XktoXk}
Let $k\geq 1$, $y_0\in I_S\cup I_W$, $f(y,y_0)\in Z_k$ and $h(\cdot,y_0)\in C^2$ such that $h(y_0,y_0)=0$ and $\Vert h \Vert_{C^2}\lesssim 1$ uniformly in $y_0$. Then,
\begin{align}
\Vert (R_{2,k,\ep}^\pm hf)(\cdot,y_0) \Vert_{X_k} \lesssim  k^{-1}\Vert f \Vert_{Z_k}  
\end{align}
uniformly in $\ep \in (0,\ep_*)$ and $y_0\in I_S\cup I_W$.
\end{lemma}

\begin{proof}
Let $g_{k,\ep}^\pm(y,y_0):= R_{2,k,\ep}^\pm[h f](y,y_0)$. As usual, we write
\begin{align*}
g_{k,\ep}^\pm(y,y_0) &= \int_{I_3(y_0)} \G_{k,\ep}^\pm (y,y_0,z) \frac{h(y,y_0)}{(v(z) - v(y_0) \pm i\ep)^2} f(z,y_0) \d z \\
&\quad + \int_{I_3^c(y_0)} \G_{k,\ep}^\pm (y,y_0,z) \frac{h(y,y_0)}{(v(z) - v(y_0) \pm i\ep)^2} f(z,y_0) \d z \\
&=:g_{\ell,k,\ep}^\pm(y,y_0) + g_{n\ell,k,\ep}^\pm(y,y_0),
\end{align*}
where $\ell$ and $n\ell$ denote the local and non-local integral contributions, respectively. Since $h(y,y_0)\in C^1$ uniformly in $y_0$ with $h(y_0,y_0) = 0$ we rapidly conclude that 
\begin{align*}
\left\Vert g_{n\ell,k,\ep}^\pm  \right\Vert_{X_k}\lesssim k^{-\frac12}\Vert f \Vert_{L^2(I_3(y_0))} \lesssim   k^{-1}\Vert f \Vert_{Z_k}
\end{align*}
thanks to Proposition \ref{prop:sobolevregdecomG} and Corollary \ref{cor:sobolevregpartialdecomG}. On the other hand, integrating by parts,
\begin{align*}
g_{\ell,k,\ep}^\pm(y,y_0) &=  -\frac{h(z,y_0)}{v'(z)} \frac{\G_{k,\ep}^\pm(y,y_0,z) f(z,y_0)}{v(z) - v(y_0) \pm i\ep_0} \Big|_{\partial I_3(y_0)}  +\int_{I_3(y_0)} \partial_z \left( \frac{h(z,y_0)}{v'(z)} \right) \frac{\G_{k,\ep}^\pm(y,y_0,z) f(z,y_0)}{v(z) - v(y_0) \pm i\ep_0} \d z \\
&\quad+ \int_{I_3(y_0)} \frac{h(z,y_0)}{v'(z)}\frac{\partial_z \left( \G_{k,\ep}^\pm(y,y_0,z) f(z, y_0) \right)}{v(z) - v(y_0) \pm i\ep} \d z \\
&=g_{1,k,\ep}^\pm(y,y_0) + g_{2,k,\ep}^\pm (y,y_0) + g_{3,k,\ep}^\pm (y,y_0).
\end{align*}
Since $h\in C^2$ uniformly in $y_0$, we can use Lemma \ref{cor:R1mapsCXktoXk} to obtain $\Vert g_{2,k,\ep}^\pm \Vert_{X_k}\lesssim k^{-1}\Vert f \Vert_{Z_k}$. Similarly, since $h(y_0,y_0)=0$ and $h$ is uniformly smooth in $y_0$ we likewise have, together with Proposition \ref{prop:sobolevregdecomG} and Corollary \ref{cor:sobolevregpartialdecomG}, that $\Vert g_{1,k,\ep}^\pm \Vert_{X_k}\lesssim k^{-1}\Vert f \Vert_{Z_k}$. On the other hand, we have
\begin{align*}
g_{3,k,\ep}^\pm(y,y_0) &= 2k\int_{I_3(y_0)} \frac{h(z,y_0)}{v'(z)}\frac{\partial_z \left( \G_{k,\ep}^\pm(y,y_0,z) f(z, y_0) \right)}{2k(z-y_0\pm i\ep_0)} \d z \\
&\quad -\int_{I_3(y_0)} \frac{h(z,y_0)}{v'(z)} \mathrm{V}_{1,\ep}^\pm(z,y_0)\partial_z \left( \G_{k,\ep}^\pm(y,y_0,z) f(z, y_0) \right) \d z \\
&=: g_{4,k,\ep}^\pm(y,y_0) + g_{5,k,\ep}^\pm(y,y_0).
\end{align*}
We shall argue for the $X_k$ bounds of $g_{4,k,\ep}^\pm$, since the proof for the $g_{5,k,\ep}^\pm$ follows the same lines. From \eqref{eq:partialzGf} we get
\begin{align*}
g_{4,k,\ep}^\pm(y,y_0) &=  2k\int_{I_3(y_0)} \frac{h(z,y_0)}{v'(z)} \partial_z \left( \left(\G_{k,\ep}^\pm\right)_\s (y,y_0,z) f_\s(z, y_0) \right) \xi^{-2\gamma_0} \d z \\
&\quad + 4k^2 (1-2\gamma_0)\int_{I_3(y_0)} \frac{h(z,y_0)}{v'(z)}  \left(\G_{k,\ep}^\pm\right)_\s(y,y_0,z) f_\s(z, y_0) \xi^{-1-2\gamma_0} \d z \\
&\quad + g_{6,k,\ep}^\pm(y,y_0)
\end{align*}
From Proposition \ref{prop:sobolevregdecomG}, Corollary \ref{cor:sobolevregpartialdecomG} and the vanishing properties of $h$ we have $\Vert g_{6,k,\ep}^\pm\Vert_{X_k^1}\lesssim k^{-1}\Vert f \Vert_{Z_k}$. Moreover, since 
\begin{align*}
\left| h(z,y_0) \xi^{-2\gamma_0}\right| \lesssim \frac{1}{k}|2k|z-y_0||^{1-2\mu_0} \lesssim \frac{1}{k} + \frac{1}{k} (1-2\mu_0) |\log(2k|z-y_0|) | \in L^1(I_3(y_0))
\end{align*}	
appealing to Proposition \ref{prop:sobolevregdecomG}, Proposition \ref{prop:sobolevregpartialdecomG} and Corollary \ref{cor:sobolevregpartialdecomG} we again conclude that
\begin{align*}
\left \Vert 2k\int_{I_3(y_0)} \frac{h(z,y_0)}{v'(z)} \partial_z \left( \left(\G_{k,\ep}^\pm\right)_\s (y,y_0,z) f_\s(z, y_0) \right) \xi^{-2\gamma_0} \d z  \right\Vert_{X_k^1}\lesssim k^{-1}\Vert f \Vert_{X_k}.
\end{align*}
Finally,
\begin{align*}
4k^2 (1-2\gamma_0)&\int_{I_3(y_0)} \frac{h(z,y_0)}{v'(z)}  \left(\G_{k,\ep}^\pm\right)_\s(y,y_0,z) f_\s(z, y_0) \xi^{-1-2\gamma_0} \d z \\
&= - 2k\frac{1-2\gamma_0}{2\gamma_0} \frac{h(z,y_0)}{v'(z)}  \left(\G_{k,\ep}^\pm\right)_\s(y,y_0,z) f_\s(z, y_0) \xi^{-2\gamma_0} \Big|_{\partial I_3(y_0)} \\
&\quad + \frac{1}{2\gamma_0}\int_{I_3(y_0)} \partial_z \left( \frac{h(z,y_0)}{v'(z)} \right) \left(\G_{k,\ep}^\pm\right)_\s(y,y_0,z) f_\s(z, y_0) \partial_z \left(\xi^{1-2\gamma_0} -1\right) \d z \\
&\quad + 2k\frac{1-2\gamma_0}{2\gamma_0}\int_{I_3(y_0)} \frac{h(z,y_0)}{v'(z)} \partial_z \left( \left(\G_{k,\ep}^\pm\right)_\s(y,y_0,z) f_\s(z, y_0) \right) \xi^{-2\gamma_0} \d z.
\end{align*}
The observation that $\left| h(z,y_0) \xi^{-2\gamma_0}\right|\lesssim k^{-1}\left( 1 + |\log |\xi| | \right)$, which is integrable whenever $|\xi|$ is bounded and it is bounded whenever $|\xi|$ is uniformly bounded from above and below, and the routine application of Propositions \ref{prop:sobolevregdecomG},  \ref{prop:sobolevregpartialdecomG} and Corollary \ref{cor:sobolevregpartialdecomG}  yield the desired $X_k^1$ bounds. In turn, the $H^1_k(I_3)$ estimates follow form the local $X_k^1$ bounds and the entanglement inequality, we omit the details.
\end{proof}

\subsection{Operator estimates for mild stratifications}
In this subsection we obtain the main estimates on the operators $R_{m,k,\ep}^\pm$ for $m=0,1,2$ when the spectral value ranges in the mildly stratified region. We begin with the analogue of Lemma \ref{lemma:R0mapsL2toXk}, whose proof is identical.

\begin{lemma}\label{lemma:R0mapsL2toLXk}
Let $k\geq 1$ and $f\in L^2(\uv, \ov)$. There holds
\begin{align*}
\Vert (R_{0,k,\ep}^\pm f)(\cdot,y_0) \Vert_{LX_k} \lesssim k^{-\frac32}\Vert f \Vert_{L^2(0,2)},
\end{align*}
uniformly in $0<\ep<\ep_*$ and in $y_0\in I_M$.
\end{lemma}

Since for $y_0\in I_M$ we have that $\mu_0\leq \tfrac14$, we now obtain a simplified and stronger version of Lemma \ref{lemma:R1mapsPLinfL2toXk} that also covers Lemma \ref{lemma:R1mapsXktoXk}, Lemma \ref{lemma:R1mapsH1toXk} in view of Lemma \ref{lemma:LinfH1bound} and Lemma \ref{lemma:R2mapsC1XktoXk}.

\begin{lemma}\label{lemma:R1mapsLinfL2toLXk}
Let $k\geq 1$ and $y_0\in I_M$. Let $f\in L^\infty (I_3)\cap L^2(I_3^c)$. There holds
\begin{align*}
\Vert (R_{1,k,\ep}^\pm f)(\cdot,y_0) \Vert_{LX_k} \lesssim k^{-1} \Vert f \Vert_{L^\infty(I_3)} +  k^{-\frac12} \Vert f \Vert_{L^2(I_3^c)},
\end{align*}
uniformly both in $0<\ep<\ep_*$, and in $y_0\in I_M$. In particular, if $f=f(\cdot,y_0) \in LZ_k$ then 
\begin{align*}
\Vert (R_{1,k,\ep}^\pm f)(\cdot,y_0) \Vert_{LX_k} \lesssim k^{-1} \Vert f \Vert_{LZ_k},
\end{align*}
while if $f\in H_k^1$ then there holds
\begin{align*}
\Vert (R_{1,k,\ep}^\pm f)(\cdot,y_0) \Vert_{LX_k} \lesssim k^{-\frac12} \Vert f \Vert_{H_k^1},
\end{align*}
uniformly both in $0<\ep<\ep_*$, and in $y_0\in I_M$.
\end{lemma}

\begin{proof}
For $g_{k,\ep}^\pm(y,y_0) = R_{1,k,\ep}^\pm(y,y_0)$, we have
\begin{align*}
g_{k,\ep}^\pm(y,y_0) &= \frac{2k}{v'(y_0)}\int_{I_3(y_0)} \G_{k,\ep}^\pm(y,y_0,z)\frac{f(z)}{\xi}\d z - \int_0^2\G_{k,\ep}^\pm(y,y_0,z)f(z) \mathrm{V}_{1,\ep}^\pm(z,y_0) \d z \\
&\quad + \frac{2k}{v'(y_0)}\int_{I_3^c(y_0)} \G_{k,\ep}^\pm(y,y_0,z)\frac{f(z)}{\xi}\d z \\
&= g_{1,k,\ep}^\pm(y,y_0) + g_{2,k,\ep}^\pm(y,y_0) + g_{3,k,\ep}^\pm.
\end{align*}
We can use Lemma \ref{lemma:R0mapsL2toLXk} to bound $\Vert g_{2,k,\ep}^\pm \Vert_{LX_k}\lesssim k^{-\frac32}\Vert f \Vert_{L^2}$. Similarly, Proposition \ref{prop:logsobolevregpartialdecomG}, Corollary \ref{cor:logsobolevregpartialdecomG} and the entanglement inequality show $\Vert g_{3,k,\ep}^\pm \Vert_{LX_k}\lesssim k^{-1}\Vert f \Vert_{L^\infty(I_3(y_0))} + k^{-\frac12}\Vert f \Vert_{L^2(I_3^c(y_0))}$. For $g_{1,k,\ep}^\pm$, we use the regularity structure of $\G_{k,\ep}^\pm$ to write
\begin{align*}
g_{1,k,\ep}^\pm(y,y_0) &=  \frac{2k}{v'(y_0)}\int_{I_3(y_0)} \left( \G_{k,\ep}^\pm\right)_\sr(y,y_0,z)f(z) \xi^{-\frac12+\gamma_0} \d z \\
&\quad + \frac{2k}{v'(y_0)}\int_{I_3(y_0)} \left( \G_{k,\ep}^\pm\right)_\s(y,y_0,z)f(z) \xi^{-\frac12+\gamma_0} \log(\xi) \Q_{\gamma_0}(\xi)\d z.
\end{align*}
The lemma then follows once we use Proposition \ref{prop:logsobolevregpartialdecomG}, Corollary \ref{cor:logsobolevregpartialdecomG} together with the observations that $|\Q_{\gamma_0}(\xi)|$ is bounded and $x^{-\frac12+\gamma_0}\log(x)$ is integrable, with uniform bounds in $|\gamma_0|\leq \tfrac14$.
\end{proof}

\subsection{Estimates for the error operator}
We are now able to bound the main error term in the reduction of the Taylor-Goldstein operator.

\begin{proposition}\label{prop:TmapsXktoXk}
Let $k\geq 1$, $y_0\in I_S\cup I_W$ and $f\in Z_k$. There holds
\begin{align*}
\Vert (T_{k,\ep}^\pm f)(\cdot,y_0) \Vert_{X_k} \lesssim k^{-1}\Vert f \Vert_{Z_k},
\end{align*}
uniformly in $0<\ep<\ep_*$ and $y_0\in I_S\cup I_W$. Similarly, for $y_0\in I_M$ and $f \in LZ_k$, we have
\begin{align*}
\Vert (T_{k,\ep}^\pm f)(\cdot,y_0) \Vert_{LX_k} \lesssim k^{-1}\Vert f \Vert_{LZ_k},
\end{align*}
uniformly in $0<\ep<\ep_*$ and $y_0\in I_M$.
\end{proposition}

\begin{proof}
We recall the error operator $\scE_{k,\ep}^\pm$ and we consider
\begin{align*}
\scE_{1,k,\ep}^\pm f(y,y_0) &:= \frac{v''(y)}{v(y)-v(y_0)\pm i\ep}f(y,y_0) , \\
\scE_{2,k,\ep}^\pm f(y,y_0)&:= \frac{\mathrm{P}(y) - \mathrm{P}(y_0)}{(v(y) - v(y_0) \pm i\ep)^2}f(y,y_0) \\
\scE_{3,k,\ep}^\pm f(y,y_0) &:= \left( \frac{\mathrm{P}(y_0)}{(v(y) - v(y_0) \pm i\ep)^2} - \frac{\cJ(y_0)}{(y-y_0\pm i\ep_0)^2}\right)f(y,y_0). 
\end{align*}
for which there holds
\begin{align*}
\scE_{k,\ep}^\pm = -\scE_{1,k,\ep}^\pm + \scE_{2,k,\ep}^\pm + \scE_{3,k,\ep}^\pm.
\end{align*}
For $g_{k,\ep}^\pm(y,y_0)= T_{k,\ep}^\pm(y,y_0)f(y,y_0)$ and $g_{j,k,\ep}^\pm(y,y_0)=R_{0,k,\ep}^\pm\sc E_{j,k,\ep}f(y,y_0)$, we clearly have $g_{k,\ep}^\pm= -g_{1,k,\ep}^\pm + g_{2,k,\ep}^\pm + g_{3,k,\ep}^\pm$. We further note that
\begin{align*}
g_{1,k,\ep}^\pm(y,y_0) = \left( R_{1,k,\ep}v''(\cdot)f(\cdot,y_0)\right)(y),
\end{align*}
while
\begin{align*}
g_{2,k,\ep}^\pm(y,y_0) = \left(R_{2,k,\ep}  h(y,y_0) f(\cdot, y_0) \right) (y), \quad h(y,y_0):= \mathrm{P}(y) - \mathrm{P}(y_0)
\end{align*}
and 
\begin{align*}
g_{3,k,\ep}^\pm(y,y_0) =\cJ(y_0) \left( R_{1,k,\ep} \mathrm{V}_{0,\ep}(\cdot,y_0)f(\cdot, y_0) \right)(y),
\end{align*}
where
\begin{align*}
\mathrm{V}_{0,\ep}^\pm(y,y_0) = \frac{(v(y) - v(y_0) - v'(y_0)(y-y_0))(v(y)-v(y_0) + v'(y_0)(y-y_0) \pm 2 i\ep)}{(v(y) - v(y_0) \pm i\ep)(y-y_0 \pm i\ep_0)^2},
\end{align*}
with $\mathrm{V}_{0,\ep}^\pm(\cdot,y_0)\in L^\infty_y$ uniformly in $y_0$ and $\ep>0$. Since 
\begin{align*}
k^{-1}\Vert f \Vert_{L^\infty(I_3(y_0))} + k^{-\frac12}\Vert f \Vert_{L^2(I_3^c(y_0))} \lesssim k^{-1}\Vert f \Vert_{Z_k},
\end{align*}
the $X_k$ estimates follows once we apply Corollary \ref{cor:R1mapsCXktoXk} to $g_{1,k,\ep}^\pm$, Lemma \ref{lemma:R2mapsC1XktoXk} to $g_{2,k,\ep}^\pm$ and Lemma \ref{lemma:R1mapsPLinfL2toXk} to $g_{3,k,\ep}^\pm$. Instead, for the $LX_k$ estimates we now write
\begin{align*}
g_{2,k,\ep}^\pm(y,y_0) = \left( R_{1,k,\ep}^\pm h(\cdot,y_0) f(\cdot, y_0) \right)(y), \quad h(y,y_0) = \frac{\P(y) - \P(y_0)}{v(y) - v(y_0)\pm i\ep}\in L^\infty_{y,y_0}.  
\end{align*}
and we use Lemma \ref{lemma:R1mapsLinfL2toLXk} for all $g_{j,k\ep}^\pm$, with $j=1,2,3$.
\end{proof}

\section{The homogeneous Taylor-Goldstein equation}\label{sec:homTG}

In this section we study solutions to the homogeneous Taylor-Goldstein equation
\begin{equation}\tag{hTG}\label{eq:homTG}
\l( \p_y^2 - k^2 -\frac{v''(y)}{v(y)-v(y_0)\pm i\ep}+\frac{\P(y)}{(v(y)-v(y_0)\pm i\ep)^2}\r) \varphi=0.
\end{equation}
Our arguments are by now standard and follow closely those presented in \cite{WZZ18} for the 2D Euler equations. Inspired by the Euler index formula, we define
\begin{equation}\label{eq:defphisigma}
\begin{split}
\phi_{\sr,k,\ep}^\pm(y,y_0) &:= (v(y) - v(y_0) \pm i\ep)^{\frac12+\gamma_0}\phi_{\sr,1,k,\ep}^\pm(y,y_0) , \\
\phi_{\s,k,\ep}^\pm(y,y_0) &:= (v(y) - v(y_0) \pm i\ep)^{\frac12-\gamma_0}\phi_{\s,1,k,\ep}^\pm(y,y_0)
\end{split}
\end{equation}
where we recall that $\gamma_0 = \sqrt{\frac14 - \cJ(y_0)}$. Then, $\phi_{\sr,k,\ep}^\pm$ and $\phi_{\s,k,\ep}^\pm$ are solutions to \eqref{eq:homTG} if $ \phi_{\sr,1,k,\ep}^\pm$ and $ \phi_{\s,1,k,\ep}^\pm$ are solutions to 
\begin{equation}\label{eq:homTGphir1}
\partial_y \left( (v(y) - v(y_0) \pm i\ep)^{1+2\gamma_0}\partial_y \phi_{\sr,1,k,\ep}^\pm \right) - k^2 \mathrm{F}_{\sr,k,\ep}^\pm(y,y_0)(v(y) - v(y_0) \pm i\ep_0)^{2\gamma_0}\phi_{\sr,1,k,\ep}^\pm = 0
\end{equation}
and
\begin{equation}\label{eq:homTGphis1}
\partial_y \left( (v(y) - v(y_0) \pm i\ep)^{1-2\gamma_0}\partial_y \phi_{\s,1,k,\ep}^\pm \right) - k^2 \mathrm{F}_{\s,k,\ep}^\pm(y,y_0)(v(y) - v(y_0) \pm i\ep_0)^{-2\gamma_0}\phi_{\s,1,k,\ep}^\pm = 0,
\end{equation}
respectively. Here, we have defined
\begin{equation}\label{eq:defFsr}
\mathrm{F}_{\sr,k,\ep}^\pm(y,y_0) := v(y) - v(y_0) \pm i\ep - \frac{(v'(y))^2}{k^2}\frac{\cJ(y) - \cJ(y_0)}{v(y) - v(y_0) \pm i\ep} +  \frac{1-2\gamma_0}{2k^2}v''(y)
\end{equation}
and
\begin{equation}\label{eq:defFs}
\mathrm{F}_{\s,k,\ep}^\pm(y,y_0) := v(y) - v(y_0) \pm i\ep - \frac{(v'(y))^2}{k^2}\frac{\cJ(y) - \cJ(y_0)}{v(y) - v(y_0) \pm i\ep} + \frac{1+2\gamma_0}{2k^2}v''(y).
\end{equation}
In order to show the existence of solutions to \eqref{eq:homTGphir1} and \eqref{eq:homTGphis1}, we define the operators
\begin{equation}\label{eq:defcTr1}
\cT_{\sr,1,k,\ep}^\pm \phi (y,y_0) = \frac{1}{(v(y) - v(y_0)\pm i\ep)^{1+2\gamma_0}} \int_{y_0}^y \mathrm{F}_{\sr,k,\ep}^\pm(z,y_0) (v(z) - v(y_0) \pm i\ep)^{2\gamma_0}\phi(z,y_0) \d z
\end{equation}
and similarly
\begin{equation}\label{eq:defcTs1}
\cT_{\s,1,k,\ep}^\pm \phi (y,y_0) = \frac{1}{(v(y) - v(y_0)\pm i\ep)^{1-2\gamma_0}} \int_{y_0}^y \mathrm{F}_{\s,k,\ep}^\pm(z,y_0) (v(z) - v(y_0) \pm i\ep)^{-2\gamma_0}\phi(z,y_0) \d z
\end{equation}
Moreover, for
\begin{align*}
\cT_{0}\phi(y,y_0) := \int_{y_0}^y \phi(z,y_0) \d z
\end{align*}
we further define
\begin{align*}
\cT_{\sr,k,\ep}^\pm := \cT_0 \circ \cT_{\sr,1,k,\ep}^\pm, \quad \cT_{\s,k,\ep}^\pm := \cT_0 \circ \cT_{\s,1,k,\ep}^\pm.
\end{align*}
Let $y,y_0\in[0,2]$  and $0<\ep<\ep_*$. For $A\geq 1$ and a function $f=f_\ep(y,y_0)$ we define the space
\begin{align*}
\Vert f \Vert_{Y} := \sup_{y\in[0,2]}\sup_{0<\ep<\ep_0} \left| \frac{f_\ep(y,y_0)}{\cosh(A(y-y_0))} \right|.
\end{align*}
The next Lemma shows the mapping properties of $\cT_{\sigma,k,\ep}^\pm$ with respect to the space $Y$.

\begin{lemma}\label{lemma:cTmapsYtoY}
Let $k\geq 1$ and $f\in Y$. Then,
\begin{align*}
 \sup_{y_0\in [\vartheta_1,\vartheta_2]}\Vert \cT_0 f \Vert_Y \leq \frac{\Vert f \Vert_Y}{A}, \quad \sup_{y_0\in [\vartheta_1,\vartheta_2]}\Vert \cT_{\sr,1,k,\ep}^\pm f \Vert_Y \lesssim \Vert f \Vert_Y,
\end{align*}
and hence $ \sup_{y_0\in [\vartheta_1,\vartheta_2]}\Vert \cT_{\sr,k,\ep}^\pm f \Vert_Y \lesssim \frac{\Vert f \Vert_Y}{A}$. Moreover, for $y_0\in(\vartheta_1,\vartheta_2)$ we have
\begin{align*}
\Vert \cT_{\s,1,k,\ep}^\pm f \Vert_Y \lesssim \frac{\Vert f \Vert_Y}{1-2\gamma_0}
\end{align*}
and thus $\Vert \cT_{\s,k,\ep}^\pm f \Vert_Y \lesssim \frac{\Vert f \Vert_Y}{A(1-2\gamma_0)}$.
\end{lemma}

\begin{proof}
The estimate for $T_0$ is straightforward, see \cite{WZZ18}. We prove the estimate of $\cT_{\sigma,1,k,\ep}^\pm$ for $y_0\in I_W$, since for $y_0\in I_S$ we have $\gamma_0\in i\R$ and the proof is similar and easier. Firstly, for $\sigma = \sr$,
\begin{align*}
\left| \frac{\cT_{\sr,1,k,\ep}(y,y_0) }{\cosh(A(y-y_0))}\right| &\lesssim \frac{\Vert f \Vert_Y}{\cosh(A(y-y_0))}\frac{1}{|v(y) - v(y_0) \pm i\ep|}\int_{y_0}^y \cosh(A(z-y_0)) \d z \\
&\leq \Vert f \Vert_{Y}\frac{\tanh (A(y-y_0))}{A(y-y_0)} \\
&\lesssim \Vert f \Vert_Y
\end{align*}
because $\Vert F_{\sr,k,\ep}(\cdot,y_0) \Vert_{L^\infty}\lesssim 1$ uniformly in $y_0$ and $\ep>0$ and $|v(z)-v(y_0) \pm i\ep|\leq |v(y) - v(y_0)\pm i\ep|$ for all $z\in [y_0,y]$ since $v(\cdot)$ is monotone increasing. Secondly, for $\sigma = \s$ we argue similarly, now integrating by parts once to obtain
\begin{align*}
\left| \frac{\cT_{\s,1,k,\ep}(y,y_0) }{\cosh(A(y-y_0))}\right| &\lesssim \frac{\Vert f \Vert_Y}{\cosh(A(y-y_0))}\frac{1}{|v(y)-v(y_0)\pm i\ep|^{1-2\gamma_0}}\int_{y_0}^y\cosh(A(z-y_0)) (z-y_0)^{-2\gamma_0} \d z \\
&\lesssim \frac{\Vert f \Vert_Y}{1-2\gamma_0} + \frac{\Vert f \Vert_Y}{\cosh(A(y-y_0))}\frac{1}{1-2\gamma_0} \int_{y_0}^y A\sinh(A(z-y_0)) \frac{(z-y_0)^{1-2\gamma_0}}{|v(y) - v(y_0) \pm i\ep|}\d z \\
&\lesssim \frac{\Vert f \Vert_Y}{1-2\gamma_0}.
\end{align*}
With this, the lemma is proved.
\end{proof}

\begin{proposition}\label{prop:existencesolphisigma}
Let $y_0\in [\vartheta_1,\vartheta_2]$ and $\ep>0$. Then there exist a unique solution $\phi_{\sr,1,k,\ep}^\pm\in Y$ to \eqref{eq:homTGphir1}  such that
\begin{align*}
\phi_{\sr,1,k,\ep}^\pm(y_0,y_0) = 1, \quad \partial_y\phi_{\sr,1,k,\ep}^\pm(y_0,y_0) = 0
\end{align*}
with $\Vert \phi_{\sr,1,k,\ep}^\pm \Vert_Y \leq C_{\sr}$, for some constant $C_{\sr}>0$ independent of $y_0$ and $\ep>0$.  Moreover, for $y_0\in (\vartheta_1,\vartheta_2)$ there also exist a unique solution $\phi_{\s,1,k,\ep}^\pm\in Y$ to \eqref{eq:homTGphis1} with 
\begin{align*}
\phi_{\s,1,k,\ep}^\pm(y_0,y_0) = 1, \quad \partial_y\phi_{\s,1,k,\ep}^\pm(y_0,y_0) = 0
\end{align*}
and $\Vert \phi_{\sr,1,k,\ep}^\pm \Vert_Y \leq \frac{C_{\s}}{1-2\mu_0}$ for some constant $C_\s>0$ independent of $y_0$ and $\ep>0$. 
\end{proposition}

\begin{proof}
Integrating \eqref{eq:homTGphir1} and \eqref{eq:homTGphis1}, we see that
\begin{align}\label{eq:fixedpointphisigma}
\phi_{\sr,1,k,\ep}^\pm (y,y_0) = 1 +k^2 \cT_{\sr,k,\ep}^\pm \phi_{\sr,1,k,\ep}^\pm (y,y_0), \quad \phi_{\sr,1,k,\ep}^\pm (y,y_0) = 1 + k^2 \cT_{\sr,k,\ep}^\pm \phi_{\sr,1,k,\ep}^\pm (y,y_0).
\end{align}
Now, from Lemma \ref{lemma:cTmapsYtoY} we know that $I-k^2\cT_{\sigma,k,\ep}$ is an invertible operator in $Y$, for $A$ large enough. Hence, we define
\begin{align}\label{eq:geometricphisigma}
\phi_{\sigma,1,k,\ep}^\pm := (I- k^2 \cT_{\sigma,k,\ep}^\pm)^{-1} 1\in Y,
\end{align}
with the bounds $\Vert \phi_{\sr,1,k,\ep}^\pm (y,y_0) \Vert_Y \leq C_\sr$ uniformly for all $y_0\in[\vartheta_1,\vartheta_2]$ and $\Vert \phi_{\s,1,k,\ep}^\pm (y,y_0) \Vert_Y \leq \frac{C_\s}{1-2\gamma_0}$ for all $y_0\in(\vartheta_1,\vartheta_2)$, for some $C_\sr, C_\s>0$.
\end{proof}

Once we have established the existence of $\phi_{\sigma,1,k,\ep}^\pm$ and its uniform bounds in $Y$, we next state two fundamental properties that will be used throughout the manuscript.

\begin{corollary}
Let $k\geq 1$ and $\ep>0$. Then,
\begin{align*}
\sup_{y_0\in [\vartheta_1,\vartheta_2]} \Vert  \partial_y \phi_{\sigma,1,k,\ep}^\pm(\cdot,y_0) \Vert_{L^\infty_y(0,2)} \lesssim_k 1,
\end{align*}
In particular, $\left| \phi_{\sigma,1,k,\ep}^\pm(\cdot,y_0) - 1 \right| \lesssim_k {|y-y_0|}$ for all $y_0\in[\vartheta_1, \vartheta_2]$. Similarly, for $y_0\in (\vartheta_1,\vartheta_2)$ there holds
\begin{align*}
 \Vert  \partial_y \phi_{\sigma,1,k,\ep}^\pm(\cdot,y_0) \Vert_{L^\infty_y(0,2)} \lesssim_k \frac{1}{1-2\mu_0},
\end{align*}
and thus $\left| \phi_{\s,1,k,\ep}^\pm(\cdot,y_0) - 1 \right| \lesssim_k \frac{|y-y_0|}{1-2\mu_0}$ as well.
\end{corollary}

\begin{proof}
We note that $\partial_y \phi_{\sigma,1,k,\ep}^\pm = k^2\cT_{\sigma,1,k,\ep}^\pm \phi_{\sigma,1,k,\ep}^\pm$ so that the result follows from Lemma \ref{lemma:cTmapsYtoY} and Proposition \ref{prop:existencesolphisigma}. 
\end{proof}

\subsection{Continuity of the homogeneous solutions}
Let $\g>0$, we define
\begin{align*}
\mathcal{Y}_I := [0,2]\times (\vartheta_1,\vartheta_2)\times [0,1]\times (0,\g], \quad \mathcal{Y}_{\overline{I}} := [0,2]\times [\vartheta_1,\vartheta_2] \times [0,1]\times [0,\g],
\end{align*}
we address the continuity of the solutions $\phi_{\s,1,k,\ep}^\pm(y,y_0)$ and $\phi_{\sr,1,k,\ep}^\pm(y,y_0)$ in the spaces $\mathcal{Y}_I$ and $\mathcal{Y}_{\overline{I}}$, respectively, where we further record the dependence on the physical gravity $\g >0$.
\begin{proposition}\label{prop:continuityphi}
Let $k\geq 1$. Then, 
\begin{align*}
\phi_{\sr,1,k}^\pm(y,y_0,\ep,\tilde\g) := \phi_{\sr,1,k,\ep,\tilde\g}^\pm(y,y_0)\in C(\mathcal{Y}_{\overline{I}} )
\end{align*}
and
\begin{align*}
\phi_{\s,1,k}^\pm(y,y_0,\ep,\tilde\g) := \phi_{\s,1,k,\ep,\tilde\g}^\pm(y,y_0)\in C(\mathcal{Y}_{{I}}  )
\end{align*}
\end{proposition}

\begin{proof}
We note that
\begin{align*}
\phi_{\sigma,1,k,\ep,\tilde\g}^\pm := (I- k^2 \cT_{\sigma,k,\ep,\tilde\g}^\pm)^{-1} 1 = \sum_{n\geq 0}  \left(k^2 \cT_{\sigma,k,\ep,\tilde\g}^\pm\right)^n 1,
\end{align*}
where convergence is uniform. Hence, the lemma follows if we show that
\begin{align*}
\cT_{\sr,k}^\pm :=\cT_{\sr,k,\ep,\tilde\g}^\pm : C(\mathcal{Y}_{\overline{I}} ) \rightarrow C(\mathcal{Y}_{\overline{I}} ) \quad \text{ and }\quad  \cT_{\s,k}^\pm :=\cT_{\s,k,\ep,\tilde\g}^\pm : C(\mathcal{Y}_{{I}} ) \rightarrow C(\mathcal{Y}_{{I}} ).
\end{align*}
To that purpose, we first argue for $\cT_{\sr,k}^\pm$, with $y_0\in [\vartheta_1,\vartheta_2]$, $\ep\in[0,1]$ and $\tilde\g\geq 0$ and we observe that
\begin{equation*}
\left( \cT_{\sr, k,\ep,\tilde\g}^\pm f\right) (y,y_0) = \int_0^1\int_0^1 \left( \mathcal{K}_{0,\sr,\ep,\tilde\g}^\pm(s,t,y,y_0) + \mathcal{K}_{1,\sr,k,\ep,\tilde\g}(s,t,y,y_0) \right)f(y_0 + st(y-y_0),y_0) \,\d s\, \d t,
\end{equation*}
where
\begin{equation*}
\mathcal{K}_{0,\sr,\ep,\tilde\g}^\pm(s,t,y,y_0)  := t(y-y_0)^2\left( \frac{v(y_0 + st(y-y_0)) - v(y_0) \pm i\ep}{v(y_0 + t(y-y_0)) - v(y_0) \pm i\ep}\right)^{1+2\gamma_0(\tilde\g)}
\end{equation*}
and
\begin{align*}
\mathcal{K}_{1,\sr,k,\ep,\tilde\g}(s,t,y,y_0) &:= \mathcal{K}_{2,\sr,k,\ep,\tilde\g}^\pm(s,y,y,y_0) \left( \frac{v(y_0 + st(y-y_0)) - v(y_0) \pm i\ep}{v(y_0 + t(y-y_0)) - v(y_0) \pm i\ep}\right)^{2\gamma_0(\tilde\g)}  
\end{align*}
with
\begin{align*}
\mathcal{K}_{2,\sr,k,\ep,\tilde\g}(s,t,y,y_0) &:= \frac{1}{k^2} \frac{t(y-y_0)^2 v''(y_0+st(y-y_0))}{v(y_0+t(y-y_0)) - v(y_0)\pm i\ep}\left( \frac12 - \gamma_0(\tilde\g)\right)  \\
&\quad -\frac{1}{k^2} \frac{t(y-y_0)^2}{v(y_0 + t(y-y_0)) - v(y_0) \pm i\ep} \frac{  \tilde\g\widetilde\cJ(y_0+st(y-y_0)) - \tilde\g\widetilde\cJ(y_0)}{v(y_0 + st(y-y_0)) - v(y_0) \pm i\ep}. 
\end{align*}
In particular, we easily see that 
\begin{align*}
\left| \mathcal{K}_{n,\sr,k,\ep,\tilde\g}(s,t,y,y_0) \right| \lesssim 1,
\end{align*}
for $n=0,1,2$, uniformly for all $(y,y_0,\ep,\tilde\g)\in \mathcal{Y}_{\overline{I}}$ and $s,t\in (0,1)$. Hence, by the Dominated Convergence Theorem we see that $\cT_{\sr,k}^\pm$ maps $C(\mathcal{Y}_{\overline{I}} ) $ to itself.

Regarding $\cT_{\s,k}^\pm$, we define $\mathcal{K}_{n,\s,k,\ep,\g}$ accordingly. We note that $\left| \mathcal{K}_{2,\sr,k,\ep,\tilde\g}(s,t,y,y_0) \right| \lesssim t\tilde\g|y-y_0|$, so that
\begin{align*}
t|y-y_0|\int_0^1 \left| \frac{v(y_0 + st(y-y_0)) - v(y_0) \pm i\ep}{v(y_0 + t(y-y_0)) - v(y_0) \pm i\ep}\right|^{-2\gamma_0(\tilde\g)}   \d s \lesssim (t|y-y_0|)^{1-2\gamma_0(\tilde\g)} \lesssim 1\in L^1_t(0,1)
\end{align*}
and thus we can use the Dominated Convergence Theorem for the term involving $\mathcal{K}_{1,\s,k,\ep,\tilde\g}(s,t,y,y_0)$. On the other hand, we still have $\left| \mathcal{K}_{0,\s,k,\ep,\tilde\g}(s,t,y,y_0) \right| \lesssim 1$ and thus we can use the Dominated Convergence Theorem for this contribution as well. With this, the proof is concluded.
\end{proof}

\subsection{Homogeneous solutions in the mild regime}\label{subsec:mild}
The two homogeneous solutions $\phi_{\sr,k,\ep}^\pm$ and $\phi_{\s,k,\ep}^\pm$ have Wronskian $\mathcal{W}\lbrace \phi_{\sr,k,\ep}^\pm(\cdot,y_0),\, \phi_{\s,k,\ep}^\pm(\cdot,y_0) \rbrace = 2\gamma_0v'(y_0)$ and thus they lose its linear independence as $\gamma_0\rightarrow 0$. Hence, they do not constitute a suitable pair of fundamental solutions with which we can construct solutions to the in-homogeneous equation. Instead, for $y_0\neq \varpi_1, \varpi_2$,  we define
\begin{equation}\label{eq:defphiL}
\phi_{\mathrm{L},k,\ep}^\pm(y,y_0) := \frac{\phi_{\sr,k,\ep}^\pm(y,y_0) - \phi_{\s,k,\ep}^\pm(y,y_0)}{2\gamma_0}
\end{equation}
for which a simple computation shows that now $\mathcal{W}\lbrace \phi_{\sr,k,\ep}^\pm(\cdot,y_0), \phi_{\rL,k,\ep}^\pm(\cdot,y_0) \rbrace = v'(y_0)$, for all $y_0\in I_M$. Moreover,
\begin{align*}
\phi_{\rL,k,\ep}^\pm(y,y_0) &= (v(y) - v(y_0) + i\ep)^{\frac12-\gamma_0} \frac{(v(y) - v(y_0) + i\ep)^{2\gamma_0} -1}{2\gamma_0}\phi_{\s,1,k,\ep}^\pm(y,y_0) \\
&\quad+ (v(y) - v(y_0) + i\ep)^{\frac12+\gamma_0}\frac{\phi_{\sr,1,k,\ep}^\pm(y,c) - \phi_{\s,1,k,\ep}^\pm(y,y_0)}{2\gamma_0} ,
\end{align*}
and for 
\begin{align*}
\phi_{\rL,1,k,\ep}^\pm(y,y_0) := \frac{\phi_{\sr,1,k,\ep}^\pm(y,c) - \phi_{\s,1,k,\ep}^\pm(y,y_0)}{2\gamma_0}, 
\end{align*}
we have that
\begin{equation}\label{eq:fixedpointeqphirL1kep}
\phi_{\rL,1,k,\ep}^\pm(y,y_0) = k^2 \cT_{\s,k,\ep}^\pm \phi_{\rL,1,k,\ep}^\pm(y,y_0) +k^2 \left(\frac{\cT_{\sr,k,\ep}^\pm-T_{\s,k,\ep}^\pm}{2\gamma_0} \right)\phi_{\sr,1,k,\ep}(y,y_0)
\end{equation}
and thus we can write
\begin{equation}\label{eq:phirLformula}
\phi_{\rL,1,k,\ep}^\pm(y,y_0) = k^2\left( I - k^2 \cT_{\s,k,\ep}^\pm\right)^{-1}\left( \frac{\cT_{\sr,k,\ep}^\pm-T_{\s,k,\ep}^\pm}{2\gamma_0}\right)\phi_{\sr,1,k,\ep}^\pm(y,y_0)
\end{equation}
The following results shows the existence of $\phi_{\rL,1,k,\ep}^\pm$.

\begin{proposition}\label{prop:existencephirL1kep}
Let $y_0\in I_M$ and $\ep>0$. Then there exists a solution $\phi_{\rL,1,k,\ep}^\pm\in Y$ to  \eqref{eq:fixedpointeqphirL1kep} with $
\phi_{\rL,1,k,\ep}^\pm(y_0,y_0) = 1$ and  $\partial_y\phi_{\rL,1,k,\ep}^\pm(y_0,y_0) = 0$ such that $\Vert \phi_{\rL,1,k,\ep}^\pm \Vert_Y \leq C$, for some constant $C>0$ independent of $y_0$ and $\ep>0$. 
\end{proposition}

\begin{proof}
The proposition follows from Proposition \ref{prop:existencesolphisigma} once we show that
\begin{align*}
\left \Vert \left(\frac{\cT_{\sr,k,\ep}^\pm-T_{\s,k,\ep}^\pm}{2\gamma_0} \right)\phi_{\sr,1,k,\ep}^\pm \right\Vert_Y \lesssim A^{-\frac12}\Vert \phi_{\sr,1,k,\ep}^\pm \Vert_Y
\end{align*}
for some implicit constant independent of $y_0\in I_M$ and $\ep>0$. To that end, recalling \eqref{eq:defcTr1} and \eqref{eq:defcTs1}, together with \eqref{eq:defFsr} and \eqref{eq:defFs}, we have
\begin{equation}\label{eq:cTrminuscTs}
\begin{split}
&\left(\frac{\cT_{\sr,1,k,\ep}^\pm-T_{\s,1,k,\ep}^\pm}{2\gamma_0} \right)\phi_{\sr,1,k,\ep}^\pm(y,y_0) \\
&=  \frac{1}{v(y) - v(y_0) \pm i\ep}\int_{y_0}^y \mathrm{F}_{\sr,k,\ep}^\pm(y,y_0)  \left(\frac{\left( \frac{v(z) - v(y_0) \pm i\ep}{v(y) - v(y_0) \pm i\ep} \right)^{2\gamma_0} -\left( \frac{v(z) - v(y_0) \pm i\ep}{v(y) - v(y_0) \pm i\ep} \right)^{-2\gamma_0}}{2\gamma_0} \right) \phi_{\sr,1,k,\ep}^\pm(z,y_0) \d z \\
&\quad -\frac{1}{k^2(v(y) - v(y_0) \pm i\ep)}\int_{y_0}^y v''(z) \left( \frac{v(z) - v(y_0) \pm i\ep}{v(y) - v(y_0) \pm i\ep} \right)^{-2\gamma_0}\phi_{\sr,1,k,\ep}^\pm(z,y_0) \d z.
\end{split}
\end{equation}
Moreover, 
\begin{align*}
\left| \frac{\left( \frac{v(z) - v(y_0) \pm i\ep}{v(y) - v(y_0) \pm i\ep} \right)^{2\gamma_0} -\left( \frac{v(z) - v(y_0) \pm i\ep}{v(y) - v(y_0) \pm i\ep} \right)^{-2\gamma_0}}{2\gamma_0} \right| &\lesssim |z-y_0|^{-2\gamma_0}|v(y) - v(y_0) \pm i\ep|^{2\gamma_0} | \log (A |z-y_0|) | \\
&\quad +|z-y_0|^{-2\gamma_0}|v(y) - v(y_0) \pm i\ep|^{2\gamma_0} | \log (A | y-y_0|) |.
\end{align*}
Since $|1-2\gamma_0|$ is uniformly bounded away from zero for $y_0\in I_M$, there holds
\begin{align*}
\left| \int_{y_0}^z \cosh(A(s-y_0)) |s-y_0|^{-2\gamma_0} | \log |s-y_0| | \d s \right| &\lesssim  \cosh(A(z-y_0))|z-y_0|^{1-2\gamma_0}|\log(A|z-y_0|)| \\
&\quad+  \left| \int_{y_0}^z \cosh(A(s-y_0))|s-y_0|^{-2\gamma_0} \d s \right| \\
&\quad + \left| A\int_{y_0}^z \sinh(A(s-y_0)) \d s \right| \\
&\lesssim \cosh(A(z-y_0)) |z-y_0|^{1-2\gamma_0}\left( 1 + |\log (A|z-y_0|) \right)
\end{align*}
and
\begin{align*}
\left| \int_{y_0}^z \cosh(A(s-y_0)) |s-y_0|^{-2\gamma_0} \d s \right| &\lesssim \cosh(A(z-y_0))|z-y_0|^{1-2\gamma_0} \\
&\quad +A  \int_{y_0}^z \sinh(A(s-y_0)) |s-y_0|^{1-2\gamma_0} \d s \\
&\lesssim \cosh(A(z-y_0))|z-y_0|^{1-2\gamma_0}, 
\end{align*}
so that now
\begin{align*}
&\left| \frac{\left(\frac{\cT_{\sr,k,\ep}^\pm-T_{\s,k,\ep}^\pm}{2\gamma_0} \right)\phi_{\sr,1,k,\ep}^\pm(y,y_0)}{\cosh(A(y-y_0))} \right| \\
&\qquad\lesssim \frac{\Vert \phi_{\sr,1,k,\ep}^\pm \Vert_Y}{\cosh(A(y-y_0))}\int_{y_0}^y\frac{\cosh(A(z-y_0))|z-y_0|^{1-2\gamma_0}}{|v(z) - v(y_0) \pm i\ep|^{1-2\gamma_0}}\left(1 + |\log (A|z-y_0|) \right) \d z \\
&\qquad\lesssim \Vert \phi_{\sr,1,k,\ep}^\pm \Vert_Y\frac{\tanh(A(y-y_0))}{A} \left( 1 + |\log (A|z-y_0|)| \right) \\
&\qquad + \frac{\Vert \phi_{\sr,1,k,\ep}^\pm \Vert_Y}{\cosh(A(z-y_0))}\int_{y_0}^y \frac{\tanh(A(z-y_0))}{A(z-y_0)}\cosh(A(z-y_0)) \d z \\
&\qquad\lesssim \Vert \phi_{\sr,1,k,\ep}^\pm \Vert_Y\frac{\tanh(A(y-y_0))}{A} \left( 1 + |\log (A|z-y_0|)| \right),
\end{align*}
for all $y\in[0,2]$ and all $y_0\in I_M$ and $\ep>0$. Further observing that $|\tanh(\zeta)\log|\zeta| |\lesssim 1$ for $|\zeta|\leq 10$,  $|\zeta^{-\frac12} |\log|\zeta||\lesssim 1$ for $|\zeta|>9$ and $\tanh(\zeta)$ is uniformly bounded, we conclude that
\begin{align*}
\left \Vert \left(\frac{\cT_{\sr,k,\ep}^\pm-T_{\s,k,\ep}^\pm}{2\gamma_0} \right)\phi_{\sr,1,k,\ep}^\pm(y,y_0) \right \Vert_Y \lesssim \frac{1}{A^\frac12}\Vert \phi_{\sr,1,k,\ep}^\pm \Vert_Y
\end{align*}
and hence we can proceed as in Proposition \ref{prop:existencesolphisigma} to show the existence and the uniform bounds in $Y$ of the solution $\phi_{\rL,1,k,\ep}^\pm$ to \eqref{eq:fixedpointeqphirL1kep}. With this, the proof is finished.
\end{proof}

We next study the homogeneous solutions in the critical point $y_0=\varpi_1$ or $y_0=\varpi_2$, for which $\cJ(y_0)=\frac14$ and $\gamma_0=0$. While $\phi_{\sr,k,\ep}^\pm(y,y_0)$ is still well defined, $\phi_{\rL,k,\ep}^\pm(y,y_0)$ as given in \eqref{eq:defphiL} is not so. However, for all $y,y_0\in [0,2]$, we now set
\begin{equation}\label{eq:defphiLcrit}
\begin{split}
\phi_{\rL,k,\ep}^\pm(y,y_0) &:= (v(y) - v(y_0) \pm i\ep)^\frac12 \log( v(y) - v(y_0) \pm i\ep) \phi_{0,1,k,\ep}^\pm(y,y_0) \\
&\quad+ (v(y) - v(y_0) \pm i\ep)^\frac12\phi_{\rL,1,k,\ep}^\pm(y,y_0),
\end{split}
\end{equation}
where
\begin{equation}\label{eq:fixedpointeqphiLcrit}
\phi_{\rL,1,k,\ep}^\pm(y,y_0) = k^2 \cT_{0,k,\ep}^\pm\phi_{\rL,1,k,\ep}^\pm(y,y_0) + k^2\cT_{\rL,k,\ep}^\pm\phi_{0,1,k,\ep}^\pm(y,y_0)
\end{equation}
with 
\begin{align*}
\cT_{\rL,k,\ep}^\pm = \cT_0 \circ \cT_{\rL,1,k,\ep}^\pm, \quad \cT_{0,k,\ep}^\pm = \cT_0 \circ \cT_{0,k,1,\ep}^\pm,
\end{align*}
where
\begin{equation}\label{eq:defTrL1kep}
\begin{split}
\cT_{\rL,1,k,\ep}^\pm f(y,y_0) &= \frac{2}{v(y) - v(y_0) \pm i\ep}\int_{y_0}^{y}\mathrm{F}_{0,k,\ep}^\pm(z,y_0) \log \left( \frac{v(z) - v(y_0) \pm i\ep}{v(y) - v(y_0) \pm i\ep} \right) f(z, y_0) \d z \\
&\quad -\frac{1}{k^2}\frac{1}{v(y) - v(y_0) \pm i\ep}\int_{y_0}^y v''(z) f(z,y_0) \d z.
\end{split}
\end{equation}
and
\begin{equation}\label{eq:defT01kep}
\begin{split}
\cT_{0,1,k,\ep}^\pm f(y,y_0) &= \frac{1}{v(y) - v(y_0) \pm i\ep}\int_{y_0}^{y}\mathrm{F}_{0,k,\ep}^\pm(z,y_0)f(z, y_0) \d z 
\end{split}
\end{equation}
where we now denote
\begin{align}
\mathrm{F}_{0,k,\ep}^\pm(y,y_0) :=(v(y) - v(y_0) \pm i\ep) - \frac{(v'(y))^2}{k^2}\frac{\cJ(y)-\frac14}{(v(y) - v(y_0)\pm i\ep)^2} 
\end{align}
and $\phi_{0,1,k,\ep}^\pm(y,y_0)$ is the unique solution to $\phi_{0,1,k,\ep}^\pm(y,y_0) = 1 + k^2 \cT_{0,k,\ep}^\pm\phi_{0,k,\ep}^\pm(y,y_0)$ arguing as in Proposition~\ref{prop:existencesolphisigma} The existence, regularity and continuity in $\ep\in[0,1]$ of $\phi_{\rL,1,k,\ep}^\pm(y,y_0)$ follows from the Neumann series' representation of Proposition \ref{prop:existencesolphisigma} once it is shown that
\begin{align*}
\Vert \cT_{\rL,k,\ep}^\pm f(y,y_0) \Vert_Y \lesssim A^{-\frac12}\Vert f \Vert_Y,
\end{align*}
which is obtained arguing as in the proof of Proposition \ref{prop:existencephirL1kep}. Hence, we have

\begin{proposition}\label{prop:existencephiLcrit}
Let $k\geq 1$ and $\ep>0$. There exists a unique solution $\phi_{\rL,1,k,\ep}^\pm(y,y_0)\in Y$ to \eqref{eq:fixedpointeqphiLcrit} with
\begin{align*}
\phi_{\rL,1,k,\ep}^\pm(y_0,y_0) = 1, \quad 
\partial_y \phi_{\rL,1,k,\ep}^\pm(y_0,y_0) = 0
\end{align*}
and such that $\Vert \phi_{\rL,1,k,\ep}^\pm(y,y_0) \Vert_Y \leq C$, for some constant $C>0$ independent and $\ep>0$. Furthermore, $\phi_{\rL,1,k,\ep}(y,y_0)$ is continuous for all $(y,y_0,\ep)\in [0,2]\times[\vartheta_1,\vartheta_2]\times[0,1]$.
\end{proposition}

\begin{proof}
    We shall only provide some details about the continuity statement. From \eqref{eq:fixedpointeqphirL1kep}, we observe that we shall only show that $\cT_{\rL,k,\ep}^\pm$ maps $C([0,2]\times[\vartheta_1,\vartheta_2]\times[0,1])$ to itself, since $\cT_{\sr,k,\ep}^\pm$ already does so thanks to Proposition~\ref{prop:continuityphi} and $\phi_{\sr,k,\ep}^\pm(y,y_0)\in C([0,2]\times[\vartheta_1,\vartheta_2]\times[0,1])$. Arguing as in the proof of Proposition~\ref{prop:continuityphi} we see with the dominated convergence theorem that $\cT_{\rL,k,\ep}^\pm$ maps $C([0,2]\times[\vartheta_1,\vartheta_2]\times[0,1])$ to itself as well.
\end{proof}

By the definition of $\phi_{\rL,k,\ep}^\pm$, the usual computation shows that
\begin{align*}
\W \lbrace \phi_{\sr,k,\ep}^\pm(\cdot, y_0), \, \phi_{\rL,k,\ep}^\pm(\cdot,y_0) \rbrace = v'(y_0)\neq 0
\end{align*}
for all $\ep>0$. Moreover, 
\begin{align*}
\left( \partial_y - k^2 -\frac{v''(y)}{v(y) - v(y_0) \pm i\ep} + \frac{\P(y)}{(v(y) - v(y_0) \pm i\ep)^2} \right) \phi_{\rL,k,\ep}^\pm(y,y_0) = 0.
\end{align*}
Indeed, there holds
\begin{align*}
\textsc{TG}_{k,\ep}^\pm &\left( (v(y) - v(y_0) \pm i\ep)^\frac12 \log(v(y) - v(y_0) \pm i\ep) \phi_{\sr,1,k,\ep}^\pm(y,y_0) \right) \\
&= \frac{v''(y)}{(v(y) - v(y_0) \pm i\ep)^\frac12}\phi_{\sr,1,k,\ep}^\pm(y,y_0) + 2\frac{v'(y)}{(v(y) - v(y_0) \pm i\ep)^\frac12} k^2\cT_{\sr,1,k,\ep}^\pm \phi_{\sr,1,k,\ep}^\pm(y,y_0) \\
&= - \textsc{TG}_{k,\ep}^\pm \left( (v(y) - v(y_0) \pm i\ep)^\frac12  \phi_{\rL,1,k,\ep}^\pm(y,y_0) \right).
\end{align*}

Hence, $\phi_{\sr,k,\ep}^\pm(y,y_0)$ and $\phi_{\rL,k,\ep}^\pm(y,y_0)$ constitute a linearly independent pair of homogeneous solutions to the Taylor-Goldstein equation for all $y_0\in(\vartheta_1,\vartheta_2)$ for which $\cJ(y_0)=\frac14$.

\section{Spectrum of the Linearised Operator}\label{sec:spectrum}
In this section we characterize the spectrum of the linearised operator $L_k$ defined in \eqref{eq:linOP}. We consider the operator $L_k$ in the domain of definition
\begin{align*}
\mathcal{D}(L_k) = \left \lbrace \begin{pmatrix}
f \\ g
\end{pmatrix} \, : \, f\in H^2(0,2)\cap H^1_0(0,2), \quad g\in L^2(0,2), \quad \text{supp } g, \,\text{supp } \Delta_k f \subseteq ( \vartheta_1, \vartheta_2) \right\rbrace.
\end{align*}
Since $\text{supp }v'', \, \text{supp }\P \subseteq (\vartheta_1, \vartheta_2)$, it is immediate to see that
\begin{align*}
L_k: \mathcal{D}(L_k)\rightarrow \mathcal{D}(L_k)
\end{align*}
continuously. As usual, we say that $\lambda\in \sigma(L_k)\subset \C$ if $(L_k-\lambda)$ is not continuously invertible. The main goal of this section is to prove Theorem \ref{thm:spectrumlinop}, namely that $\sigma(L_k) = [v(\vartheta_1),v(\vartheta_2)]$, through a combination of several results. First, we write
\begin{align*}
L_k = \widetilde L_k + \mathcal{K}_k, \quad \widetilde L_k :=\begin{pmatrix}
v(y) & 0 \\
-\mathrm{P}(y) & v(y)
\end{pmatrix}, \quad \mathcal{K}_k := \begin{pmatrix}
\Delta_k^{-1}\left( [v,\Delta_k] - v''(y) \right) & \g\Delta_k^{-1} \\
0 & 0 
\end{pmatrix}
\end{align*}
{and we note that $\mathcal{K}_k$ is a compact perturbation of $\widetilde L_k$. Moreover, we easily see that $[v(\vartheta_1),v(\vartheta_2)]\subseteq\sigma_{ess}(\widetilde L_k)$ since $(\widetilde L_k-\lambda)$ is not surjective for all $\lambda\in [v(\vartheta_1),v(\vartheta_2)]$. We then conclude that $[v(\vartheta_1),v(\vartheta_2)]\subseteq \sigma(L_k)$. The rest of the section is devoted to showing that $L_k-\lambda$ is invertible for all $\lambda\not \in [v(\vartheta_1),v(\vartheta_2)]$ and that no $\lambda \in [v(\vartheta_1),v(\vartheta_2)]$ is an embedded eigenvalue of $L_k$}.

\subsection{The non-stratified region}
Firstly, we investigate those $\lambda\in \C$ such that $\Re(\lambda) \not \in (\vartheta_1, \vartheta_2)=\text{supp }\P$.

\begin{proposition}\label{prop:rangerealcomplementresolvent}
Let $v(y)$ and $\P(y)$ satisfy {H$\P$, H$v$} and $H1$. Let $\lambda\in \C$ such that $\Re(\lambda)\not \in [v(0),v(2)]$. Then, $\lambda\not\in \sigma(L_k)$.
\end{proposition}

\begin{proof}
To show that $\lambda$ is in the resolvent of $L_k$, we shall see that for any pair $\begin{pmatrix}
f \\ g
\end{pmatrix}\in \mathcal{D}(L_k)$ we can find a pair $\begin{pmatrix}
\psi \\ \rho
\end{pmatrix}\in D(L_k)$ such that
\begin{equation}\label{eq:resolventeq}
\left(L_k-\lambda\right)\begin{pmatrix}
\psi \\ \rho
\end{pmatrix} = \begin{pmatrix}
f \\ g
\end{pmatrix}
\end{equation}
with
\begin{equation}\label{eq:resolventbound}
\Vert \psi \Vert_{H^2(0,2)\cap H^1_0(0,2)} + \Vert \rho \Vert_{L^2(0,2)} \lesssim \Vert f \Vert_{H^2(0,2)} + \Vert g \Vert_{L^2(0,2)}.
\end{equation}
By definition of $L_k$, we find
\begin{equation}\label{eq:resolventrho}
\rho(y) = \frac{g(y) + \P(y)\psi(y)}{v(y) - \lambda},
\end{equation}
where $\psi(y)$ is a solution to
\begin{equation}\label{eq:resolventTG}
\partial_y^2 \psi(y) - k^2\psi(y) - \frac{v''(y)}{v(y) - \lambda}\psi(y) + \frac{\P(y)}{(v(y) - \lambda)^2}\psi(y) = \frac{h(y)}{v(y) - \lambda} - \frac{\g g(y)}{(v(y) - \lambda)^2}
\end{equation}
with boundary conditions $\psi(0) = \psi(2)=0$ and where we denote $h(y) = \Delta_k f(y)$. Firstly, since $\Re(\lambda)\not \in [v(0),v(2)]$ and $\text{supp }v'', \, \text{supp }\P, \, \text{supp } h, \, \text{supp }g \subseteq (\vartheta_1, \vartheta_2)$, we see that $\psi\in H^2(0,2)\cap H^1_0(0,2)$ is a classical solution to the equation, with $\omega = \Delta_k \psi$ having its support contained in $(\vartheta_1, \vartheta_2)$. Hence, $\rho\in L^2(0,2)$ and $\text{supp } \rho \in (\vartheta_1, \vartheta_2)$ as well. Now, in order to show the continuity of $(L_k - \lambda)^{-1}:\mathcal{D}(L_k) \rightarrow\mathcal{D}(L_k) $, we
 multiply by $\overline{\psi}$ and integrate by parts to see that
\begin{equation}\label{eq:energyideigen}
\int_0^2  |\p_y \psi|^2 + k^2|\psi|^2 + \frac{v''(y)}{v(y) - \lambda}|\psi|^2 - \int_0^2 P(y) \frac{|\psi|^2}{(v(y)-\lambda)^2} =  - \int_0^2 \frac{h(y) \overline{\psi(y)}}{v(y) - \lambda} \d y +\g \int_0^2 \frac{g(y) \overline{\psi(y)}}{(v(y) - \lambda)^2} \d y
\end{equation}
Now, since $\Re(\lambda)\not \in [v(0),v(2)]$, say $\Re(\lambda) < v(0)$ so that there holds
\begin{equation*}
\left| \int_0^2 \frac{v''(y)}{|v(y) - \lambda|}|\psi(y)|^2 \,\d y \right| \leq \int_0^2 \frac{|v''(y)|}{v(y) - v(0)}|\psi(y)|^2 \,\d y \leq \frac{\Vert v''' \Vert_{L^\infty}}{c_0}\Vert \psi \Vert_{L^2}^2 
\end{equation*}
with also
\begin{equation*}
\left|\int_0^2 \P(y) \frac{|\psi(y)|^2}{(v(y) - \lambda)^2} \, \d y\right| \leq  \frac{\Vert \P'' \Vert_{L^\infty}}{2c_0^2}\Vert \psi \Vert_{L^2}^2 
\end{equation*}
We define $\epsilon := \frac{\Vert v''' \Vert_{L^\infty}}{c_0} +  \frac{\Vert \P'' \Vert_{L^\infty}}{2c_0^2} < 1$. On the other hand, since $\text{supp }h,\, \text{supp }g \in (\vartheta_1, \vartheta_2)$, for $\vartheta=\min(\vartheta_1,2-\vartheta_2)$, we have
\begin{align*}
\left| \int_0^2 \frac{h(y) \overline{\psi}(y)}{v(y) - \lambda}\d y \right| \leq \int_0^2 \frac{|h(y)|\psi(y)|}{v(y) - v(0)} \d y \leq \frac{1}{c_0\vartheta} \int_{\vartheta_1}^{\vartheta_2}|h(y)| |\psi(y)| \leq C_{\vartheta,\epsilon}\Vert h \Vert_{L^2(0,2)}^2 + \frac{1-\epsilon}{4}\Vert \psi \Vert_{L^2(0,2)}^2
\end{align*}
and similarly
\begin{align*}
\left| \g\int_0^2 \frac{g(y) \overline{\psi}(y)}{(v(y) - \lambda)^2}\d y \right| \leq C_{\vartheta,\epsilon}\Vert g \Vert_{L^2(0,2)}^2 + \frac{1-\epsilon}{4}\Vert \psi \Vert_{L^2(0,2)}^2
\end{align*}
Thus, we conclude that
\begin{align*}
\frac{1-\epsilon}{2}\left( \int_0^2 \left( |\partial_y \psi(y) |^2 + k^2 |\psi(y)|^2 \right)\d y \right) \leq C_{\vartheta,\epsilon}\left( \Vert h \Vert_{L^2(0,2)}^2 + \Vert g \Vert_{L^2(0,2)}^2 \right).
\end{align*}
For $\Re(\lambda) > v(2)$ we argue similarly, and the lemma follows.
\end{proof}

The next results shows that under H1 the spectrum does not include the range of the shear flow restricted to the non-stratified region.

\begin{proposition}\label{prop:nonstratifiedresolvent}
Assume that $v(y)$ and $\P(y)$ satisfy {H$\P$, H$v$} and H1. Let $\lambda\in \C$ such that $\Re(\lambda) \not \in [v(\vartheta_1), v(\vartheta_2)]$. Then, $\lambda\not \in \sigma(L_k)$.
\end{proposition}

\begin{proof}
We argue as in the proof of Proposition \ref{prop:rangerealcomplementresolvent}, so that for any pair $\begin{pmatrix}
f \\ g
\end{pmatrix}\in \mathcal{D}(L_k)$ we can find a pair $\begin{pmatrix}
\psi \\ \rho
\end{pmatrix}\in D(L_k)$ such that \eqref{eq:resolventeq} holds. As before, $\psi(y)$ is given by \eqref{eq:resolventTG} and $\rho(y)$ is given by \eqref{eq:resolventrho}. Thanks to the support assumptions on $v''$, $\P$, $h$ and $g$, we see that $\psi$ is a classical solution to \eqref{eq:resolventTG} and that $\begin{pmatrix}
\psi \\ \rho
\end{pmatrix}\in \mathcal{D}(L_k)$. To prove \eqref{eq:resolventbound}, we now note that for $\Re(\lambda)\in [v(0), v(\vartheta_1))$, there is a unique $y_\lambda\in [0, \vartheta_1)$ such that $v(y_\lambda) = \Re(\lambda)$. Moreover,
\begin{equation*}
\left|\int_0^2 \P(y) \frac{|\psi(y)|^2}{(v(y) - \lambda)^2} \, \d y\right| \leq  \int_{\vartheta_1}^{\vartheta_2} \P(y) \frac{|\psi(y)|^2}{(v(y) - v(y_\lambda))^2}\, \d y \leq \frac{\Vert \P'' \Vert_{L^\infty}}{2 c_0^2} \Vert \psi \Vert_{L^2}^2
\end{equation*}
and
\begin{align*}
\left| \int_0^2 \frac{v''(y)}{v(y) - \lambda}|\psi(y)|^2 \d y \right| &\leq \int_{\vartheta_1}^{\vartheta_2} \frac{|v''(y)|}{v(y) - v(y_\lambda)}|\psi(y)|^2 \d y \leq \frac{\Vert v''' \Vert_{L^\infty(0,2)}}{c_0} \Vert \psi \Vert_{L^2(0,2)}^2
\end{align*}
for which we note that
\begin{align*}
\epsilon_1 := \frac{\Vert \P'' \Vert_{L^\infty}}{2 c_0^2} + \frac{\Vert v''' \Vert_{L^\infty(0,2)}}{c_0}  < 1.
\end{align*}
On the other hand, since $\text{supp }h,\, \text{supp }g \subseteq [\vartheta_1, \vartheta_2]$ and $y_\lambda < \vartheta_1$, we see that
\begin{align*}
\left| \int_0^2 \frac{h(y) \overline{\psi}(y)}{v(y) - \lambda} \d y \right| +  \left| \int_0^2 \frac{g(y) \overline{\psi}(y)}{(v(y) - \lambda)^2} \d y \right| &\leq \left( \frac{\Vert h \Vert_{L^2(0,2)}}{c_0(\vartheta_1-y_\lambda)} + \frac{\Vert g \Vert_{L^2(0,2)}}{c_0^2(\vartheta_1- y_\lambda)^2} \right) \Vert \psi \Vert_{L^2(0,2)} \\
&\leq \frac{1-\epsilon_1}{2}\Vert \psi \Vert_{L^2(0,2)} + C_{\epsilon_1, \vartheta_1, \lambda}\left( \Vert h \Vert_{L^2(0,2)} + \Vert g \Vert_{L^2(0,2)} \right)
\end{align*}
from which the resolvent estimate \eqref{eq:resolventbound} follows. The proof for the case $\Re(\lambda) \in (v(\vartheta_2), v(2)]$ follows the same lines. With this, the proposition is established.
\end{proof}

We can further exploit {H$\P$, H$v$} and H1 to show that the limiting values $\lambda\in\C$ with $\Re(\lambda) = v(\vartheta_1)$ or $\Re(\lambda) = v(\vartheta_2)$ are not discrete eigenvalues of $L_k$.

\begin{lemma}\label{lemma:boundarysuppnoteigen}
Let $v(y)$ and $\P(y)$ satisfy {H$\P$, H$v$} and H1. Let $\lambda\in \C$ with $\Re(\lambda) = v(\vartheta_1)$ or $\Re(\lambda) = v(\vartheta_2)$. Then,  $\lambda\not \in \sigma_{disc}(L_k)$.
\end{lemma}

\begin{proof}
We shall see that there is no non-zero solution $\begin{pmatrix}
\psi \\ \rho
\end{pmatrix}\in \mathcal{D}(L_k)$ to 
\begin{equation}\label{eq:eigenvalueTG}
\partial_y^2 \psi(y) - k^2\psi(y) - \frac{v''(y)}{v(y) - \lambda}\psi(y) + \frac{\P(y)}{(v(y) - \lambda)^2}\psi(y) = 0, \quad \rho(y) = \frac{\P(y)\psi(y)}{v(y) - \lambda}
\end{equation}
with $\psi(0)=\psi(2)=0$ and $\Re(\lambda) = v(\vartheta_1)$ or $\Re(\lambda) = v(\vartheta_2)$. Due to the support assumptions on $v''$ and $\P$, we readily see that $\begin{pmatrix}
\psi \\ \rho
\end{pmatrix} \in \mathcal{D}(L_k)$ for any such solution $\begin{pmatrix}
\psi \\ \rho
\end{pmatrix}$ to \eqref{eq:eigenvalueTG}. Say now that $\Re(\lambda) = v(\vartheta_1)$, we easily see that
\begin{align*}
\left| \int_0^2  \frac{v''(y)|\psi(y)|^2}{v(y) - v(\vartheta_1)}\d y \right| + \left| \int_0^2  \frac{\P(y)|\psi(y)|^2}{(v(y) - v(\vartheta_1))^2}\d y \right| &\leq \left( \frac{\Vert v'''\Vert_{L^\infty(0,2)}}{c_0} + \frac{\Vert \P'' \Vert_{L^\infty(0,2)}}{2c_0^2} \right) \Vert \psi \Vert_{L^2(0,2)}
\\&< \Vert \psi \Vert_{L^2(0,2)}
\end{align*}
and thus the usual energy estimate directly yields that $\psi = 0$ and thus $\rho = 0$. 
\end{proof}

\subsection{Fine properties of the homogeneous solutions}
{In this subsection we derive several key properties satisfied by the homogeneous solutions of the Taylor-Goldstein equation under {H$\P$, H$v$} and the spectral assumptions H1--H3, that will be used to rule out the existence of embedded eigenvalues in Propositions~\ref{prop:noembedweak}-\ref{prop:noembedstrong} and the existence eigenvalues with non-zero imaginary part in Proposition ~\ref{prop:noimagspectrum}. They will also be essential for the Limiting Absorption Principles of Propositions~\ref{prop:LAPZk} and \ref{prop:LAPLZk} below}.

Assumptions {H$\P$, H$v$} and H1 provide a lower bound control on the homogeneous solution $\phi_{\sr,1,k,\ep}^\pm(y,y_0)$ introduced in \eqref{eq:defphisigma} for $\vartheta_1 < y_0 < \vartheta_1 + \delta$ and $\vartheta_2 - \delta < y_0 < \vartheta_2$, for some $\delta>0$ small.

\begin{lemma}\label{lemma:lowerboundphisigma}
Let $k\geq 1$. There exists $\delta>0$ such that 
\begin{align*}
\left| \phi_{\sr,1,k,\ep}^\pm(y,y_0) \right| \geq \frac12, 
\end{align*} 
for all $y\in[0,2]$, $y_0\in [\vartheta_1, \vartheta_1+\delta) \cup (\vartheta_2 - \delta, \vartheta_2]$ and $0 \leq \ep \leq \delta$.
\end{lemma}

\begin{proof}
Thanks to the continuity of $\phi_{\sr,1,k,\ep}^\pm(y,y_0)$ with respect to $y\in[0,2]$, $y_0\in [\vartheta_1, \vartheta_2]$ and $\ep\in [0,2]$ from Proposition \ref{prop:continuityphi}, there exists some $\delta>0$ such that
\begin{align*}
\Vert \phi_{\sr,1,k,\ep}^\pm(\cdot, y_0) - \phi_{\sr,1,k,0}(\cdot, \vartheta_i)\Vert_{L^\infty(0,2)} \leq \frac12 
\end{align*}
for all $\ep\in (0,\delta)$ and all $y_0\in [\vartheta_1, \vartheta_1 + \delta)$ or $y_0\in (\vartheta_2-\delta, \vartheta_2]$. Hence, the lemma follows once we show that
\begin{align*}
\left| \phi_{\sr,1,k,0}(y, \vartheta_i) \right| \geq 1,
\end{align*}
for all $y\in(0,2)$. We now argue for $\vartheta_1$ and we recall that
\begin{align*}
\phi_{\sr,1,k,0}^\pm(y,\vartheta_1) = \sum_{n\geq 0} \left( k^2 \cT_{\sr,k,0}^\pm \right)^n(1)(y,\vartheta_1).
\end{align*}
In particular, if $\cT_{\sr,k,0}^\pm f(y,\vartheta_1)\geq 0$ whenever $f(y,\vartheta_1)\geq 0$ for all $y\in [0,2]$, we then deduce that $\phi_{\sr,k,0}^\pm(y,\vartheta_1)\geq 1$, for all $y\in[0,2]$. Thus, since
\begin{align*}
\cT_{\sr,k,0}^\pm f(y,\vartheta_1) = \int_{\vartheta_1}^y \frac{1}{(v(z) - v(\vartheta_1))^2}\int_{\vartheta_1}^z \left( k^2(v(s) - v(\vartheta_1))^2 - \P(s) \right) f(s,\vartheta_1) \d s \d z
\end{align*} 
and 
\begin{align*}
k^2(v(s) - v(\vartheta_1))^2 - \P(s) \geq c_0^2 (s-\vartheta_1)^2 - \frac{\Vert \P'' \Vert_{L^\infty(0,2)}}{2}(s-\vartheta_1)^2 \geq 0
\end{align*}
because of $\P(\vartheta) = \P'(\vartheta) = 0$ and H1, we see that $\inf_{y\in[0,2]}\cT_{\sr,k,0}^\pm f(y,\vartheta_1)\geq 0$ whenever $\inf_{y\in[0,2]} f(y,\vartheta_1)\geq 0$ and the proof is concluded.
\end{proof}

\begin{lemma}\label{lemma:nonzerophisrphis}
Let $v(y)$ and $\P(y)$ satisfy H2. Then, there exists $\delta>0$ such that
\begin{enumerate}
\item If $y_0\in \left( \vartheta_1, \varpi_1 \right]$, then $\phi_{\sr,1,k,\ep}^\pm(y,y_0)\phi_{\s,1,k,\ep}^\pm(y,y_0)\neq 0$, for all $0 \leq \ep < \delta$ and all $y\in [0,y_0]$.
\item If $y_0 \in \left[ \varpi_2, \vartheta_2 \right)$, then $\phi_{\sr,1,k,\ep}(y,y_0)\phi_{\s,1,k,\ep}^\pm(y,y_0)\neq 0$, for all $0 \leq \ep < \delta$ and all $y\in[y_0,2]$.
\end{enumerate}
\end{lemma}

\begin{proof}
We argue for $\phi_{\s,1,k,\ep}^\pm$ and for $y_0\in [\varpi_1, \vartheta_2)$, for which $\gamma_0= \mu_0 \in \left[ 0 , \frac12 \right)$, the other combinations follow the same ideas. Arguing as in the proof of Lemma \ref{lemma:lowerboundphisigma}, since $\gamma_0 < \frac12$ we appeal to Proposition \ref{prop:continuityphi} to deduce that
\begin{align*}
\Vert \phi_{\s,k,\ep}^\pm(\cdot,y_0) - \phi_{\s,k,0}^\pm(\cdot,y_0) \Vert_{L^\infty(0,2)} < \frac12
\end{align*}
for all $0 \leq \ep < \delta$, for some $\delta>0$. Hence, since
\begin{align*}
\phi_{\s,1,k,0}^\pm(y,y_0) = \sum_{n\geq 0} \left( k^2 \cT_{\s,k,0}^\pm \right)^n(1)(y,y_0).
\end{align*} 
we shall check that
\begin{equation*}
\cT_{\s,k,\ep}^\pm f(y,y_0) = \int_{y_0}^{y}\frac{k^2}{(v(z) - v(y_0) \pm i\ep)^{1-2\mu_0}}\int_{y_0}^z F_{\s,k,\ep}^\pm(s,y_0)(v(s) - v(y_0) \pm i\ep)^{-2\mu_0}f(s,y_0)\d s \d z \geq 0
\end{equation*}
for all $y\in [y_0,2]$ whenever $f(y,y_0)\geq 0$ for all $y\in [y_0,2]$, to conclude that $\phi_{\s,1,k,\ep}^\pm(y,y_0)\neq 0$ for all $y\in[y_0,2]$. Further recalling that
\begin{equation*}
F_{\s,k,0}^\pm(y,y_0) = (v(y)-v(y_0)) - \frac{v'(y)^2}{k^2} \frac{\cJ(y)-\cJ(y_0)}{v(y)-v(y_0)} + \frac{\frac12 + \mu_0}{k^2}v''(y),
\end{equation*}
since $v'(y) \geq c_0 > 0$, $\cJ'(y) \leq 0$ and $v''(y)\geq 0$ for all $y\in [y_0,2]\subseteq [\tilde y,2]$ we deduce that $F_{\s,k,0}^\pm(y,y_0)\geq 0$ for all $y\in [y_0,2]$ and thus $\phi_{\s,1,k,0}^\pm(y,y_0)\geq 1$ for all $y\in [y_0,2]$.
\end{proof}

The following result is a consequence of Lemma \ref{lemma:nonzerophisrphis}.

\begin{proposition}\label{prop:nonzeroWweakstrat}
Let $k\geq 1$. Let $y_j\in I_W$ and $\ep_j>0$ such that $\ep_j\rightarrow 0$ and $y_j\rightarrow y_0\in I_W$ such that $\cJ(y_0) \in (0, \frac14)$. Then,
\begin{align*}
\lim_{j\rightarrow \infty} \phi_{\s,k,\ep_j}^\pm(2,y_j)\phi_{\sr,k,\ep_j}^\pm(0,y_j) - \phi_{\s,k,\ep_j}^\pm(0,y_j)\phi_{\sr,k,\ep_j}^\pm(2,y_j) \neq 0.
\end{align*}
\end{proposition}

\begin{proof}
Recalling that
\begin{align*}
\phi_{\sr,k,\ep_j}^\pm(y,y_j) &= (v(y) - v(y_j)\pm i\ep_j)^{\frac12+\gamma_j}\phi_{\sr,1,k,\ep_j}^\pm(y,y_j), \\ 
\phi_{\s,k,\ep_j}^\pm(y,y_j) &= (v(y) - v(y_j)\pm i\ep_j)^{\frac12-\gamma_j}\phi_{\s,1,k,\ep_j}^\pm(y,y_j),
\end{align*}
since $\ep_j>0$ and $\phi_{\tau,1,k,\ep_j}^\pm(y,y_j)$ is continuous in $\ep_j$ and $y_j$ uniformly in $y\in [0,2]$, confer Proposition \ref{prop:continuityphi}, we deduce that
\begin{align*}
\lim_{j\rightarrow \infty}\phi_{\sr,k,\ep_j}^\pm(0,y_j) &= -ie^{-i\mu_0\pi}(v(y_0)-v(0))^{\frac12+\mu_0}\phi_{\sr,1,k,0}^\pm(0,y_0), \\ 
\lim_{j\rightarrow \infty}\phi_{\s,k,\ep_j}^\pm(0,y_j) &= -ie^{i\mu_0\pi}(v(y_0)-v(0))^{\frac12-\mu_0} \phi_{\s,1,k,0}^\pm(0,y_0),
\end{align*}
and 
\begin{align*}
\lim_{j\rightarrow \infty}\phi_{\sr,k,\ep_j}^\pm(2,y_j) = (v(2) - v(y_0))^{\frac12+\mu_0}\phi_{\sr,1,k,0}^\pm(2,y_0), \\ 
\lim_{j\rightarrow \infty}\phi_{\s,k,\ep_j}^\pm(2,y_j) = (v(2) - v(y_0))^{\frac12-\mu_0}{\phi_{\s,1,k,0}^\pm(2,y_0)},
\end{align*}
Hence,
\begin{align*}
\lim_{j\rightarrow \infty} &\phi_{\s,k,\ep_j}^\pm(2,y_j)\phi_{\sr,k,\ep_j}^\pm(0,y_j) - \phi_{\s,k,\ep_j}^\pm(0,y_j)\phi_{\sr,k,\ep_j}^\pm(2,y_j) \\
&= -i(v(y_0)-v(0))^{\frac12-\mu_0}(v(2)-v(y_0))^{\frac12-\mu_0} \left( e^{-i\mu_0\pi} (v(y_0)-v(0))^{2\mu_0}\phi_{\sr,1,k,0}^\pm(0,y_0){\phi_{\s,1,k,0}^\pm(2,y_0)}\right. \\
&\qquad\qquad\qquad\qquad\qquad\qquad \left.- e^{i\mu_0\pi}(v(2)-v(y_0))^{2\mu_0}\phi_{\sr,1,k,0}^\pm(2,y_0){\phi_{\s,1,k,0}^\pm(0,y_0)} \right)
\end{align*}
For
\begin{align*}
a_0 &= (v(y_0)-v(0))^{2\mu_0}\phi_{\sr,1,k,0}^\pm(0,y_0){\phi_{\s,1,k,0}^\pm(2,y_0)}, \\ 
a_1 &= (v(2)-v(y_0))^{2\mu_0}\phi_{\sr,1,k,0}^\pm(2,y_0){\phi_{\s,1,k,0}^\pm(0,y_0)}
\end{align*}
we observe that 
\begin{align*}
\lim_{j\rightarrow \infty} &\phi_{\s,k,\ep_j}^\pm(2,y_j)\phi_{\sr,k,\ep_j}^\pm(0,y_j) - \phi_{\s,k,\ep_j}^\pm(0,y_j)\phi_{\sr,k,\ep_j}^\pm(2,y_j) \\
&=i(v(y_0)-v(0))^{\frac12-\mu_0}(v(2)-v(y_0))^{\frac12-\mu_0} \left( (a_0-a_1)\cos(\mu_0\pi) - i (a_0+a_1)\sin(\mu_0\pi) \right).
\end{align*}
If the above limit is zero, since $\mu_0\in (0,\frac{\pi}{2})$ we have that $\cos(\mu_0\pi)\sin(\mu_0\pi)\neq 0$ so that, together with the monotonicity of $v$, we deduce that we must have $a_0=a_1=0$. However, assume now that there exists some $y_0\in \left(y_>, \vartheta_2\right)$ for which $\gamma_0\in \left(0,\frac12\right)$ and $a_0=0$. From Lemma \ref{lemma:nonzerophisrphis},  since $\phi_{\s,1,k,0}^\pm(2,y_0)\neq 0$, this forces $\phi_{\sr,1,k,0}^\pm(0,y_0) = 0$. As $\phi_{\sr,k,0}^\pm(\cdot,y_0)$ and $\phi_{\s,k,0}^\pm(\cdot,y_0)$ are linearly independent solutions (they have non-zero Wronskian), we see that $\phi_{\s,1,k,0}^\pm(0,y_0)\neq 0$. On the other hand, Lemma \ref{lemma:nonzerophisrphis} also gives $\phi_{\sr,1,k,\ep}^\pm(2,y_0)\neq 0$ and thus we conclude that $a_1\neq 0$, and thus the limit is non-zero. This same argument shows that $a_0\neq 0$ if $a_1 = 0$ and it extends to the case where $y_0\in (\vartheta_1, \varpi_1)$.
\end{proof}

We next state a result analogous to Proposition \ref{prop:nonzeroWweakstrat}, which now addresses the case of strong stratification.

\begin{proposition}\label{prop:nonzeroWstrongstrat}
Let $k\geq 1$. Let $y_j\in I_W$ and $\ep_j>0$ such that $\ep_j\rightarrow 0$ and $y_j\rightarrow y_0\in I_S$ such that $\cJ(y_0)> \frac14$. Then,
\begin{align*}
\lim_{j\rightarrow \infty} \phi_{\s,k,\ep_j}^\pm(2,y_j)\phi_{\sr,k,\ep_j}^\pm(0,y_j) - \phi_{\s,k,\ep_j}^\pm(0,y_j)\phi_{\sr,k,\ep_j}^\pm(2,y_j) \neq 0.
\end{align*}
\end{proposition}

\begin{proof}
We note that as $\cJ(y_0)>\frac14$, there holds $\gamma_0 = i \nu_0$ with $\nu_0>0$ so that $\gamma_j = i\nu_j$ with $\nu_j>0$, for $j$ large enough. Thanks to the continuity of $\phi_{\tau,1,k,\ep_j}^\pm(y,y_j)$ and the relation $\overline{\phi_{\s,1,k,\ep_j}^\mp(y,y_j)} = \phi_{\sr,1,k,\ep_j}^\pm(y,y_j)$, which is deduced from \eqref{eq:defphisigma}, \eqref{eq:fixedpointphisigma} and \eqref{eq:geometricphisigma} we obtain 
\begin{align*}
\lim_{\ep\rightarrow 0}\phi_{\sr,1,k,\ep_j}^\pm(0,y_j) &= -ie^{\nu_0\pi}(v(y_0)-v(0))^{\frac12+i\nu_0}\phi_{\sr,1,k,0}^\pm(0,y_0), \\ 
\lim_{\ep\rightarrow 0}\phi_{\s,1,k,\ep_j}^\pm(0,y_0) &= -ie^{-\nu_0\pi}(v(y_0)-v(0))^{\frac12-i\nu_0}\overline{\phi_{\sr,1,k,0}^\mp(0,y_0)},
\end{align*}
and 
\begin{align*}
\lim_{\ep\rightarrow 0}\phi_{\sr,1,k,\ep_j}^\pm(2,y_j) &= (v(2) - v(y_0))^{\frac12+i\nu_0}\phi_{\sr,1,k,0}^\pm(2,y_0), \\ 
\lim_{\ep\rightarrow 0}\phi_{\s,1,k,\ep_j}^\pm(2,y_0) &= (v(2) - v(y_0))^{\frac12-i\nu_0}\overline{\phi_{\sr,1,k,0}^\mp(2,y_0)}.
\end{align*}
Moreover, due to the continuity of $\phi_{\sr,1,k,\ep_j}^\pm(y,y_j)$ with respect to $\ep_j>0$, there holds
\begin{align*}
{\phi_{\sr,1,k,0}^\mp(0,y_0)} = {\phi_{\sr,1,k,0}^\pm(0,y_0)}, \quad {\phi_{\sr,1,k,0}^\mp(2,y_0)} = {\phi_{\sr,1,k,0}^\pm(2,y_0)}.
\end{align*}
Then, we reach
\begin{align*}
\lim_{j\rightarrow \infty} \phi_{\s,k,\ep_j}^\pm(2,y_j)\phi_{\sr,k,\ep_j}^\pm(0,y_j) &- \phi_{\s,k,\ep_j}^\pm(0,y_j)\phi_{\sr,k,\ep_j}^\pm(2,y_j) \\
&=i(v(y_0)-v(0))^\frac12(v(2)-v(y_0))^\frac12 \left( e^{\nu_0\pi}\zeta - e^{-\nu_0\pi}\overline{\zeta}\right)
\end{align*}
where
\begin{align*}
\zeta := (v(y_0)-v(0))^{i\nu_0}(v(2)-v(y_0))^{-i\nu_0}\phi_{\sr,1,k,0}^\pm(0,y_0)\overline{\phi_{\sr,1,k,0}^\pm(2,y_0)}.
\end{align*}
Since $\nu_0>0$, we see that the limit vanishes if and only if $\zeta=0$. Hence, if the limit vanishes then either $\phi_{\sr,1,k,0}^\pm(0,y_0)=0$ or $\phi_{\sr,1,k,0}^\pm(2,y_0)=0$. If $\phi_{\sr,1,k,0}^\pm(2,y_0) = 0$, from the Wronskian invariance we have
\begin{equation*}
-2\gamma_0v'(y_0) = -\phi_{\s,k,0}^\pm(2,y_0)\partial_y\phi_{\sr,k,\ep}^\pm(2,y_0)
\end{equation*}
and thus $\phi_{\s,1,k,0}^\pm(2,y_0)\neq 0$. On the other hand,
\begin{equation*}
\phi_{\s,1,k,0}^\pm(y,y_0) = \lim_{j\rightarrow \infty} \phi_{\s,1,k,\ep_j}^\pm(y,y_j) = \lim_{j\rightarrow \infty} \overline{\phi_{\sr,1,k,\ep_j}^\mp(y,y_j)} = \overline{\phi_{\sr,1,k,0}^\pm(y,y_0)}=0,
\end{equation*}
thus reaching a contradiction. A similar contradiction is obtained supposing that $\phi_{\sr,1,k,0}^\pm(0,y_0) = 0$. As a result, $\phi_{\sr,1,k,0}^\pm(0,y_0)\phi_{\sr,1,k,0}^\pm(2,y_0)\neq 0$ and the corollary follows.
\end{proof}

To obtain the analogue of Propositions ~\ref{prop:nonzeroWweakstrat} and ~\ref{prop:nonzeroWstrongstrat} for the mild stratification regime, we first show an intermediate result.

\begin{lemma}\label{lemma:realphirphirL}
Let $k\geq 1$, $y_j,\,y_0\in I_M$, with $\cJ(y_0)=\frac14$, $y_j\rightarrow y_0$ and $\ep_j\rightarrow 0^+$ as $j\rightarrow \infty$. Then,
\begin{align*}
\phi_{\sr,1,k,0}^\pm(y,y_0):= \lim_{j\rightarrow \infty}\phi_{\sr,1,k,\ep_j}^\pm(y,y_j) \in \R, \quad \phi_{\rL,1,k,0}^\pm(y,y_0):= \lim_{j\rightarrow \infty}\phi_{\rL,1,k,\ep_j}^\pm(y,y_j) \in \R, 
\end{align*}
for all $y\in [0,2]$.
\end{lemma}

\begin{proof}
Since $\phi_{\sigma,1,k,\ep}^\pm := (I- k^2 \cT_{\sigma,k,\ep}^\pm)^{-1} 1$, we see that $\lim_{j\rightarrow\infty}\phi_{\sigma,1,k,\ep_j}^\pm(y,y_j)\in \R$ provided that $\lim_{j\rightarrow\infty}\cT_{\sigma,k,\ep_j}^\pm (1)(y,y_j)\in \R$. Indeed, recall that
\begin{align*}
\cT_{\sigma,k,\ep_j}^\pm(1)(y,y_j) = \int_{y_j}^y\frac{1}{(v(z) - v(y_j) \pm i \ep_j)^{1 + \mathbf{s}(\sigma)2\gamma_j}}\int_{y_j}^z \mathrm{F}_{\sigma,k,\ep_j}^\pm(s,y_j) (v(s) - v(y_j) \pm i\ep_j)^{\mathbf{s}(\sigma)2\gamma_j} \d s \d z,
\end{align*}
where $\mathbf{s}(\sigma)=1$ for $\sigma = \sr$ and $\mathbf{s}(\sigma) = -1$ for $\sigma = \s$. Since $\gamma_j\rightarrow 0$ as $j\rightarrow \infty$ and $\mathrm{F}_{\sr,k,0}^\pm(y,y_0):=\lim_{j\rightarrow \infty} \mathrm{F}_{\sr,k,\ep_j}^\pm(y,y_j)\in \R$ for all $y\in[0,2]$, the Dominated Convergence Theorem shows that
\begin{align*}
\lim_{j\rightarrow \infty} \cT_{\sigma,k,\ep_j}^\pm(1)(y,y_j) =\int_{y_0}^y \frac{1}{v(z) - v(y_0)}\int_{y_0}^z \mathrm{F}_{\sigma,k,0}^\pm(y,y_0) \d s \d z := \cT_{\sigma,k,0}^\pm(1)(y,y_0)\in \R
\end{align*}
and thus $\phi_{\sigma,k,0}^\pm(y,y_0):=\lim_{j\rightarrow\infty}\phi_{\sigma,1,k,\ep_j}^\pm(y,y_j)\in \R$, for $\sigma\in \lbrace \sr, \s \rbrace$. In particular, this same argument shows that 
\begin{align*}
\lim_{j\rightarrow \infty}\cT_{\s,k,\ep_j}^\pm \phi_{k,\ep_j}(y,y_j) \in \R
\end{align*}
if $ \phi_{k,\ep_j}(y,y_j)\in L^\infty$ uniformly in $y_j\in I_M$ and $\ep_j>0$ and $\lim_{j\rightarrow\infty}  \phi_{k,\ep_j}(y,y_j) \in \R$, for all $y\in[0,2]$. Hence, for the second part of the lemma, thanks to \eqref{eq:phirLformula} we just need to see that 
\begin{equation}\label{eq:limTrLphir}
\lim_{j\rightarrow \infty} \left(\frac{\cT_{\sr,k,\ep_j}^\pm-T_{\s,k,\ep_j}^\pm}{2\gamma_j} \right)\phi_{\sr,1,k,\ep}^\pm(y,y_j)\in \R
\end{equation}
for all $y\in [0,2]$. To that end, we note that
\begin{align*}
&\left| \frac{\left( \frac{v(z) - v(y_j) \pm i\ep_j}{v(y) - v(y_j) \pm i\ep_j} \right)^{2\gamma_j} -\left( \frac{v(z) - v(y_j) \pm i\ep_j}{v(y) - v(y_j) \pm i\ep_j} \right)^{-2\gamma_j}}{2\gamma_j} - 2\log \left( \frac{v(z) - v(y_j) \pm i\ep_j}{v(y) - v(y_j) \pm i\ep_j} \right) \right| \\
&\qquad\qquad\lesssim \gamma_j^2 \left|  \frac{v(z) - v(y_j) \pm i\ep_j}{v(y) - v(y_j) \pm i\ep_j} \right|^{-\frac12}  \left| \log^3 \left|  \frac{v(z) - v(y_j) \pm i\ep_j}{v(y) - v(y_j) \pm i\ep_j} \right| \right| 
\end{align*}
for all $y,z\in[0,2]$, all $y_j\in I_M$ and all $\ep_j>0$. In particular, we now have
\begin{align*}
&\left(\frac{\cT_{\sr,k,\ep_j}^\pm-T_{\s,k,\ep_j}^\pm}{2\gamma_j} \right)\phi_{\sr,1,k,\ep}^\pm(y,y_j) \\
&\qquad= 2\int_{y_j}^y \frac{1}{v(z) - v(y_j) \pm i\ep_j}\int_{y_j}^z \mathrm{F}_{\sr,k,\ep_j}^\pm(s,y_j)\phi_{\sr,1,k,\ep_j}^\pm(s,y_j) \log \left( \frac{v(s) - v(y_j) \pm i\ep_j}{v(z) - v(y_j) \pm i\ep_j} \right) \d s \d z \\
&\qquad -\int_{y_j}^y \frac{1}{k^2(v(z) - v(y_j) \pm i\ep_j)}\int_{y_j}^s v''(s) \left( \frac{v(s) - v(y_j) \pm i\ep_j}{v(z) - v(y_j) \pm i\ep_j} \right)^{-2\gamma_j}\phi_{\sr,1,k,\ep_j}^\pm(s,y_j) \d s \d z \\
&\qquad + O(\gamma_j^2) \int_{y_j}^y \frac{1}{v(z) - v(y_j) \pm i\ep_j}\int_{y_j}^z \mathrm{F}_{\sr,k,\ep_j}^\pm(s,y_j)\phi_{\sr,1,k,\ep_j}^\pm(s,y_j) \log^3 \left( \frac{v(s) - v(y_j) \pm i\ep_j}{v(z) - v(y_j) \pm i\ep_j} \right) \d s \d z. 
\end{align*}
The third contribution clearly vanishes in the limit. For the second contribution we can use the Dominated Convergence Theorem directly and show it converges to
\begin{align*}
\int_{y_0}^y \frac{1}{k^2(v(z) - v(y_0))}\int_{y_0}^z v''(s) \phi_{\sr,1,k,0}^\pm(s,y_0) \d s \d z \in \R
\end{align*}
while for the first contribution we note that
\begin{align*}
&\left| \frac{1}{v(z) - v(y_j) \pm i\ep_j}\int_{y_j}^z \mathrm{F}_{\sr,k,\ep_j}^\pm(s,y_j)\phi_{\sr,1,k,\ep_j}^\pm(s,y_j) \log \left( \frac{v(s) - v(y_j) \pm i\ep_j}{v(z) - v(y_j) \pm i\ep_j} \right) \d s \right| \\
&\quad\lesssim \frac{1}{|v(z) - v(y_j) \pm i\ep_j|}\int_{y_j}^z \left| \frac{v(s) - v(y_j) \pm i\ep_j}{v(z) - v(y_j) \pm i\ep_j} \right|^{-\frac12} \d s \\
&\quad \lesssim 1,
\end{align*}
and thus we can use the Dominated Convergence Theorem as well to show that its limit is
\begin{align*}
2\int_{y_0}^y \frac{1}{v(z) - v(y_0) }\int_{y_0}^z \mathrm{F}_{\sr,k,0}^\pm(s,y_0)\phi_{\sr,1,k,0}^\pm(s,y_0) \log \left( \frac{v(s) - v(y_0)}{v(z) - v(y_0) } \right) \d s \d z \in \R.
\end{align*}
With this we conclude that \eqref{eq:limTrLphir} holds and the lemma follows.
\end{proof}

We are now in position to present the analogue of Propositions ~\ref{prop:nonzeroWweakstrat} and ~\ref{prop:nonzeroWstrongstrat} in the mildly stratified region.

\begin{proposition}\label{prop:limitWronkianMild}
Let $k\geq 1$, $y_j\rightarrow y_0$ with $\cJ(y_0) = \frac14$, $\cJ(y_j)\neq\frac14$ and  $\ep_j\rightarrow 0^+$ as $j\rightarrow \infty$. Then,
\begin{align*}
\lim_{j\rightarrow\infty} \phi_{\sr,k,\ep_j}^\pm(0,y_j)\phi_{\rL,k,\ep_j}^\pm(2,y_j) - \phi_{\sr,k,\ep_j}^\pm(2,y_j)\phi_{\rL,k,\ep_j}^\pm(0,y_j) \neq 0.
\end{align*}
\end{proposition}

\begin{proof}
From Proposition~\ref{prop:continuityphi} and Lemma~\ref{lemma:realphirphirL} we have that
\begin{align*}
\lim_{j\rightarrow \infty}\phi_{\sr,k,\ep_j}^\pm(0,y_j)  = \pm i(v(y_0) - v(0))^{\frac12}\phi_{\sr,1,k,0}(0,y_0),
\end{align*}
and
\begin{align*}
\lim_{j\rightarrow \infty}\phi_{\rL,k,\ep_j}^\pm(2,y_j) = (v(2) - v(y_0))^{\frac12} \left( \log(v(2) - v(y_0)) \phi_{\sr,1,k,0}^\pm(2,y_0) + \phi_{\rL,1,k,0}^\pm(2,y_0) \right)
\end{align*}
while
\begin{align*}
\lim_{j\rightarrow \infty} \phi_{\sr,k,\ep_j}^\pm(2,y_j) = (v(2) - v(y_0))^{\frac12}\phi_{\sr,1,k,0}^\pm(2,y_0) 
\end{align*}
and
\begin{align*}
\lim_{j\rightarrow \infty} \phi_{\rL,k,\ep_j}^\pm(0,y_j) = \pm i(v(y_0) - v(0))^\frac12 \left(   \left(\log(v(y_0) - v(0)) \pm i \pi \right) \phi_{\sr,1,k,0}^\pm(0,y_0) + \phi_{\rL,1,k,0}^\pm(0,y_0) \right)
\end{align*}
Therefore, we have that
\begin{align*}
\lim_{j\rightarrow\infty} &\phi_{\sr,k,\ep_j}^\pm(0,y_j)\phi_{\rL,k,\ep_j}^\pm(2,y_j) - \phi_{\sr,k,\ep_j}^\pm(2,y_j)\phi_{\rL,k,\ep_j}^\pm(0,y_j)  \\
&= \pi (v(2) - v(y_0))^\frac12 (v(y_0) - v(0))^\frac12 \phi_{\sr,1,k,0}^\pm(0,y_0) \phi_{\sr,1,k,0}^\pm(2,y_0) \\
&\quad \pm i(v(y_0) - v(0))^{\frac12}(v(2) - v(y_0))^{\frac12}\phi_{\sr,1,k,0}(0,y_0) \left( \log(v(2) - v(y_0)) \phi_{\sr,1,k,0}^\pm(2,y_0) + \phi_{\rL,1,k,0}^\pm(2,y_0) \right) \\
&\quad \mp i(v(y_0) - v(0))^{\frac12}(v(2) - v(y_0))^{\frac12}\phi_{\sr,1,k,0}(2,y_0) \left( \log(v(2) - v(y_0)) \phi_{\sr,1,k,0}^\pm(0,y_0) + \phi_{\rL,1,k,0}^\pm(0,y_0) \right)
\end{align*}
with further $\phi_{\sigma,1,k,0}^\pm(y,y_0)\in \R$ for all $y\in [0,2]$ and for $\sigma\in \lbrace \sr, \rL \rbrace$. Assume now that $y_0 = \varpi_2$, so that $\phi_{\sr,1,k,0}^\pm(2,y_0)\geq 1$ due to Lemma \ref{lemma:nonzerophisrphis}. Then, if $\phi_{\sr,1,k,0}^\pm(0,y_0)\neq 0$ we see that the limit has non-zero real part. Instead, if $\phi_{\sr,1,k,0}^\pm(0,y_0)= 0$, since $\W \lbrace\phi_{\sr,k,0}^\pm,\phi_{\rL,k,0}^\pm \rbrace = v'(y_0)\neq 0$, we deduce that $\phi_{\rL,1,k,0}^\pm(0,y_0)\neq  0$ and
\begin{align*}
\lim_{j\rightarrow\infty} &\phi_{\sr,k,\ep_j}^\pm(0,y_j)\phi_{\rL,k,\ep_j}^\pm(2,y_j) - \phi_{\sr,k,\ep_j}^\pm(2,y_j)\phi_{\rL,k,\ep_j}^\pm(0,y_j) \\
&= \mp i(v(y_0) - v(0))^{\frac12}(v(2) - v(y_0))^{\frac12}\phi_{\sr,1,k,0}(2,y_0) \phi_{\rL,1,k,0}^\pm(0,y_0) \neq 0.
\end{align*}
For $y_0=\varpi_1$ we argue similarly, now observing that $\phi_{\sr,1,k,0}^\pm(0,y_0)\neq 0$. We omit the details.
\end{proof}

We shall need later in Proposition ~\ref{prop:noimagspectrum} a precise description of the limiting homogeneous solutions for the critical interphase $y_0=\varpi_1$ and $y_0=\varpi_2$. Arguing as in Lemma \ref{lemma:realphirphirL} we now have

\begin{lemma}\label{lemma:realphiLcrit}
Let $k\geq 1$, $\g_j\rightarrow \tilde\g>0$, $y_j\rightarrow y_0=\tilde\varpi_1$ or $y_0=\tilde\varpi_2$, with $\g_j\mathrm{P}(y_j)=\frac14$ and $\tilde\g\mathrm{P}(\tilde\varpi_n)=\frac14$ and $\ep_j>0$ with $\ep_j\rightarrow 0$ as $j\rightarrow \infty$. Then,
\begin{align*}
\phi_{rL,1,k,0}^\pm(y,y_0):= \lim_{j\rightarrow\infty} \phi_{\rL,1,k,\ep_j}^\pm(y,y_j) \in \R,
\end{align*}
for all $y\in[0,2]$.
\end{lemma}

{
\begin{proof} Since $\g_j\mathrm{P}(y_j)=\frac14$ and $\tilde\g\mathrm{P}(\tilde\varpi_n)=\frac14$ for all $j\geq 1$, we shall use \eqref{eq:defphiLcrit}. From \eqref{eq:fixedpointeqphiLcrit} we have that
\begin{align*}
    \phi_{\rL,1,k,\ep_j}^\pm(y,y_0) =  \left( 1-k^2 \cT_{0,k,\ep_j}^\pm \right)^{-1} \left( k^2 \cT_{\rL,k,\ep_j}^\pm\phi_{0,1,k,\ep_j}^\pm \right)(y,y_0).
\end{align*}
As $\ep_j\rightarrow 0$, we see that $\cT_{\rL,k,0}:=\lim_{j\rightarrow \infty}\cT_{\rL,k,\ep}^\pm$ given by \eqref{eq:defTrL1kep} is a real valued operator acting on real functions. Since $\phi_{0,1,k,0}(y,y_0)\in \R$ and $\cT_{0,k,0}$ is a real valued operator acting on real functions due to Lemma \ref{lemma:realphirphirL}, we deduce that $\phi_{rL,1,k,0}(y,y_0)\in \R$ as well.
\end{proof}
}
With the above lemma at hand, we next obtain 

\begin{proposition}\label{prop:limitWronskianMildCrit}
Let $k\geq 1$,  $\g_j\rightarrow \tilde\g>0$, $y_j\rightarrow y_0=\tilde\varpi_1$ or $y_0=\tilde\varpi_2$, with $\g_j\mathrm{P}(y_j)=\frac14$ and $\tilde\g\mathrm{P}(\tilde\varpi_n)=\frac14$ and $\ep_j>0$ with $\ep_j\rightarrow 0$ as $j\rightarrow \infty$. Then,
\begin{align*}
\lim_{j\rightarrow\infty} \phi_{\sr,k,\ep_j}^\pm(0,y_j)\phi_{\rL,k,\ep_j}^\pm(2,y_j) - \phi_{\sr,k,\ep_j}^\pm(2,y_j)\phi_{\rL,k,\ep_j}^\pm(0,y_j) \neq 0.
\end{align*}
\end{proposition}

\begin{proof}
Let $C_v:=(v(2) - v(y_0))^\frac12 (v(y_0) - v(0))^\frac12>0$. Due to Proposition~\ref{prop:existencephiLcrit}, a routine computation shows that for $\phi_{\rL,k,\ep_j}^\pm(\cdot,y_j)$ given by \eqref{eq:defphiLcrit}, we have
\begin{align*}
\lim_{j\rightarrow\infty} &\phi_{\sr,k,\ep_j}^\pm(0,y_j)\phi_{\rL,k,\ep_j}^\pm(2,y_j) - \phi_{\sr,k,\ep_j}^\pm(2,y_j)\phi_{\rL,k,\ep_j}^\pm(0,y_j) \\
&= C_v\pi \phi_{\sr,1,k,0}(0,y_0) \phi_{\sr,1,k,0}(2,y_0) \\
&\quad \pm C_v i  \phi_{\sr,1,k,0}(0,y_0) \phi_{\sr,1,k,0}(2,y_0) \log \left( \frac{v(2) - v(y_0)}{v(y_0) - v(0)} \right) \\
&\quad \pm C_v i \left( \phi_{\sr,1,k,0}(0,y_0) \phi_{\rL,1,k,0}(2,y_0)  - \phi_{\sr,1,k,0}(2,y_0) \phi_{\rL,1,k,0}(0,y_0)  \right) .
\end{align*}
Hence, if $\phi_{\sr,1,k,0}(0,y_0) \phi_{\sr,1,k,\ep}(2,y_0) \neq 0$ then the limit has non-zero real part, and it is thus non-zero. Assume next that $\phi_{\sr,1,k,0}(0,y_0) \phi_{\sr,1,k,0}(2,y_0) =0$ and assume further that $y_0 = \widetilde\varpi_2$, so that $\phi_{\sr,1,k,0}(2,y_0) \neq 0$ due to Lemma \ref{lemma:nonzerophisrphis} and thus $\phi_{\sr,1,k,0}(0,y_0) =0$. Since $\W\lbrace \phi_{\sr,k,0}^\pm, \phi_{\rL,k,0}^\pm \rbrace = v'(y_0)\neq 0$, we have $\phi_{\rL,1,k,0}(0,y_0)\neq 0$ and thus
\begin{align*}
\lim_{j\rightarrow\infty} &\phi_{\sr,k,\ep_j}^\pm(0,y_j)\phi_{\rL,k,\ep_j}^\pm(2,y_j) - \phi_{\sr,k,\ep_j}^\pm(2,y_j)\phi_{\rL,k,\ep_j}^\pm(0,y_j) = \mp C_v i \phi_{\sr,1,k,0}(2,y_0) \phi_{\rL,1,k,0}(0,y_0) \neq 0.
\end{align*}
For $y_0=\widetilde\varpi_1$ we now have $\phi_{\sr,1,k,0}(0,y_0)\neq 0$ and we argue similarly, we omit the details.
\end{proof}

\subsection{Absence of embedded eigenvalues}
In this section we show that $\lambda= v(y_0)$ is not an embedded eigenvalue of $L_k$, for all $y_0\in (\vartheta_1,\vartheta_2)$. In the following three propositions we assume as usual that {H$\P$-H$v$} and H1--H3 hold. 

\begin{proposition}\label{prop:noembedweak}
Let $y_0\in (\vartheta_1, \varpi_1) \cup (\varpi_2,\vartheta_2)$. Then $\lambda= v(y_0)$ is not an embedded eigenvalue of $L_k$.
\end{proposition}

\begin{proof}
    Assume without loss of generality that $y_0\in (\varpi_2,\vartheta_2)$ and that there exists a non-zero solution $\psi\in H^2\cap H^1_0$ to \eqref{eq:eigenvalueTG}, with $\lambda = v(y_0)$ so that $\gamma_0 = \mu_0\in (0,\frac14)$. For $y>y_0$ we observe that
    \begin{align*}
        \psi(y) = C \left( \phi_{\s,k,0}(2,y_0)\phi_{\sr,k,0}(y,y_0) - \phi_{\sr,k,0}(2,y_0)\phi_{\s,k,0}(y,y_0) \right)
    \end{align*}
for some $C>0$ and where we recall that
\begin{align*}
    \phi_{\sr,k,0}(y,y_0) &= (v(y) - v(y_0))^{\frac12+\mu_0}\phi_{\sr,1,k,0}(y,y_0), \\
    \phi_{\s,k,0}(y,y_0) &= (v(y) - v(y_0))^{\frac12-\mu_0}\phi_{\s,1,k,0}(y,y_0), 
\end{align*}
with $\phi_{\sigma,1,k,0}(y_0,y_0)=1$ and $\left| \partial_y \phi_{\sigma,1,k,0}(y,y_0) \right| \lesssim 1$, confer Propositions \ref{prop:existencesolphisigma} and \ref{prop:continuityphi}. We shall now see that $\psi(y)\not \in H^1(y_0,2)$. To that purpose, we note that 
\begin{align*}
    \psi'(y) &= Cv'(y_0)  \phi_{\s,k,0}(2,y_0) \left(\frac12+\mu_0 \right)(v(y) - v(y_0))^{-\frac12+\mu_0}\phi_{\sr,1,k,0}(y,y_0) \\
    &\quad - Cv'(y_0)\phi_{\sr,k,0}(2,y_0)\left(\frac12-\mu_0 \right)(v(y) - v(y_0))^{-\frac12-\mu_0} \phi_{\s,1,k,0}(y,y_0) \\
    &\quad + C  \phi_{\s,k,0}(2,y_0) (v(y) - v(y_0))^{\frac12+\mu_0}\partial_y\phi_{\sr,1,k,0}(y,y_0) \\
    &\quad - C\phi_{\sr,k,0}(2,y_0)(v(y) - v(y_0))^{\frac12-\mu_0} \partial_y\phi_{\s,1,k,0}(y,y_0)
\end{align*}
It is immediate to see that the last two contributions are in $H^1(y_0,2)$ and so we focus on the first two. Further recalling that $\phi_{\sigma,1,k,0}(y_0,y_0)=1$ and $\left| \partial_y \phi_{\sigma,1,k,0}(y,y_0) \right| \lesssim 1$, we shall see that 
\begin{align*}
    (v(y) - v(y_0))^{-\frac12-\mu_0} \left( \phi_{\s,k,0}(2,y_0) \left(\frac12+\mu_0 \right)(v(y) - v(y_0))^{2\mu_0} - \phi_{\sr,k,0}(2,y_0) \left(\frac12-\mu_0 \right)\right)\not \in L^2(y_0,2).
\end{align*}
Indeed, since $\mu_0\neq\frac12$ and also $\phi_{\sr,k,0}(2,y_0)\neq 0$, confer Lemma \ref{lemma:nonzerophisrphis}, we have that
\begin{align}
    \left| \phi_{\s,k,0}(2,y_0) \left(\frac12+\mu_0 \right)(v(y) - v(y_0))^{2\mu_0} \right| \leq \frac12 \left| \phi_{\sr,k,0}(2,y_0) \left(\frac12-\mu_0 \right) \right|
\end{align}
for $y>y_0$ sufficiently close to $y_0$. Moreover, since $(v(y) - v(y_0))^{-\frac12-\mu_0}\not \in L^2(y_0,2)$ we conclude that $\psi'(y)\not \in L^2(y_0,2)$.
\end{proof}

We next argue for the mild regime.

\begin{proposition}\label{prop:noembedmild}
Let $y_0=\varpi_1$ or $y_0=\varpi_2$. Then, $\lambda= v(y_0)$ is not an embedded eigenvalue of $L_k$.
\end{proposition}

\begin{proof}
Assume $y_0=\varpi_2$, with $\gamma_0=0$. Appealing to the homogeneous solutions constructed in \ref{prop:existencesolphisigma} and \ref{prop:existencephiLcrit}, we now have    
\begin{align*}
    \psi(y) = C \left( \phi_{\sr,k,0}(2,y_0)\phi_{\rL,k,0}(y,y_0) - \phi_{\rL,k,0}(2,y_0)\phi_{\sr,k,0}(y,y_0) \right),
\end{align*}
where we recall 
\begin{align*}
\phi_{\rL,k,0}(y,y_0) &= (v(y) - v(y_0))^\frac12 \log( v(y) - v(y_0)) \phi_{\sr,1,k,0}(y,y_0) \\
&\quad+ (v(y) - v(y_0))^\frac12\phi_{\rL,1,k,0}(y,y_0),
\end{align*}
with further $\phi_{\rL,1,k,0}(y_0,y_0) = 1$ and $\left|\partial_y \phi_{\rL,1,k,0}(y,y_0) \right| \lesssim 1$. Therefore, 
\begin{align*}
    \psi'(y) &= \frac{C}{2}v'(y)(v(y) - v(y_0))^{-\frac12} (v(2) - v(y_0))^\frac12 \phi_{\sr,1,k,0}(2,y_0) \log \left( \frac{v(y)-v(y_0)}{v(2) - v(y_0)} \right) \\
    &\quad + Cv'(y) (v(y) - v(y_0))^{-\frac12} (v(2) - v(y_0))^\frac12\phi_{\sr,1,k,0}(2,y_0) \\
    &\quad + \frac{C}{2}v'(y) (v(y) - v(y_0))^{-\frac12} (v(2) - v(y_0))^\frac12 \left( \phi_{\rL,1,k,0}(2,y_0) - \phi_{\sr,1,k,0}(2,y_0) \right)  + \widetilde\psi(y)
\end{align*}
with $\widetilde\psi\in L^2(0,2)$. For $y>y_0$ close enough, we see from the above that the most singular factor is 
\begin{align*}
    v'(y)(v(y) - v(y_0))^{-\frac12} (v(2) - v(y_0))^\frac12 \phi_{\sr,1,k,0}(2,y_0) \log (v(y)-v(y_0)) \not \in L^2(y_0,2)
\end{align*}
since $v'(y)\geq c_0>0$ and $\phi_{\sr,1,k,0}(2,y_0)\neq 0$ due to Lemma \ref{lemma:nonzerophisrphis}. With this, we see that $\psi'(y)\not \in L^2(y_0,2)$.
\end{proof}

We last show there are no embedded eigenvalues in the strongly stratified region.

\begin{proposition}\label{prop:noembedstrong}
Let $y_0\in (\varpi_1,\varpi_2)$. Then $\lambda= v(y_0)$ is not an embedded eigenvalue of $L_k$.
\end{proposition}

\begin{proof}
For $y_0\in (\varpi_1,\varpi_2)$ we now have $\gamma_0 = i\nu_0$ with $\nu_0>0$. Arguing as in the proof of Proposition~\ref{prop:noembedweak}, we shall now see that
\begin{align*}
 \phi_{\s,k,0}(2,y_0) &\left( \frac12+i\nu_0 \right)(v(y) - v(y_0))^{-\frac12+i\nu_0} \\
 &\quad - \phi_{\sr,k,0}(2,y_0) \left( \frac12-i\nu_0 \right)(v(y) - v(y_0))^{-\frac12-i\nu_0} \not \in L^2(y_0,2).
\end{align*}
Since $\overline{\phi_{\sr,k,0}(2,y_0)} = \phi_{\s,k,0}(2,y_0)\neq 0$, we must show that
\begin{align*}
    (v(y) - v(y_0))^{-\frac12}\Im \left( \phi_{\s,k,0}(2,y_0) \left( \frac12+i\nu_0 \right)(v(y) - v(y_0))^{i\nu_0} \right)\not \in L^2(y_0,2).
\end{align*}
To do so, we note that $\phi_{\s,k,0}(2,y_0) \left( \tfrac12+i\nu_0 \right)= re^{i\alpha}$, for some $r>0$ and $\alpha\in [0,2\pi)$. Hence, 
\begin{align*}
  \left| \Im \left( \phi_{\s,k,0}(2,y_0) \left( \frac12+i\nu_0 \right)(v(y) - v(y_0))^{i\nu_0} \right) \right| = r \left| \sin( \alpha  + \nu_0\log(v(y) - v(y_0))) \right|
\end{align*}
Hence, with the change of variables $x=v(y) - v(y_0)$, noting that $v'(y)\geq c_0>0$ we shall show that
\begin{align*}
\int_0^2 x^{-1}\sin^2(\alpha+\nu_0\log(x)) \d x = +\infty.   
\end{align*}
We next note that 
\begin{align*}
    \left| \sin( \alpha  + \nu_0\log(x)) \right|^2 \geq \frac12
\end{align*}
for all $x>0$ such that $\alpha + \nu_0\log(v(y) - v(y_0)) \in ( \frac{\pi}{4} - 2\pi k, \frac{3\pi}{4} - 2\pi k)$, namely for \newline
all $x\in (e^{\nu_0^{-1}(\frac{\pi}{4} - 2\pi k - \alpha)}, e^{\nu_0^{-1}(\frac{3\pi}{4} - 2\pi k - \alpha)} )$, for all $k\geq 1$. In particular,
\begin{align*}
    \int_0^2 x^{-1}\sin^2(\alpha+\nu_0\log(x)) \d x \geq \frac12\sum_{k\geq 1}^\infty \int_{e^{\nu_0^{-1}(\frac{\pi}{4} - 2\pi k - \alpha)}}^ {e^{\nu_0^{-1}(\frac{3\pi}{4} - 2\pi k - \alpha)}}x^{-1}\d x &= \frac12 \sum_{k\geq 1}^\infty \log(x) \Big|_{x=e^{\nu_0^{-1}(\frac{\pi}{4} - 2\pi k - \alpha)}}^{x=e^{\nu_0^{-1}(\frac{3\pi}{4} - 2\pi k - \alpha)}} \\
    &= \frac14 \sum_{k\geq 1}^{\infty} \frac{\pi}{\nu_0} = +\infty.
\end{align*}
With this, we conclude that $\psi'(y)\not \in L^2(y_0,2)$.
\end{proof}

\subsection{Absence of imaginary spectrum}

To conclude the proof of Theorem \ref{thm:spectrumlinop}, it only remains to see there are no eigenvalues with non-zero imaginary part.
\begin{proposition}\label{prop:noimagspectrum}
Let $v(y)$ and $\P(y)$ be such that {H$\P$, H$v$} and H1--H3 hold. Then, $\lambda\not \in \sigma(L_k)$, for all $\lambda\in \C$ with $\Im(\lambda)\neq 0$.
\end{proposition}

\begin{proof}
Let $\tilde{\g} \in [0,\g]$. Let $\lambda=\lambda_r +i\lambda_i$, with $\lambda_i\neq 0$ and $\lambda_r\in [v(0),v(2)]$. {Since $\lambda_i\neq 0$, by classical ODE theory we note that $\lambda\in \sigma (L_k)$ if and only if $\lambda$ is a discrete eigenvalue of $L_k$}. Hence, we shall see that there is no non-zero solution $\psi$ to 
\begin{equation}
\left( \p_y^2 - k^2 -\frac{v''(y)}{v(y) - \lambda}+ \frac{\tilde{\g}\mathrm{P}(y)}{(v(y)-\lambda)^2} \right)\psi = 0, \quad \psi(0)=\psi(2)=0
\end{equation}
for all $\tilde{\g}\in [0, \g]$. We begin by remarking that for $\tilde{\g}=0$, the Taylor-Goldstein equation reduces to the Rayleigh equation
\begin{equation}
\left( \p_y^2 - k^2 -\frac{v''(y)}{v(y) - \lambda} \right)\psi = 0, \quad \psi(0)=\psi(2)=0,
\end{equation}
which has no non-zero solution, thanks to the assumption H3 on the absence of eigenvalues of the linearised Euler operator. Let $\varphi(y,\lambda,\g)$ be the unique solution to 
\begin{equation}\label{eq:TGlambdaeigenvalue}
\left( \p_y^2 - k^2 -\frac{v''(y)}{v(y) - \lambda}+\frac{ \tilde{\g}\mathrm{P}(y)}{(v(y)-\lambda)^2} \right)\psi = 0, \quad \psi(2)=0,\quad \psi'(2)=1,
\end{equation}
so that $\varphi(y,\lambda,\g)$ is an eigenfunction associated to the eigenvalue $\lambda$ provided that $\varphi(0,\lambda,\g)=0$. Hence, we define the Wronskian $\W(\lambda,\g) = \varphi(0,\lambda,\g)$, which is continuous in $\lambda\in \C\setminus [v(0),v(2)])$. Set also
\begin{equation}
\mathfrak{G} = \lbrace \tilde{\g}\in [0, \g ]: \, \exists \lambda \in \mathcal{S}\setminus \mathcal{R}: \mathcal{W}(\tilde{\g},\lambda) = 0 \rbrace ,
\end{equation}
where $\mathcal{R} = [v(0),v(2)]$ is the range of the shear flow $v(y)$ and 
\begin{align*}
\mathcal{S}=\left \lbrace \lambda=\lambda_r + i\lambda \in \C: \lambda_r \in [v(0), v(2)], \quad \lambda_i\in [0,2] \right\rbrace 
\end{align*} 
For all $\lambda\in \mathcal{S}\setminus\mathcal{R}$, we have $\Im(\lambda) \neq 0$ and thus the solution $\varphi(y,\lambda,\g)$ is classical and $C^2$ smooth. We shall see that $\mathfrak{G}$ is both open and closed in the subset topology of $[0,\g]$ with respect to $\R$, a connected space. Hence, it is either $[0, \g]$ or the empty set. Since $0\not \in \mathfrak{G}$ because of H3, we conclude that $\mathfrak{G}$ is empty.

Firstly, to see that $\mathfrak{G}$ is open, assume otherwise and let $\tilde{\g}\in \mathfrak{G}$ such that for all $\delta>0$ there exists some $\mathfrak{h}\in B_\delta(\tilde{\g})$ with $\mathfrak{h}\not\in \mathfrak{G}$. Let $\lambda_0\in \mathcal{S}\setminus \mathcal{R}$ be such that $\mathcal{W}(\lambda_0,\tilde\g) = 0$. Since $\lambda_0$ is an isolated zero of  $\W(\lambda,\tilde\g)$, which is analytic in $\lambda\in \mathcal{S}\setminus\mathcal{R}$,  there exists some $\delta_1>0$ sufficiently small such that for $U_{\delta_1}(\lambda_0)=\lbrace \lambda\in \mathcal{S}\setminus\mathcal{R}: \lambda\in B_{\delta_1}(\lambda_0) \rbrace$, there holds $\mathcal{W}(\lambda,\tilde\g)\neq 0$, for all $\lambda\in \p U_{\delta_1}(\lambda_0)$. Set $\varepsilon=\frac12 \min_{\lambda\in\p U_{\delta_1}(\lambda_0)}|\mathcal{W}(\lambda,\g)|>0$. Since $\mathcal{W}(\lambda,\tilde\g)$ is continuous in both $\lambda\in \C\setminus\mathcal{R}$ and $\tilde\g\geq 0$, there exists some $\delta_2>0$ such that 
\begin{equation}
|\mathcal{W}(\lambda,\mathfrak{h}) - \mathcal{W}(\lambda,\tilde{\g})| < \varepsilon < |\mathcal{W}(\lambda,\tilde\g)|,
\end{equation}
for all $|\tilde{\g} - \mathfrak{h}| < \delta_2$. Now, by Rouche's Theorem, the number of zeroes of $\mathcal{W}(\lambda,\tilde\g)$ in $U_{\delta_1}(\lambda_0)$ (which is one, $\lambda_0$) coincides with the number of zeroes of $\mathcal{W}(\lambda,\mathfrak{h})$ in $U_{\delta_1}(\lambda_0)$, which is 0 by assumption! Thus we reach a contradiction and conclude that $\mathfrak{G}$ is open.

On the other hand, in order to show that $\mathfrak{G}$ is closed, let $(\tilde{\g}_j)_{j\geq1}\subset \mathfrak{G}$ and assume that $\tilde{\g_j}\rightarrow\tilde{\g}\in \R$, we shall see that $\tilde{\g}\in \mathfrak{G}$. Note that for each $\tilde{\g}_j$ there exists some $\lambda_j\in \mathcal{S}\setminus\mathcal{R}$ such that 
\begin{equation}\label{eq:vanishingvarphilambdaj}
\varphi(0,\lambda_j, \tilde\g_j) = \mathcal{W}(\lambda_j,\tilde{\g}_j) = 0
\end{equation}
for all $j\geq 1$. Moreover, up to a subsequence, $\lambda_j\rightarrow \lambda\in \mathcal{S}$. If $\lambda\not\in \mathcal{R}$, since $\mathcal{W}$ is continuous at $(\lambda,\tilde{\g})\in\mathcal{S}\setminus\mathcal{R}\times [0,\g]$ and in a small neighbourhood of it, there holds $\mathcal{W}(\lambda,\tilde{\g}) = 0$, and thus $\tilde{\g}\in \mathfrak{G}$, it is closed. We shall now argue that we cannot have $\lambda\in \mathcal{R}$.

\bullpar{ Case $\lambda\in \mathcal{R}$ and $\tilde{\g}>0$} We shall see this scenario cannot take place. If $\lambda_j\rightarrow\lambda\in \mathcal{R}\subset \R$, we have that $\Im(\lambda_j)\rightarrow 0$. Moreover, Proposition \ref{prop:nonstratifiedresolvent} shows that $\lambda\in [v(\vartheta_1), v(\vartheta_2)]$. We next distinguish $\lambda$ according to its stratified strength and its approximating sequence. In what follows, let $y_j\in [\vartheta_1,\vartheta_2]$ be such that $v(y_j) = \Re(\lambda_j)$ and $\ep_j= \Im(\lambda_j)>0$, so that $v(y_0) = \lambda$ and $\ep_j\rightarrow 0$. Due to H2, for all $\tilde\g_j>0$, the set $\lbrace  y\in [0,2]: \tilde{\g}_j \tilde\cJ(y) = \frac14 \rbrace$ has at most two connected components. Each connected component is either a point or a closed interval. We shall assume here that they are points, and we comment afterwards the main modifications for the interval case. Hence, let $\widetilde\varpi_j$ denote the unique solution if $\max_{y\in [0,2]} \tilde\g_j\tilde\cJ = \frac14$ and $\widetilde \varpi_{j,1} < \widetilde \varpi_{j,2}$ denote the two distinct solutions if $\max_{y\in [0,2]} \tilde\g_j\tilde\cJ > \frac14$. Likewise, we define $\widetilde \varpi$ and $\widetilde \varpi_1  <  \widetilde\varpi_2$ the solutions associated to the equation $\tilde\g \widetilde\cJ(y) = \frac14$.

\bullsubpar{Case $y_0 = \vartheta_n$ and $y_j=\vartheta_n$, for $n=1,2$} We have that $\mathcal{W}(\lambda_j,\tilde\g_j) = \varphi(0,\lambda_j,\tilde\g_j) = 0$. Since $\Re(\lambda_j) = v(\vartheta_n)$ for all $j\geq 1$ and $\tilde{\g}_j\leq \g$, we see with Lemma \ref{lemma:boundarysuppnoteigen} that $\varphi(0,\lambda_j,\tilde\g_j)=0$ forces $\varphi(y,\lambda_j ,\tilde{\g}_j)$ to be identically zero, thus $\lambda_j$ is not an eigenvalue, a contradiction with $\tilde\g_j\in \mathfrak{G}$.

\bullsubpar{Case $y_0 = \vartheta_n$ and $\P(y_j)>0$} Then, for $j$ large enough we have from Lemma \ref{lemma:lowerboundphisigma} that \newline $|\phi_{\sr,1,k,\ep_j}^+(y,y_j)|\geq \frac12$ for all $y\in [0,2]$ and thus
\begin{align*}
\varphi(y,\lambda_j,\tilde{\g}_j) = \phi_{\sr,k,\ep_j,\tilde{\g}_j}^+(y,y_j)\int_0^y \frac{1}{(\phi_{\sr,k,\ep_j,\tilde{\g}_j}^+(z,y_j))^2} \d z - \phi_{\sr,k,\ep_j,\tilde{\g}_j}^+(y,y_j) \int_0^2 \frac{1}{(\phi_{\sr,k,\ep_j,\tilde{\g}_j}^+(z,y_j))^2}\d z 
\end{align*}
is the unique well-defined solution to 
\begin{align*}
\left(\partial_y - k^2 - \frac{v''(y)}{v(y) - v(y_j)\pm i\ep_j} + \frac{\tilde{\g}_j\mathrm{P}(y)}{(v(y) - v(y_j)\pm i\ep_j)^2}\right) \phi_{\sr,k,\ep_j,\tilde{g}_j}^+(y,y_0) =0.
\end{align*}
 with $\varphi(2,\lambda_j,\tilde{\g}_j) = 0$ and $\partial_y\varphi(2,\lambda_j,\tilde{\g}_j)=1$. Thanks to Proposition \ref{prop:continuityphi} we next note that  
\begin{align*}
\phi_{\sr,1,k,0,\tilde{\g}}^\pm(y,y_0)=\lim_{j\rightarrow\infty} \phi_{\sr,1,k,\ep_j,\tilde{\g}_j}^+(y,y_j)
\end{align*}
satisfies
\begin{align*}
\phi_{\sr,1,k,0,\tilde{\g}}^+(y,y_0)= 1 + k^2\int_{y_0}^y \frac{1}{(v(z) - v(y_0))^2}\int_{y_0}^z \mathrm{F}_{\sr,k,0,\tilde{\g}}^+(s,y_0) (v(s) - v(y_0)) \phi_{\sr,1,k,0,\tilde{\g}}^+(s,y_0) \d s \d z
\end{align*}
so that $\phi_{\sr,k,0,\tilde{\g}}(y,y_0) := (v(y) - v(y_0)) \phi_{\sr,1,k,0,\tilde{\g}}^+(y,y_0)$ satisfies
\begin{align*}
\left(\partial_y - k^2 - \frac{v''(y)}{v(y) - v(y_0)} + \frac{\tilde{\g}\mathrm{P}(y)}{(v(y) - v(y_0))^2}\right) \phi_{\sr,k,0,\tilde{\g}}^+(y,y_0) =0.
\end{align*}
Moreover, we have from Proposition \ref{prop:limitWronkianFragile} and Lemma \ref{lemma:vanishingsecondinhomsol} that for $\varphi(y,\lambda,\tilde{\g}) = \lim_{j\rightarrow \infty} \varphi(y, \lambda_j, \tilde{\g}_j)$ we have
\begin{align*}
\varphi(y,\lambda,\tilde{\g}) =\phi_{\sr,k,0,\tilde{\g}}^+(y,y_0)\int_0^y \frac{1}{(\phi_{\sr,k,0,\tilde{\g}}^+(z,y_0))^2} \d z - \phi_{\sr,k,0,\tilde{\g}}^+(y,y_0) \int_0^2 \frac{1}{(\phi_{\sr,k,0,\tilde{\g}}^+(z,y_0))^2}\d z 
\end{align*}
is well-defined and, moreover, it satisfies
\begin{align*}
\left(\partial_y - k^2 - \frac{v''(y)}{v(y) - v(y_0)} + \frac{\tilde{\g}\mathrm{P}(y)}{(v(y) - v(y_0))^2}\right) \varphi(y,\lambda,\tilde{\g}) =0
\end{align*}
with $\varphi(2,\lambda,\tilde{\g})=0$. In particular, we also have $\varphi(0,\lambda,\tilde{\g})=0$. However, Proposition \ref{prop:limitWronkianFragile} and $\phi_{\sr,k,0,\tilde{\g}}^+(y,y_0)\neq 0$ for all $y\in[0,2]\setminus \lbrace y_0 \rbrace$ due to Lemma \ref{lemma:lowerboundphisigma} show that $\varphi(0,\lambda,\tilde\g)\neq 0$, reaching a contradiction. 

\bullsubpar{Case $y_0 \in (\vartheta_1, \vartheta_2)$ and $\tilde{\g}\widetilde\cJ(y_0)\neq \frac14$} We now have that $y_j=v^{-1}(\Re(\lambda_j))$ is such that $\tilde{\g}_j\widetilde\cJ(y_j)\neq \frac14$ and thus $\gamma_j\not\rightarrow0$, for $j\geq 1$ large enough and thus
\begin{equation}\label{eq:varphieigenstrongweak}
\varphi(y,\lambda_j,\tilde{\g}_j) = \frac{\phi_{\s,k,\ep_j,\tilde{\g}_j}^+(2,y_j)\phi_{\sr,k,\ep_j,\tilde{\g}_j}^+(y,y_j) - \phi_{\sr,k,\ep_j,\tilde{\g}_j}^+(2,y_j)\phi_{\s,k,\ep_j,\tilde{\g}_j}^+(y,y_j)}{2\gamma_j v'(y_j)}
\end{equation}  
solves \eqref{eq:TGlambdaeigenvalue} with $\varphi(2,\lambda_j,\tilde{\g}_j) = 0$ and $\partial_y\varphi(2,\lambda_j,\tilde{\g}_j)=1$. If $\widetilde\varpi_1 < \widetilde\varpi_2$, 
Proposition \ref{prop:nonzeroWstrongstrat} for $y_0\in (\widetilde\varpi_1, \widetilde\varpi_2)$  and Proposition \ref{prop:nonzeroWweakstrat} for $y_0\in (\vartheta_1,\widetilde\varpi_1)\cup (\widetilde\varpi_2, \vartheta_2)$, show that $|\varphi(0,\lambda_j,\tilde{\g}_j)| \geq c_0 > 0$, for some $c_0>0$, for all $j$ large enough, thus contradicting \eqref{eq:vanishingvarphilambdaj}. If $\widetilde\varpi$ is the unique root to $\tilde\g\widetilde\cJ(y) = \frac14$, then $\widetilde\varpi = \tilde y$ and we use Proposition \ref{prop:nonzeroWweakstrat}. If there is no such root, $\tilde\g\widetilde\cJ(\tilde(y)) < \frac14$ and we also use Proposition \ref{prop:nonzeroWweakstrat}. In all cases we conclude that $|\varphi(0,\lambda_j,\tilde{\g}_j)| \geq c_0 > 0$ for all $j$ large enough, reaching a contradiction with \eqref{eq:vanishingvarphilambdaj}.

\bullsubpar{Case $y_0 =  \widetilde\varpi_1$ and $y_j\neq \widetilde\varpi_{j,1}$} In this setting, $\gamma_j\neq 0$ but $\gamma_j\rightarrow 0$ so that we cannot use \eqref{eq:varphieigenstrongweak}. Instead, we write
\begin{equation}
\varphi(y,\lambda_j,\tilde{\g}_j) = \frac{\phi_{\sr,k,\ep_j,\tilde\g_j}^+(2,y_j) \phi_{\rL,k,\ep_j,\tilde\g_j}^+(y,y_j) - \phi_{\sr,k,\ep_j,\tilde\g_j}^+(y,y_j) \phi_{\rL,k,\ep_j,\tilde\g_j}^+(2,y_j)}{v'(y_j)}
\end{equation}
where $\phi_{\rL,k,\ep_j,\tilde\g_j}^\pm(y,y_j)$ is given by \eqref{eq:defphiL}, which again solves \eqref{eq:TGlambdaeigenvalue} with $\varphi(2,\lambda_j,\tilde{\g}_j) = 0$ and $\partial_y\varphi(2,\lambda_j,\tilde{\g}_j)=1$. Now, Proposition \ref{prop:limitWronkianMild} shows that $\lim_{j\rightarrow \infty}\varphi(0,\lambda_j, \tilde{\g}_j) \neq 0$, contradicting again \eqref{eq:vanishingvarphilambdaj}.

\bullsubpar{Case $y_0 = \widetilde\varpi_1$ and $y_j=\widetilde\varpi_{j,1}$} We now have $\tilde\g_j\widetilde\cJ(y_j)=\frac14$ for all $j \geq 1$ so that $\gamma_j = 0$ already. Hence, we cannot use $\phi_{\rL,k,\ep_j}^\pm(y,y_j)$ as given in \eqref{eq:defphiL} but instead we take
\begin{equation*}
\varphi(y,\lambda_j,\tilde{\g}_j) = \frac{\phi_{\sr,k,\ep_j,\tilde\g_j}^+(2,y_j) \phi_{\rL,k,\ep_j,\tilde\g_j}^+(y,y_j) - \phi_{\sr,k,\ep_j,\tilde\g_j}^+(y,y_j) \phi_{\rL,k,\ep_j,\tilde\g_j}^+(2,y_j)}{v'(y_j)}
\end{equation*}
where now $\phi_{\rL,k,\ep_j\tilde\g_j}^\pm(y,y_j)$ is given by \eqref{eq:defphiLcrit}. Thanks to Proposition \ref{prop:limitWronskianMildCrit} we obtain $\lim_{j\rightarrow \infty}\varphi(0,\lambda_j, \tilde{\g}_j) \neq 0$, contradicting again \eqref{eq:vanishingvarphilambdaj}. This case also covers the setting where $\lbrace y\in [0,2] : \widetilde \g \widetilde\cJ(y) = \frac14 \rbrace$ contains an interval.

\bullpar{Case $\lambda\in \mathcal{R}$, $\tilde{\g}=0$ and $\tilde\g_j>0$} Since $\tilde\g_j\rightarrow 0$, we have that $\tilde\g_j\widetilde\cJ(y) < \frac14$ for all $y\in[0,2]$ and $\gamma_j\in (\frac14,\frac12)$, for $j\geq 1$ large enough. We distinguish according to the limiting value $\lambda\in \mathcal{R}$. As before, let $y_0,\, y_j \in[0,2]$ be such that $v(y_0) = \lambda$ and $v(y_j) = \Re(\lambda_j)$, respectively.

\bullsubpar{Case $y_0=\vartheta_n$ and $\mathrm{P}(y_j)=0$} As before, this case cannot happen due to Lemma \ref{lemma:boundarysuppnoteigen}.

\bullsubpar{Case $y_0 = \vartheta_n$ and $\mathrm{P}(y_j)>0$} We argue as before, now noting that 
\begin{align}
\phi_{\sr,1,k,0,0}(y,y_0) = \lim_{j\rightarrow \infty}\phi_{\sr,1,k,\ep_j,\tilde\g_j}^+(y,y_j)
\end{align}
satisfies
\begin{align*}
\phi_{\sr,1,k,0,0}(y,y_0)= 1 + k^2\int_{y_0}^y \frac{1}{(v(z) - v(y_0))^2}\int_{y_0}^z (v(s) - v(y_0))^2 \phi_{\sr,1,k,0,0}(s,y_0) \d s \d z
\end{align*}
so that $\phi_{\sr,k,0,0}(y,y_0) := (v(y) - v(y_0)) \phi_{\sr,1,k,0,0}(y,y_0)$ satisfies
\begin{align*}
\left(\partial_y - k^2 - \frac{v''(y)}{v(y) - v(y_0)} \right) \phi_{\sr,k,0,0}(y,y_0) =0.
\end{align*}
Moreover, we have from Proposition \ref{prop:limitWronkianFragile} and Lemma \ref{lemma:vanishingsecondinhomsol} that for $\varphi(y,\lambda,0) = \lim_{j\rightarrow \infty} \varphi(y, \lambda_j, \tilde{\g}_j)$ we have
\begin{align*}
\varphi(y,\lambda,0) =\phi_{\sr,k,0,0}(y,y_0)\int_0^y \frac{1}{(\phi_{\sr,k,0,0}(z,y_0))^2} \d z - \phi_{\sr,k,0,0}(y,y_0) \int_0^2 \frac{1}{(\phi_{\sr,k,0,0}(z,y_0))^2}\d z 
\end{align*}
is well-defined and, moreover, satisfies
\begin{align*}
\left(\partial_y - k^2 - \frac{v''(y)}{v(y) - v(y_0)} \right) \varphi(y,\lambda,0) =0
\end{align*}
with $\varphi(2,\lambda,0)=0$ and $\partial_y\varphi(2,\lambda,0)=1$.  In particular, we also have $\varphi(0,\lambda,0)=0$. However, Proposition \ref{prop:limitWronkianFragile} and $\phi_{\sr,k,0,\tilde{\g}}(y,y_0)\neq 0$ for all $y\in[0,2]\setminus \lbrace y_0 \rbrace$ due to Lemma \ref{lemma:lowerboundphisigma} show again that $\varphi(0,\lambda,0)\neq 0$, thus obtaining a contradiction.

\bullsubpar{Case $y_0\in (\vartheta_1, \vartheta_2)$} We proceed as in the previous case, where now $\varphi(y,\lambda,0)$ satisfies 
\begin{align*}
\left(\partial_y - k^2 - \frac{v''(y)}{v(y) - v(y_0)} \right) \varphi(y,\lambda,0) =0
\end{align*}
together with $\varphi(2,\lambda,0)=0$, $\varphi(0,\lambda,0)=0$ and $\partial_y\varphi(2,\lambda,0)=1$. Hence, $\lambda \in (v(\vartheta_1), v(\vartheta_2))$ and $\varphi(y,\lambda,\g)\neq 0$ constitute an eigenvalue and corresponding non-zero eigenfunction of the linearised Euler equations, a contradiction with the spectral assumption.

\bullpar{Case $\lambda\in \mathcal{R}$, $\tilde{\g}=\tilde\g_j=0$} This setting correspond to having $\lambda_j= v(y_j)$ an embedding eigenvalue of the linearised Euler operator, which is in contradiction with H3.
\end{proof}

\begin{remark}
As one can see from the stability conditions and the proof of Theorem \ref{thm:spectrumlinop}, increasing the size of the channel from $[0,2]$ to $[0,\mathsf H]$ for some $\mathsf H>2$ relaxes the conditions on the background steady state and allows us to find more examples  where the region $J(y) \geq \frac14$ fills most of the channel. In a sense, the absence of physical boundaries may be understood as a stabilising mechanism. For instance, for the stably stratified Couette flow the linearised operator is easily understood in the periodic strip in \cite{CZN23strip}, while in the periodic channel it presents sequences of neutral eigenvalues and may exhibit unstable eigenvalues, see \cite{CZN25chan}.
\end{remark}

\section{The fragile regime}\label{sec:fragile}
In this section we study the behaviour of the homogeneous and in-homogeneous solutions of the Taylor-Goldstein equation for $k\geq 1$ fixed when the spectral parameter $y_0\in I_W$ is close to the non-stratified region. As $y_0\rightarrow \vartheta_n$, for $n=1,2$, the \emph{most singular} homogeneous solution $\phi_{\s,k,\ep}^\pm(y,y_0)$ degenerates, in that its bounds are not uniform in $y_0$ and blow up as $y_0\rightarrow \vartheta_n$. Consequently, we shall use a different homogeneous solution. Once it is established, we use it to construct the solution $h_{k,\ep}^\pm(y,y_0)$ to 
\begin{equation}\label{eq:inhomTgeqhr}
\textsc{TG}_{k,\ep}^\pm h_{k,\ep}^\pm(y,y_0) = \mathrm{g}_{k,\ep}^\pm(y,y_0),
\end{equation}
with $h_{k,\ep}^\pm(0,y_0)=h_{k,\ep}^\pm(2,y_0)=0$ and for some source term $\mathrm{g}_{k,\ep}^\pm(y,y_0)$. The next result is immediate from Proposition 6.5 in \cite{WZZ18}.

\begin{proposition}\label{prop:defhsolTGr}
Let $k\geq 1$ fixed, $\ep>0$ and $y_0\in I_S\cup I_W$. The solution $h_{k,\ep}^\pm(y,y_0)$ to \eqref{eq:inhomTgeqhr} is given by
\begin{equation}\label{eq:defhsolTGr}
\begin{split}
h_{k,\ep}^\pm(y,y_0) &= \phi_{\sr,k,\ep}^\pm(y,y_0) \int_0^y \frac{1}{(\phi_{\sr,k,\ep}^\pm(y,y_0))^2} \int_{y_0}^z \phi_{\sr,k,\ep}^\pm(s,y_0) \mathrm{g}_{k,\ep}^\pm(s,y_0) \d s \d z \\
&\quad+ \mathrm{W}_{k,\ep}^\pm(y_0) \phi_{\sr,k,\ep}^\pm(y,y_0)\int_0^y \frac{1}{(\phi_{\sr,k,\ep}^\pm(z,y_0))^2} \d z
\end{split}
\end{equation}
where
\begin{align*}
\mathrm{W}_{k,\ep}^\pm(y_0) := -\frac{\int_0^2 \frac{1}{(\phi_{\sr,k,\ep}^\pm(y,y_0))^2} \int_{y_0}^y \phi_{\sr,k,\ep}^\pm(z,y_0) \mathrm{g}_{k,\ep}^\pm(z,y_0) \d z \d y }{\int_0^2 \frac{1}{(\phi_{\sr,k,\ep}^\pm(y,y_0))^2} \d y}
\end{align*}
and $\phi_{\sr,k,\ep}^\pm(y,y_0)$ is the homogeneous solution to the Taylor-Goldstein operator given by \eqref{eq:defphisigma} and Proposition \ref{prop:existencesolphisigma}.
\end{proposition}

While $h_{k,\ep}^\pm(y,y_0)$ is well-defined for all $\ep>0$ and all $y_0\in I_S$, here we investigate the solution in the fragile regime, namely when $\ep\rightarrow 0$ and $y_j\rightarrow y_0$ with $\P(y_0)=0$. The main result of the section is the following.

\begin{theorem}\label{thm:vanishinghjfragile}
Fix $k\geq 1$. Let $y_0\in (0,2)$ such that $\P(y_0)=0$. Let $y_j\neq y_0$, with $\P(y_j)>0$, $y_j\rightarrow y_0$ and $\ep_j\rightarrow 0^+$ as $j\rightarrow \infty$. Let $h_{k,\ep_j}^\pm(y,y_j)$ given by \eqref{eq:defhsolTGr} be the solution to \eqref{eq:inhomTgeqhr}, with  $\mathrm{g}_j(y,y_j):=-\textsc{E}_{k_j,\ep_j}^\pm(y,y_j) g_j(y)$ and $g_j\in Z_k$ with $\Vert g_j \Vert_{Z_k}\rightarrow 0$ as $j\rightarrow \infty$. Then,
\begin{align*}
\lim_{j\rightarrow \infty} \Vert h_{k,\ep_j}^\pm(y,y_j) \Vert_{Z_k} = 0.
\end{align*}
\end{theorem}
The theorem follows from Proposition \ref{prop:limitWronkianFragile},  Lemma  \ref{lemma:vanishingfirstinhomsol} and Lemma \ref{lemma:vanishingsecondinhomsol} below.

\subsection{Homogeneous solutions in the fragile regime}
We first investigate finer properties of the homogeneous solution $\phi_{\sr,k,\ep}^\pm(y,y_0)$, when $y_0=\vartheta_n$, $n=1,2$, or close to it. Firstly, due to Lemma \ref{lemma:boundarysuppnoteigen} we observe that $\lambda=v(\vartheta_1)$ and $\lambda=v(\vartheta_2)$ are not embedded eigenvalues of $L_k$. In particular, for $y_0=\vartheta_n$ with $n=1,2$ and
\begin{align*}
    \Psi(y) := -\phi_{\sr,k,0}(2,y_0)\phi_{\sr,k,0}(y,y_0)\int_y^2 \frac{\d z}{(\phi_{\sr,k,0}(z,y_0))^2}
\end{align*}
we note that $\Psi(y)$ is a non-zero solution to \eqref{eq:eigenvalueTG} due to Lemma \ref{lemma:lowerboundphisigma}, with further $\Psi(2)=0$, $\Psi'(2)=1$, and 
\begin{align*}
    \Psi(y) &= - \frac{\phi_{\sr,1,k,0}(2,y_0)\phi_{\sr,1,k,0}(y,y_0)}{v'(y_0)}(v(2) - v(y)) \\
    &\quad+ \frac{\phi_{\sr,k,0}(2,y_0)\phi_{\sr,k,0}(y,y_0)}{v'(y_0)}\int_y^2 \frac{v'(z) - v'(y_0)}{(v(z) - v(y_0)^2} \d z \\
    &\quad - \phi_{\sr,k,0}(2,y_0)\phi_{\sr,k,0}(y,y_0)\int_{y}^2 \frac{1}{(v(z) - v(y_0))^2}\left( \frac{1}{(\phi_{\sr,1,k,0}(z,y_0))^2} -1 \right) \d z.
\end{align*}
Now, since  $\lambda=v(\vartheta_1)$ and $\lambda=v(\vartheta_2)$ are not embedded eigenvalues of $L_k$ we have that $\Psi(0)\neq 0$ and thus
\begin{equation}\label{eq:nonzeroWborder}
\begin{split}
    0 &\neq - \frac{\phi_{\sr,1,k,0}(2,y_0)\phi_{\sr,1,k,0}(0,y_0)}{v'(y_0)}(v(2) - v(0)) \\
    &\quad+ \frac{\phi_{\sr,k,0}(2,y_0)\phi_{\sr,k,0}(0,y_0)}{v'(y_0)}\int_0^2 \frac{v'(z) - v'(y_0)}{(v(z) - v(y_0)^2} \d z \\
    &\quad - \phi_{\sr,k,0}(2,y_0)\phi_{\sr,k,0}(0,y_0)\int_{0}^2 \frac{1}{(v(z) - v(y_0))^2}\left( \frac{1}{(\phi_{\sr,1,k,0}(z,y_0))^2} -1 \right) \d z.
    \end{split}
\end{equation}
 
Due to Proposition~\ref{prop:continuityphi}, we can extend by continuity the non-vanishing of $\phi_{\sr,k,0}(y,\vartheta_n)$.

\begin{lemma}\label{lemma:fragilephir1knonzero}
Fix $k\geq 1$. Let $y_0\in (0,2)$ such that $\P(y_0)=0$. Let $y_j\neq y_0$, with $\P(y_j)>0$, $y_j\rightarrow y_0$ and $\ep_j\rightarrow 0^+$ as $j\rightarrow \infty$. Then,
\begin{align*}
|\phi_{\sr,1,k,\ep_j}^\pm(y,y_j)|\geq \frac12,
\end{align*}
for all $y\in [0,2]$, for $j\geq 0$ sufficiently large.
\end{lemma}

With this uniform lower bound at hand, we now obtain the next result.

\begin{proposition}\label{prop:limitWronkianFragile}
Let $y_0\in (0,2)$ such that $\P(y_0)=0$, Let $y_j\neq y_0$, with $\P(y_j)>0$, $y_j\rightarrow y_0$ and $\ep_j\rightarrow 0^+$ as $j\rightarrow \infty$. Then,
\begin{align*}
\lim_{j\rightarrow \infty} \int_0^2 \frac{\d y}{(\phi_{\sr,k,\ep_j}^\pm(y,y_j))^2} &= -\frac{1}{v'(y_0)}  \frac{v(2) - v(0)}{(v(2) - v(y_0))(v(y_0) - v(0))}  \\
&\quad -\frac{1}{v'(y_0)}\left( \text{P.V.} \int_0^2 \frac{v'(y) - v'(y_0)}{(v(y) - v(y_0))^2} \d y \pm i \pi \frac{v''(y_0)}{(v'(y_0))^2} \right) \\ 
&\quad + \int_0^2 \frac{1}{(v(y) - v(y_0))^2} \left( \frac{1}{(\phi_{\sr,1,k,0}(y,y_0))^2} -1 \right) \d y.
\end{align*}
In particular, the limit is non-zero.
\end{proposition}

\begin{proof}
Firstly, that the limit is non-zero follows from \eqref{eq:nonzeroWborder}, Lemma~\ref{lemma:lowerboundphisigma} and $\vartheta_n \in (0,2)$, for $n=0,1$. We shall now show the convergence to the claimed limit. Since $y_j\rightarrow y_0$ and $\gamma(y_0)=\frac12$, we have that $\gamma_j:=\gamma(y_j) \rightarrow \frac12$. We first write
\begin{align*}
\int_0^2 \frac{\d y}{(\phi_{\sr,k,\ep_j}^\pm(y,y_j))^2} &= \int_0^2 \frac{\d y}{(v(y) - v(y_j) \pm i\ep_j)^{1+2\gamma_j}} \\
&\quad + \int_0^2 \frac{1}{(v(y) - v(y_j) \pm i\ep_j)^{1+2\gamma_j}} \left( \frac{1}{(\phi_{\sr,1,k,\ep_j}(y,y_j))^2} -1 \right) \d y.
\end{align*}
The first contribution is studied separately in Lemma~\ref{lemma:limitvelquotient} below. For the second contribution, to ease notation we denote $\phi_{j}(y):=\phi_{\sr,1,k,\ep_j}(y,y_j)$ and $\phi_{0}(y):=\phi_{\sr,1,k,0}(y,y_0)$. We further integrate by parts to get
\begin{align*}
\int_0^2 \frac{1}{(v(y) - v(y_j) \pm i\ep_j)^{1+2\gamma_j}} &\left( \frac{1}{(\phi_{j}(y,y_j))^2} -1 \right) \d y \\
&= -\frac{1}{2\gamma_j v'(y)}( v(y) - v(y_j) \pm i\ep_j)^{-2\gamma_j} \left( \frac{1}{\phi_j^2(y) } -1 \right) \Big|_{y=0}^{y=2} \\
&+\frac{1}{2\gamma_j}\int_0^2 \partial_y\left( \frac{1}{v'(y)} \right) (v(y) - v(y_j) \pm i\ep_j)^{-2\gamma_j} \left( \frac{1}{\phi_j^2(y) } -1 \right)  \d y \\
&\quad -\frac{1}{\gamma_j}\int_0^2 \frac{1}{v'(y)}\frac{\phi_j'(y)}{\phi_j^3(y)} (v(y) - v(y_0) \pm i\ep_j)^{-2\gamma_j} \d y.
\end{align*}
The convergence for the solid boundary term is obvious since the involved functions are continuous in $y_j$ uniformly a neighbourhood of $y=0$ and $y=2$. The convergence of the two integral contributions are studied in Lemma~\ref{lemma:firstconvergencelemma} and Lemma~\ref{lemma:secondconvergencelemma} below, where integration by parts arguments are carried out in order to smooth the singularities so that the Dominated Convergence Theorem can be applied. Once the passage to the limit is achieved, the Proposition follows by undoing the integration by parts.
\end{proof}

\begin{lemma}\label{lemma:limitvelquotient}
Let $k\geq 1$. Let $y_j\rightarrow y_0$, with $\P(y_0) =0$ and $\ep_j \rightarrow 0$ as $j\rightarrow \infty$. There holds
\begin{align*}
\lim_{j\rightarrow \infty}\int_0^2 \frac{\d y}{(v(y) - v(y_j) \pm i\ep_j)^{1+2\gamma_j}} &= -\frac{1}{v'(y_0)}  \frac{v(2) - v(0)}{(v(2) - v(y_0))(v(y_0) - v(0))}  \\
&\quad -\frac{1}{v'(y_0)}\left( \text{P.V.} \int_0^2 \frac{v'(y) - v'(y_0)}{(v(y) - v(y_0))^2} \d y \pm i \pi \frac{v''(y_0)}{(v'(y_0))^2} \right)
\end{align*}
\end{lemma}

\begin{proof}
We decompose
\begin{align*}
\int_0^2 \frac{\d y}{(v(y) - v(y_j) \pm i\ep_j)^{1+2\gamma_j}} &= \frac{1}{v'(y_j)}\int_0^2 \frac{v'(y)}{(v(y) - v(y_j) \pm i\ep_j)^{1+2\gamma_j}} \d y \\
&\quad- \frac{1}{v'(y_j)}\int_0^2 \frac{v'(y) - v'(y_j)}{(v(y) - v(y_j) \pm i\ep_j)^{1+2\gamma_j}} \d y.
\end{align*}
On one hand, we have
\begin{align*}
\frac{1}{v'(y_j)}\int_0^2 \frac{v'(y)}{(v(y) - v(y_j) \pm i\ep_j)^{1+2\gamma_j}} \d y = -\frac{1}{v'(y_j)}\frac{1}{2\gamma_j} (v(y) - v(y_j) \pm i\ep_j)^{-2\gamma_j}\Big|_{y=0}^{y=2}
\end{align*}
and we observe 
\begin{align*}
(v(y) -v(y_j) \pm i\ep_j)^{-2\gamma_j} = \frac{1}{v(y) - v(y_j) \pm i\ep_j}\left( 1 + (v(y) -v(y_j) \pm i\ep_j)^{1-2\gamma_j} -1 \right)
\end{align*}
so that for $y\neq y_0$ there holds 
\begin{align*}
\lim_{j\rightarrow \infty} (v(y) -v(y_j) \pm i\ep_j)^{-2\gamma_j} = \frac{1}{v(y) - v(y_0)}
\end{align*}
because $\left|  (v(y) -v(y_j) \pm i\ep_j)^{1-2\gamma_j} -1  \right| \lesssim (1-2\gamma_j) \left| \log( v(y) - v(y_j) \pm i\ep_j) \right|$. On the other hand, we claim that
\begin{align*}
\lim_{j\rightarrow \infty}\int_0^2 \left( v'(y) - v'(y_j) \right) \left( \frac{1}{(v(y) - v(y_j) \pm i\ep_j)^{1+2\gamma_j}} - \frac{1}{(v(y) - v(y_j) \pm i\ep_j)^2} \right) \d y =0.
\end{align*}
Indeed, let $\mathrm{v}(y,y_j):= v'(y) - v'(y_j) - \frac{v''(y_j)}{(v'(y_j))^2}v'(y)(v(y) - v(y_j))$. We note that $\mathrm{v}(y_j,y_j) = \partial_y \mathrm{v}(y_j,y_j) = 0$ and that $| \mathrm{v}(y,y_j) | \lesssim |y-y_j|^2$. Then,
\begin{align*}
\int_0^2 \mathrm{v}(y,y_j)& \left( \frac{1}{(v(y) - v(y_j) \pm i\ep_j)^{1+2\gamma_j}} - \frac{1}{(v(y) - v(y_j) \pm i\ep_j)^2} \right) \d y \\
&= \int_0^2 \frac{\mathrm{v}(y,y_j)}{(v(y) - v(y_j) \pm i\ep_j)^2} \left( (v(y) - v(y_j) \pm i\ep_j)^{1-2\gamma_j} -1 \right) \d y
\end{align*}
Since $\left|  (v(y) -v(y_j) \pm i\ep_j)^{1-2\gamma_j} -1  \right| \lesssim (1-2\gamma_j) \left| \log( v(y) - v(y_j) \pm i\ep_j) \right|$ and $\left| \log( v(y) - v(y_j) \pm i\ep_j) \right|$ is integrable uniformly in $y_j$ and $\ep_j>0$ we conclude that 
\begin{align*}
\lim_{j\rightarrow \infty} \int_0^2 \frac{\mathrm{v}(y,y_j)}{(v(y) - v(y_j) \pm i\ep_j)^2} \left( (v(y) - v(y_j) \pm i\ep_j)^{1-2\gamma_j} -1 \right) \d y = 0.
\end{align*}
Similarly, integrating by parts twice we now see that
\begin{align*}
\int_0^2 &\frac{v'(y)(v(y) - v(y_j))}{(v(y) - v(y_j)\pm i\ep_j)^{1+2\gamma_j}} \d y - \int_0^2 \frac{v'(y)(v(y) - v(y_j)}{(v(y) - v(y_j)\pm i\ep_j)^2} \d y \\
&=  (v(y) - v(y_j)) \left( \frac{1}{v(y)-v(y_j) \pm i\ep_j} - \frac{1}{2\gamma_j}\frac{1}{(v(y) - v(y_j) \pm i\ep_j)^{1+2\gamma_j}} \right) \Big|_{y=0}^{y=2} \\
&\quad + \frac{1}{2\gamma_j}\frac{(v(y) - v(y_j)\pm i\ep_j)^{1-2\gamma_j}-1}{1-2\gamma_j} - \log(v(y) - v(y_j) \pm i\ep_j) \Big|_{y=0}^{y=2} \\
&\quad - \int_0^2 \left( \frac{1}{2\gamma_j}\frac{(v(y) - v(y_j)\pm i\ep_j)^{1-2\gamma_j}-1}{1-2\gamma_j} - \log(v(y) - v(y_j) \pm i\ep_j) \right) \d y
\end{align*}
Further observing that for $\zeta\in C$ there holds
\begin{align*}
\frac{\zeta^{1-2\gamma_j}-1}{1-2\gamma_j} = \log(\zeta) + (1-2\gamma_j)\log^2(\zeta)\int_0^2r\int_0^2 e^{(1-2\gamma_j)r s \log(\xi)} \d s \d r, 
\end{align*}
since $y_j$ is bounded away from $0$ and $2$ uniformly and $v$ is monotone, we conclude that 
\begin{align*}
\lim_{j\rightarrow \infty} \int_0^2 &\frac{v'(y)(v(y) - v(y_j)}{(v(y) - v(y_j)\pm i\ep_j)^{1+2\gamma_j}} \d y - \int_0^2 \frac{v'(y)(v(y) - v(y_j)}{(v(y) - v(y_j)\pm i\ep_j)^2} \d y = 0
\end{align*}
and the claim follows. Therefore,
\begin{align*}
\lim_{j\rightarrow \infty} \frac{1}{v'(y_j)}\int_0^2 \frac{v'(y) - v'(y_j)}{(v(y) - v(y_j) \pm i\ep_j)^{1+2\gamma_j}} \d y &= \lim_{j\rightarrow \infty}\frac{1}{v'(y_j)}\int_0^2 \frac{v'(y) - v'(y_j)}{(v(y) - v(y_j) \pm i\ep_j)^{2}} \d y.
\end{align*}
With this and Lemma 6.3 of \cite{WZZ18} the proof is concluded.
\end{proof}

\begin{lemma}\label{lemma:firstconvergencelemma}
In the setting of Proposition \ref{prop:limitWronkianFragile}, there holds
\begin{align*}
\lim_{j\rightarrow \infty} \frac{1}{2\gamma_j}\int_0^2 \partial_y\left( \frac{1}{v'(y)} \right) \frac{  \frac{1}{\phi_j^2(y) } -1}{(v(y) - v(y_j) \pm i\ep_j)^{2\gamma_j}} \d y = \int_0^2 \partial_y\left( \frac{1}{v'(y)} \right) \frac{  \frac{1}{\phi_0^2(y) } -1}{v(y) - v(y_0)} \d y
\end{align*}
\end{lemma}

\begin{proof}
We integrate by parts again:
\begin{align*}
\int_0^2 \partial_y\left( \frac{1}{v'(y)} \right) &(v(y) - v(y_j) \pm i\ep_j)^{-2\gamma_j} \left( \frac{1}{\phi_j^2(y) } -1 \right)  \d y \\
&= \partial_y \left( \frac{1}{v'(y)} \right) \frac{(v(y) - v(y_j) \pm i\ep_j)^{1-2\gamma_j}-1}{1-2\gamma_j} \left( \frac{1}{\phi_j^2(y) } -1 \right)  \Big|_{y=0}^{y=2} \\
&\quad - \int_0^2 \partial_y \left(  \partial_y \left( \frac{1}{v'(y)} \right) \left( \frac{1}{\phi_j^2(y) } -1 \right) \right) \frac{(v(y) - v(y_j) \pm i\ep_j)^{1-2\gamma_j}-1}{1-2\gamma_j} \d y
\end{align*}
As before, the solid boundary term converges, while now for the integral term we recall that
\begin{equation}\label{eq:gammajexpansion}
 \frac{(v(y) - v(y_j) \pm i\ep_j)^{1-2\gamma_j}-1}{1-2\gamma_j} = \log( v(y) - v(y_j) \pm i\ep_j) + (1-2\gamma_j)\log^2 ( v(y) - v(y_j) \pm i\ep_j) \widetilde\Q_{\frac12 -\gamma_j}(\xi),
\end{equation}
for $\xi= v(y) - v(y_j) \pm i\ep_j$ and with $\left| \widetilde\Q_{\frac12-\gamma}(\xi) \right| \lesssim 1$. Then, since $\left| 1 - \phi_j^2(y) \right| \lesssim |y-y_j|\lesssim |v(y) - v(y_j)|$, the dominated convergence theorem shows that
\begin{align*}
\lim_{j\rightarrow \infty} \int_0^2 \partial_y^2 \left( \frac{1}{v'(y)} \right) &\left( \frac{1}{\phi_j^2(y) } -1 \right)  \frac{(v(y) - v(y_j) \pm i\ep_j)^{1-2\gamma_j}-1}{1-2\gamma_j} \d y \\
&= \int_0^2 \partial_y^2 \left( \frac{1}{v'(y)} \right) \left( \frac{1}{\phi_0^2(y) } -1 \right)  \log(v(y) - v(y_0) \pm i\ep_j) \d y.
\end{align*}
On the other hand,
\begin{align*}
\int_0^2    \partial_y \left( \frac{1}{v'(y)} \right) &\partial_y \left( \frac{1}{\phi_j^2(y) } -1  \right) \frac{(v(y) - v(y_j) \pm i\ep_j)^{1-2\gamma_j}-1}{1-2\gamma_j} \d y \\
&= -2\int_0^2 \left( \frac{1}{v'(y)} \right)' \frac{\phi_j'(y)}{\phi_j^3(y)} \log(v(y) - v(y_j) \pm i\ep_j) \d y \\
&\quad -2(1-2\gamma_j)\int_0^2 \left( \frac{1}{v'(y)} \right)' \frac{\phi_j'(y)}{\phi_j^3(y)} \log^2(v(y) - v(y_j) \pm i\ep_j)  \widetilde\Q_{\frac12-\gamma_j}(\xi) \d y
\end{align*}
and we readily observe that the second term vanishes. For the first one, since
\begin{equation}\label{eq:pyphij}
\phi_j'(y)=\cT_{\sr,1,j}\phi_j(y) =\frac{1}{(v(y) - v(y_j) \pm i\ep_j)^{1+2\gamma_j}}\int_{y_0}^y \mathrm{F}_{\sr,j}^\pm (z,y_0) (v(y) - v(y_j) \pm i\ep_j)^{2\gamma_j} \phi_j(z) \d z
\end{equation}
we integrate by parts again to obtain 
\begin{align*}
&\int_0^2 \left( \frac{1}{v'(y)} \right)'  \frac{\phi_j'(y)}{\phi_j^3(y)} \log(v(y) - v(y_j) \pm i\ep_j) \d y \\
&= -\frac{1}{2\gamma_j v'(y)}\left( \frac{1}{v'(y)} \right)' \frac{\log(v(y) - v(y_j) \pm i\ep_j)}{\phi_j^3(y)} \frac{\int_{y_0}^y \mathrm{F}_{\sr,j}(z,y_0) (v(y) - v(y_j) \pm i\ep_j)^{2\gamma_j}\phi_j(z) \d z}{(v(y) - v(y_j) \pm i\ep_j)^{2\gamma_j}}  \Big|_{y=0}^{y=2}  \\ 
&\quad + \frac{1}{2\gamma_j}\int_0^2 \left( \frac{1}{v'(y)} \right)' \frac{\cT_{\sr,1,j}\phi_j(y)}{\phi_j^{3}(y)} \d y  \\
&\quad + \frac{1}{2\gamma_j}\int_0^2 \frac{1}{v'(y)}\left( \frac{1}{v'(y)} \right)' \frac{\log (v(y) - v(y_j) \pm i\ep_j)}{\phi_j^2(y)}\mathrm{F}_{\sr,j}(y,y_0) \d y \\
&\quad + \frac{1}{2\gamma_j} \int_0^2 \partial_y \left( \frac{1}{v'(y) \phi_j^3(y)}\left( \frac{1}{v'(y)} \right)' \right) (v(y) - v(y_j)\pm i\ep_j)\log(v(y) - v(y_j) \pm i\ep_j)\cT_{\sr,1,j}\phi_j(y) \d y 
\end{align*}
Since $\Vert \cT_{\sr,1,j}\phi_j\Vert_Y \lesssim 1$ uniformly in $j\geq 1$, we readily obtain convergence of the first and third integral using the Dominated Convergence Theorem. As for the second integral, we now exploit the structure of $\mathrm{F}_{\sr,j}(y,y_0)$, 
\begin{align*}
\int_0^2 &\frac{1}{v'(y)}\left( \frac{1}{v'(y)} \right)' \frac{\log (v(y) - v(y_j) \pm i\ep_j)}{\phi_j^2(y)}\mathrm{F}_{\sr,j}(y,y_0)  \d y \\
&=\frac12\int_0^2 \partial_y \left(\frac{1}{v'(y)}\right)^2 \frac{(v(y) - v(y_j) \pm i\ep_j)\log(v(y) - v(y_j) \pm i\ep_j)}{\phi_j^2(y)} \d y \\
&\quad +\frac{1}{2k^2} \int_0^2 \partial_y \left(\frac{1}{v'(y)}\right)^2 \frac{\P(y) - \P(y_j) - \cJ(y_j)(v'(y)^2 -v'(y_j)^2)}{\phi_j^2(y)}\frac{\log(v(y) - v(y_j) \pm i\ep_j)}{v(y) - v(y_j) \pm i\ep_j)} \d y \\
&\quad +\frac{1-2\gamma_j}{4k^2}\int_0^2 \partial_y \left( \frac{1}{v'(y)} \right)^2 \frac{\log(v(y) - v(y_j) \pm i\ep_j)}{\phi_j^2(y)}v''(y) \d y.
\end{align*}
Clearly, the third term vanishes and we can use the Dominated Convergence Theorem on the first integral. For the second integral we shall integrate by parts once again and reach 
\begin{align*}
\int_0^2 &\partial_y \left(\frac{1}{v'(y)}\right)^2 \frac{\P(y) - \P(y_j) - \cJ(y_j)(v'(y)^2 -v'(y_j)^2)}{\phi_j^2(y)}\frac{\log(v(y) - v(y_j) \pm i\ep_j)}{v(y) - v(y_j) \pm i\ep_j} \d y \\
&=\frac{1}{6}\partial_y \left(\frac{1}{v'(y)} \right)^3  \frac{\P(y) - \P(y_j) - \cJ(y_j)(v'(y)^2 -v'(y_j)^2)}{\phi_j^2(y)}\log^2(v(y) - v(y_j) \pm i\ep_j) \Big|_{y=0}^{y=2} \\
&\quad - \frac{1}{6}\int_0^2 \partial_y \left( \frac{1}{\phi_j^2(y)}  \partial_y \left(\frac{1}{v'(y)} \right)^3 \right)  \left( \P(y) - \P(y_j) - \cJ(y_j)(v'(y)^2 -v'(y_j)^2) \right) \log^2(v(y) - v(y_j) \pm i\ep_j) \d y \\
&\quad - \frac{1}{6} \int_0^2 \partial_y \left(\frac{1}{v'(y)} \right)^3\frac{\P'(y) - 2\cJ(y_j)v'(y) v''(y)}{\phi_j^2(y)} \log^2(v(y) - v(y_j) \pm i\ep_j) \d y.
\end{align*}
Further observing that $\P'$, $\cJ$ are $C^1$ functions with $\P'(y_0)=\cJ(y_0)=0$, we can take the limits in all the above integrals using the Dominated Convergence Theorem. Integrating by parts backwards we reach the claim.
\end{proof}

\begin{lemma}\label{lemma:secondconvergencelemma}
In the setting of Proposition \ref{prop:limitWronkianFragile}, there holds
\begin{align*}
\lim_{j\rightarrow\infty}\frac{1}{\gamma_j}\int_0^2 \frac{1}{v'(y)}\frac{\phi_j'(y)}{\phi_j^3(y)} (v(y) - v(y_j) \pm i\ep_j)^{-2\gamma_j} \d y = 2\int_0^2	\frac{1}{v'(y)}\frac{\phi_0'(y)}{\phi_0^3(y)}\frac{\d y}{v(y) - v(y_0)}.
\end{align*}
\end{lemma}

\begin{proof}
Firstly, we split
\begin{align*}
\int_0^2 \frac{1}{v'(y)}\frac{\phi_j'(y)}{\phi_j^3(y)} (v(y) - v(y_j) \pm i\ep_j)^{-2\gamma_j} \d y &= \int_0^2 \frac{\phi_j'(y)}{v'(y)}\left( \frac{1}{\phi_j^3(y)} - 1 \right) (v(y) - v(y_j) \pm i\ep_j)^{-2\gamma_j} \d y \\
&\quad + \int_0^2 \frac{\phi_j'(y)}{v'(y)}(v(y) - v(y_j) \pm i\ep_j)^{-2\gamma_j} \d y .
\end{align*}
For the first integral, we use the Dominated Convergence Theorem, since $\Vert \phi_j' \Vert_{L^\infty}\lesssim 1$ uniformly in $j\geq 1$, $|\phi_j|$ is bounded from below and from above uniformly in $j\geq 1$, with $|\phi_j(y) -1 |\lesssim |y-y_j|$ and 
\begin{align*}
|y-y_j|^{1-2\gamma_j}\lesssim 2^{1-2\gamma_j}\lesssim 1
\end{align*}
since $\zeta^{1-2\gamma_j}$ is increasing in $\zeta$. Hence, we obtain
\begin{align*}
\lim_{j\rightarrow \infty}\int_0^2 \frac{\phi_j'(y)}{v'(y)}\left( \frac{1}{\phi_j^3(y)} - 1 \right) (v(y) - v(y_j) \pm i\ep_j)^{-2\gamma_j} \d y = \int_0^2 \frac{\phi_0'(y)}{v'(y)}\frac{ \frac{1}{\phi_0^3(y)} - 1}{v(y) - v(y_0)} \d y.
\end{align*}
For the second integral, using \eqref{eq:pyphij} and integrating by parts, we have
\begin{align*}
\int_0^2 \frac{\phi_j'(y)}{v'(y)} &(v(y) - v(y_j) \pm i\ep_j)^{-2\gamma_j} \d y \\
&= -\frac{1}{4\gamma_j} \left(\frac{1}{v'(y)}\right)^2 \frac{\int_{y_j}^y \mathrm{F}_{\sr,j}(z,y_j)(v(z)-v(y_j)\pm i\ep_j)^{2\gamma_j}\phi_j(z) \d z}{(v(y) - v(y_j) \pm i\ep_j)^{4\gamma_j}}\Big|_{y=0}^{y=2} \\
&\quad + \frac{1}{4\gamma_j}\int_0^2 \partial_y\left(\frac{1}{v'(y)}\right)^2 \frac{\int_{y_j}^y \mathrm{F}_{\sr,j}(z,y_j)(v(z)-v(y_j)\pm i\ep_j)^{2\gamma_j}\phi_j(z) \d z}{(v(y) - v(y_j) \pm i\ep_j)^{4\gamma_j}} \d y \\
&\quad + \frac{1}{4\gamma_j}\int_0^2 \left( \frac{1}{v'(y)} \right)^2 \mathrm{F}_{\sr, j}(y,y_j)\frac{\phi_j(y)}{(v(y) - v(y_j) \pm i\ep_j)^{2\gamma_j}} \d y.
\end{align*}
The convergence of the boundary term is immediate, while that of the first integral follows from the Dominated Convergence Theorem. Thus, we focus on the second integral, which we can further write as
\begin{align*}
\int_0^2 \left( \frac{1}{v'(y)} \right)^2 \frac{\mathrm{F}_{\sr, j}(y,y_j)\phi_j(y)}{(v(y) - v(y_j) \pm i\ep_j)^{2\gamma_j}} \d y &= \int_0^2 \left( \frac{1}{v'(y)}\right)^2 (v(y) - v(y_j) \pm i\ep_j)^{1-2\gamma_j}\phi_j(y) \d y \\
&\quad - \frac{1}{k^2}\int_0^2 \frac{\cJ(y) - \cJ(y_j)}{(v(y) - v(y_j) \pm i\ep_j)^{1+2\gamma_j}}\phi_j(y) \d y \\
&\quad + \frac{1-2\gamma_j}{k^2}\int_0^2 \frac{v''(y)}{(v'(y))^2}\frac{\phi_j(y)}{(v(y) - v(y_j) \pm i\ep_j)^{2\gamma_j} } \d y.
\end{align*}
The first contribution converges thanks to the Dominated Convergence Theorem, while the third contribution vanishes in the limit. Indeed, 
\begin{align*}
\frac{1-2\gamma_j}{k^2}\int_0^2 \frac{v''(y)}{(v'(y))^2}\frac{\phi_j(y)}{(v(y) - v(y_j) \pm i\ep_j)^{2\gamma_j} } \d y &= \frac{v''(y)\phi_j(y)}{k^2(v'(y))^3}(v(y) - v(y_j) \pm i\ep_j)^{1-2\gamma_j}\Big|_{y=0}^{y=2} \\
&\quad - \int_0^2 \left( \frac{v''(y)\phi_j(y)}{k^2(v'(y))^3} \right)' (v(y) - v(y_j) \pm i\ep_j)^{1-2\gamma_j} \d y
\end{align*}
and we note that $(v(y) - v(y_j) \pm i\ep_j)^{1-2\gamma_j}\rightarrow 1$ weakly in $L^1_y$ and pointwise for all $y\neq y_0$. Concerning the second contribution, we integrate by parts to find
\begin{align*}
\int_0^2 \frac{\cJ(y) - \cJ(y_j)}{(v(y) - v(y_j) \pm i\ep_j)^{1+2\gamma_j}}\phi_j(y) \d y &= -\frac{1}{2\gamma_j}\frac{\cJ(y) - \cJ(y_j)}{(v(y) - v(y_j) \pm i\ep_j^{2\gamma_j}}\frac{\phi_j(y)}{v'(y)}\Big|_{y=0}^{y=2} \\
&\quad + \frac{1}{2\gamma_j} \int_0^2 \left( \frac{\phi_j(y)}{v'(y)}\right)'\frac{\cJ(y) - \cJ(y_j)}{(v(y) - v(y_j) \pm i\ep_j)^{2\gamma_j}} \d y \\
&\quad + \frac{1}{2\gamma_j} \int_0^2  \frac{\phi_j(y)}{v'(y)} \frac{\cJ'(y)}{(v(y) - v(y_j) \pm i\ep_j)^{2\gamma_j}} \d y.
\end{align*}
The convergence of the first two terms being obvious (using the Dominated Convergence Theorem for the integral), we focus on the third term, for which we integrate by parts again. Appealing to \eqref{eq:gammajexpansion}, we see that
\begin{align*}
\int_0^2  \frac{\phi_j(y)}{v'(y)} &\frac{\cJ'(y)}{(v(y) - v(y_j) \pm i\ep_j)^{2\gamma_j}} \d y \\
&= \frac{\phi_j(y) \cJ'(y)}{(v'(y))^2}\frac{(v(y) - v(y_j) \pm i\ep_j)^{1-2\gamma_j} - 1}{1-2\gamma_j} \Big|_{y=0}^{y=2} \\
&\quad - \int_0^2 \partial_y\left( \frac{\phi_j(y) \cJ'(y)}{(v'(y))^2} \right) \log( v(y) - v(y_j) \pm i\ep_j) \d y \\
&\quad - (1-2\gamma_j)\int_0^2 \partial_y\left( \frac{\phi_j(y) \cJ'(y)}{(v'(y))^2} \right) \mathcal{Q}_{2,\gamma_j}(\xi_j)\log^2( v(y) - v(y_j) \pm i\ep_j) \d y, 
\end{align*}
where we denote $\xi_j=v(y) - v(y_j) \pm i\ep_j$. The convergence of the first term is immediate, while the third term vanishes in the limit. To prove the convergence of the second term we argue as in the proof of Lemma \ref{lemma:firstconvergencelemma}, we omit the routine details.
\end{proof}

\subsection{In-homogeneous solutions in the fragile regime}
Next, we address the convergence of $h_{k,\ep_j}^\pm(y,y_j)$. To ease notation, we denote 
\begin{align*}
\phi_{\sr,j}(y) := (v(y) - v(y_j) \pm i\ep_j)^{\frac12+\gamma_j}\phi_j(y)
\end{align*} 
where we recall that $\phi_j(y)=\phi_{\sr,1,k,\ep_j}^\pm(y,y_j)$.

\begin{lemma}\label{lemma:vanishingfirstinhomsol}
Let $k\geq 1$, $y_j\rightarrow y_0$ with $\P(y_0)=0$ and $\ep_j\rightarrow 0^+$ for $j\rightarrow \infty$. Let $\mathrm{g}_j:=-\textsc{E}_{k_j,\ep_j}^\pm g_j$ and $g_j\in Z_k$ with $\Vert g_j \Vert_{Z_k}\rightarrow 0$ as $j\rightarrow \infty$. Then,
\begin{align*}
\lim_{j\rightarrow \infty}  \left \Vert \phi_{\sr,j}(y) \int_0^y \frac{1}{(\phi_{\sr,j}(z))^2} \int_{y_j}^z \phi_{\sr,j}(s) \mathrm{g}_j(s) \d s \d z \right\Vert_{Z_k} = 0,
\end{align*}
In particular, 
\begin{align*}
\lim_{j\rightarrow \infty} \int_0^2 \frac{1}{(\phi_{\sr,j}(z))^2} \int_{y_j}^z \phi_{\sr,j}(s) \mathrm{g}_j(s) \d s \d z = 0
\end{align*}
since $\phi_{\sr,j}(2)$ is uniformly bounded away from 0.
\end{lemma}

\begin{proof}
We decompose
\begin{align*}
\phi_{\sr,j}(y)\int_0^y \frac{1}{(\phi_{\sr,j}(z))^2} &\int_{y_j}^z \phi_{\sr,j}(s) \mathrm{g}_j(s) \d s \d z \\
&= \phi_{\sr,j}(y)\int_0^y \frac{1}{(\phi_{\sr,j}(z))^2}\int_{y_j}^z (v(s) - v(y_j) \pm i\ep_j)^{\frac12+\gamma_j}\left( \phi_{\sr,1,j}(s) -1 \right)  \mathrm{g}_j(s) \d s \d z \\
&\quad + \phi_{\sr,j}(y)\int_0^y \frac{\int_{y_j}^z (v(s) - v(y_j) \pm i\ep_j)^{\frac12+\gamma_j} \mathrm{g}_j(s) \d s }{(v(z) - v(y_j)\pm i\ep_j)^{1+2\gamma_j} } \left( \frac{1}{\phi_{\sr,1,j}^2(z)} - 1 \right) \d z \\
&\quad + \phi_{\sr,j}(y)\int_0^y \frac{\int_{y_j}^z (v(s) - v(y_j) \pm i\ep_j)^{\frac12+\gamma_j} \mathrm{g}_j(s) \d s }{(v(z) - v(y_j)\pm i\ep_j)^{1+2\gamma_j} }  \d z \\
&= \sum_{n=1}^3h_{n,j}(y) 
\end{align*}
We next recall that $\mathrm{g}_j:=-\textsc{E}_{k_j,\ep_j}^\pm g_j$ and $g_j\in Z_k$ with $\Vert g_j \Vert_{Z_k}\rightarrow 0$ as $j\rightarrow \infty$. Moreover, we have from \eqref{eq:deferroroperator} that
\begin{align*}
\left| \textsc{E}_{k_j,\ep_j}^\pm(z) \right| \lesssim \frac{1}{|v(z) -v(y_j) \pm i\ep_j|}
\end{align*}
and with \eqref{eq:defbB} and  \eqref{eq:boundbB} we conclude that
\begin{equation}\label{eq:rmrjbound}
\left| \mathrm{g}_j(y) \right| \lesssim  |v(y)-v(y_j)\pm i\ep_{j}|^{-\frac12-\gamma_j} \Vert g_j \Vert_{Z_k}.
\end{equation}
for all $y\in (0,2)$, since now $k\geq 1$ is fixed. Hence, thanks to Lemma \ref{lemma:fragilephir1knonzero}, we deduce
\begin{align*}
\left \Vert \frac{1}{(\phi_{\sr,j}(z))^2}\int_{y_j}^z (v(s) - v(y_j) \pm i\ep_j)^{\frac12+\gamma_j}\left( \phi_{\sr,1,j}(s) -1 \right)  \mathrm{g}_j(s) \d s \right \Vert_{L^2_z} \lesssim \Vert g_j \Vert_{Z_k}
\end{align*}
and thus $ \left \Vert h_{1,j}(y)  \right \Vert_{Z_k} \lesssim \Vert g_j \Vert_{Z_k}$. Similarly, 
\begin{align*}
\left \Vert \frac{\int_{y_j}^z (v(s) - v(y_j) \pm i\ep_j)^{\frac12+\gamma_j} \mathrm{g}_j(s) \d s }{(v(z) - v(y_j)\pm i\ep_j)^{1+2\gamma_j} } \left( \frac{1}{\phi_{\sr,1,j}^2(z)} - 1 \right) \right \Vert_{L^2_z} \lesssim \Vert g_j \Vert_{Z_k}
\end{align*}
so that $\left \Vert h_{2,j} \right\Vert_{Z_k} \lesssim \Vert g_j \Vert_{Z_k}$. For the last term, we integrate by parts and obtain
\begin{align*}
h_{3,j}(y) &= -\frac{\phi_{\sr,j}(y)}{2\gamma_j v'(z)} \frac{\int_{y_j}^z (v(s) - v(y_j) \pm i\ep_j)^{\frac12+\gamma_j} \mathrm{g}_j(s) \d s }{(v(z) - v(y_j)\pm i\ep_j)^{2\gamma_j} } \Big|_{z=0}^{z=y} \\
&\quad + \frac{\phi_{\sr,j}(y)}{2\gamma_j}\int_0^y \left( \frac{1}{v'(z)} \right)' \frac{\int_{y_j}^z (v(s) - v(y_j) \pm i\ep_j)^{\frac12+\gamma_j} \mathrm{g}_j(s) \d s }{(v(z) - v(y_j)\pm i\ep_j)^{2\gamma_j} }  \d z \\
&\quad + \frac{\phi_{\sr,j}(y)}{2\gamma_j} \int_0^y (v(z) - v(y_j) \pm i\ep_j)^{\frac12-\gamma_j} \frac{\mathrm{g}_j(z)}{v'(z)} \d z \\
&= h_{4,j}(y) + h_{5,j}(y) + h_{6,j}(y) + h_{7,j}(y)
\end{align*}
We further note that
\begin{align*}
h_{4,j}(y) = \left( \frac{\phi_j(y)}{2\gamma_j v'(y)}\int_{y_j}^y (v(z) - v(y_j) \pm i\ep_j)^{\frac12+\gamma_j} \mathrm{g}_j(z) \d z \right) (v(y) - v(y) \pm i\ep_j)^{\frac12-\gamma_j}
\end{align*}
and since
\begin{align*}
\left \Vert \frac{\phi_j(y)}{2\gamma_j v'(y)}\int_{y_j}^y (v(z) - v(y_j) \pm i\ep_j)^{\frac12+\gamma_j} \mathrm{g}_j(z) \d z \right \Vert_{L^\infty(I_3(y_j)) \cap \dot{H}^1(I_3(y_j))} \lesssim \Vert g_j \Vert_{Z_k}
\end{align*}
we deduce that $\left \Vert h_{4,j}(y) \right \Vert_{Z_k} \lesssim \Vert g_j \Vert_{Z_k}$. Likewise, 
\begin{align*}
\left \Vert -\frac{\phi_j(y)}{2\gamma_j v'(0)} \frac{\int_{y_j}^0 (v(s) - v(y_j) \pm i\ep_j)^{\frac12+\gamma_j} \mathrm{g}_j(s) \d s }{(v(0) - v(y_j)\pm i\ep_j)^{2\gamma_j} } \right \Vert_{L^\infty(I_3(y_j)) \cap \dot{H}^1(I_3(y_j))} \lesssim \Vert g_j \Vert_{Z_k}
\end{align*}
and thus $\Vert h_{5,j} \Vert_{Z_k}\lesssim \Vert g_j \Vert_{Z_k}$. Next, regarding $h_{6,j}$ we note that
\begin{align*}
\left \Vert \left( \frac{1}{v'(z)} \right)' \frac{\int_{y_j}^z (v(s) - v(y_j) \pm i\ep_j)^{\frac12+\gamma_j} \mathrm{g}_j(s) \d s }{(v(z) - v(y_j)\pm i\ep_j)^{2\gamma_j} }  \right\Vert_{L^2_z} \lesssim \Vert g_j \Vert_{Z_k} 
\end{align*}
from which we deduce that $\Vert h_{6,j} \Vert \lesssim \Vert g_j \Vert_{Z_k}$. Lastly, since
\begin{align*}
\mathrm{g}_j(y) &= -\frac{v''(y)}{v(y) - v(y_j) \pm i\ep_j}g_j(y) + \frac{\mathrm{P}(y) - \mathrm{P}(y_j)}{(v(y) - v(y_j) \pm i\ep_j)^2}g_j(y) \\
&\quad+\left( \frac{\mathrm{P}(y_j)}{v(y) - v(y_j) \pm i\ep_j)^2} - \frac{\cJ(y_j)}{(y-y_j\pm i\ep_{0,j})^2}\right)g_j(y)
\end{align*}
we further decompose $h_{7,j}(y) = \sum_{n=8}^{10}h_{n,j}(y)$. Firstly,
\begin{align*}
h_{8,j}(y)&=  -\phi_{\sr,j}(y) \int_0^y \frac{v''(z)}{v'(z)}\frac{g_j(z)}{(v(z) - v(y_j) \pm i\ep_j)^{\frac12+\gamma_j}} \d z \\
&= -\phi_{\sr,j}(y)\int_0^y \frac{v''(z)}{v'(z)}\left( \frac{z-y_j\pm i\ep_{0,j}}{v(z) - v(y_j) \pm i\ep_j} \right) ^{\frac12+\gamma_j} (g_j)_{\sr}(z) \d z \\
&\quad -\phi_{\sr,j}(y)\int_0^y \frac{v''(z)}{v'(z)}\left( \frac{z-y_j\pm i\ep_{0,j}}{v(z) - v(y_j) \pm i\ep_j} \right) ^{\frac12-\gamma_j}  \frac{(g_j)_\s(z)}{(v(z) - v(y_j) \pm i\ep_j)^{2\gamma_j}} \d z \\
&= h_{11,j}(y) + h_{12,j}(y).
\end{align*}
Using \eqref{eq:boundbB} we see that $\Vert h_{11,j} \Vert_{Z_k} \lesssim \Vert g_j \Vert_{Z_k}$. On the other hand, for $h_{12,j}$ we integrate by parts once,
\begin{align*}
h_{12,j}(y) &=  \phi_{\sr,j}(y)\frac{v''(z)}{(v'(z))^2}  \left( B_{\ep_j}^\pm(z,y_j) \right) ^{\frac12-\gamma_j} (g_j)_\sr(z) \frac{\xi^{1-2\gamma_j} - 1 }{1-2\gamma_j} \Big|_{z=0}^{z=y} \\
&\quad + \phi_{\sr,j}(y)\int_0^y \partial_ z \left( \frac{v''(z)}{(v'(z))^2}  \left( B_{\ep_j}^\pm(z,y_j) \right) ^{\frac12-\gamma_j} (r_j)_\sr(z)  \right) \mathcal{Q}_{\frac12-\gamma_j}(\xi)\log(\xi) \d z \\
&= h_{13,j}(y) + h_{14,j}(y),
\end{align*}
where $\xi:=v(z) - v(y_j) \pm i\ep_j$. We now write
\begin{align*}
h_{13,j}(y) &= \left(h_{13,j}\right)_\s(y)(v(y) - v(y_j) \pm i\ep_j)^{\frac12-\gamma_j} , \\
\left(h_{13,j}\right)_\s(y) &:= \frac{v''(y)}{(v'(y))^2}\left( B_{\ep_j}^\pm(y,y_j) \right)^{\frac12-\gamma_j} \phi_j(y) (g_j)_\sr(y) (v(y) - v(y_j) \pm i\ep_j)^{2\gamma_j} \frac{\eta^{1-2\gamma_j}-1}{1-2\gamma_j}
\end{align*}
with $\eta = v(y) - v(y_j) \pm i\ep_j$ and
\begin{align*}
\Vert \left(h_{13,j}\right)_\s \Vert_{L^\infty(I_3(y_j))} + \Vert \partial_y \left(h_{13,j}\right)_\s(y) \Vert_{L^2(I_3(y_j))} \lesssim \Vert g_j \Vert_{Z_k}
\end{align*}
because $\Vert \zeta^{2\gamma_j-1}\log(\zeta)\Vert_{L^2}\lesssim 1$ uniformly for $\gamma_j\rightarrow \frac12$. Hence, $\Vert h_{13,j} \Vert_{Z_k} \lesssim \Vert g_j \Vert_{Z_k}$. Arguing as before, 
\begin{align*}
\left \Vert \partial_ z \left( \frac{v''(z)}{(v'(z))^2}  \left( B_{\ep_j}^\pm(z,y_j) \right) ^{\frac12-\gamma_j} (r_j)_\sr(z)  \right) \mathcal{Q}_{\frac12-\gamma_j}(\xi)\log(\xi) \right \Vert_{L^2_z(0,2)}\lesssim \Vert g_j \Vert_{Z_k}
\end{align*}
and thus $\Vert h_{14,j} \Vert_{Z_k} \lesssim \Vert g_j \Vert_{Z_k}$. This completes the $Z_k$ estimate of $h_{8,j}$. For $h_{9,j}(y)$, we indicate how to obtain analogous bounds. We write
\begin{align*}
\phi_{\sr,j}(y)\int_0^y &\frac{\P(z) - \P(y_j)}{(v(z) - v(y_j) \pm i\ep_j)^{\frac32+\gamma_j}}\frac{g_j(z)}{v'(z)} \d z \\
&= \phi_{\sr,j}(y)\int_0^y \frac{\P(z) - \P(y_j)}{v(z) - v(y_j) \pm i\ep_j}\left( \frac{z-y_j \pm i\ep_{0,j}}{v(z) - v(y_j) \pm i\ep_j} \right)^{\frac12+\gamma_j}\frac{(g_j)_\sr(z)}{v'(z)} \d z \\
&\quad + \phi_{\sr,j}(y)\int_0^y \left( \frac{z-y_j \pm i\ep_{0,j}}{v(z) - v(y_j) \pm i\ep_j} \right)^{\frac12-\gamma_j}\frac{(g_j)_\sr(z)}{v'(z)}\frac{\P(z) - \P(y_j)}{(v(z) - v(y_j) \pm i\ep_j)^{1+2\gamma_j}} \d z 
\end{align*} 
and as before we easily observe that the first integral is bounded in $Z_k$ by $\Vert g_j \Vert_{Z_k}$. For the second integral, we integrate by parts once more
\begin{align*}
\phi_{\sr,j}(y) &\int_0^y \left( \frac{z-y_j \pm i\ep_{0,j}}{v(z) - v(y_j) \pm i\ep_j} \right)^{\frac12-\gamma_j}\frac{(g_j)_\s(z)}{v'(z)}\frac{\P(z) - \P(y_j)}{(v(z) - v(y_j) \pm i\ep_j)^{1+2\gamma_j}} \d z \\
&= -\frac{\phi_{\sr,j}(y)}{2\gamma_j} \left( \frac{z-y_j \pm i\ep_{0,j}}{v(z) - v(y_j) \pm i\ep_j} \right)^{\frac12-\gamma_j}\frac{(g_j)_\s(z)}{(v'(z))^2}\frac{\P(z) - \P(y_j)}{(v(z) - v(y_j) \pm i\ep_j)^{2\gamma_j}} \Big|_{z=0}^{z=y} \\
&\quad + \frac{\phi_{\sr,j}(y)}{2\gamma_j}\int_0^y \frac{\P'(z)}{(v'(z))^2}\left( \frac{z-y_j \pm i\ep_{0,j}}{v(z) - v(y_j) \pm i\ep_j} \right)^{\frac12-\gamma_j} \frac{(g_j)_\s(z)}{(v(z) - v(y_j) \pm i\ep_j)^{2\gamma_j}} \d z \\
&\quad +\frac{\phi_{\sr,j}(y)}{2\gamma_j} \int_0^y \partial_z \left( \left( \frac{z-y_j \pm i\ep_{0,j}}{v(z) - v(y_j) \pm i\ep_j} \right)^{\frac12-\gamma_j}\frac{(g_j)_\sr(z)}{(v'(z))^2} \right) \frac{\P(z) - \P(y_j)}{(v(z) - v(y_j) \pm i\ep_j)^{2\gamma_j}} \d z. 
\end{align*}
We argue as in $h_{4,j}$ for the solid boundary term. For the first integral term we argue as for $h_{12,j}$ while for the second integral term we argue as for $h_{6,j}$, we omit the routine details. Finally, for $h_{10,j}$ we first write
\begin{align*}
\frac{\P(y_j)}{v(z) - v(y_j) \pm i\ep_j)^2} &- \frac{\cJ(y_j)}{(z-y_j\pm i\ep_{0,j})^2} \\
&= \frac{v(z) - v(y_j) - v'(y_j)(z-y_j)}{v(z) - v(y_j) \pm i\ep_j} \frac{1}{(z-y_j \pm i\ep_{0,j})^2} \\
&+ \frac{v(z) - v(y_j) - v'(y_j)(z-y_j)}{z-y_j \pm i\ep_{0,j}} \frac{v'(y_j)}{(v(z) - v(y_j) \pm i\ep_j)^2}
\end{align*}
Hence,
\begin{align*}
h_{10,j}(y) &= \phi_{\sr,j}(y)\int_0^y \frac{v(z) - v(y_j) - v'(y_j)(z-y_j)}{(v(z) - v(y_j) \pm i\ep_j)^{\frac12+\gamma_j}} \frac{1}{(z-y_j \pm i\ep_{0,j})^{\frac32-\gamma_j}}\frac{(g_j)_\sr(z)}{v'(z)} \d z \\
&\quad + \phi_{\sr,j}(y)\int_0^y \frac{v(z) - v(y_j) - v'(y_j)(z-y_j)}{(z-y_j \pm i\ep_{0,j})^{\frac12-\gamma_j}} \frac{v'(y_j)}{(v(z) - v(y_j) \pm i\ep_j)^{\frac32+\gamma_j}} \frac{(g_j)_\sr(z)}{v'(z)} \d z \\
&\quad +\phi_{\sr,j}(y) \int_0^y \frac{v(z) - v(y_j) - v'(y_j)(z-y_j)}{v(z) - v(y_j) \pm i\ep_j} \left( b_{\ep_j}^\pm(z,y_j)\right)^{\frac32+\gamma_j} \frac{(g_j)_\s(z)}{(v(z) - v(y_j) \pm i\ep_j)^{1+2\gamma_j}}\frac{ \d z}{v'(z)} \\
&\quad +\phi_{\sr,j}(y) \int_0^y \frac{v(z) - v(y_j) -v'(y_j)(z-y_j)}{z-y_j\pm i\ep_{0,j}} \left( B_{\ep_j}^\pm(z,y_j) \right)^{\frac12-\gamma_j} \frac{(g_j)_\s(z)}{(v(z) - v(y_j) \pm i\ep_j)^{1+2\gamma_j}}\frac{ \d z}{v'(z)},
\end{align*}
where $b_{\ep_j}^\pm$ and $B_{\ep_j}^\pm$ are defined in \eqref{eq:defbB}. Further observing that 
\begin{align*}
\left| \frac{v(z) - v(y_j) - v'(y_j)(z-y_j)}{v(z) - v(y_j) \pm i\ep_j} \right| + \left|  \frac{v(z) - v(y_j) - v'(y_j)(z-y_j)}{z-y_j \pm i\ep_{0,j}}  \right| \lesssim |z-y_j|
\end{align*}
and
\begin{align*}
\left| \partial_z \left( \frac{v(z) - v(y_j) - v'(y_j)(z-y_j)}{v(z) - v(y_j) \pm i\ep_j} \right) \right| + \left| \partial_z \left( \frac{v(z) - v(y_j) - v'(y_j)(z-y_j)}{z-y_j \pm i\ep_{0,j}} \right) \right| \lesssim 1
\end{align*}
uniformly in $z,y_j\in (0,2)$ and $\ep_j>0$, we see that the first two integrals can be bounded as for $h_{6,j}$, while for the last two integrals we argue as for $h_{9,j}$, we omit the details.
\end{proof}

Once the next lemma is established, the proof of Theorem \ref{thm:vanishinghjfragile} is complete.

\begin{lemma}\label{lemma:vanishingsecondinhomsol}
Let $k\geq1$, $y_j\rightarrow y_0\in (0,2)$ such that $\P(y_0)=0$ with $\P(y_j)>0$ and $\ep_j\rightarrow 0$ as $j\rightarrow \infty$. Then,
\begin{align*}
\left\Vert  \phi_{\sr,j}(y)\int_0^y \frac{1}{(\phi_{\sr,j}(z))^2} \d z \right\Vert_{Z_k} \lesssim 1.
\end{align*}
\end{lemma}

\begin{proof}
The arguments follow closely those of Proposition \ref{prop:limitWronkianFragile} and Lemma \ref{lemma:vanishingfirstinhomsol}. Integrating by parts,
\begin{align*}
\phi_{\sr,j}(y)\int_0^y \frac{1}{(\phi_{\sr,j}(z))^2} \d z &= -\frac{\phi_{\sr,j}(y)}{2\gamma_j v'(z)}(v(z) - v(y_j) \pm i\ep_j)^{-2\gamma_j}\phi_j^{-2}(z) \Big|_{z=0}^{z=y} \\
&\quad + \frac{\phi_{\sr,j}(y)}{2\gamma_j} \left( \frac{1}{v'(z)} \right)' \phi_{j}^{-2}(z) \frac{(v(z) - v(y_j) \pm i\ep_j)^{1-2\gamma_j}-1}{1-2\gamma_j} \Big|_{z=0}^{z=y} \\
&\quad - \frac{\phi_{\sr,j}(y)}{\gamma_j}\int_0 ^ y (v(z) - v(y_j) \pm i\ep_j)^{-2\gamma_j} \frac{\phi_j'(z) }{v'(z)\phi_j^3(z)} \d z \\
&\quad - \frac{\phi_{\sr,j}(y)}{2\gamma_j}\int_0^y \partial_z \left( \left( \frac{1}{v'(z)} \right)' \frac{1}{\phi_{j}^2(z)} \right) \frac{ (v(z) - v(y_j) \pm i\ep_j)^{1-2\gamma_j}-1}{1-2\gamma_j} \d z
\end{align*}
so that we argue as in $h_{4,j}$ for the first boundary term and as in $h_{13,j}$ for the second boundary term, while we address the second integral as in $h_{14,j}$. For the first integral, since $\phi_j'(y) = \cT_{\sr,1,k,\ep_j}\phi_j$ we have
\begin{align*}
\frac{\phi_{\sr,j}(y)}{\gamma_j}  &\int_0 ^ y (v(z) - v(y_j) \pm i\ep_j)^{-2\gamma_j} \frac{\phi_j'(z) }{v'(z)\phi_j^3(z)} \d z \\
&= -\frac{\phi_{\sr,j}(y)}{2\gamma_j^2} (v(z) - v(y_j) \pm i\ep_j)^{-2\gamma_j} \frac{\int_{y_j}^z F_{\sr,k,\ep_j}(s,y_j)(v(s) - v(y_j) \pm i\ep_j)^{2\gamma_j}\phi_j(s) \d s}{(v'(z))^2 \phi_j^3(z)} \Big|_{z=0}^{z=y} \\
&\quad + \frac{\phi_{\sr,j}(y)}{2\gamma_j^2} \int_0^y (v(z) - v(y_j) \pm i\ep_j)^{-2\gamma_j}  \partial_z \left( \frac{\int_{y_j}^z F_{\sr,k,\ep_j}(s,y_j)(v(s) - v(y_j) \pm i\ep_j)^{2\gamma_j}\phi_j(s) \d s}{(v'(z))^2 \phi_j^3(z)}\right) \d z
\end{align*}
And thus we argue as in $h_{4,j}$ for the solid term, and as in $h_{6,j}$ for the integral contribution.
\end{proof}

\section{The limiting absorption principle for the stratified regime}\label{sec:LAPstrat}
We are now in position to establish coercive bounds via the limiting absorption principle in the $Z_k$ space.
\begin{proposition}\label{prop:LAPZk}
There exists some $\kappa>0$ such that
\begin{align*}
\Vert f \Vert_{Z_k} \leq \kappa \Vert f + T_{k,\ep}^\pm f \Vert_{Z_k},
\end{align*}
uniformly for all $y_0\in I_S\cup I_W$ and $0<\ep<\ep_*$.
\end{proposition}

\begin{proof}
Assume towards a contradiction that we can find a sequence of parameters $k_j\geq 1$, $y_j\in I_S\cup I_W$ such that $y_j\rightarrow y_*$, $\ep_j\rightarrow 0^+$ and $f_j\in Z_{k_j}$, with $\Vert f_j\Vert_{Z_{k_j}}=1$ such that 
\begin{align}\label{eq:failureLAP}
\Vert f_j + T_{k_j,\ep_j}^\pm f_j(\cdot,y_j)\Vert_{Z_{k_j}} \rightarrow 0
\end{align}
as $j\rightarrow \infty$. From Proposition \ref{prop:TmapsXktoXk}, $X_k\subset Z_k$ and \eqref{eq:failureLAP}, we deduce that $|k_j|\lesssim 1$ and hence $k_j\rightarrow k_*\in \mathbb{Z}\setminus \lbrace 0 \rbrace$ up to a subsequence. In particular, $k_j \equiv k_*$ for all $j$ sufficiently large. In what follows, we already consider $j$ large enough so that $k_j=k_*$. Now, let $g_j(y):=f_j(y) + T_{k_*,\ep_j}^\pm f_j(y,y_j)$ and define $h_j(y) = f_j(y) - g_j(y)$. It is such that $\lim_{j\rightarrow \infty} \Vert h_j \Vert_{Z_{k_*}}=1$ and
\begin{align}\label{eq:defRj}
h_j(y) + T_{k_*,\ep_j}^\pm h_j(y,y_j) = -T_{k_*,\ep_j}^\pm g_j(y,y_j) 
\end{align}
In particular, we now have that $h_j(0)=h_j(2)=0$. Next, applying $\text{RTG}_{k_*,\ep_j}^\pm$ to both sides we obtain
\begin{equation}\label{eq:TGhj}
\textsc{TG}_{k_*,\ep_j}^\pm h_j(y,y_j) = -\textsc{E}_{k_*,\ep_j}^\pm g_j(y,y_j).
\end{equation}
To solve the equation, we distinguish whether the limiting spectral parameter $y_0$ is such that $\P(y_0)=0$ or not.

\bullpar{Case $\P(y_0)\neq 0$} Then $\gamma_0\neq \frac12$ and since $h_j(0)=h_j(2)=0$, we have that
\begin{align*}
h_j(y)&= \left( \frac{\phi_{\s,j}(2)\int_0^2 \phi_{\sr,j}(z) \mathrm{g}_j(z) \d z - \phi_{\sr,j}(2)\int_0^2 \phi_{\s,j}(z) \mathrm{g}_j(z)\d z }{\phi_{\sr,j}(0)\phi_{\s,j}(2) - \phi_{\sr,j}(2)\phi_{\s,j}(0)} \right)  \frac{\phi_{\sr,j}(0)\phi_{\s,j}(y) - \phi_{\s,j}(0)\phi_{\sr,j}(y)}{2\gamma_j v'(y_j)} \\
&\quad + \frac{1}{2\gamma_jv'(y_j)} \left( \phi_{\s,j}(y) \int_0^y \phi_{\sr,j}(z)\mathrm{g}_j(z) \d z - \phi_{\sr,j}(y) \int_0^y \phi_{\s,j}(z)\mathrm{g}_j(z) \d z \right)
\end{align*}
where here $\phi_{\sigma,j}(y)=\phi_{\sigma,k_j,\ep_j}(y,y_j)$, for $\sigma = \s,\,\sr$, denote the pair of linearly independent homogeneous solutions of the $\textsc{TG}_{k_j,\ep_j}^\pm$ operator well defined in Proposition  \ref{prop:existencesolphisigma} because $\gamma_j\rightarrow \gamma_0\neq\frac12$ and $\mathrm{g}_j:=-\textsc{E}_{k_j,\ep_j}^\pm g_j$. Since $k_*$ is fixed, for $\ep<\ep_*$ we have from Proposition \ref{prop:nonzeroWstrongstrat} if $y_j\rightarrow y_0\in I_S$ and Proposition \ref{prop:nonzeroWweakstrat} if $y_j\rightarrow y_0\in I_W$ that 
\begin{equation}\label{eq:nonzeroweakstrongWlim}
\lim_{j\rightarrow \infty} \phi_{\sr,j}(0)\phi_{\s,j}(2) - \phi_{\sr,j}(2)\phi_{\s,j}(0) \neq 0.
\end{equation}
Next, for
\begin{equation}\label{eq:defbB}
b_{\ep_j}^\pm(y,y_j) := \frac{v(y) - v(y_j) \pm i\ep_j}{y-y_j \pm i\ep_{0,j}}, \quad B_{\ep_j}^\pm(y,y_j) := \frac{1}{b_{\ep_j}^\pm(y,y_j)}
\end{equation}
a direct computations shows that 
\begin{equation}\label{eq:boundbB}
\Vert b_{\ep}^\pm \Vert_{L^\infty_y} + \Vert B_{\ep}^\pm \Vert_{L^\infty_y} +  \Vert \partial_y b_{\ep}^\pm \Vert_{L^\infty_y} + \Vert  \partial_y B_{\ep}^\pm \Vert_{L^\infty_y} \lesssim 1,
\end{equation}
uniformly for all $y_j \in (0,2)$ and all $\ep_j>0$. Hence, 
since $|\textsc{E}_{k_*,\ep_j}^\pm(y,y_j)| \lesssim |v(y) - v(y_j) \pm i\ep_j|^{-1}$ and $k_*$ is fixed, we deduce that $\Vert \phi_{\sigma,j} \Vert_{Z_{k_*}} \lesssim 1$ for $\sigma \in \lbrace \sr, \s \rbrace$ and
\begin{align*}
\left\Vert  \phi_{\s,j}(y) \int_0^y \phi_{\sr,j}(z) \mathrm{g}_j(z) \d z \right\Vert_{Z_{k_*}} &+ \left\Vert  \phi_{\sr,j}(y)\int_0^y \phi_{\s,j}(z) \textsc{E}_{k_*,\ep_j}^\pm (y-y_j \pm i\ep_{0,j} )^{\frac12+\gamma_j} \left(g_j \right)_\sr (z) \d z \right\Vert_{Z_{k_*}}  \\
&\qquad\lesssim \Vert g_j \Vert_{Z_{k_*}}.
\end{align*}
We next claim that 
\begin{align*}
\left\Vert  \phi_{\sr,j}(y)\int_0^y \phi_{\s,j}(z) \textsc{E}_{k_*,\ep_j}^\pm (z-y_j \pm i\ep_{0,j} )^{\frac12-\gamma_j} \left(g_j \right)_\s (z) \d z \right\Vert_{Z_{k_*}}  \lesssim \Vert g_j \Vert_{Z_{k_*}},
\end{align*}
for which we note that
\begin{align*}
\phi_{\sr,j}(y) &\int_0^y \phi_{\s,j}(z) \textsc{E}_{1,k_*,\ep_j}^\pm(z,y_j) (z-y_j \pm i\ep_{0,j})^{\frac12-\gamma_j} \left( g_j \right)_\s(z) \d z \\
&= \phi_{\sr,j}(y) \int_0^y (v(z) - v(y_j) \pm i\ep_j)^{-2\gamma_j} v''(z) (B_{\ep_j}^\pm(z,y_j))^{\frac12-\gamma_j} \phi_{\s,j}(z)(g_j)_\s(z) \d z \\
&= -\frac{1}{1-2\gamma_j}\phi_{\sr,j}(y) (v(z) - v(y_j) \pm i\ep_j)^{1-2\gamma_j} \mathsf{g}_j^\pm(z,y_j) \Big|_{z=0}^{z=y} \\
&\quad + \frac{1}{1-2\gamma_j}\phi_{\sr,j}(y)\int_0^y (v(z) - v(y_j) \pm i\ep_j)^{1-2\gamma_j} \partial_z \mathsf{g}_j^\pm(z,y_j) \d z
\end{align*}
where $\mathsf{g}_j^\pm(z,y_j):= \frac{v''(z)}{v'(z)}(B_{\ep_j}^\pm(z,y_j))^{\frac12-\gamma_j} \phi_{\s,j}(z)(g_j)_\s(z)$ and $\Vert \mathsf{g}_j^\pm(\cdot,y_j)\Vert_{W^{1,\infty}_y(I_3(y_j))}\lesssim \Vert g_j \Vert_{Z_{k_*}}$, uniformly in $y_j\in I_S\cup I_W$ and $\ep_j>0$. In particular, we have
\begin{align*}
\frac{1}{|1-2\gamma_j|}\left \Vert \int_0^y (v(z) - v(y_j) \pm i\ep_j)^{1-2\gamma_j} \partial_z \mathsf{g}_j^\pm(z,y_j) \d z \right\Vert_{W^{1,\infty}_y(I_3(y_j))}\lesssim \Vert g_j \Vert_{Z_{k_*}}
\end{align*}
and
\begin{align*}
\frac{1}{1-2\gamma_j}&\phi_{\sr,j}(y) (v(y) - v(y_j) \pm i\ep_j)^{1-2\gamma_j} \mathsf{g}_j^\pm(y,y_j) \\
&= -\left( \frac{v(y) - v(y_j) \pm i\ep_j}{1-2\gamma_j}\phi_{\sr,1,j}(y)\mathsf{g}_j^\pm(y,y_j)  \right) (v(y) - v(y_j)\pm i\ep_j)^{\frac12-\gamma_j}
\end{align*}
with now $\left \Vert \frac{v(y) - v(y_j) \pm i\ep_j}{1-2\gamma_j}\phi_{\sr,1,j}(y)\mathsf{g}_j^\pm(y,y_j)  \right\Vert_{W^{1,\infty}_y(I_3(y_j))}\lesssim \Vert g_j \Vert_{Z_{k_*}}$, since $1-2\gamma_j$ is uniformly bounded away from 0 because $\gamma_0\neq \frac12$. The estimates for $\textsc{E}_{2,k_*,\ep_j}^\pm$ and $\textsc{E}_{3,k_*,\ep_j}^\pm$ are deduced similarly, integrating by parts to soften out the potential singularities in the spirit of Lemma \ref{lemma:R1mapsPLinfL2toXk} and Lemma \ref{lemma:R2mapsC1XktoXk}, we omit the routine details. Finally, the above arguments also provide
\begin{align*}
\left\Vert \phi_{\sigma,j}(y) \int_0^2 \phi_{\tau,j}(z) g_j(z) \d z \right\Vert_{Z_{k_*}}\lesssim \Vert g_j \Vert_{Z_{k_*}}
\end{align*}
for all $\sigma,\tau\in \lbrace \sr, \s \rbrace$. As a result, this shows that
\begin{align*}
\lim_{j\rightarrow \infty}\Vert h_j \Vert_{Z_{k_*}} \lesssim \lim_{j\rightarrow \infty}\Vert g_j \Vert_{Z_{k_*}} = 0,
\end{align*}
a contradiction with $\lim_{j\rightarrow \infty}\Vert h_j \Vert_{Z_{k_*}} = 1$.

\bullpar{Case $\P(y_0)=0$} Now $\gamma_0=\frac12$ and we can no longer use Proposition \ref{prop:existencesolphisigma} to find the solution to \eqref{eq:TGhj} by means of $\phi_{\sr,j}$ and $\phi_{\s,j}$, since $\phi_{\s,j}$ is a priori not well-defined. Instead, we appeal to Proposition \ref{prop:defhsolTGr} and Theorem \ref{thm:vanishinghjfragile} to conclude that $h_j(y)\rightarrow 0$ for all $y\in(0,2)$, thus reaching a contradiction with $\Vert h_j \Vert_{X_k}=1$.
\end{proof} 

\begin{proposition}\label{prop:LAPLZk}
There exists some $\kappa>0$ such that 
\begin{align*}
\Vert  f \Vert_{LZ_k} \leq \kappa \Vert f + T_{k,y_0,\ep}^\pm f  \Vert_{LZ_k},
\end{align*}
uniformly for all $y_0\in I_M$ and $0<\ep<\ep_*$.
\end{proposition}

\begin{proof}
Arguing as in Proposition \ref{prop:LAPZk}, we assume towards a contradiction that there exist $k\geq 1$, $y_j\in I_M$, $\ep_j>0$, $h_j\in LX_k$ and $r_j\in LX_k$ such that
\begin{align*}
\textsc{TG}_{k_*,\ep_j}^\pm h_j(y,y_j) = -\textsc{E}_{k_*,\ep_j}^\pm(y,y_j) g_j(y,y_j)
\end{align*}
with $h_j(0,y_j) = h_j(2,y_j)=0$ and
\begin{align*}
y_j\rightarrow y_0\in I_M, \quad \ep_j\rightarrow 0, \quad \Vert h_j \Vert_{LX_k}\rightarrow 1, \quad \Vert g_j \Vert_{LX_k}\rightarrow 0,
\end{align*}
as $j\rightarrow \infty$. Just as before, we distinguish two cases according to the limiting $y_0\in I_M$.

\bullpar{ Case $\cJ(y_0)\neq \frac14$} Then $\gamma_0\neq\frac12$ and also $\gamma_0\neq 0$, so that we can write 
\begin{align*}
h_j(y)&= \left( \frac{\phi_{\s,j}(2)\int_0^2 \phi_{\sr,j}(z) \mathrm{g}_j(z) \d z - \phi_{\sr,j}(2)\int_0^2 \phi_{\s,j}(z) \mathrm{g}_j(z)\d z }{\phi_{\sr,j}(0)\phi_{\s,j}(2) - \phi_{\sr,j}(2)\phi_{\s,j}(0)} \right)  \frac{\phi_{\sr,j}(0)\phi_{\s,j}(y) - \phi_{\s,j}(0)\phi_{\sr,j}(y)}{2\gamma_j v'(y_j)} \\
&\quad + \frac{1}{2\gamma_j v'(y_j)} \left( \phi_{\s,j}(y) \int_0^y \phi_{\sr,j}(z)\mathrm{g}_j(z) \d z - \phi_{\sr,j}(y) \int_0^y \phi_{\s,j}(z)\mathrm{g}_j(z) \d z \right)
\end{align*}
as before, where $\phi_{\sigma,j}(y)=\phi_{\sigma,k_j,\ep_j}(y,y_j)$, for $\sigma = \s,\,\sr$ are defined as in \eqref{eq:defphisigma}. In particular, since $\gamma_0\neq 0$ we see from Proposition \ref{prop:nonzeroWstrongstrat} for $\gamma_0\in i\R$ and Proposition \ref{prop:nonzeroWweakstrat} for $\gamma_0\in \R$ that 
\begin{align*}
\lim_{j\rightarrow\infty}\phi_{\sr,j}(0)\phi_{\s,j}(2) - \phi_{\sr,j}(2) \phi_{\s,j}(0) \neq 0.
\end{align*}
Furthermore, we also have
\begin{align*}
\left| \phi_{\s,j}(2)\int_0^2 \phi_{\sr,j}(z) \mathrm{g}_j(z) \d z - \phi_{\sr,j}(2)\int_0^2 \phi_{\s,j}(z) \mathrm{g}_j(z)\d z \right| \lesssim \Vert g_j \Vert_{LZ_k}
\end{align*}
since $\zeta^{-2\gamma_j}\log(\zeta) \in L^1_{loc}$ uniformly for all $\gamma_j\rightarrow \gamma_0 \neq \frac12$, while we now write 
\begin{align*}
\frac{\phi_{\sr,j}(0)\phi_{\s,j}(y) - \phi_{\s,j}(0)\phi_{\sr,j}(y)}{2\gamma_j} &= \frac{\phi_{\sr,j}(0) - \phi_{\s,j}(0)}{2\gamma_j}\phi_{\sr,j}(y) - \phi_{\sr,j}(0) \frac{\phi_{\sr,j}(y) - \phi_{\s,j}(y)}{2\gamma_j} 
\end{align*}
with further 
\begin{align*}
&\frac{\phi_{\sr,j}(y) - \phi_{\s,j}(y)}{2\gamma_j}   \\
&= (y-y_j \pm i\ep_{0,j})^{\frac12+\gamma_j}\frac{(b_{\ep_j}^\pm(y,y_j))^{\frac12+\gamma_j} \phi_{\sr,k,1,k,\ep_j}(y,y_j) - (b_{\ep_j}^\pm(y,y_j))^{\frac12-\gamma_j} \phi_{\s,k,1,k,\ep_j}(y,y_j)}{2\gamma_j} \\
&\quad+ (b_{\ep_j}^\pm(y,y_j))^{\frac12-\gamma_j} \phi_{\s,k,1,k,\ep_j}(y,y_j)(y-y_j \pm i\ep_{0,j})^{\frac12-\gamma_j}\log (y-y_j \pm i\ep_{0,j})\Q_{1,\gamma_j}( y-y_j \pm i\ep_{0,j}).
\end{align*} 
Hence, we conclude that
\begin{align*}
\left \Vert \frac{\phi_{\sr,j}(0)\phi_{\s,j}(y) - \phi_{\s,j}(0)\phi_{\sr,j}(y)}{2\gamma_j} \right\Vert_{LZ_k} \lesssim 1.
\end{align*}
for $\gamma_j\rightarrow \gamma_0\neq 0$. Moreover, since
\begin{align*}
\Vert \phi_{\sigma,j} \mathrm{g}_j \Vert_{L^2(0,2)} \lesssim \Vert g_j \Vert_{LZ_k}
\end{align*}
for $\sigma\in \lbrace \sr, \s \rbrace$, we also deduce that
\begin{align*}
\left \Vert \frac{1}{2\gamma_j} \left( \phi_{\s,j}(y) \int_0^y \phi_{\sr,j}(z)\mathrm{g}_j(z) \d z - \phi_{\sr,j}(y) \int_0^y \phi_{\s,j}(z)\mathrm{g}_j(z) \d z \right) \right \Vert_{LZ_k} \lesssim \Vert g_j \Vert_{LZ_k}
\end{align*}
and thus we obtain $\Vert h_j \Vert_{LZ_k}\lesssim \Vert g_j \Vert_{LZ_k}\rightarrow 0$ as $j\rightarrow \infty$, contradicting $\lim_{j\rightarrow\infty} \Vert h_j \Vert_{LZ_k} =1$.

\bullpar{Case $\cJ(y_0)=\frac14$} Now we have $\gamma_0=0$, for which we use the homogeneous solutions $\phi_{\sr,j}(y):=\phi_{\sr,k,\ep_j}^\pm(y,y_j)$ and $\phi_{\rL,j}(y):=\phi_{\rL,k,\ep_j}^\pm(y,y_j)$ given by \eqref{eq:defphiL} and whose Wronskian is $v'(y_j)$, to construct 
\begin{align*}
h_j(y)&= \left( \frac{\phi_{\rL,j}(2)\int_0^2 \phi_{\sr,j}(z) \mathrm{g}_j(z) \d z - \phi_{\sr,j}(2)\int_0^2 \phi_{\rL,j}(z) \mathrm{g}_j(z)\d z }{\phi_{\sr,j}(0)\phi_{\rL,j}(2) - \phi_{\sr,j}(2)\phi_{\rL,j}(0)} \right)  \frac{\phi_{\sr,j}(0)\phi_{\rL,j}(y) - \phi_{\rL,j}(0)\phi_{\sr,j}(y)}{v'(y_j)} \\
&\quad +  \frac{1}{v'(y_j)}\left( \phi_{\rL,j}(y) \int_0^y \phi_{\sr,j}(z)\mathrm{g}_j(z) \d z - \phi_{\sr,j}(y) \int_0^y \phi_{\rL,j}(z)\mathrm{g}_j(z) \d z \right).
\end{align*}
Now, Proposition \ref{prop:limitWronkianMild} shows that
\begin{align*}
\lim_{j\rightarrow\infty} \phi_{\sr,j}(0)\phi_{\rL,j}(2) - \phi_{\sr,j}(2)\phi_{\rL,j}(0) \neq 0.
\end{align*}
Further recalling that 
\begin{align*}
\phi_{\rL,j}(y) = \eta^{\frac12+\gamma_j}\phi_{\rL,1,j}(y) + \eta^{\frac12-\gamma_j}\log(\eta) \Q_{\gamma_j}(\eta) \phi_{\s,j,1}(y)
\end{align*}
with $\eta= v(y) - v(y_j) \pm i\ep_j$, we readily see that
\begin{align*}
\Vert \phi_{\sr,j}(0)\phi_{\rL,j}(y) - \phi_{\rL,j}(0)\phi_{\sr,j}(y) \Vert_{LZ_k}\lesssim 1
\end{align*}
and
\begin{align*}
\left| \phi_{\rL,j}(2)\int_0^2 \phi_{\sr,j}(z) \mathrm{g}_j(z) \d z - \phi_{\sr,j}(2)\int_0^2 \phi_{\rL,j}(z) \mathrm{g}_j(z)\d z \right| \lesssim \Vert g_j \Vert_{LZ_k}
\end{align*}
as $\zeta^{-2\gamma_j}\log^2(\zeta)\in L^1$ uniformly for $\gamma_j\rightarrow \gamma_0=0$. Hence, we also have that $\Vert \phi_{\sigma,j}\mathrm{g}_j \Vert_{L^2(0,2)} \lesssim \Vert g_j \Vert_{LZ_k}$ and thus 
\begin{align*}
\left\Vert \phi_{\rL,j}(y) \int_0^y \phi_{\sr,j}(z)\mathrm{g}_j(z) \d z - \phi_{\sr,j}(y) \int_0^y \phi_{\rL,j}(z)\mathrm{g}_j(z) \d z \right\Vert_{LZ_k} \lesssim \Vert g_j \Vert_{LZ_k},
\end{align*}
from which we deduce that $\Vert h_j \Vert_{LZ_k}\lesssim \Vert g_j \Vert_{LZ_k}$, reaching a contradiction again with $\Vert h_j \Vert_{LZ_k}\rightarrow 1$ as $j\rightarrow \infty$.
\end{proof}

\begin{remark}[Spectral conditions]
For the limiting absorption principle to hold, an inspection of the proof shows that Lemma~\ref{lemma:nonzerophisrphis}, Lemma~\ref{lemma:fragilephir1knonzero}, and Proposition~\ref{prop:limitWronkianFragile} are key to obtain Proposition~\ref{prop:nonzeroWweakstrat} and Theorem~\ref{thm:vanishinghjfragile}, respectively. Thus, the coercive estimate in Propositions \ref{prop:LAPZk} and \ref{prop:LAPLZk} hold for all background stably stratified monotone shear flows for which Lemma~\ref{lemma:nonzerophisrphis}, Lemma~\ref{lemma:fragilephir1knonzero}, and Proposition~\ref{prop:limitWronkianFragile} remain true, which essentially preclude the existence of generalized embedded eigenvalues.
\end{remark}

\section{The non-stratified regime}\label{sec:nonstrat}
This section aims at proving estimates analogue to those of Propositions \ref{prop:LAPZk} and \ref{prop:LAPLZk} for $y_0\in I_E$, the non-stratified regime. Now, for $y_0\in I_E$ the Taylor-Goldstein equation is not as singular and we can use the ideas of \cites{Jia20, JiaGev20}. For $k\geq 1$, let $G_k(y,z)$ denote the Green's function of the operator $\partial_y^2 - k^2$ with homogeneous boundary conditions at $y=0$ and $y=2$. It is given explicitly by
\begin{equation}\label{eq:defGk}
G_k(y,z) = -\frac{1}{k\sinh(k)}\begin{cases}
\sinh(k(2-z))\sinh(ky), & \text{for }y\leq z, \\
\sinh(kz)\sinh(k(2-y)), & \text{for }z\leq y
\end{cases}
\end{equation}
and it is such that
\begin{align}
\sup_{y\in[0,2]}\Vert \partial_y^\alpha G_k(y,z) \Vert_{L^2_z(0,2)} &\lesssim k^{-\frac32+\alpha} \label{eq:L2boundGk} \\
\sup_{y\in[0,2]} \Vert \partial_z^\alpha G_{k_0}(y,z) \log(v(z) - v(y_0) \pm i\ep) \Vert_{L^1_z(0,2)}&\lesssim k^{-2+\alpha}\log(k)\label{eq:L1boundGklog}
\end{align}
for $\alpha=0,1$, $y_0\in[0,2]$ and $0\leq \ep\leq 1$. For each $f \in L^2(0,2)$ we define
\begin{equation}\label{eq:defTE}
T_{E,k,\ep}^\pm f(y,y_0) := \int_0^2 G_k(y,z) \left( -\frac{v''(z)}{v(z) - v(y_0) \pm i\ep} + \frac{\P(z)}{(v(z) - v(y_0) \pm i\ep)^2} \right) f(z) \d z
\end{equation}
and
\begin{equation}\label{eq:defREm}
R_{E,m,k,\ep}^\pm f(y,y_0) := \int_0^2 G_k(y,z) \frac{f(z)}{(v(z) - v(y_0) \pm i\ep)^m} \d z
\end{equation}
for $m=0,1$. Therefore, we can rewrite \eqref{eq:introTGvarphi} as
\begin{equation}\label{eq:LAPTGvarphi}
\varphi_{k,\ep}^\pm(y,y_0) + T_{E,k,\ep}^\pm\varphi_{k,\ep}^\pm(y,y_0) = R_{E,1,k,\ep}^\pm\omega_k^0(y,y_0) + R_{E,0,k,\ep}^\pm \varrho_k^0(y,y_0)
\end{equation}
with $\varphi_{k,\ep}^\pm(0,y_0) = \varphi_{k,\ep}^\pm(2,y_0)=0$. 

\subsection{Operator estimates for the non-stratified regime}
To capture the precise regularity in the wave-number $k\geq 1$, we recall the space
\begin{equation*}
\Vert g \Vert_{H_k^1}:=\Vert g \Vert_{L^2(0,2)} + k^{-1}\Vert g' \Vert_{L^2(0,2)}.
\end{equation*}
The next result extends \cite[Lemma 3.1]{JiaGev20} to the operator $T_{E,k,\ep}^\pm$.

\begin{lemma}\label{lemma:TEmapsH1ktoH1k}
Let $k\geq 1$, $y_0\in I_E$  and $\ep\in (0,1)$. There holds
\begin{equation}\label{eq:H1boundTkyep}
\Vert T_{E,k,\ep}^\pm f \Vert_{H_k^1} \lesssim |k|^{-\frac13}\Vert f \Vert_{H_k^1}
\end{equation}
for all $f\in H_k^1$. Furthermore, we also have
\begin{equation}\label{eq:pyTEboundW11}
\left\Vert \partial_y T_{E,k,\ep}^\pm f(y,y_0)  - \frac{v''(y)}{v(y)}f(y)\log(v(y) - v(y_0) \pm i\ep) \right\Vert_{W_y^{1,1}(0,2)} \lesssim k^\frac43\Vert f \Vert_{H_k^1},
\end{equation}
for all $f\in H_k^1$
\end{lemma}

\begin{proof}
For the first part of the lemma, thanks to  \cite[Lemma 3.1]{JiaGev20}, we shall only prove
\begin{align*}
\left \Vert \int_0^2 G_k(y,z)  \frac{\P(z)}{(v(z) - v(y_0) \pm i\ep)^2} f(z) \d z \right\Vert_{H_k^1}\lesssim |k|^{-\frac13}\Vert f \Vert_{H_k^1}.
\end{align*}
The estimate \eqref{eq:L2boundGk} together with $\P(y_0) = \P'(y_0)=0$ and $\P\in C^2$, and the usual Cauchy-Schwarz inequality gives the desired bound. The second statement follows similarly, appealing to  \cite[Lemma 3.1]{JiaGev20} we just need to show that 
\begin{align*}
\left \Vert \partial_y \int_0^2 G_k(y,z)  \frac{\P(z)}{(v(z) - v(y_0) \pm i\ep)^2} f(z) \d z \right\Vert_{W_y^{1,1}(0,2)}\lesssim |k|^{\frac12}\Vert f \Vert_{H_k^1}.
\end{align*}
Now, \eqref{eq:L2boundGk} directly gives 
\begin{align*}
\left \Vert \partial_y \int_0^2 G_k(y,z)  \frac{\P(z)}{(v(z) - v(y_0) \pm i\ep)^2} f(z) \d z \right\Vert_{L^1_y(0,2)}\lesssim |k|^{-\frac12}\Vert f \Vert_{H_k^1},
\end{align*}
while for the $\dot{W}^{1,1}$ estimate, we note that $\partial_y^2 G_k(y,z) = k^2g_k(y,z) + \delta(y-z)$. Hence,
\begin{align*}
\left \Vert \partial^2_y \int_0^2 G_k(y,z)  \frac{\P(z)}{(v(z) - v(y_0) \pm i\ep)^2} f(z) \d z \right\Vert_{L^1_y(0,2)} &\leq \left \Vert \frac{\P(y)}{(v(y) - v(y_0) \pm i\ep)^2}f(y) \right\Vert_{L^1_y(0,2)} \\
&\quad + k^2 \left \Vert \int_0^2  \frac{G_k(y,z) \P(z)f(z)}{(v(z) - v(y_0) \pm i\ep)^2}  \d z \right\Vert_{L^1_y(0,2)} \\
&\lesssim k^\frac12\Vert f \Vert_{H_k^1}
\end{align*}
and the lemma follows.
\end{proof}

The following $H_k^1$ bound on $R_{E,0,k,\ep}^\pm$ is immediate from \eqref{eq:L2boundGk} as well.
\begin{lemma}\label{lemma:RE0mapsL2toH1k}
Let $k\geq 1$, $y_0\in I_E$ and $\ep>0$. Then,
\begin{align*}
\Vert R_{E,0,k,\ep}^\pm f(y,y_0) \Vert_{H_k^1} \lesssim k^{-\frac32}\Vert f \Vert_{L^2(0,2)},
\end{align*}
for all $f\in L^2(0,2)$.
\end{lemma}

Similarly, we have the following.

\begin{lemma}[Lemma 3.1 in \cite{JiaGev20}]\label{lemma:RE1mapsH1ktoH1k}
Let $k\geq 1$, $y_0\in I_E$ and $\ep>0$. Then,
\begin{align*}
\Vert R_{E,1,k,\ep}^\pm f(y,y_0) \Vert_{H_k^1} \lesssim k^{-\frac13}\Vert f \Vert_{H_k^1}
\end{align*}
for all $f\in H_k^1$.
\end{lemma}

\subsection{The limiting absorption principle in the non-stratified regime}
Next, we show the main limiting absorption principle coercive bound:

\begin{proposition}\label{prop:LAPHk1}
There exists $\kappa>0$ such that 
\begin{equation}
\Vert f\Vert_{H^1_k} \leq \kappa \Vert f + (T_{E,k,\ep}^\pm f)(\cdot,y_0)\Vert_{H^1_k}
\end{equation}
for all $y_0\in I_E$, $k\geq 1$, $\ep>0$ and $f\in H_k^1(0,2)$.
\end{proposition}

\begin{proof}
We argue as in the proof of \cite[Lemma 3.2]{JiaGev20}: by contradiction, assume there exist sequences $k_j\in \mathbb{Z} \setminus \lbrace 0 \rbrace$, $y_j\in I_E$, $\ep_j>0$ and functions $f_j\in H_{k_j}^1$ with $\Vert f_j \Vert_{H_{k_j}^1(0,2)}=1$, for all $j\geq 1$, such that 
\begin{equation*}
k_j \rightarrow k_0\in (\mathbb{Z}\setminus \lbrace 0 \rbrace) \cup \lbrace \pm \infty \rbrace, \quad y_j \rightarrow y_0\in I_E, \quad \ep_j\rightarrow 0
\end{equation*}
and
\begin{equation}\label{eq:limH1norm}
\Vert f_j + T_{E,k_j,\ep_j}^\pm f_j(\cdot,y_j)\Vert_{H_{k_j}^1} \rightarrow 0
\end{equation}
as $j\rightarrow \infty$. Thanks to \eqref{eq:H1boundTkyep} and \eqref{eq:limH1norm}, it is clear that $|k_j|\lesssim 1$, and thus $k_0\in \Z\setminus \lbrace 0 \rbrace$. Since $\Vert g_j \Vert_{H_{k_j}^1}=1$, the embedding $W^{1,1}(0,2)\rightarrow L^2(0,2)$ is compact and we have the estimate \eqref{eq:pyTEboundW11}, we deduce that $T_{E,k_0,\ep_j}f(\cdot,y_j)$ converges in $H^1_{k_0}$ up to subsequences. Hence, from  \eqref{eq:limH1norm} we see that $f_j\rightarrow f$ in $H^1_{k_0}$, with $\Vert f \Vert_{H^1_{k_0}} = 1$. Then, since 
\begin{align*}
T_{E,k_0,\ep_j}^\pm f_j(y,y_j) &= \int_0^2 \partial_z \left( G_{k_0}(y,z)\frac{v''(z)}{v'(z)}f_j(z) \right) \log(v(z) - v(y_j) \pm i\ep_j) \d z \\
&\quad + \int_0^2 G_{k_0}(y,z)\frac{\P(z)}{(v(z) - v(y_j) \pm i\ep_j)^2}f_j(z) \d z
\end{align*}
and  for $\alpha=0,1$ uniformly in $y_j\in I_E$ and $\ep_j>0$ we see that 
\begin{equation*}
f(y) - \lim_{j\rightarrow \infty}  \int_0^2 G_{k_0}(y,z) \frac{v''(z)}{v(z) - v(y_j) + i\ep_j}f(z) \d z +\int_0^2 G_{k_0}(y,z)\frac{P(z)}{(v(z) - v(y_0) )^2}  f(z) \d z = 0
\end{equation*}
almost everywhere in $[0,2]$. Then, applying $\partial_y^2 - k^2$ we reach
\begin{align*}
\partial_y^2 f(y) - k^2 f(y) - \lim_{j\rightarrow \infty} (v(y) - v(y_j))\frac{v''(y) f(y)}{(v(y) - v(y_j))^2 + \ep_j^2} &+i C_{y_0} v''(y_0) f(y_0) \delta(y-y_0) \\
&+ \frac{\P(y)}{(v(y) - v(y_0))^2}f(y) = 0
\end{align*}
in the sense of distributions for $y\in (0,2)$. Multiplying by $\overline{f(y)}$, integrating by parts over $(0,2)$ and taking the imaginary part we reach 
\begin{align*}
0 = v''(y_0)f(y_0)
\end{align*}
with also $v''f\in H_{k_0}^1$. Therefore, since $\frac{v''(y)f(y)}{v(y) - v(y_0)}\in L^2$ and $\frac{P(y)}{(v(y) - v(y_0))^2}\in L^\infty$, we have \begin{equation}\label{eq:defHeigenfunction}
\mathrm{H}(y) := \frac{v''(y)}{v(y) - v(y_0)}f(y) - \frac{\P(y)}{(v(y) - v(y_0))^2}f(y)\in L^2_y(0,2)
\end{equation}
and $\Delta_k f = \mathrm H$. Moreover, we also define
\begin{equation}\label{eq:defUpsilon}
\Upsilon(y) := \frac{\mathcal{P}(y)}{v(y) - v(y_0)}f(y) \in L^2_y(0,2)
\end{equation}
and we further note that
\begin{equation}\label{eq:HUpsiloneigen}
\begin{split}
(v(y) - v(y_0)) \mathrm H(y) - \frac{v''(y)}{v(y) - v(y_0)}\int_0^2 G_{k_0}(y,z)\mathrm H(z) \d z  + \Upsilon(y) &= 0, \\
(v(y) - v(y_0))\Upsilon(y) - \mathcal{P}(y)\int_0^2 G_{k_0}(y,z)\mathrm H(z) \d z  &=0.
\end{split}
\end{equation}
This shows that the pair $(\mathrm H,\Upsilon)$ is an $L^2$ eigenfunction of the linearised operator, with embedded eigenvalue $v(y_0)$, a contradiction with Proposition~\ref{prop:nonstratifiedresolvent} and Lemma~\ref{lemma:boundarysuppnoteigen}.
\end{proof}

As in the limiting absorption principle for $Z_k$ and $LZ_k$, the coercive estimate of Proposition~\ref{prop:LAPHk1} holds if the linearised operator $L_k$ does not have embedding eigenvalues in the non-stratified region.

\subsection{Sobolev regularity in the non-stratified regime}

With Proposition \ref{prop:LAPHk1} at hand, Lemma \ref{lemma:RE0mapsL2toH1k}, Lemma \ref{lemma:RE1mapsH1ktoH1k} and \eqref{eq:LAPTGvarphi} yield

\begin{proposition}\label{prop:Hk1varphi}
Let $\varphi_{k,\ep}^\pm$ solve \eqref{eq:LAPTGvarphi}. Then,
\begin{align*}
\Vert \varphi_{k,\ep}^\pm \Vert_{H_k^1}\lesssim k^{-\frac13}\Vert w_k^0\Vert_{H_k^1} + k^{-\frac32}\Vert q_k^0 \Vert_{L^2}
\end{align*}
for all $k\geq 1$, $y_0\in I_E$ and $\ep>0$. 
\end{proposition}
Since the initial data $\omega_k^0$ and $\varrho_k^0$ is supported in the stratified regime, we next show that the spectral density function in the non-stratified regime vanishes as $\ep\rightarrow 0$.
\begin{lemma}\label{lemma:vanishingpsiIE}
Let $k\geq 1$. Then,
\begin{equation*}
\lim_{\ep \rightarrow 0}\psi_{k,\ep}^-(y,y_0) - \psi_{k,\ep}^+(y,y_0) = 0,
\end{equation*}
for all $y_0\in I_E$ and all $y\in [0,2]$.
\end{lemma}

\begin{proof}
Since $\psi_{k,\ep}^\pm(y,y_0) = \varphi_{k,\ep}^\pm(y,y_0) - \varrho_k^0(y)$, we have that
\begin{align*}
\psi_{k,\ep}(y,y_0) := \psi_{k,\ep}^-(y,y_0) - \psi_{k,\ep}^+(y,y_0) = \varphi_{k,\ep}^-(y,y_0) - \varphi_{k,\ep}^+(y,y_0)
\end{align*}
and it is straightforward to see that $\psi_{k,\ep}(y,y_0)$ satisfies
\begin{align*}
\text{TG}_{k,\ep}^-\psi_{k,\ep}(y,y_0) &= \frac{2i\ep}{(v(y) - v(y_0))^2 + \ep^2}w_k^0(y) + \frac{2i\ep v''(y)}{(v(y) - v(y_0))^2 + \ep^2}\psi_{k,\ep}^+(y,y_0) \\
&-\P(y)\frac{2i\ep(v(y) - v(y_0))}{((v(y)-v(y_0))^2 + \ep^2)^2}\psi_{k,\ep}^+(y,y_0).
\end{align*}
Observing that $(\sp \omega_k^0 \cup \sp v'' )\cap I_E = \emptyset$ and $\P\equiv 0$ in $I_E$ with $\P\in C^2$ we deduce that
\begin{align*}
\psi_{k,0}(y,y_0) := \lim_{\ep\rightarrow 0}\psi_{k,\ep}(y,y_0) \in L^2(0,2)
\end{align*}
satisfies
\begin{align*}
\D_k\psi_{k,0}(y,y_0) - \frac{v''(y)}{v(y) - v(y_0)}\psi_{k,0}(y,y_0) + \frac{\P(y)}{(v(y) - v(y_0))^2}\psi_{k,0}(y,y_0) = 0
\end{align*}
in the sense of distributions. As before, setting 
\begin{align*}
\mathrm H(y,y_0) &:= \frac{v''(y)}{v(y) - v(y_0)}\psi_{k,0}(y,y_0) - \frac{\P(y)}{(v(y) - v(y_0))^2}\psi_{k,0}(y,y_0), \\ 
\Upsilon(y,y_0) &:= \frac{\P(y)}{v(y) - v(y_0)}\psi_{k,0}(y,y_0),
\end{align*}
we see that $\mathrm H(\cdot,y_0),\Upsilon(\cdot,y_0) \in L^2_y(0,2)$ and they satisfy \eqref{eq:HUpsiloneigen}. Hence, $v(y_0)$ is an embedded eigenvalue, a contradiction with Proposition~\ref{prop:nonstratifiedresolvent} and Lemma~\ref{lemma:boundarysuppnoteigen}.
\end{proof}

\section{Sobolev regularity of the spectral density function}\label{sec:SobolevReg}
With the limiting absorption principle coercive estimates at hand we now show Sobolev regularity of the solutions ${\varphi}_{k,\ep}^\pm(y,y_0)$ to
\begin{equation}\label{eq:TGeqvarphiSobolevReg}
\l(\D_k-\frac{v''(y)}{v(y)-v(y_0)\pm i\ep}+\frac{\P(y)}{(v(y)-v(y_0)\pm i\ep)^2}\r){\varphi}_{k,\ep}^\pm(y,y_0)=\frac{w_k^0(y)}{v(y)-v(y_0)\pm i\ep} + q_k^0(y).
\end{equation}
To keep track of the regularity of the initial data, for $k\geq 1$ and $j\geq 0$ we define
\begin{align}\label{eq:defSnorm}
\cS_{k,j} = k^j \left( \Vert w_k^0 \Vert_{H_k^{j+1}} + k^{-1} \Vert  q_k^0 \Vert_{H_k^j} \right),
\end{align}
where we further define
\begin{align*}
\Vert f \Vert_{H_k^j} := \sum_{n=0}^j k^{-n} \Vert \partial_y^n f \Vert_{L^2}
\end{align*}
for all $f\in H^j$. In particular, we note that $k\cS_{k,j}\leq \cS_{k,j+1}$ .

\begin{proposition}\label{prop:Zkvarphi}
Let $y_0\in I_S\cup I_W$ and ${\varphi}_{k,\ep}^\pm(y,y_0)$ a solution to \eqref{eq:TGeqvarphiSobolevReg}. Then,
\begin{align*}
\sup_{y_0\in I_S\cup I_W}\Vert {\varphi}_{k,\ep}^\pm(\cdot,y_0) \Vert_{Z_k}\lesssim k^{-\frac12}\cS_{k,0}.
\end{align*}
uniformly in $0< \ep< \ep_*$. 
\end{proposition}

\begin{proof}
Note that
\begin{equation}\label{eq:resolventvarphi}
\varphi_{k,\ep}^\pm(y,y_0) + T_{k,y_0,\ep}^\pm \varphi_{k,\ep}^\pm (y,y_0) = \left(R_{1,k,\ep}^\pm\omega_k^0\right)(y,y_0) + \left( R_{0,k,\ep}^\pm \varrho_k^0 \right) (y,y_0)
\end{equation}
where we recall that $T_{k,y_0,\ep}^\pm$, $R_{0,k,\ep}^\pm$ and $R_{1,k,\ep}^\pm$ are given by \eqref{eq:defsolopT} and \eqref{eq:defsolopR}, respectively. Now, thanks to Proposition \ref{prop:LAPZk},  Lemma \ref{lemma:R1mapsH1toXk} and Lemma \ref{lemma:R0mapsL2toXk} there holds
\begin{align*}
\Vert \varphi_{k,\ep}^\pm(\cdot,y_0) \Vert_{Z_k} &\lesssim \Vert \varphi_{k,\ep}^\pm(\cdot,y_0) + T_{k,y_0,\ep}^\pm \varphi_{k,\ep}^\pm (\cdot,y_0) \Vert_{Z_k}  \\
&\lesssim \Vert \varphi_{k,\ep}^\pm(\cdot,y_0) + T_{k,y_0,\ep}^\pm \varphi_{k,\ep}^\pm (\cdot,y_0) \Vert_{X_k}  \\
&\lesssim \Vert R_{1,k,\ep}^\pm\omega_k^0\Vert_{X_k} + \Vert R_{0,k,\ep}^\pm  \varrho_k^0 \Vert_{X_k} \\
&\lesssim k^{-\frac12}\Vert \omega_k^0 \Vert_{H^1_k} + k^{-\frac32}\Vert  \varrho_k^0\Vert_{L^2},
\end{align*}
uniformly for all $y_0\in I_S\cup I_W$, and all $0< \ep< \ep_*$. 
\end{proof}

With the uniform $Z_k$ bounds at hand for $\varphi_{k,\ep}^\pm$, we are able to use \eqref{eq:TGeqvarphiSobolevReg} to upgrade them to uniform $X_k$ bounds.

\begin{proposition}\label{prop:Xkvarphi}
Let $y_0\in I_S\cup I_W$ and ${\varphi}_{k,\ep}^\pm(y,y_0)$ a solution to \eqref{eq:TGeqvarphiSobolevReg}. Then,
\begin{align*}
\sup_{y_0\in I_S\cup I_W}\Vert {\varphi}_{k,\ep}^\pm(\cdot,y_0) \Vert_{X_k}\lesssim k^{-\frac12}\cS_{k,0}.
\end{align*}
uniformly in $0< \ep< \ep_*$. Moreover, we have
\begin{align}\label{eq:decomXkpyvarphi}
\partial_y\varphi_{k,\ep}^\pm(y,y_0) =  \varphi_{1,1,\sr,k,\ep}^\pm(y,y_0) \eta^{-\frac12+\gamma_0} + \varphi_{1,1,\s,k,\ep}^\pm(y,y_0) \eta^{-\frac12+\gamma_0} + \widetilde\varphi_{0,k,\ep}^\pm(y,y_0) 
\end{align}
where
\begin{align*}
\varphi_{1,1,\sr,k,\ep}^\pm (y,y_0) := k(1+2\mu_0)\left( \varphi_{k,\ep}^\pm \right)_\sr, \quad \varphi_{1,1,\s,k,\ep}^\pm (y,y_0) := k(1-2\mu_0)\left( \varphi_{k,\ep}^\pm \right)_\s
\end{align*}
and $\Vert \widetilde\varphi_{0,k,\ep}^\pm \Vert_{L^\infty(I_3(y_0))} \lesssim k^{\frac12}\cS_{k,0}$ uniformly for all $0< \ep< \ep_*$. Additionally, $\varphi_{k,\ep}^\pm$ enjoys the more precise bound
\begin{align*}
\Vert \partial_{y}\varphi_{k,\ep}^\pm \Vert_{H_k^1(I_3^c(y_0))} \lesssim  \cS_{k,0}, \quad \Vert \partial_y^n\varphi_{k,\ep}^\pm \Vert_{L^{\infty}(I_3^c(y_0))} \lesssim k^{n-\frac12}\cS_{k,0}, 
\end{align*}
for $n=0,1,2$, uniformly for all $y_0\in I_S\cup I_{W}$ and all $0 < \ep\leq \ep_*$.
\end{proposition}

\begin{proof}
From \eqref{eq:resolventvarphi} we have
\begin{align*}
\Vert \varphi_{k,\ep}^\pm \Vert_{X_k} \leq \Vert T_{k,\ep}^\pm \varphi_{k,\ep}^\pm (y,y_0) \Vert_{X_k} + \Vert R_{1,k,\ep}^\pm\omega_k^0\Vert_{X_k} + \Vert R_{0,k,\ep}^\pm  \varrho_k^0 \Vert_{X_k}
\end{align*}
and the estimate swiftly follows from Proposition \ref{prop:TmapsXktoXk}, Lemma \ref{lemma:R1mapsH1toXk}, Lemma \ref{lemma:R0mapsL2toXk} and the $Z_k$ bounds on $\varphi_{k,\ep}^\pm$ from Proposition \ref{prop:Zkvarphi}. The decomposition \eqref{eq:decomXkpyvarphi} and the bounds for its coefficients are immediate. Likewise, $H^1_k(I_3(y_0))$ estimate on $\partial_y\varphi_{k,\ep}^\pm$ is a consequence of \eqref{eq:TGeqvarphiSobolevReg} and the $X_k$ estimates on $\varphi_{k,\ep}^\pm$. We next address the pointwise estimates. 

\diampar{Proof for $n = 0$}
Let $y\in I_3^c(y_0)$ and assume first that $y\in I_6(y_0)\cap I_3^c(y_0)$, with $y > y_0$, say. For $y_3 = y_0 + \frac{3\beta}{k}$ we have 
\begin{align*}
\varphi_{k,\ep}^\pm (y,y_0) = \varphi_{k,\ep}^\pm(y_3,y_0) + \int_{y_3}^y \partial_y\varphi_{k,\ep}^\pm(s, y_0) \d s,
\end{align*} 
with $| \partial_y\varphi_{k,\ep}^\pm(y_3,y_0) | \lesssim k^{\frac12}\cS_{k,0}$ due to the local bounds and 
$$\left| \int_{y_3}^y \partial_y\varphi_{k,\ep}^\pm(s, y_0) \d s \right| \lesssim |y-y_3|^\frac12 \Vert \partial_y\varphi_{k,\ep}^\pm \Vert_{L^2(I_3^c(y_0))} \lesssim k^{-\frac12}\cS_{k,0}$$
as well. On the other hand, if $y\in I_6^c{(y_0)}$, then $I_3(y)\cap I_3(y_0)= \emptyset$, and thus we can use the Sobolev-type inequality of Lemma \ref{lemma:LinfH1bound} to show that
\begin{align*}
|\varphi_{k,\ep}^\pm(y,y_0)| \leq \Vert \varphi_{k,\ep}^\pm(z,y_0) \Vert_{L^\infty_z(I_3(y))} \lesssim k^\frac12 \Vert \varphi_{k,\ep}^\pm(z,y_0) \Vert_{H_k^1(I_3(y))} \lesssim k^\frac12\Vert \varphi_{k,\ep}^\pm(z,y_0) \Vert_{H_k^1(I_3^c(y_0))} \lesssim k^{-\frac12}\cS_{k,0}
\end{align*}
uniformly for all $y_0\in I_S\cup I_W$.

\diampar{Proof for $n=1$} We proceed as before, now writing
\begin{align*}
\partial_y\varphi_{k,\ep}^\pm (y,y_0) = \partial_y\varphi_{k,\ep}^\pm(y_3,y_0) + \int_{y_3}^y \partial_y^2\varphi_{k,\ep}^\pm(s, y_0) \d s,
\end{align*} 
for $y\in I_6(y_0) \cap I_3^c(y_0)$. We have $| \partial_y\varphi_{k,\ep}^\pm(y_3,y_0) | \lesssim k^{\frac12}\cS_{k,0}$ and further using the Taylor-Goldstein equation \eqref{eq:TGeqvarphiSobolevReg}, the fact that the integration takes place in an interval of size roughly $k^{-1}$ and the bound $\Vert \varphi_{k,\ep}^\pm\Vert_{L^2(I_3^c(y_0))} \leq k^{-\frac12}\Vert \varphi_{k,\ep}^\pm\Vert_{X_k}$ we obtain the desired estimate. Similarly, for $y\in I_6^c{(y_0)}$,we use Sobolev-type inequality of Lemma \ref{lemma:LinfH1bound} as before.

\diampar{Proof for $n=2$} Here we just use \eqref{eq:TGeqvarphiSobolevReg} and the previous $L^\infty(I_3^c(y_0))$ bounds for $\varphi_{k,\ep}^\pm$.
\end{proof}

On the other hand, for $y_0\in I_M$ we have

\begin{proposition}\label{prop:LZkvarphi}
Let $k\geq 1$, $y_0\in I_M$ and $\varphi_{k,\ep}^\pm(y,y_0)$ a solution to \eqref{eq:TGeqvarphiSobolevReg}. Then,
\begin{align*}
\sup_{y_0\in I_M}\Vert \varphi_{k,\ep}^\pm(\cdot,y_0) \Vert_{LZ_k} \lesssim k^{-\frac12}\cS_{k,0}, 
\end{align*}
uniformly for all $0 < \ep< \ep_*$.
\end{proposition}

\begin{proof}
From \eqref{eq:resolventvarphi}, we use Proposition \ref{prop:LAPLZk} together with  Lemma \ref{lemma:R0mapsL2toLXk} and Lemma \ref{lemma:R1mapsLinfL2toLXk}.
\end{proof}

As before, with the $LZ_k$ bounds for $\varphi_{k,\ep}^\pm$ at hand we can improve them to $LX_k$ estimates.

\begin{proposition}\label{prop:LXkvarphi}
Let $k\geq 1$, $y_0\in I_M$ and $\varphi_{k,\ep}^\pm(y,y_0)$ a solution to \eqref{eq:TGeqvarphiSobolevReg}. Then,
\begin{align*}
\sup_{y_0\in I_M}\Vert \varphi_{k,\ep}^\pm(\cdot,y_0) \Vert_{LX_k} \lesssim k^{-\frac12}\cS_{k,0}, 
\end{align*}
uniformly for all $0 < \ep< \ep_*$. Moreover, we have
\begin{equation}\label{eq:decomLXkpyvarphi}
\begin{split}
\partial_{y}\varphi_{k,\ep}^\pm &= \varphi_{\rL,1,1,\sr,k,\ep}^\pm(y,y_0) \eta^{-\frac12+\gamma_0} + \varphi_{\rL,1,1,\rL,k,\ep}^\pm(y,y_0) \eta^{-\frac12-\gamma_0}\log(\eta) \Q_{\gamma_0}(\eta) \\
&\quad +\varphi_{\rL,1,1,\s,k,\ep}^\pm(y,y_0) \eta^{-\frac12-\gamma_0} + \widetilde\varphi_{\rL,0,k,\ep}^\pm(y,y_0)
\end{split}
\end{equation}
where
\begin{align*}
\varphi_{\rL,1,1,\sr,k,\ep}^\pm(y,y_0) &:= k(1+2\gamma_0)\left( \varphi_{k,\ep}^\pm \right)_\sr(y,y_0), \\ 
\varphi_{\rL,1,1,\rL,k,\ep}^\pm(y,y_0) &:= k(1-2\gamma_0)\left( \varphi_{k,\ep}^\pm \right)_\s(y,y_0), \\ 
\varphi_{\rL,1,1,\s,k,\ep}^\pm(y,y_0) &:= 2k\left( \varphi_{k,\ep}^\pm \right)_\s(y,y_0), \\ 
\end{align*}
with $\Vert \widetilde\varphi_{\rL,0,k,\ep}^\pm \Vert_{L^\infty(I_3(y_0))} \lesssim k^{\frac12}\cS_{k,0}$
Moreover, $\varphi_{k,\ep}^\pm$ enjoys the more precise bound
\begin{align*}
\Vert \partial_{y}\varphi_{k,\ep}^\pm \Vert_{H_k^1(I_3^c(y_0))} \lesssim  \cS_{k,0}, \quad \Vert \partial_{y}^n\varphi_{k,\ep}^\pm \Vert_{L^\infty(I_3^c(y_0))} \lesssim k^{n-\frac12}\cS_{k,0}, 
\end{align*}
for $n=0,1,2$, uniformly for all $y_0\in I_M$ and all $0 < \ep\leq \ep_*$.
\end{proposition}

\subsection{Sobolev regularity for the first derivative}
To obtain Sobolev regularity for $\partial_{y_0}\varphi_{k,\ep}^\pm$, we note that 
\begin{align*}
\partial_{y_0}\varphi_{k,\ep}^\pm =\left( \partial_{y}\varphi_{k,\ep}^\pm + \partial_{y_0}\varphi_{k,\ep}^\pm \right) - \partial_y\varphi_{k,\ep}^\pm.
\end{align*}
Let $\varphi_{1,k,\ep}^\pm(y,y_0) := \left( \partial_{y}\varphi_{k,\ep}^\pm + \partial_{y_0}\varphi_{k,\ep}^\pm \right)(y,y_0)$. There holds
\begin{equation}\label{eq:TGvarphi1}
\begin{split}
\text{TG}_{k,\ep}^\pm\varphi_{1,k,\ep}^\pm &= \frac{v'''(y)}{v(y) - v(y_0) \pm i\ep}\varphi_{k,\ep}^\pm - \frac{\P'(y)}{(v(y) - v(y_0)\pm i\ep)^2}\varphi_{k,\ep}^\pm \\
&\quad + 2\P(y)\frac{v'(y) - v'(y_0)}{(v(y) - v(y_0) \pm i\ep)^3}\varphi_{k,\ep}^\pm - v''(y) \frac{v'(y)-v'(y_0)}{(v(y) - v(y_0)\pm i\ep)^2} \varphi_{k,\ep}^\pm \\
&\quad + \frac{\partial_y w_k^0(y)}{v(y) - v(y_0)\pm i\ep} - w_k^0(y)\frac{v'(y) - v'(y_0)}{(v(y) - v(y_0) \pm i\ep)^2} + \partial_y q_k^0(y),
\end{split}
\end{equation}
with now
\begin{align*}
\varphi_{1,k,\ep}^\pm(0,y_0) = \partial_y\varphi_{k,\ep}^\pm(0,y_0), \quad  \varphi_{1,k,\ep}^\pm(0,y_0) = \partial_y\varphi_{k,\ep}^\pm(2,y_0).
\end{align*}
Therefore, we have
\begin{equation}\label{eq:LAPvarphi1}
\varphi_{1,k,\ep}^\pm(y,y_0) + T_{k,\ep}^\pm\varphi_{1,k,\ep}^\pm(y,y_0) = \left(R_{0,k,\ep}^\pm\text{TG}_{k,\ep}^\pm \varphi_{1,k,\ep}^\pm \right)(y,y_0) + \B_{0,k,\ep}^\pm(y,y_0).
\end{equation}
where we define
\begin{equation}\label{eq:defB0}
\B_{0,k,\ep}^\pm(y,y_0) := \partial_z\G_{k,\ep}^\pm(y,y_0,z)\varphi_{1,k,\ep}^\pm(z,y_0)\Big|_{z=0}^{z=2}.
\end{equation}

We next argue according to the stratification regime.

\subsubsection{Strong and weak stratifications}
In order to use Proposition \ref{prop:LAPZk} on \eqref{eq:LAPvarphi1} we first obtain $X_k$ estimates on $\B_{0,k,\ep}^\pm$. 
\begin{lemma}\label{lemma:XkB0}
Let $k\geq 1$. Then,  
\begin{align*}
\Vert \B_{0,k,\ep}^\pm \Vert_{X_k}\lesssim k^\frac12\cS_{k,0}
\end{align*}
uniformly in $0<\ep<\ep_*$ and $y_0\in I_S \cup I_W$.
\end{lemma}

\begin{proof}
Firstly, since $\varphi_{k,\ep}^\pm(z,y_0)=0$ for $z=0,2$ for all $y_0\in (0,2)$, we note that $|\varphi_{1,k,\ep}^\pm(z,y_0)| = |\partial_y\varphi_{k,\ep}^\pm(z,y_0)|\lesssim k^\frac12\cS_{k,0}$ for $z=0,2$ due to Proposition \ref{prop:Xkvarphi}. We next argue for $\partial_z\G_{k,\ep}^\pm(y,y_0,z)$ for $z=0$. We have
\begin{align*}
\partial_z\G_{k,\ep}^\pm(y,y_0,0) = \frac{\phi_{\su,k,\ep}^\pm(y,y_0)}{\W_{k,\ep}^\pm(y_0)}\partial_y\phi_{\sl,k,\ep}^\pm(0,y_0).
\end{align*}
where there further holds, see \cite{NIST}, 
\begin{equation}\label{eq:pyphil}
\partial_y\phi_{\sl, k,\ep}^\pm(0,y_0) = -4k\gamma_0
\end{equation}
Then, 
\begin{align*}
\Vert \partial_z\G_{k,\ep}^\pm(y,y_0,0)  \Vert_{X_k} \lesssim 1
\end{align*}
follows from \eqref{eq:pyphil}, Proposition \ref{prop:sobolevregdecomG} and Corollary \ref{cor:sobolevregpartialdecomG}. The lemma is then proved. 
\end{proof}

We are now in position to prove the main regularity properties of $\varphi_{1,k,\ep}^\pm$ in $X_k$.

\begin{proposition}\label{prop:Xkvarphi1}
Let $k\geq 1$. There holds 
\begin{align*}
\Vert \varphi_{1,k,\ep}^\pm \Vert_{X_k} \lesssim k^{-\frac12}\cS_{k,1}
\end{align*}
uniformly for all $y_0\in I_S\cup I_{W_*}$ and all $0 < \ep\leq \ep_*$. Moreover, we have
\begin{align*}
\Vert \partial_y^n \varphi_{1,k,\ep}^\pm \Vert_{L^\infty(I_3(y_0))} \lesssim k^{-\frac12}\cS_{k,n+1}
\end{align*}
for $n=0,1$, uniformly for all $y_0\in I_S\cup I_{W}$ and all $0 < \ep\leq \ep_*$.
\end{proposition}

\begin{proof}
We use the limiting absorption principle Proposition \ref{prop:LAPZk} on \eqref{eq:LAPvarphi1} which provides
\begin{align*}
\Vert \varphi_{1,k,\ep}^\pm \Vert_{Z_k}\lesssim \Vert R_{0,k,\ep}^\pm\text{TG}_{k,\ep}^\pm\varphi_{1,k,\ep}^\pm \Vert_{X_k} + \Vert \B_{0,k,\ep}^\pm \Vert_{X_k}.
\end{align*}
From \eqref{eq:TGvarphi1} we see that
\begin{align*}
\Vert \text{TG}_{k,\ep}^\pm\varphi_{1,k,\ep}^\pm \Vert_{X_k}  &\leq \left\Vert R_{1,k,\ep}^\pm\left( v'''\varphi_{k,\ep}^\pm \right) \right\Vert_{X_k} + \left\Vert R_{2,k,\ep}^\pm \left(\P'\varphi_{k,\ep}^\pm\right) \right\Vert_{X_k} \\
&\quad + 2 \left\Vert R_{3,k,\ep}^\pm \left( \P h_1(\cdot,y_0) \varphi_{k,\ep}^\pm  \right)  \right\Vert_{X_k}  + \left\Vert R_{2,k,\ep}^\pm \left(v'' h_1(\cdot,y_0)\varphi_{k,\ep}^\pm \right)  \right\Vert_{X_k} \\
&\quad + \left\Vert R_{1,k,\ep}^\pm \partial_y w_k^0 \right\Vert_{X_k} + \left\Vert R_{2,k,\ep}^\pm \left(h_1(\cdot,y_0) w_k^0\right) \right \Vert_{X_K} + \left\Vert R_{0,k,\ep}^\pm \partial_y q_k^0 \right\Vert_{X_k}
\end{align*}
where we have defined $h_1(y,y_0):=v'(y) - v'(y_0)$. Thanks to Proposition \ref{prop:Xkvarphi}, we have $\Vert \varphi_{k,\ep}^\pm \Vert_{X_k}\lesssim k^{-\frac12}\cS_{k,0}$. Furthermore,  Corollary \ref{cor:R1mapsCXktoXk} gives
\begin{align*}
\left\Vert R_{1,k,\ep}^\pm\left( v'''\varphi_{k,\ep}^\pm \right) \right\Vert_{X_k} &\lesssim k^{-1}\Vert \varphi_{k,\ep}^\pm \Vert_{X_k} \lesssim k^{-\frac32}\cS_{k,0}.
\end{align*}
Similarly, Corollary \ref{cor:R2mapsCXktoXk} shows that
\begin{align*}
\left\Vert R_{2,k,\ep}^\pm \left(\P'\varphi_{k,\ep}^\pm\right) \right\Vert_{X_k} \lesssim \Vert \varphi_{k,\ep}^\pm \Vert_{X_k}\lesssim k^{-\frac12}\cS_{k,0},
\end{align*}
while Lemma \ref{lemma:R3mapsC1XktoXk} gives
\begin{align*}
\left\Vert R_{3,k,\ep}^\pm \left( \P h_1(\cdot,y_0) \varphi_{k,\ep}^\pm  \right)  \right\Vert_{X_k} \lesssim \Vert \varphi_{k,\ep}^\pm \Vert_{X_k}\lesssim k^{-\frac12}\cS_{k,0}
\end{align*}
and Lemma \ref{lemma:R2mapsC1XktoXk} provides
\begin{align*}
\left\Vert R_{2,k,\ep}^\pm \left(v'' h_1(\cdot,y_0)\varphi_{k,\ep}^\pm \right)  \right\Vert_{X_k}\lesssim k^{-1}\Vert \varphi_{k,\ep}^\pm \Vert_{X_k}\lesssim k^{-\frac32}\cS_{k,0}.
\end{align*}
Finally, Lemma \ref{lemma:R1mapsH1toXk}, Lemma \ref{lemma:R2mapsC1H1toXk} and Lemma \ref{lemma:R0mapsL2toXk} yield
\begin{align*}
\left\Vert R_{1,k,\ep}^\pm \partial_y w_k^0 \right\Vert_{X_k} &+ \left\Vert R_{2,k,\ep}^\pm \left(h_1(\cdot,y_0) w_k^0\right) \right \Vert_{X_K} + \left\Vert R_{0,k,\ep}^\pm \partial_y q_k^0 \right\Vert_{X_k} \\
&\lesssim k^{-\frac12}\Vert \partial_y w_k^0\Vert_{H_k^1} + k^{-\frac12}\Vert w_k^0\Vert_{H_k^1} + k^{-\frac32}\Vert \partial_y q_k^0\Vert_{L^2} \\
&\lesssim k^{-\frac12}\cS_{k,1}.
\end{align*}
With this and Lemma \ref{lemma:XkB0}, the $Z_k$ estimate is established. The $X_k$ estimate is then obtained as in Proposition \ref{prop:Xkvarphi}. The proof of the pointwise estimate is the same as in Proposition\ref{prop:Xkvarphi}.
\end{proof}

We are now in position to provide a useful description of $\partial_{y_0}\varphi_{k,\ep}^\pm$.

\begin{proposition}\label{prop:Xkpynotvarphi}
Let $k\geq 1$ and $y_0\in I_S\cup I_W$, we have
\begin{equation}\label{eq:decomXkpynotvarphi}
\begin{split}
\partial_{y_0}\varphi_{k,\ep}^\pm &= -\varphi_{1,1,\sr,k,\ep}^\pm(y,y_0) \eta^{-\frac12+\gamma_0} -\varphi_{1,1,\s,k,\ep}^\pm(y,y_0) \eta^{-\frac12-\gamma_0} + \widetilde\varphi_{1,k,\ep}^\pm(y,y_0)
\end{split}
\end{equation}
where $\varphi_{1,1,\sigma,k,\ep}^\pm(y,y_0)$ are given as in Proposition \ref{prop:Xkvarphi}, with further $\Vert \widetilde\varphi_{1,k,\ep}^\pm \Vert_{L^\infty(I_3(y_0))} \lesssim  k^{-\frac12}\cS_{k,1}$. Moreover, $\partial_{y_0}\varphi_{k,\ep}^\pm$ enjoys the more precise bound
\begin{align*}
\Vert \partial_{y_0}\varphi_{k,\ep}^\pm \Vert_{H_k^1(I_3^c(y_0))} \lesssim  k^{-1}\cS_{k,1}, \quad \Vert \partial_{y_0}\varphi_{k,\ep}^\pm \Vert_{L^\infty(I_3^c(y_0))} \lesssim k^{-\frac12}\cS_{k,1}, 
\end{align*}
uniformly for all $y_0\in I_S\cup I_{W}$ and all $0 < \ep\leq \ep_*$.
\end{proposition}

\begin{proof}
There holds $\partial_{y_0}\varphi_{k,\ep}^\pm(y,y_0) = -\partial_y\varphi_{k,\ep}^\pm(y,y_0) + \varphi_{1,k,\ep}^\pm(y,y_0)$ and
\begin{align*}
\partial_y\varphi_{k,\ep}^\pm(y,y_0) &=  \partial_y \left(\varphi_{k,\ep}^\pm \right)_\sr(y,y_0)\eta^{\frac12+\gamma_0} + 2k\left( \frac12 +\gamma_0\right) \left(\varphi_{k,\ep}^\pm \right)_\sr(y,y_0) \eta^{-\frac12+\gamma_0} \\
&\quad+ \partial_y \left(\varphi_{k,\ep}^\pm \right)_\s(y,y_0)\eta^{\frac12-\gamma_0} + 2k\left( \frac12 -\gamma_0\right) \left(\varphi_{k,\ep}^\pm \right)_\s(y,y_0)  \eta^{-\frac12-\gamma_0}.
\end{align*}
Hence, we define
\begin{align*}
\widetilde\varphi_{1,k,\ep}^\pm := \varphi_{1,k,\ep}^\pm - \widetilde\varphi_{0,k,\ep}^\pm
\end{align*}
Thanks to Proposition \ref{prop:Xkvarphi} and Proposition \ref{prop:Xkvarphi1}, we obtain the $L^\infty(I_3(y_0))$ estimate for $\widetilde\varphi_{1,k,\ep}^\pm$. Next, the $H_{k}^1(I_3^c(y_0))$ and $L^\infty(I_3^c(y_0))$ bounds on $\partial_{y_0}\varphi_{k,\ep}^\pm$ also follow from Propositions~\ref{prop:Xkvarphi} and \ref{prop:Xkvarphi1}.
\end{proof}

\subsubsection{Mild stratifications}
In what follows we study the regularity of $\varphi_{1,k,\ep}^\pm$ when $y_0\in I_M$. Firstly, the same arguments of Lemma \ref{lemma:XkB0}, Proposition \ref{prop:logsobolevregdecomG} and Corollary \ref{cor:logsobolevregpartialdecomG} now provides
\begin{lemma}\label{lemma:LXkB0}
Let $k\geq 1$. Then,  
\begin{align*}
\Vert \B_{0,k,\ep}^\pm \Vert_{LX_k}\lesssim k^\frac12\cS_{k,0},
\end{align*}
uniformly in $0<\ep<\ep_*$ and $y_0\in I_M$.
\end{lemma}

We next record the $LX_k$ regularity of $\varphi_{1,k,\ep}^\pm$.

\begin{proposition}\label{prop:LXkvarphi1}
Let $k\geq 1$. There holds 
\begin{align*}
\Vert \varphi_{1,k,\ep}^\pm \Vert_{LX_k} \lesssim k^{-\frac12}\cS_{k,1}
\end{align*}
uniformly for all $y_0\in I_M$ and all $0 < \ep\leq \ep_*$. 
\end{proposition}

\begin{proof}We show only the $LZ_k$ estimate, since once it is established it can be upgraded to $LX_k$ bounds as in Proposition \ref{prop:LXkvarphi}. We use the limiting absorption principle Proposition \ref{prop:LAPLZk} on \eqref{eq:LAPvarphi1} which gives
\begin{align*}
\Vert \varphi_{1,k,\ep}^\pm \Vert_{LX_k}\lesssim \Vert R_{0,k,\ep}^\pm\text{TG}_{k,\ep}^\pm\varphi_{1,k,\ep}^\pm \Vert_{LX_k} + \Vert \B_{0,k,\ep}^\pm \Vert_{LX_k}.
\end{align*}
From \eqref{eq:TGvarphi1} we see that
\begin{align*}
\Vert R_{0,k,\ep}^\pm\text{TG}_{k,\ep}^\pm\varphi_{1,k,\ep}^\pm \Vert_{LX_k}  &\leq \left\Vert R_{1,k,\ep}^\pm\left( v'''\varphi_{k,\ep}^\pm \right) \right\Vert_{X_k} + \left\Vert R_{2,k,\ep}^\pm \left(\P'\varphi_{k,\ep}^\pm\right) \right\Vert_{X_k} \\
&\quad + 2 \left\Vert R_{3,k,\ep}^\pm \left( \P h_1(\cdot,y_0) \varphi_{k,\ep}^\pm  \right)  \right\Vert_{X_k}  + \left\Vert R_{2,k,\ep}^\pm \left(v'' h_1(\cdot,y_0)\varphi_{k,\ep}^\pm \right)  \right\Vert_{X_k} \\
&\quad + \left\Vert R_{1,k,\ep}^\pm \partial_y w_k^0 \right\Vert_{X_k} + \left\Vert R_{2,k,\ep}^\pm \left(h_1(\cdot,y_0) w_k^0\right) \right \Vert_{X_K} + \left\Vert R_{0,k,\ep}^\pm \partial_y q_k^0 \right\Vert_{X_k}
\end{align*}
where we have defined $h_1(y,y_0):=v'(y) - v'(y_0)$. Thanks to Proposition \ref{prop:LXkvarphi}, we have $\Vert \varphi_{k,\ep}^\pm \Vert_{X_k}\lesssim k^{-\frac12}\cS_{k,0}$. Furthermore, Lemma \ref{lemma:R1mapsLinfL2toLXk} gives
\begin{align*}
\left\Vert R_{1,k,\ep}^\pm\left( v'''\varphi_{k,\ep}^\pm \right) \right\Vert_{LX_k} &\lesssim k^{-1}\Vert \varphi_{k,\ep}^\pm \Vert_{LX_k} \lesssim k^{-\frac32}\cS_{k,0}.
\end{align*}
Similarly, Corollary \ref{cor:R2mapsCLXktoLXk} shows that
\begin{align*}
\left\Vert R_{2,k,\ep}^\pm \left(\P'\varphi_{k,\ep}^\pm\right) \right\Vert_{LX_k} \lesssim \Vert \varphi_{k,\ep}^\pm \Vert_{LX_k}\lesssim k^{-\frac12}\cS_{k,0},
\end{align*}
while Lemma \ref{lemma:R3mapsC1LXktoLXk} gives
\begin{align*}
\left\Vert R_{3,k,\ep}^\pm \left( \P h_1(\cdot,y_0) \varphi_{k,\ep}^\pm  \right)  \right\Vert_{LX_k} \lesssim \Vert \varphi_{k,\ep}^\pm \Vert_{LX_k}\lesssim k^{-\frac12}\cS_{k,0}
\end{align*}
and Lemma \ref{lemma:R1mapsLinfL2toLXk} provides
\begin{align*}
\left\Vert R_{2,k,\ep}^\pm \left(v'' h_1(\cdot,y_0)\varphi_{k,\ep}^\pm \right)  \right\Vert_{X_k}\lesssim k^{-1}\Vert \varphi_{k,\ep}^\pm \Vert_{LX_k}\lesssim k^{-\frac32}\cS_{k,0}.
\end{align*}
Finally, Lemma \ref{lemma:R1mapsLinfL2toLXk}, Lemma \ref{lemma:R0mapsL2toLXk} and Lemma \ref{lemma:LinfH1bound} yield
\begin{align*}
\left\Vert R_{1,k,\ep}^\pm \partial_y w_k^0 \right\Vert_{LX_k} &+ \left\Vert R_{2,k,\ep}^\pm \left(h_1(\cdot,y_0) w_k^0\right) \right \Vert_{LX_K} + \left\Vert R_{0,k,\ep}^\pm \partial_y q_k^0 \right\Vert_{LX_k} \lesssim k^{-\frac12}\cS_{k,1}.
\end{align*}
With this and Lemma \ref{lemma:XkB0} the proof is finished.
\end{proof}

We next establish a working formula for $\partial_{y_0}\varphi_{k,\ep}^\pm$ for $y_0\in I_M$.

\begin{proposition}\label{prop:LXkpynotvarphi}
Let $k\geq 1$ and $y_0\in I_M$, we have
\begin{equation}\label{eq:decomLXkpynotvarphi}
\begin{split}
\partial_{y_0}\varphi_{k,\ep}^\pm &= -\varphi_{\rL,1,1,\sr,k,\ep}^\pm(y,y_0) \eta^{-\frac12+\gamma_0} - \varphi_{\rL,1,1,\rL,k,\ep}^\pm(y,y_0) \eta^{-\frac12-\gamma_0}\log(\eta) \Q_{1,\gamma_0}(\eta) \\
&\quad -\varphi_{\rL,1,1,\s,k,\ep}^\pm(y,y_0) \eta^{-\frac12-\gamma_0} + \widetilde\varphi_{\rL,1,k,\ep}^\pm(y,y_0)
\end{split}
\end{equation}
where $\varphi_{\rL,1,1,\sigma,k,\ep}^\pm(y,y_0)$ are given as in Proposition \ref{prop:LXkvarphi}, for $\sigma\in \lbrace \sr, \s, \rL \rbrace$, and $\Vert \widetilde\varphi_{\rL,1,k,\ep}^\pm \Vert_{L^\infty(I_3(y_0))} \lesssim  k^{-\frac12}\cS_{k,1}$. Moreover, $\partial_{y_0}\varphi_{k,\ep}^\pm$ enjoys the more precise bound
\begin{align*}
\Vert \partial_{y_0}\varphi_{k,\ep}^\pm \Vert_{H_k^1(I_3^c(y_0))} \lesssim  k^{-1}\cS_{k,1}, \quad \Vert \partial_{y_0}\varphi_{k,\ep}^\pm \Vert_{L^\infty(I_3^c(y_0))} \lesssim k^{-\frac12}\cS_{k,1}, 
\end{align*}
uniformly for all $y_0\in I_M$ and all $0 < \ep\leq \ep_*$.
\end{proposition}

\begin{proof}
The decomposition \eqref{eq:decomLXkpynotvarphi} follows from Proposition \ref{prop:LXkvarphi}. The pointwise estimates are then a consequence of Propositions \ref{prop:LXkvarphi} and \ref{prop:LXkvarphi1} while the proof for the $H_{k}^1(I_3^c(y_0))$ and $L^\infty(I_3^c(y_0))$ estimate is identical to that of Proposition \ref{prop:Xkvarphi1}. 
\end{proof}

\subsection{Sobolev regularity for the second derivative}
Finally, to obtain the main regularity structures of $\partial^2_{y_0}\varphi_{k,\ep}^\pm$, we recall $\varphi_{1,k,\ep}^\pm=(\partial_y + \partial_{y_0})\varphi_{k,\ep}^\pm$ and we record the identity
\begin{align*}
\partial_{y_0}^2\varphi_{k,\ep}^\pm = \partial_y^2\varphi_{k,\ep}^\pm + \left( \partial_y + \partial_{y_0}\right)\varphi_{1,k,\ep}^\pm -2\partial_y \varphi_{1,k,\ep}^\pm.
\end{align*}
Then, for $\varphi_{2,k,\ep}^\pm :=(\partial_y + \partial_{y_0})\varphi_{1,k,\ep}^\pm$ we have
\begin{align*}
\partial_{y_0}^2\varphi_{k,\ep}^\pm &= k^2\varphi_{k,\ep}^\pm + \frac{v''(y)}{v(y) - v(y_0) \pm i\ep}\varphi_{k,\ep}^\pm - \frac{\P(y)}{(v(y) - v(y_0) \pm i\ep)^2}\varphi_{k,\ep}^\pm + \frac{w_k^0}{v(y) - v(y_0) \pm i\ep} + q_k^0 \\ 
&\quad+ \varphi_{2,k,\ep}^\pm -2\partial_y \varphi_{1,k,\ep}^\pm
\end{align*}
with 
\begin{equation}\label{eq:TGvarphi2}
\begin{split}
\textsc{TG}_{k,\ep}^\pm\varphi_{2,k,\ep}^\pm &= 2\frac{v'''(y)}{v(y) - v(y_0) \pm i\ep}\varphi_{1,k,\ep}^\pm - 2v''(y) \frac{v'(y)-v'(y_0)}{(v(y) - v(y_0)\pm i\ep)^2} \varphi_{1,k,\ep}^\pm  \\
&\quad + 4\P(y)\frac{v'(y) - v'(y_0)}{(v(y) - v(y_0) \pm i\ep)^3}\varphi_{1,k,\ep}^\pm - 2\frac{\P'(y)}{(v(y) - v(y_0)\pm i\ep)^2}\varphi_{1,k,\ep}^\pm  \\
&\quad + \frac{v^{(4)}(y)}{v(y) - v(y_0)\pm i\ep}\varphi_{k,\ep}^\pm  -2 v'''(y) \frac{v'(y) - v'(y_0)}{(v(y) - v(y_0)\pm i\ep)^2}\varphi_{k,\ep}^\pm \\
&\quad + 2v''(y) \frac{(v'(y) - v'(y_0))^2}{(v(y) - v(y_0)\pm i\ep)^3}\varphi_{k,\ep}^\pm - v''(y) \frac{v''(y) - v''(y_0)}{(v(y) - v(y_0)\pm i\ep)^2}\varphi_{k,\ep}^\pm \\   
&\quad +4\P'(y)\frac{v'(y) - v'(y_0)}{(v(y) - v(y_0)\pm i\ep)^3}\varphi_{k,\ep}^\pm - \frac{\P''(y)}{(v(y) - v(y_0) \pm i\ep)^2}\varphi_{k,\ep}^\pm     \\
&\quad + 2\P(y) \frac{v''(y) - v''(y_0)}{(v(y) - v(y_0) \pm i\ep)^3}\varphi_{k,\ep}^\pm - 6 \P(y)\frac{(v'(y) - v'(y_0))^2}{(v(y) - v(y_0) \pm i\ep)^4}\varphi_{k,\ep}^\pm\\
&\quad + \frac{\partial_y^2 w_k^0(y)}{v(y) - v(y_0)\pm i\ep} - 2\partial_y w_k^0(y)\frac{v'(y) - v'(y_0)}{(v(y) - v(y_0) \pm i\ep)^2} \\
&\quad + 2 w_k^0(y)\frac{(v'(y) - v'(y_0))^2}{(v(y) - v(y_0)\pm i\ep)^3} - w_k^0(y) \frac{v''(y) - v''(y_0)}{(v(y) - v(y_0)\pm i\ep)^2} + \partial_y^2 q_k^0(y)\\
&=\sum_{j=1}^{17} \F_{j,k,\ep}^\pm(y,y_0)
\end{split}
\end{equation}
and 
\begin{equation}\label{eq:LAPvarphi2}
\varphi_{2,k,\ep}^\pm(y,y_0) + T_{k,\ep}^\pm \varphi_{2,k,\ep}^\pm(y,y_0) =\sum_{j=1}^{17} \left( R_{0,k,\ep}^\pm \F_{j,k,\ep}^\pm \right)(y,y_0) + \B_{2,k,\ep}^\pm(y,y_0),
\end{equation}
where we further have
\begin{equation}\label{eq:defB2}
\B_{2,k,\ep}^\pm(y,y_0) := \partial_z \G_{k,\ep}^\pm(y,y_0,z) \varphi_{2,k,\ep}^\pm(z,y_0)\Big|_{z=0}^{z=2}.
\end{equation}

We first obtain a more amenable form for $\varphi_{2,k,\ep}^\pm$ at $z=0,2$.

\begin{lemma}\label{lemma:tracevarphi2}
Let $k\geq 1$, $0<\ep<\ep_*$ and $y_0\in [0,2]$. There holds
\begin{align*}
\varphi_{2,k,\ep}^\pm(z,y_0) = 2\partial_y\varphi_{1,k,\ep}^\pm(z,y_0)
\end{align*}
for $z=0,2$.
\end{lemma}

\begin{proof}
We observe that $\varphi_{2,k,\ep} = \partial_y^2\varphi_{k,\ep}^\pm + 2\partial_y\partial_{y_0}\varphi_{k,\ep}^\pm + \partial_{y_0}^2\varphi_{k,\ep}^\pm$. Moreover $\partial_y\partial_{y_0}\varphi_{k,\ep}^\pm = \partial_y\varphi_{1,k,\ep} -\partial_y^2\varphi_{k,\ep}^\pm$. Consequently,
\begin{align*}
\varphi_{2,k,\ep}^\pm = -\partial_y^2\varphi_{k,\ep}^\pm + 2\partial_y\varphi_{1,k,\ep}^\pm + \partial_{y_0}^2\varphi_{k,\ep}^\pm
\end{align*}
and the Lemma follows from \eqref{eq:TGeqvarphiSobolevReg}, the compact support of $\omega_k^0$ and $\varrho_k^0$, and the fact that $\varphi_{k,\ep}^\pm(z,y_0)=0$, for $z=0,2$, for all $y_0\in [0,2]$.
\end{proof}

\subsubsection{Strong and weak stratifications}

To use Proposition \ref{prop:LAPZk} on \eqref{eq:TGvarphi2}, we first obtain $X_k$ estimates on $\B_{2,k,\ep}^\pm$.

\begin{lemma}\label{lemma:XkB2}
Let $k\geq 1$. There holds
\begin{align*}
\Vert \B_{2,k,\ep}^\pm \Vert_{X_k}\lesssim k^\frac12\cS_{k,1}
\end{align*}
uniformly for all $y_0\in I_S\cup I_{W}$ and all $0 < \ep\leq \ep_*$.
\end{lemma}

\begin{proof}
From Lemma \ref{lemma:tracevarphi2} we see that once $\left|\partial_y\varphi_{1,k,\ep}^\pm(z,y_0) \right|\lesssim k^\frac12\cS_{k,1}$ the conclusion follows as in the proof of Lemma \ref{lemma:XkB0}. Indeed, now $\varphi_{1,k,\ep}^\pm$ solves \eqref{eq:TGvarphi1} and $\Vert \varphi_{1,k,\ep}^\pm \Vert_{X_K}\lesssim k^{-\frac12}\cS_{k,1}$ so that one can argue as in the proof of Lemma \ref{lemma:XkB0} to show that $\left|\partial_y\varphi_{1,k,\ep}^\pm(z,y_0) \right|\lesssim k^\frac12\cS_{k,1}$, we omit the details.
\end{proof}

\begin{proposition}\label{prop:Xkvarphi2}
Let $k\geq 1$. There holds
\begin{align*}
 \Vert \varphi_{2,k,\ep}^\pm\Vert_{X_k}\lesssim k^{-\frac12}\cS_{k,2}
\end{align*}
uniformly for all $y_0\in I_S\cup I_{W}$ and all $0 < \ep\leq \ep_*$.
\end{proposition}

\begin{proof}
We shall use Proposition \ref{prop:LAPZk} on \eqref{eq:LAPvarphi2} to obtain $Z_k$ bounds and then upgrade to $X_k$ estimates. We recall from Propositions \ref{prop:Xkvarphi} and \ref{prop:Xkvarphi1} that 
\begin{align*}
\Vert \varphi_{k,\ep}^\pm\Vert_{X_k}\lesssim k^{-\frac12}\cS_{k,0}, \quad \Vert \varphi_{1,k,\ep}^\pm\Vert_{X_k}\lesssim k^{-\frac12}\cS_{k,1}. 
\end{align*}
Let $F_{j,k,\ep}^\pm(y,y_0) = \left(R_{0,k,\ep}^\pm\F_{j,k,\ep}^\pm \right)(y,y_0)$, for $j=1,...,17$. Corollary \ref{cor:R1mapsCXktoXk} shows
\begin{align*}
\Vert F_{1,k,\ep}^\pm \Vert_{X_k} + \Vert F_{5,k,\ep}^\pm \Vert_{X_k} \lesssim k^{-\frac32}\cS_{k,1}
\end{align*}
while Lemma \ref{lemma:R2mapsC1XktoXk} gives
\begin{align*}
\Vert F_{2,k,\ep}^\pm \Vert_{X_k} + \Vert F_{6,k,\ep}^\pm \Vert_{X_k} + \Vert F_{8,k,\ep}^\pm \Vert_{X_k} \lesssim k^{-\frac32}\cS_{k,1}.
\end{align*}
Furthermore, Lemma \ref{lemma:R3mapsC1XktoXk} yields
\begin{align*}
\Vert F_{3,k,\ep}^\pm \Vert_{X_k} + \Vert F_{9,k,\ep}^\pm \Vert_{X_k} + \Vert F_{11,k,\ep}^\pm \Vert_{X_k} \lesssim k^{-\frac12}\cS_{k,1}
\end{align*}
whereas Corollary \ref{cor:R2mapsCXktoXk} provides
\begin{align*}
\Vert F_{4,k,\ep}^\pm \Vert_{X_k} + \Vert F_{10,k,\ep}^\pm \Vert_{X_k} \lesssim k^{-\frac12}\cS_{k,1}.
\end{align*}
From Lemma \ref{lemma:R3mapsC2XktoXk} we obtain
\begin{align*}
\Vert F_{7,k,\ep}^\pm \Vert_{X_k} \lesssim k^{-\frac32}\cS_{k,0}
\end{align*}
and from Lemma \ref{lemma:R4mapsC1XktoXk} we deduce
\begin{align*}
\Vert F_{1,k,\ep}^\pm \Vert_{X_k} \lesssim k^{-\frac12}\cS_{k,0}.
\end{align*}
Finally, using Lemma \ref{lemma:R1mapsH1toXk} for $F_{13,k,\ep}^\pm$, Lemma \ref{lemma:R2mapsC1H1toXk} for $F_{14,k,\ep}^\pm$ and $F_{16,k,\ep}^\pm$, Lemma \ref{lemma:R3mapsC2H1toXk} for $F_{15,k,\ep}^\pm$ and Lemma \ref{lemma:R0mapsL2toXk} for $F_{17,k,\ep}^\pm$ we reach
\begin{align*}
\sum_{j=13}^{17}\Vert F_{j,k,\ep}^\pm \Vert_{X_k} \lesssim k^{-\frac12}\cS_{k,2}.
\end{align*}
With this and Lemma \ref{lemma:XkB2} the result follows.
\end{proof}

We are now ready to present a working formula for $\partial_{y_0}^2\varphi_{k,\ep}^\pm$.

\begin{proposition}\label{prop:Xkpynotpynotvarphi}
Let $k\geq 1$ and $y_0\in I_S\cup I_W$, we have
\begin{equation}\label{eq:decomXkpynotpynotvarphi}
\begin{split}
\partial_{y_0}^2\varphi_{k,\ep}^\pm &= \varphi_{2,1,\sr,k,\ep}^\pm(y,y_0) \eta^{-\frac12+\gamma_0} + \varphi_{2,1,\s,k,\ep}^\pm(y,y_0) \eta^{-\frac12-\gamma_0} \\
&\quad + \varphi_{2,2,\sr,k,\ep}^\pm(y,y_0) \eta^{-\frac32+\gamma_0} + \varphi_{2,2,\s,k,\ep}^\pm(y,y_0) \eta^{-\frac32-\gamma_0} \\ 
&\quad + \frac{w_k^0(y)}{v(y) - v(y_0) \pm i\ep} + q_k^0(y) + \widetilde\varphi_{2,k,\ep}^\pm(y,y_0) 
\end{split}
\end{equation}
where
\begin{align*}
\varphi_{2,1,\sr,k,\ep}^\pm(y,y_0) &:= \left( 2kv''(y)B_{\ep}^\pm(y,y_0)  \left( \varphi_{k,\ep}^\pm \right)_\sr - 2k(1+2\gamma_0) \left( \varphi_{1,k,\ep}^\pm \right)_\sr \right),  \\
\varphi_{2,1,\s,k,\ep}^\pm(y,y_0) &:=\left( 2kv''(y)B_{\ep}^\pm(y,y_0)  \left( \varphi_{k,\ep}^\pm \right)_\s - 2k(1-2\gamma_0) \left( \varphi_{1,k,\ep}^\pm \right)_\s \right), \\
\varphi_{2,2,\sigma,k,\ep}^\pm(y,y_0) &:= - 4k^2\P(y) (B_{\ep}^\pm(y,y_0))^2  \left( \varphi_{k,\ep}^\pm \right)_\sigma, 
\end{align*}
and $\Vert \widetilde{\varphi}_{2,k,\ep}^\pm \Vert_{L^\infty(I_3(y_0))} \lesssim k^{-\frac12}\cS_{k,2}$. Moreover, $\partial_{y_0}^2\varphi_{k,\ep}^\pm$ enjoy the more precise bound
\begin{align*}
\Vert \partial_{y_0}^2\varphi_{k,\ep}^\pm \Vert_{H_k^1(I_3^c)} \lesssim  k^{-1}\cS_{k,2}, \quad \Vert \partial_{y_0}^2\varphi_{k,\ep}^\pm \Vert_{L^\infty(I_3^c)} \lesssim  k^{-\frac12}\cS_{k,2},
\end{align*}
uniformly for all $y_0\in I_S\cup I_{W}$ and all $0 < \ep\leq \ep_*$.
\end{proposition}

\begin{proof}
The decomposition \eqref{eq:decomXkpynotpynotvarphi}, its local estimates, the ${H_k^1(I_3^c(y_0))}$ and $L^\infty(I_3^c(y_0))$ estimates are a consequence of the identity $\partial_{y_0}^2\varphi_{k,\ep}^\pm = \partial_y^2\varphi_{k,\ep}^\pm + \varphi_{2,k,\ep}^\pm -2\partial_y \varphi_{1,k,\ep}^\pm$ and Propositions \ref{prop:Xkvarphi}, \ref{prop:Xkvarphi1} and \ref{prop:Xkvarphi2}.
\end{proof}

\subsubsection{Mild stratifications}
Analogously to Lemma \ref{lemma:XkB2}, we now have

\begin{lemma}\label{lemma:LXkB2}
Let $k\geq 1$. There holds
\begin{align*}
\Vert \B_{2,k,\ep}^\pm \Vert_{LX_k}\lesssim k^\frac12\cS_{k,1}
\end{align*}
uniformly for all $y_0\in I_M$ and all $0 < \ep\leq \ep_*$.
\end{lemma}

We next state the analogue of Proposition \ref{prop:Xkvarphi2} for the mild stratified region $I_M$. Its proof follows from Proposition \ref{prop:LAPLZk}, Lemmas \ref{lemma:R0mapsL2toLXk},  \ref{lemma:R1mapsLinfL2toLXk}, \ref{lemma:LXkB0}, \ref{lemma:LXkB2}, \ref{lemma:R2mapsLXktoLXk}, \ref{lemma:R3mapsC1LXktoLXk}, \ref{lemma:R4mapsC2LXktoC2LXk} and Corollary \ref{cor:R2mapsCLXktoLXk}. We omit the details.

\begin{proposition}\label{prop:LXkvarphi2}
Let $k\geq 1$ and $y_0\in I_S\cup I_W$. There holds
\begin{align*}
 \Vert \varphi_{2,k,\ep}^\pm\Vert_{X_k}\lesssim k^{-\frac12}\cS_{k,2},
\end{align*}
uniformly for all $y_0\in I_S\cup I_{W}$ and all $0 < \ep\leq \ep_*$.
\end{proposition}

We finish the section presenting a formula for $\partial_{y_0}^2\varphi_{k,\ep}^\pm$ when $y_0\in I_M$.

\begin{proposition}\label{prop:LXkpynotpynotvarphi}
Let $k\geq 1$ and $y_0\in I_M$, we have
\begin{equation}\label{eq:decomLXkpynotpynotvarphi}
\begin{split}
\partial_{y_0}^2\varphi_{k,\ep}^\pm &= \varphi_{\rL,2,1,\sr,k,\ep}^\pm(y,y_0) \eta^{-\frac12+\gamma_0} + \varphi_{\rL,2,1,\rL,k,\ep}^\pm(y,y_0) \eta^{-\frac12-\gamma_0}\log(\eta)\Q_{1,\gamma_0}(\eta) \\
&\quad+ \varphi_{\rL,2,1,\s,k,\ep}^\pm(y,y_0) \eta^{-\frac12-\gamma_0} \\
&\quad + \varphi_{\rL,2,2,\sr,k,\ep}^\pm(y,y_0) \eta^{-\frac32+\gamma_0} + \varphi_{\rL,2,2,\rL,k,\ep}^\pm(y,y_0) \eta^{-\frac32-\gamma_0}\log(\eta) \Q_{1,\gamma_0}(\eta) \\ 
&\quad + \frac{\omega_k^0(y)}{v(y) - v(y_0) \pm i\ep} + \varrho_k^0(y) + \widetilde\varphi_{\rL,2,k,\ep}^\pm(y,y_0) 
\end{split}
\end{equation}
where
\begin{align*}
\varphi_{\rL,2,1,\sr,k,\ep}^\pm(y,y_0) &:= \left( 2kv''(y)B_{\ep}^\pm(y,y_0)  \left( \varphi_{k,\ep}^\pm \right)_\sr - 2k(1+2\gamma_0) \left( \varphi_{1,k,\ep}^\pm \right)_\sr \right),  \\
\varphi_{\rL,2,1,\rL,k,\ep}^\pm(y,y_0) &:=\left( 2kv''(y)B_{\ep}^\pm(y,y_0)  \left( \varphi_{k,\ep}^\pm \right)_\s - 2k(1-2\gamma_0) \left( \varphi_{1,k,\ep}^\pm \right)_\s \right), \\
\varphi_{\rL,2,1,\s,k,\ep}^\pm(y,y_0) &:=-4k \left( \varphi_{1,k,\ep}^\pm \right)_\s , \\
\varphi_{\rL,2,2,\sigma,k,\ep}^\pm(y,y_0) &:= - 4k^2\P(y) (B_{\ep}^\pm(y,y_0))^2  \left( \varphi_{k,\ep}^\pm \right)_\sigma, 
\end{align*}
for $\sigma \in \lbrace \sr, \rL \rbrace$ and $\Vert \widetilde{\varphi}_{\rL,2,k,\ep}^\pm \Vert_{L^\infty(I_3(y_0))} \lesssim k^{-\frac12}\cS_{k,2}$. Moreover, $\partial_{y_0}^2\varphi_{k,\ep}^\pm$ enjoy the more precise bound
\begin{align*}
\Vert \partial_{y_0}^2\varphi_{k,\ep}^\pm \Vert_{H_k^1(I_3^c)} \lesssim  k^{-1}\cS_{k,2}, \quad \Vert \partial_{y_0}^2\varphi_{k,\ep}^\pm \Vert_{L^\infty(I_3^c)} \lesssim  k^{-\frac12}\cS_{k,2},
\end{align*}
uniformly for all $y_0\in I_M$ and all $0 < \ep\leq \ep_*$.
\end{proposition}

\begin{proof}
The decomposition \eqref{eq:decomXkpynotpynotvarphi} follows from $\partial_{y_0}^2\varphi_{k,\ep}^\pm = \partial_y^2\varphi_{k,\ep}^\pm + \varphi_{2,k,\ep}^\pm -2\partial_y \varphi_{1,k,\ep}^\pm$ and Propositions \ref{prop:Xkvarphi}, \ref{prop:Xkvarphi1} and \ref{prop:Xkvarphi2}. Likewise, the ${H_k^1(I_3^c)}$ estimate is proven arguing as in the proof of Proposition \ref{prop:Xkpynotvarphi}, we omit the details.
\end{proof}

\subsection{Sobolev regularity for mixed derivatives}
The purpose of this section is to obtain working formulas for $\partial_{y,y_0}^2\varphi_{k,\ep}^\pm$. Since $\partial_{y_0}\varphi_{k,\ep}^\pm = \varphi_{1,k,\ep}^\pm - \partial_y\varphi_{k,\ep}^\pm$, with $\varphi_{1,k,\ep}^\pm\in X_k$, we observe that
\begin{align*}
\partial_{y,y_0}^2\varphi_{k,\ep}^\pm(y,y_0) = \partial_y \varphi_{1,k,\ep}^\pm(y,y_0) - \partial_y^2\varphi_{k,\ep}^\pm(y,y_0).
\end{align*}
We again present the estimates according to the location of $y_0$.

\subsubsection{Strong and weak stratifications}
Here we consider the case where $y_0\in I_W\cup I_S$. Then, Propositions \ref{prop:Xkvarphi}, \ref{prop:Xkvarphi1} and \eqref{eq:TGeqvarphiSobolevReg} provide:

\begin{proposition}\label{prop:Xkpypynotvarphi}
Let $k\geq 1$ and $y_0\in I_S\cup I_W$. Then,
\begin{equation}\label{eq:decomXkpypynotvarphi}
\begin{split}
\partial_{y,y_0}^2\varphi_{k,\ep}^\pm &= \varphi_{3,1,\sr,k,\ep}^\pm(y,y_0) \eta^{-\frac12+\gamma_0} + \varphi_{3,1,\s,k,\ep}^\pm(y,y_0) \eta^{-\frac12-\gamma_0} \\
&\quad + \varphi_{3,2,\sr,k,\ep}^\pm(y,y_0) \eta^{-\frac32+\gamma_0} + \varphi_{3,2,\s,k,\ep}^\pm(y,y_0) \eta^{-\frac32-\gamma_0} \\ 
&\quad - \frac{ w_k^0(y)}{v(y) - v(y_0) \pm i\ep} - q_k^0(y) + \widetilde\varphi_{3,k,\ep}^\pm(y,y_0) 
\end{split}
\end{equation}
where
\begin{align*}
\varphi_{3,1,\sr,k,\ep}^\pm(y,y_0) &:= -\left( 2kv''(y)B_{\ep}^\pm(y,y_0)  \left( \varphi_{k,\ep}^\pm \right)_\sr - k(1+2\gamma_0) \left( \varphi_{1,k,\ep}^\pm \right)_\sr \right),  \\
\varphi_{3,1,\s,k,\ep}^\pm(y,y_0) &:=-\left( 2kv''(y)B_{\ep}^\pm(y,y_0)  \left( \varphi_{k,\ep}^\pm \right)_\s - k(1-2\gamma_0) \left( \varphi_{1,k,\ep}^\pm \right)_\s \right), \\
\varphi_{3,2,\sigma,k,\ep}^\pm(y,y_0) &:=  4k^2\P(y) (B_{\ep}^\pm(y,y_0))^2  \left( \varphi_{k,\ep}^\pm \right)_\sigma, 
\end{align*}
and $\Vert \widetilde{\varphi}_{3,k,\ep}^\pm \Vert_{L^\infty(I_3(y_0))} \lesssim k^{-\frac12}\cS_{k,2}$. Moreover, $\partial_{y,y_0}^2\varphi_{k,\ep}^\pm$ enjoy the more precise bound
\begin{align*}
\Vert \partial_{y,y_0}^2\varphi_{k,\ep}^\pm \Vert_{L^2(I_3^c)} \lesssim  k^{-1}\cS_{k,2}, \quad \Vert \partial_{y,y_0}^2\varphi_{k,\ep}^\pm \Vert_{L^\infty(I_3^c)} \lesssim k^{-\frac12}\cS_{k,2},
\end{align*}
uniformly for all $y_0\in I_S\cup I_{W}$ and all $0 < \ep\leq \ep_*$.
\end{proposition}

\subsubsection*{Mild stratifications}
For $y_0\in I_M$ we now turn to Propositions \ref{prop:LXkvarphi} and \ref{prop:LXkvarphi1} to obtain the following:

\begin{proposition}\label{prop:LXkpypynotvarphi}
Let $k\geq 1$ and $y_0\in I_M$. Then, we have
\begin{equation}\label{eq:decomLXkpypynotvarphi}
\begin{split}
\partial_{y,y_0}^2\varphi_{k,\ep}^\pm &= \varphi_{\rL,3,1,\sr,k,\ep}^\pm(y,y_0) \eta^{-\frac12+\gamma_0} + \varphi_{\rL,3,1,\rL,k,\ep}^\pm(y,y_0) \eta^{-\frac12-\gamma_0}\log(\eta)\Q_{1,\gamma_0}(\eta) \\
&\quad+ \varphi_{\rL,3,1,\s,k,\ep}^\pm(y,y_0) \eta^{-\frac12-\gamma_0} \\
&\quad + \varphi_{\rL,3,2,\sr,k,\ep}^\pm(y,y_0) \eta^{-\frac32+\gamma_0} + \varphi_{\rL,2,3,\rL,k,\ep}^\pm(y,y_0) \eta^{-\frac32-\gamma_0}\log(\eta) \Q_{1,\gamma_0}(\eta) \\ 
&\quad + \frac{w_k^0(y)}{v(y) - v(y_0) \pm i\ep} + q_k^0(y) + \widetilde\varphi_{\rL,3,k,\ep}^\pm(y,y_0) 
\end{split}
\end{equation}
where
\begin{align*}
\varphi_{\rL,3,1,\sr,k,\ep}^\pm(y,y_0) &:= -\left( 2kv''(y)B_{\ep}^\pm(y,y_0)  \left( \varphi_{k,\ep}^\pm \right)_\sr - k(1+2\gamma_0) \left( \varphi_{1,k,\ep}^\pm \right)_\sr \right),  \\
\varphi_{\rL,3,1,\rL,k,\ep}^\pm(y,y_0) &:= -\left( 2kv''(y)B_{\ep}^\pm(y,y_0)  \left( \varphi_{k,\ep}^\pm \right)_\s - k(1-2\gamma_0) \left( \varphi_{1,k,\ep}^\pm \right)_\s \right), \\
\varphi_{\rL,3,1,\s,k,\ep}^\pm(y,y_0) &:= 2k \left( \varphi_{1,k,\ep}^\pm \right)_\s , \\
\varphi_{\rL,3,2,\sigma,k,\ep}^\pm(y,y_0) &:=  4k^2\P(y) (B_{\ep}^\pm(y,y_0))^2  \left( \varphi_{k,\ep}^\pm \right)_\sigma, 
\end{align*}
for $\sigma \in \lbrace \sr, \rL \rbrace$ and $\Vert \widetilde{\varphi}_{\rL, 3,k,\ep}^\pm \Vert_{L^\infty(I_3(y_0))} \lesssim k^{-\frac12}\cS_{k,2}$. Moreover, $\partial_{y,y_0}^2\varphi_{k,\ep}^\pm$ enjoy the more precise bound
\begin{align*}
\Vert \partial_{y,y_0}^2\varphi_{k,\ep}^\pm \Vert_{L^2(I_3^c)} \lesssim  k^{-1}\cS_{k,2}.
\end{align*}
uniformly for all $y_0\in I_M$ and all $0 < \ep\leq \ep_*$.
\end{proposition}

\section{Inviscid damping estimates}\label{sec:IDestimates}
In this section we present the proof of time-decay estimates on Theorem \ref{thm:mainID}. 
\subsection{Decay estimates for $\psi_k(t,y)$}
Recall that we can write
\begin{align*}
\psi_k(t,y) = \frac{1}{2\pi i }\lim_{\ep\rightarrow 0}\int_0^2 e^{-ikv(y_0) t} \left( \psi_{k,\ep}^-(y,y_0) - \psi_{k,\ep}^+(y,y_0) \right) v'(y_0) \d y_0.
\end{align*}
Thanks to Lemma \ref{lemma:vanishingpsiIE}, Proposition \ref{prop:Hk1varphi} and the Dominated Convergence Theorem we have
\begin{align}\label{eq:psioscint}
\psi_k(t,y) = \frac{1}{2\pi i }\lim_{\ep\rightarrow 0}\int_{\vartheta_1}^{\vartheta_2} e^{-ikv(y_0) t} \left( \psi_{k,\ep}^-(y,y_0) - \psi_{k,\ep}^+(y,y_0) \right) v'(y_0) \d y_0
\end{align}
and integrating by parts we reach
\begin{align*}
\psi_k(t,y) = \frac{1}{ikt}\frac{1}{2\pi i }\lim_{\ep\rightarrow 0}\int_{\vartheta_1}^{\vartheta_2} e^{-ikv(y_0) t} \left( \partial_{y_0}\varphi_{k,\ep}^-(y,y_0) - \partial_{y_0}\varphi_{k,\ep}^+(y,y_0) \right)  \d y_0
\end{align*}
where we have used Lemma \ref{lemma:vanishingpsiIE} to show that the boundary terms vanish in the limit as $\ep\rightarrow 0$ and $\partial_{y_0}\psi_{k,\ep}^\pm = \partial_{y_0}\varphi_{k,\ep}^\pm$. We first prove point-wise inviscid damping in the strongly stratified regime.

\begin{proposition}\label{prop:IDpsistrong}
Let $k\geq 1$. Then,
\begin{align*}
\Vert \psi_k(t,\cdot) \Vert_{L^\infty(I_S)}  \lesssim k^{-\frac52} t^{-\frac32}\cS_{k,2}
\end{align*}
for all $t\geq 1$.
\end{proposition}

\begin{proof}
Let $\mathsf{m}=\min_{j,n,=1,2}|\varpi_j-\varpi_{j,n}|$, $\delta_0= \min \left( \frac{3\beta}{k}, \frac12 \mathsf{m} \right)$ and let $\delta\in (0,\frac{\delta_0}{2})$. Let $B_\delta=B_\delta(y)$ denote tha ball of radius $\delta$ centred at $y$ and $B_{\delta}^c=(\vartheta_1,\vartheta_2)\setminus B_\delta$. We split
\begin{align*}
\psi_k(t,y) &= \frac{1}{ikt}\frac{1}{2\pi i }\lim_{\ep\rightarrow 0} \left( \int_{B_\delta} + \int_{B_\delta^c} \right) e^{-ikv(y_0) t} \left( \partial_{y_0}\psi_{k,\ep}^-(y,y_0) - \partial_{y_0}\psi_{k,\ep}^+(y,y_0) \right)  \d y_0 \\
&= \frac{1}{ikt}\frac{1}{2\pi i} \lim_{\ep\rightarrow 0} \left( \cI_{\ep,1} + \cI_{\ep,2} \right)
\end{align*}
We begin with $\cI_{\ep,1}$. Since $|y-y_0|\leq\frac{\delta}{2}< \frac{3\beta}{k}$, we have that $y\in I_3(y_0)$ and $y_0\in I_S$ so that $\gamma_0\in i \R$ and we can use Proposition \ref{prop:Xkpynotvarphi} to write
\begin{align*}
\left| \partial_{y_0}\varphi_{k,\ep}^\pm(y,y_0) \right| \lesssim k^{-\frac12}(k|y-y_0|)^{-\frac12}\cS_{k,1}
\end{align*}
for all $y\in I_S$ and all $y_0\in I_3(y)$. Hence, $|\cI_{\ep,1}|\lesssim k^{-\frac32}(k\delta)^\frac12 \cS_{k,1}$ for all $y\in I_S$. For $\cI_{\ep,2}$ we integrate by parts once more, 
\begin{align*}
\cI_{\ep,2} &= \frac{1}{ikt} \frac{e^{-ikv(y_0)t }}{v'(y_0)}\left( \partial_{y_0}\psi_{k,\ep}^-(y,y_0) - \partial_{y_0}\psi_{k,\ep}^+(y,y_0) \right)  \Big|_{y_0\in \partial B_\delta^c} \\
&\quad - \frac{1}{ikt}\int_{B_\delta^c}  \frac{e^{-ikv(y_0)t }}{v'(y_0)}\left( \partial^2_{y_0}\psi_{k,\ep}^-(y,y_0) - \partial^2_{y_0}\psi_{k,\ep}^+(y,y_0) \right)  \d y_0 \\
&\quad - \frac{1}{ikt}\int_{B_\delta^c}  {e^{-ikv(y_0)t }}\partial_{y_0}\left(\frac{1}{v'(y_0)}\right)\left( \partial_{y_0}\psi_{k,\ep}^-(y,y_0) - \partial_{y_0}\psi_{k,\ep}^+(y,y_0) \right)  \d y_0\\
&= \frac{1}{ikt}\left( - \cI_{\ep,2,1} + \cI_{\ep,2,2} + \cI_{\ep,2,3} \right).
\end{align*}
To estimate $\cI_{\ep,2,1}$ we use the local bounds and the global pointwise bounds of Proposition \ref{prop:Xkpynotvarphi} when $y_0\in \partial B_{\delta}\cap B_{\delta_0}$ and $y_0\in \partial B_\delta^c \cap B_{\delta_0}^c$, respectively, to obtain
\begin{align*}
\left| \cI_{\ep,2,1} \right| \lesssim k^{-\frac12}\left( 1 + (k\delta)^{-\frac12} \right) \cS_{k,1}.
\end{align*} 
For $\cI_{\ep,2,2}$ we further split
\begin{align*}
\cI_{\ep,2,2} &= \left( \int_{B_\delta^c\cap B_{\delta_0}}  + \int_{B_{\delta_0}^c} \right)  \frac{e^{-ikv(y_0)t }}{v'(y_0)}\left( \partial^2_{y_0}\psi_{k,\ep}^-(y,y_0) - \partial^2_{y_0}\psi_{k,\ep}^+(y,y_0) \right)  \d y_0 \\
&= \cI_{\ep,2,2,1} + \cI_{\ep,2,2,2}.
\end{align*}
Moreover, thanks to \eqref{eq:decomXkpynotpynotvarphi}, we note that
\begin{align*}
\lim_{\ep\rightarrow 0} \int_{B_\delta^c}  \frac{e^{-ikv(y_0)t }}{v'(y_0)}\left( \frac{w_k^0(y)}{v(y) - v(y_0) - i\ep} - \frac{w_k^0(y)}{v(y) - v(y_0) + i\ep} \right)  \d y_0 = 0
\end{align*}
due to the dominated convergence theorem because $|y-y_0|>\frac{\delta}{2}>0$. Next, for all $y_0\in B_\delta^c\cap B_{\delta_0}$ we have $y_0\in I_S$ and thus integrating the local decomposition and estimates of Proposition \ref{prop:Xkpynotpynotvarphi} yields
\begin{align*}
\left| \lim_{\ep\rightarrow 0} \cI_{\ep,2,2,1}  \right| \lesssim k^{-\frac32}(k\delta)^{-\frac12}\cS_{k,2}.
\end{align*} 
On the other hand, the global pointwise bounds of Propositions \ref{prop:Xkpynotpynotvarphi} and \ref{prop:LXkpynotpynotvarphi} when $\delta_0=\frac{3\beta}{k}$ and the local and global estimates of Propositions \ref{prop:Xkpynotpynotvarphi} and \ref{prop:LXkpynotpynotvarphi} when $\delta_0=\frac12\mathsf{m}$ give
\begin{align*}
\left| \lim_{\ep\rightarrow 0} \cI_{\ep,2,2,2}  \right| \lesssim k^{-\frac12}\cS_{k,2}.
\end{align*}
Hence, $ \left| \lim_{\ep\rightarrow 0} \cI_{\ep,2,2}  \right| \lesssim k^{-\frac12}(k\delta)^{-\frac12}\cS_{k,2}$.  Similarly, we also deduce that $\left| \lim_{\ep\rightarrow 0} \cI_{\ep,2,3}  \right| \lesssim k^{-\frac12}(k\delta)^{-\frac12}\cS_{k,2}$. Therefore,
\begin{align*}
\left| \psi_k(t,y) \right| \lesssim (kt)^{-1}k^{-\frac32}(k\delta)^{\frac12}\cS_{k,1} + (kt)^{-2}k^{-\frac12}(k\delta)^{-\frac12}\cS_{k,2}.
\end{align*}
Optimizing for $\delta$, we see that $\delta= \frac{1}{kt}\min \left( \beta, \frac{k\mathsf{m}}{4}\right) $ gives
\begin{align*}
\left| \psi_k(t,y) \right| \lesssim k^{-\frac52}t^{-\frac32}\cS_{k,2}
\end{align*}
and the proof is concluded.
\end{proof}

We next address the bounds for $\psi_k(t,y)$ when $y\in I_M$. Recall that $I_M = [\varpi_{1,1}, \varpi_{1,2}]\cup [\varpi_{2,1} , \varpi_{2,2}] \setminus \lbrace \varpi_1, \varpi_2 \rbrace$. We have

\begin{proposition}\label{prop:IDpsimild}
Let $k\geq 1$ and $t\geq 1$. Then, 
\begin{align*}
\left| \psi_k(t,y) \right| \lesssim k^{-\frac52}t^{-\frac32 + \mu(y)}(1+ \log(t))\cS_{k,2}
\end{align*}
for $y\in [\varpi_{1,1}, \varpi_{1,2}]\cup [\varpi_{2,1}, \varpi_{2,2}]$.
\end{proposition}

\begin{proof}
We shall prove the bounds for $y\in [\varpi_{2,1},\varpi_{2,2}]$. Define $\delta_0 =  \frac{3\beta}{k}$ and let $\delta\in (0, \frac{\delta_0}{2})$. We write
\begin{align*}
\int_{\vartheta_1}^{\vartheta_2} e^{-ikv(y_0)t} & \left( \partial_{y_0}\varphi_{k,\ep}^-(y,y_0) - \partial_{y_0}\varphi_{k,\ep}^+(y,y_0) \right) \d y_0 \\
&= \int_{B_\delta} e^{-ikv(y_0)t} \left( \partial_{y_0}\varphi_{k,\ep}^-(y,y_0) - \partial_{y_0}\varphi_{k,\ep}^+(y,y_0) \right) \d y_0\\
&\quad + \int_{B_\delta^c} e^{-ikv(y_0)t} \left( \partial_{y_0}\varphi_{k,\ep}^-(y,y_0) - \partial_{y_0}\varphi_{k,\ep}^+(y,y_0) \right) \d y_0\\
&= \cI_{\ep,1} + \cI_{\ep,2},
\end{align*}
where we recall $B_\delta=B_\delta(y)$ denotes the ball of radius $\delta>0$ centred at $y$ and $B_\delta^c = (\vartheta_1, \vartheta_2)\setminus B_\delta$.

\bullpar{Estimates for $\cI_{\ep,1}$} Since $\delta < \frac{3\beta}{k}$ we have that $y\in I_3(y_0)$, for all $y_0\in B_\delta$. Thus, Proposition \ref{prop:Xkpynotvarphi} for $y_0\in B_\delta\cap (I_S\cup I_W)$ and Proposition \ref{prop:LXkpynotvarphi} for $y_0\in B_\delta\cap I_M$ yield
\begin{align*}
|\cI_1| \lesssim k^{-\frac12}\cS_{k,1} \int_{B_\delta}(k|y-y_0|)^{-\frac12+\mu_0} +  (k|y-y_0|)^{-\frac12-\mu_0}  + (k|y-y_0|)^{-\frac12-\mu_0} |\log(k|y-y_0|)| \d y_0.
\end{align*}
Recall that $\mu_0>0$ for $y_0\in (\varpi_2,\varpi_{2,2}]$ while it is identically zero for $y_0\in [\varpi_{2,1},\varpi_2)$. Moreover, 
\begin{align*}
(k|y_0-y|)^{-\mu_0} &= (k|y_0-y|)^{-\mu(y)} e^{(\mu(y) - \mu(y_0)) \log(k|y_0-y|)} \\
&\lesssim (k|y_0-y|)^{-\mu(y)}
\end{align*}
because $e^{(\mu(y) - \mu(z)) \log(k|y_0-y|)}\lesssim e^{k^{-1} (k|y-y_0|)\log( k|y-y_0|)} \lesssim 1$ as $(k|y-y_0|)\log( k|y-y_0|) \lesssim 1$ since $k|y-y_0|\lesssim 1$. Hence,
\begin{align*}
\left| \cI_{\ep,1} \right| \lesssim k^{-\frac32}(k\delta)^{\frac12-\mu(y)}\left( 1 - \log (k\delta) \right) \cS_{k,1},
\end{align*}
where we have used that
\begin{align*}
\left| \int_0^x u^{-\frac12-\mu}\log(u) \d u \right| = -\int_0^x u^{-\frac12-\mu}\log(u) \d u &= - \int_0^x \partial_u \left( \frac{u^{\frac12-\mu}}{\frac12-\mu} \right) \log(u) \d u \\
&\lesssim x^{\frac12-\mu} \left( 1 + |\log(x) | \right),
\end{align*}
for all $x<1$ and all $\mu < \frac12$. This is indeed the case for all $y\in I_M$.

\bullpar{Estimates for $\cI_{\ep,2}$} We integrate by parts again,
\begin{align}
\cI_{\ep,2} &= -\frac{1}{ikt}\frac{e^{-ikv(y_0)t}}{v'(y_0)} \left( \partial_{y_0}\varphi_{k,\ep}^-(y,y_0) - \partial_{y_0}\varphi_{k,\ep}^+(y,y_0) \right) \Big|_{y_0\in \partial B^c_{\delta}}\label{eq:cItwoone} \\
&\quad + \frac{1}{ikt}\int_{B_\delta^c} \frac{e^{-ikv(y_0)t}}{v'(y_0)} \left( \partial_{y_0}^2\varphi_{k,\ep}^-(y,y_0) - \partial_{y_0}^2\varphi_{k,\ep}^+(y,y_0) \right) \d y_0
 \label{eq:cItwotwo}\\
&\quad + \frac{1}{ikt}\int_{B_\delta^c} e^{-ikv(y_0)t} \partial_{y_0}\left(\frac{1}{v'(y_0)} \right) \left( \partial_{y_0}\varphi_{k,\ep}^-(y,y_0) - \partial_{y_0}\varphi_{k,\ep}^+(y,y_0) \right) \d y_0 
 \label{eq:cItwothree} \\
&= \frac{1}{ikt} \left( - \cI_{\ep,2,1} + \cI_{\ep,2,2} + \cI_{\ep,2,3} \right).
\end{align}
We treat each term separately.

\diampar{Estimates for $\cI_{\ep,2,1}$} For $y_0\in B_{\delta_0}$ we use Proposition \ref{prop:Xkpynotvarphi} if $y_0\in (I_S\cup I_W)$ and Proposition \ref{prop:LXkpynotvarphi} if $y_0\in I_M$, while for $y_0\in B_{\delta_0}^c$ we use the pointwise global estimates of Propositions~\ref{prop:Xkpynotvarphi} and  \ref{prop:LXkpynotvarphi} to deduce that
\begin{align*}
|\cI_{\ep,2,1}| \lesssim k^{-\frac12}(k\delta)^{-\frac12-\mu(y)}(1 - \log(k\delta) )\cS_{k,1}.
\end{align*}

\diampar{Estimates for $\cI_{2,2}$} We subdivide again
\begin{align*}
\cI_{2,2} &= \left( \int_{B_\delta^c\cap I_3(y)} + \int_{B_\delta^c\cap I_3^c(y)} \right)\frac{e^{-ikv(y_0)t}}{v'(y_0)}  \left( \partial_{y_0}^2\varphi_{k,\ep}^-(y,y_0) - \partial_{y_0}^2\varphi_{k,\ep}^+(y,y_0) \right) \d y_0\\
&= \cI_{\ep,2,2,1} + \cI_{\ep,2,2,2}.
\end{align*}
Arguing as in $\cI_{2,2,2}$ above, from Propositions \ref{prop:Xkpynotpynotvarphi} and \ref{prop:LXkpynotpynotvarphi} we deduce that integrating the local decomposition and bounds give
\begin{align*}
|\lim_{\ep\rightarrow 0}\cI_{\ep, 2,2,1}| \lesssim k^{-\frac32}(k\delta)^{-\frac12-\mu(y)}(1 - \log(k\delta) ) \cS_{k,2}
\end{align*}
while the global pointwise estimates provide
\begin{align*}
|\lim_{\ep\rightarrow 0}\cI_{\ep, 2,2,2}| \lesssim k^{-\frac12}\cS_{k,2}.
\end{align*}
Hence, we have $|\lim_{\ep\rightarrow 0}\cI_{\ep, 2,2,2}| \lesssim k^{-\frac12}(k\delta)^{-\frac12-\mu(y)}(1-\log(k\delta))\cS_{k,2}$.

\diampar{Estimates for $\cI_{\ep,2,3}$} Proceeding as in $\cI_{\ep,2,2}$, we see from Propositions~\ref{prop:Xkpynotvarphi} and \ref{prop:LXkpynotvarphi} that
\begin{align*}
|\lim_{\ep\rightarrow 0}\cI_{\ep, 2,3}| \lesssim k^{-\frac12}(k\delta)^{-\frac12-\mu(y)}( 1 - \log(k\delta) )\cS_{k,2}
\end{align*}
as well.

\bullpar{End of proof}
Gathering all the estimates, we read
\begin{align*}
\left| \psi_k(t,y) \right| \lesssim k^{-\frac32}(kt)^{-1}(k\delta)^{\frac12-\mu(y)}(1-\log(k\delta))\cS_{k,2} + k^{-\frac12}(kt)^{-2}(k\delta)^{-\frac12-\mu(y)}( 1 - \log(k\delta) ) \cS_{k,2}
\end{align*}
which yields the desired estimate for $\delta= \frac{\beta}{kt}$.
\end{proof}

We next show the decay rates in the weak regime.

\begin{proposition}\label{prop:IDpsiweak}
Let $k\geq 1$ and $t\geq 1$. Then, 
\begin{align*}
\left| \psi_k(t,y) \right| \lesssim k^{-\frac52}t^{-\frac32 + \mu(y)}\cS_{k,2}
\end{align*}
for all $y\in  (\vartheta_1,\varpi_{1,1})\cup (\varpi_{2,2},\vartheta_2)$.
\end{proposition}

\begin{proof}
We show the estimates for $y\in (\varpi_{2,2}, \vartheta_2)$. Proceeding as before, let $\delta_0 = \min \left( \frac{3\beta}{k}, \frac{\varpi_{2,2} - \varpi_{2}}{2}\right) $, let $\delta\in ( 0, \frac{\delta_0}{2})$ and we consider 
\begin{align*}
\cI_{\ep,1} = \int_{B_\delta(y)}e^{-ikv(y_0) y} \left( \partial_{y_0}\varphi_{k,\ep}^-(y,y_0) - \partial_{y_0} \varphi_{k,\ep}^+(y,y_0) \right) \d y_0
\end{align*}
and
\begin{align*}
\cI_{\ep,2} = \int_{B_\delta^c(y)}e^{-ikv(y_0) y} \left( \partial_{y_0}\varphi_{k,\ep}^-(y,y_0) - \partial_{y_0} \varphi_{k,\ep}^+(y,y_0) \right) \d y_0
\end{align*}
as before, so that
\begin{align*}
\psi_k(t,y) = \frac{1}{ikt}\frac{1}{2\pi i} \lim_{\ep\rightarrow 0} \left( \cI_{\ep,1} + \cI_{\ep,2} \right).
\end{align*}

\bullpar{Estimates for $\cI_{\ep,1}$} As usual, since $\delta<\frac{3\beta}{k}$ we have $y\in I_3(y_0)$ and $y_0 \in I_W$ for all $y_0\in B_\delta$. Then, Proposition \ref{prop:Xkpynotvarphi} gives
\begin{align*}
|\cI_{\ep,1}| &\lesssim k^{-\frac12}\cS_{k,1}\int_{y-\delta}^{y+\delta} \left( (k|y-y_0|)^{-\frac12+\mu_0} + (1-2\mu_0)(k|y-y_0|)^{-\frac12-\mu_0}   \right) \d y_0 \\
&\lesssim k^{-\frac12}\cS_{k,1}\int_{y-\delta}^{y+\delta} \left( k^{-1} + (k|y-y_0|)^{-\frac12+\mu(y)} +(1-2\mu(y))(k|y-y_0|)^{\frac12-\mu(y)}  \right) \d y_0 \\
&\lesssim k^{-\frac32}(k\delta)^{\frac12-\mu(y)}\cS_{k,1},
\end{align*}
where we have used that $|\mu_0-\mu(y)|\lesssim |y-y_0|$ and that $(k|y-y_0|)^{-\mu_0} \lesssim (k|y-y_0|)^{-\mu(y)}$ because $k|y-y_0| \leq 1$.

\bullpar{Estimates for $\cI_{\ep,2}$} We further integrate by parts in $\cI_{\ep,2}$ and split into $\cI_{\ep,2,1}$, $\cI_{\ep,2,2}$ and $\cI_{\ep,2,3}$ as in \eqref{eq:cItwoone} - \eqref{eq:cItwothree}. 

\diampar{Estimates for $\cI_{\ep,2,1}$} For
\begin{align*}
\cI_{\ep,2,1} =\frac{e^{-ikv(y_0)t}}{v'(y_0)}\left( \partial_{y_0}\varphi_{k,\ep}^-(y,y_0) - \partial_{y_0}\varphi_{k,\ep}^+(y,y_0) \right) \Big|_{y_0\in \partial B_\delta^c} 
\end{align*}
we note that for $|y-y_0|=\delta$ we can use the local estimates of Proposition \ref{prop:Xkpynotvarphi}, while for $y_0\rightarrow \vartheta_2$ or $y_0\rightarrow\vartheta_1$ we can use either the local or global bounds of Proposition \ref{prop:Xkpynotvarphi} to deduce that
\begin{align*}
|\cI_{\ep,2,1}|\lesssim k^{-\frac12}(k\delta)^{-\frac12-\mu(y)}\cS_{k,1}
\end{align*}

\diampar{Estimates for $\cI_{\ep,2,2}$} Next,
\begin{align*}
\cI_{\ep,2,2} &= \int_{B_\delta^c(y)\cap B_{\delta_0}(y)} \frac{e^{-ikv(y_0)t}}{v'(y_0)}\left( \partial_{y_0}^2 \varphi_{k,\ep}^-(y,y_0) - \partial_{y_0}^2 \varphi_{k,\ep}^+(y,y_0) \right) \d y_0 \\
&\quad +\int_{B_{\delta_0}^c(y)} \frac{e^{-ikv(y_0)t}}{v'(y_0)}\left( \partial_{y_0}^2 \varphi_{k,\ep}^-(y,y_0) - \partial_{y_0}^2 \varphi_{k,\ep}^+(y,y_0) \right) \d y_0 \\
&= \cI_{\ep,2,2,1} + \cI_{\ep,2,2,2}
\end{align*}
where $B_{\delta_0}^c(y) = (\vartheta_1,\vartheta_2)\setminus B_{\delta_0}(y)$. Firstly, for $\cI_{\ep,2,2,1}$ we see from the local decomposition and estimates of Proposition \ref{prop:Xkpynotpynotvarphi} that
\begin{align*}
|\lim_{\ep\rightarrow 0}\cI_{\ep,2,2,1}|\lesssim   k^{-\frac32} \left( \log\left( \frac{\delta_0}{\delta} \right) +  (k\delta)^{-\frac12-\mu(y)} \right)\cS_{k,2}
\end{align*}
because 
\begin{align*}
\int_{B_\delta^c(y) \cap B_{\delta_0}(y)}(k|y-y_0|)^{-\frac12-\mu_0} \d y_0 \lesssim -k^{-1}\log\left( \frac{\delta_0}{\delta} \right)
\end{align*}
while for $\cI_{\ep,2,2,2}$ we use the global pointwise estimates of Proposition \ref{prop:Xkpynotpynotvarphi} and \ref{prop:LXkpynotpynotvarphi} to deduce that
\begin{align*}
|\lim_{\ep\rightarrow 0}\cI_{\ep,2,2,2}|\lesssim   k^{-\frac12} \cS_{k,2}.
\end{align*}
Consequently, $|\lim_{\ep\rightarrow 0}\cI_{\ep,2,2}|\lesssim   k^{-\frac12} \left( \log\left( \frac{\delta_0}{\delta} \right) +  (k\delta)^{-\frac12-\mu(y)} \right)\cS_{k,2}$.

\diampar{Estimates for $\cI_{\ep,2,3}$} Arguing as for $\cI_{\ep,2,2}$, Propositions \ref{prop:Xkpynotvarphi} and \ref{prop:LXkpynotvarphi} give
\begin{align*}
|\cI_{\ep,2,3}| \lesssim k^{-\frac12}(k\delta)^{\frac12-\mu(y)} \cS_{k,1}.
\end{align*}

\bullpar{End of proof}
Gathering the estimates we obtain
\begin{align*}
|\psi_k(t,y)| &\lesssim (kt)^{-1}  k^{-\frac32}(k\delta)^{\frac12-\mu(y)} \cS_{k,1} + k^{-\frac12}(kt)^{-2}\left( -\log \left(\frac{\delta_0}{\delta} \right) +  (k\delta)^{-\frac12-\mu(y)}  \right)\cS_{k,2}
\end{align*}
Setting $\delta=\frac{\delta_0}{4t}$, we obtain the desired result.
\end{proof}

In the non-stratified regime we obtain optimal Euler-type decay rates.

\begin{proposition}\label{prop:IDpsinonstrat}
Let $k\geq 1$ and $y\in [0,\vartheta_1]\cup [\vartheta_2,2]$. Then,
\begin{align*}
|\psi_k(t,y) | \lesssim k^{-\frac52}t^{-2}\cS_{k,2}
\end{align*} 
for all $t\geq 1$.
\end{proposition}

\begin{proof}
We assume $y\in [\vartheta_2, 1]$.  Let $\delta_0=\frac{3\beta}{k}$. Integrating by parts once more, 
\begin{align*}
\psi_k(t,y) &=  -\frac{1}{(ikt)^2}\frac{1}{2\pi i}\lim_{\ep\rightarrow 0} \frac{e^{-ikv(y_0)t}}{v'(y_0)} \left( \partial_{y_0}\varphi_{k,\ep}^-(y,y_0) - \partial_{y_0}\varphi_{k,\ep}^+(y,y_0) \right) \Big|_{y_0=\vartheta_1}^{y_0=\vartheta_2} \\
&\quad + \frac{1}{(ikt)^2}\frac{1}{2\pi i}\lim_{\ep\rightarrow 0} \int_{(\vartheta_1,\vartheta_2)\cap B_{\delta_0}(y)} {e^{-ikv(y_0)t}} \partial_{y_0} \left( \frac{1}{v'(y_0)} \left( \partial_{y_0}\varphi_{k,\ep}^-(y,y_0) - \partial_{y_0}\varphi_{k,\ep}^+(y,y_0) \right) \right) \d y_0 \\
&\quad + \frac{1}{(ikt)^2}\frac{1}{2\pi i}\lim_{\ep\rightarrow 0} \int_{(\vartheta_1,\vartheta_2)\cap B_{\delta_0}^c(y)} {e^{-ikv(y_0)t}} \partial_{y_0} \left( \frac{1}{v'(y_0)} \left( \partial_{y_0}\varphi_{k,\ep}^-(y,y_0) - \partial_{y_0}\varphi_{k,\ep}^+(y,y_0) \right) \right) \d y_0 \\
&= \frac{1}{(ikt)^2}\frac{1}{2\pi i}\lim_{\ep\rightarrow 0} \left( -\cI_{\ep,1} + \cI_{\ep,2} + \cI_{\ep,3} \right).
\end{align*}
We argue for each term. Using the global pointwise estimates of Proposition \ref{prop:Xkpynotvarphi} we deduce that
\begin{align*}
\left| \cI_{\ep,1} \right| \lesssim k^{-\frac12}\cS_{k,1}.
\end{align*} 
With the local decomposition and estimates of Propositions \ref{prop:Xkpynotvarphi} and \ref{prop:Xkpynotpynotvarphi}, together with the support assumptions on $\P(y)$ and $v''(y)$ and also on $\omega_k^0$ and $\varrho_k^0$, we see that
\begin{align*}
\left| \cI_{\ep,2} \right| \lesssim k^{-\frac32}\cS_{k,2}.
\end{align*}
Lastly, we use the global pointwise estimates of Propositions \ref{prop:Xkpynotvarphi} and \ref{prop:Xkpynotpynotvarphi} to obtain
\begin{align*}
\left| \cI_{\ep,2} \right| \lesssim k^{-\frac12}\cS_{k,2}.
\end{align*}
Gathering the estimates, the proof is concluded.

Combining it with the estimate for $\cI_1$ and choosing $\delta= \frac{\delta_0}{4t}$ we obtain the desired result.
\end{proof}

The time decay rates of Proposition \ref{prop:IDpsiweak} degenerate as $y\rightarrow\vartheta_1$ or $y\rightarrow\vartheta_2$, where the stratification of the background density vanishes and where $\psi_k(t,y)$ decays faster according to Proposition \ref{prop:IDpsinonstrat}. We next provide a more precise description for $\psi_k(t,y)$ for $y$ near $\vartheta_n$, which allows us to locally improve the inviscid damping decay rates. To that purpose, we observe that
\begin{align*}
\psi_k(t,y) &= \frac{1}{2\pi i}\lim_{\ep\rightarrow 0}\int_{\vartheta_1}^{\vartheta_2}e^{-ikv(y_0)t}\left( \varphi_{k,\ep}^-(y,y_0) - \varphi_{k,\ep}^+(y,y_0) \right) v'(y_0) \d y_0 \\
&=\frac{1}{ikt} \frac{1}{2\pi i}\lim_{\ep\rightarrow 0}\int_{\vartheta_1}^{\vartheta_2}e^{-ikv(y_0)t} \left( \partial_{y_0}\varphi_{k,\ep}^-(y,y_0) - \partial_{y_0}\varphi_{k,\ep}^+(y,y_0) \right) \d y_0 \\
&= \frac{1}{ikt}\frac{1}{2\pi i}\lim_{\ep\rightarrow 0}\int_{\vartheta_1}^{\vartheta_2}e^{-ikv(y_0)t} \left( \varphi_{1,k,\ep}^-(y,y_0) - \varphi_{1,k,\ep}^+(y,y_0) \right) \d y_0 \\
&\quad - \frac{1}{ikv'(y) t}\frac{1}{2\pi i} \lim_{\ep\rightarrow 0}\int_{\vartheta_1}^{\vartheta_2}e^{-ikv(y_0)t} \partial_y \left( \varphi_{k,\ep}^-(y,y_0) - \varphi_{k,\ep}^+(y,y_0) \right)v'(y_0) \d y_0 \\
&\quad -\frac{1}{ikv'(y) t}\frac{1}{2\pi i} \lim_{\ep\rightarrow 0}\int_{\vartheta_1}^{\vartheta_2}e^{-ikv(y_0)t} \partial_y \left( \varphi_{k,\ep}^-(y,y_0) - \varphi_{k,\ep}^+(y,y_0) \right)(v'(y)-v'(y_0)) \d y_0 
\end{align*}
Hence, for 
\begin{align*}
\psi_{k,\ep}(t,y) := \frac{1}{2\pi i }\int_{\vartheta_1}^{\vartheta_2}e^{-ikv(y_0)t}\left( \varphi_{k,\ep}^-(y,y_0) - \varphi_{k,\ep}^+(y,y_0) \right) v'(y_0) \d y_0 
\end{align*}
we have
\begin{align*}
ikv'(y)t \psi_{k,\ep}(t,y) + \partial_y \psi_{k,\ep}(t,y) &= v'(y)\int_{\vartheta_1}^{\vartheta_2}e^{-ikv(y_0)t} \left( \varphi_{1,k,\ep}^-(y,y_0) - \varphi_{1,k,\ep}^+(y,y_0) \right) \d y_0 \\
&\quad- \int_{\vartheta_1}^{\vartheta_2}e^{-ikv(y_0)t} \partial_y \left( \varphi_{k,\ep}^-(y,y_0) - \varphi_{k,\ep}^+(y,y_0) \right)(v'(y)-v'(y_0)) \d y_0.
\end{align*}
Integrating from $y$ to $\vartheta_2$, we reach
\begin{align*}
e^{ikv(y)t}\psi_{k,\ep}(t,y) &= e^{ikv(\vartheta_2)t}\psi_{k,\ep}(t,\vartheta_2) \\
&\quad + \frac{1}{2\pi i}\int_{y}^{\vartheta_2}\int_{\vartheta_1}^{\vartheta_2} e^{-ik(v(y_0)-v(z)) t}\partial_z \left( \varphi_{k,\ep}^-(z,y_0) - \varphi_{k,\ep}^+(z,y_0) \right) (v'(z) - v'(y_0) ) \d y_0\d z\\
&\quad - \frac{1}{2\pi i}\int_y^{\vartheta_2} v'(z) \int_{\vartheta_1}^{\vartheta_2}e^{-ik(v(y_0)-v(z))t} \left( \varphi_{1,k,\ep}^-(z,y_0) - \varphi_{1,k,\ep}^+(z,y_0) \right) \d y_0 \d z \\
&:=e^{ikv(\vartheta_2)t}\psi_{k,\ep}(t,\vartheta_2) + \psi_{1,k,\ep}(t,y) + \psi_{2,k,\ep}(t,y)
\end{align*}
for all $y\in (\varpi_{2,2}, \vartheta_2)$. We likewise reach similar expressions for $y\in (\vartheta_1,\varpi_{1,1})$, where now we integrate from $\vartheta_1$ to $y$. As a result, we now have

\begin{proposition}\label{prop:improvedIDpsiweak}
Let $k\geq 1$ and $t\geq 1$. Then, 
\begin{align*}
\left| \psi_k(t,y) \right| \lesssim k^{-\frac32}\min \left( t^{-\frac32+\mu(y)}, t^{-2} + (\vartheta_2-y)t^{-1} \right) \cS_{k,2},
\end{align*}
for all $y\in(\varpi_{2,2},\vartheta_2)$.
\end{proposition}

The proposition is a consequence of Proposition \ref{prop:IDpsinonstrat} and Lemmas \ref{lemma:firstintegralIDpsiweak} and \ref{lemma:secondintegralIDpsiweak} below.

\begin{lemma}\label{lemma:firstintegralIDpsiweak}
Let $k\geq 1$ and $y\in (\varpi_{2,2}, \vartheta_2)$. Then,
\begin{align*}
\left| \psi_{1,k,\ep}(t,y) \right| \lesssim k^{-\frac52}\left( t^{-2} + (\vartheta_2-y)t^{-1} \right)\cS_{k,2} 
\end{align*}
uniformly for all $0<\ep<\ep_*$.
\end{lemma}

\begin{proof}
We integrate by parts once in $z\in (y,2)$ to write
\begin{align*}
\psi_{1,k,\ep}(t,y) &= \frac{1}{ikt}\frac{1}{2\pi i}\int_{\vartheta_1}^{\vartheta_2} \frac{e^{-ik(v(y_0)-v(z)) t}}{v'(z)} \partial_y \left( \varphi_{k,\ep}^-(z,y_0) - \varphi_{k,\ep}^+(z,y_0) \right) (v'(z) - v'(y_0) ) \d y_0 \Big|_{z=y}^{z=\vartheta_2} \\
&\quad -\frac{1}{ikt}\frac{1}{2\pi i}\int_y^1 \int_{\vartheta_1}^{\vartheta_2} \frac{e^{-ik(v(y_0)-v(z))t}}{v'(z)} \left( \partial_z^2\varphi_{k,\ep}^-(z,y_0) - \partial_z^2\varphi_{k,\ep}^+(z,y_0) \right) (v'(z) - v'(y_0))\d y_0 \d z \\
&\quad - \frac{1}{ikt}\frac{1}{2\pi i}\int_y^1 \int_{\vartheta_1}^{\vartheta_2} e^{-ik(v(y_0)-v(z))t} \partial_z \left( \frac{v'(z) - v'(y_0)}{v'(z)} \right) \partial_z\left( \varphi_{k,\ep}^-(z,y_0) - \varphi_{k,\ep}^+(z,y_0) \right) \d y_0 \d z \\
&= \frac{1}{ikt}\frac{1}{2\pi i}\left(  \cI_{\ep,1} - \cI_{\ep,2} - \cI_{\ep,3} + \cI_{\ep,4} \right).
\end{align*}

\bullpar{Estimates for $\cI_{\ep,1}$ and $\cI_{\ep,2}$} Let $z=\vartheta_2$ for $j=1$ and $z=y$ for $j=2$, we have
\begin{align*}
\cI_{\ep,j} &= \int_{\vartheta_1}^{\vartheta_2} \frac{e^{-ik(v(y_0)-v(z)) t}}{v'(z)} \partial_y \left( \varphi_{k,\ep}^-(z,y_0) - \varphi_{k,\ep}^+(z,y_0) \right) (v'(z) - v'(y_0) ) \d y_0 \\
&= - \frac{1}{ikt} \frac{e^{-ik(v(y_0) - v(z))t}}{v'(z)v'(y_0)} \partial_y \left( \varphi_{k,\ep}^-(z,y_0) - \varphi_{k,\ep}^+(z,y_0) \right) (v'(z) - v'(y_0) ) \Big|_{y_0=\vartheta_1}^{y_0=\vartheta_2} \\
&\quad + \frac{1}{ikt}\int_{\vartheta_1}^{\vartheta_2}\frac{e^{-ik(v(y_0) - v(z))t}}{v'(z)} \partial_{y_0} \left( \frac{v'(z) - v'(y_0)}{v'(y_0)}\partial_y \left( \varphi_{k,\ep}^-(z,y_0) - \varphi_{k,\ep}^+(z,y_0) \right) \right) \d y_0 \\
&=\frac{1}{ikt}\left( -\cI_{\ep,j,1} + \cI_{\ep,j,2} + \cI_{\ep,j,3} \right).
\end{align*}
From Proposition \ref{prop:Xkvarphi} we see $\left| \partial_y \left( \varphi_{k,\ep}^-(z,y_0) - \varphi_{k,\ep}^+(z,y_0) \right) (v'(z) - v'(y_0) ) \right| \lesssim k^{-\frac32}\cS_{k,1}$, which yields
\begin{align*}
| \cI_{\ep,j,1} | + | \cI_{\ep,j,2} | \lesssim k^{-\frac32}\cS_{k,1}.
\end{align*}
For $\cI_{\ep,j,3}$, we next split
\begin{align*}
\cI_{\ep,j,3} &=\int_{I_3(y)}\frac{e^{-ik(v(y_0) - v(z))t}}{v'(z)} \partial_{y_0} \left( \frac{v'(z) - v'(y_0)}{v'(y_0)}\partial_y \left( \varphi_{k,\ep}^-(z,y_0) - \varphi_{k,\ep}^+(z,y_0) \right) \right) \d y_0 \\
&\quad + \int_{I_3^c(y)}\frac{e^{-ik(v(y_0) - v(z))t}}{v'(z)} \partial_{y_0} \left( \frac{v'(z) - v'(y_0)}{v'(y_0)}\partial_y \left( \varphi_{k,\ep}^-(z,y_0) - \varphi_{k,\ep}^+(z,y_0) \right) \right) \d y_0 \\
&= \cI_{\ep,j,3,1} + \cI_{\ep,j,3,2}.
\end{align*}
where now $I_3^c(y) = (\vartheta_1,\vartheta_2)\setminus I_3(y)$. Using the local decomposition and bounds of Propositions \ref{prop:Xkvarphi} and \ref{prop:Xkpypynotvarphi}, noting that $(k|z-y_0|)^{-\frac12-\mu_0)}\lesssim (k|z-y_0|)^{-\frac12-\mu(z)}$ because  $(k|z-y_0|)\lesssim 1$ for all  $y_0\in I_3(z)$,  we have
\begin{align*}
|\cI_{\ep,j,3,1}|&\lesssim k^{-\frac12}\cS_{k,2} \int_{I_3(z)} \left( 1 + \left( \P(z) + (1-2\mu(z)) \right)(k|z-y_0|)^{-\frac12-\mu(z)} \right) \d y_0 \\
&\lesssim k^{-\frac32}\cS_{k,2}
\end{align*}
since $\frac{\P(z)}{\frac12-\mu(z)} = \frac{\frac12+\mu(z)}{v'(z)}$ is uniformly bounded. Likewise the global pointwise bounds of Propositions \ref{prop:Xkvarphi} and \ref{prop:Xkpypynotvarphi} yield
\begin{align*}
|\cI_{\ep,j,3,2}| \lesssim k^{-\frac12}\cS_{k,2}
\end{align*}
and we conclude that $|\cI_{\ep,j}| \lesssim k^{-\frac12}(kt)^{-1}\cS_{k,2}$, for $j=1,2$.

\bullpar{Estimates for $\cI_{\ep,3}$} 
We first write
\begin{align*}
\cI_{\ep,3} = \int_y^{\vartheta_2} \frac{e^{ikv(z) t}}{v'(z)}  \widetilde\cI_{\ep,3}(z) \d z,
\end{align*}
where
\begin{align*}
\widetilde\cI_{\ep,3}(z) &:= \int_{\vartheta_1}^{\vartheta_2} e^{-ikv(y_0)t}\left( \partial_z^2\varphi_{k,\ep}^-(z,y_0) - \partial_z^2\varphi_{k,\ep}^+(z,y_0) \right) (v'(z) - v'(y_0))\d y_0 \\
&= \int_{\vartheta_1}^{\vartheta_2} e^{-ikv(y_0)t} V_{k,\ep}^-(z,y_0)\varphi_{k,\ep}^-(z,y_0)(v'(z) - v'(y_0)) \d y_0 \\
& \quad - \int_{\vartheta_1}^{\vartheta_2} e^{-ikv(y_0)t} V_{k,\ep}^+(z,y_0) \varphi_{k,\ep}^+(z,y_0)(v'(z) - v'(y_0)) \d y_0 \\
& \quad+ w_k^0(z)\int_{\vartheta_1}^{\vartheta_2}e^{-ikv(y_0)t}\frac{2i\ep(v'(z) - v'(y_0))}{(v(z) - v(y_0))^2 + \ep^2} \d y_0 \\
&=:\widetilde\cI_{\ep,3}^- - \widetilde\cI_{\ep,3}^+ + \widehat{\cI}_{\ep,3},
\end{align*}
where we defined
\begin{align*}
V_{k,\ep}^\pm(z,y_0) = k^2 + \frac{v''(z)}{v(z) - v(y_0) \pm i\ep} - \frac{\P(z)}{(v(z) - v(y_0) \pm i\ep)^2} .
\end{align*}
The dominated convergence theorem directly yields $\lim_{\ep\rightarrow 0} \widehat{\cI}_{\ep,3}=0$ so we focus on $\widetilde\cI_{\ep,3}^\pm$. Set $\delta_0 =\frac{3\beta}{k}$ and $\delta\in (0,\frac{\delta_0}{2})$. Let $B_r:=B_r(z)$ the ball of radius $r>0$ centred at $z$. We further set
\begin{align*}
\widetilde\cI_{\ep,3,1}^\pm = \int_{B_\delta(z)\cap(\vartheta_1,\vartheta_2)}e^{-ikv(y_0)t}V_{k,\ep}^\pm(z,y_0)\varphi_{k,\ep}^\pm(z,y_0) (v'(z) - v'(y_0)) \d y_0
\end{align*}
and
\begin{align*}
\widetilde\cI_{\ep,3,2}^\pm = \int_{B_\delta^c(z)\cap(\vartheta_1,\vartheta_2)}e^{-ikv(y_0)t}V_{k,\ep}^\pm(z,y_0)\varphi_{k,\ep}^\pm(z,y_0) (v'(z) - v'(y_0)) \d y_0
\end{align*}
so that $\widetilde\cI_{\ep,3}^\pm = \widetilde\cI_{\ep,3,1}^\pm + \widetilde\cI_{\ep,3,1}^\pm$. From Proposition \ref{prop:Xkvarphi} we have
\begin{align*}
|\widetilde\cI_{\ep,3,1}^\pm| & = \left|  \int_{z-\delta}^{z+\delta} e^{-ikv(y_0)t} V_{k,\ep}^\pm(z,y_0) \varphi_{k,\ep}^\pm(z,y_0)(v'(z) - v'(y_0)) \d y_0 \right| \\
&\quad\lesssim k^{-\frac12}\cS_{k,2}\int_{z-\delta}^{z+\delta} \left( 1  + \P(z) \left( (k|z-y_0|) ^{-\frac12+\mu_0} + (k|z-y_0|) ^{-\frac12-\mu_0} \right) \right) \d y_0 \\
&\quad\lesssim k^{-\frac32}\left( (k\delta) +  (k\delta)^{\frac12-\mu(z)} +  (k\delta)^{\frac12+\mu(z)} \right) \cS_{k,2}
\end{align*}
Here we have used that $|\mu(z) - \mu_0||\log(k|z-y_0|) \lesssim k^{-1}$ for $|z-y_0|\leq \delta \leq\frac{1}{2k}$ so that
\begin{align*}
\eta^{-\frac12+\mu_0} \lesssim \eta^{-\frac12+\mu(z)}, \quad \eta^{-\frac12-\mu_0} \lesssim \eta^{-\frac12-\mu(z)}, 
\end{align*}
for $\eta=k|z-y_0|$ and also that $\frac{\P(z)}{\frac12-\mu(z)}= \frac{\frac12+\mu(z)}{v'(z)}\lesssim 1$. On the other hand, for $\widetilde\cI_{3,2}$, we integrate by parts in $y_0\in (\vartheta_1,\vartheta_2)\cap B_\delta^c(z)$ thus obtaining
\begin{align*}
\widetilde\cI_{\ep,3,2}^\pm &= -\frac{1}{ikt}\frac{e^{-ikv(y_0)t}}{v'(y_0)} V_{k,\ep}^\pm(z,y_0) \varphi_{k,\ep}^\pm(z,y_0)(v'(z) - v'(y_0)) \Big|_{y_0\in \partial B_\delta^c\cap(\vartheta_1,\vartheta_2)} \\
&\quad + \frac{1}{ikt}\int_{B_{\delta}^c(z)\cap B_{\delta_0}(z)\cap(\vartheta_1,\vartheta_2)}e^{-ikv(y_0)t}\partial_{y_0} \left( \frac{v'(z) - v'(y_0)}{v'(y_0)}V_{k,\ep}^\pm(z,y_0) \varphi_{k,\ep}^\pm(z,y_0) \right)  \d y_0 \\
&\quad + \frac{1}{ikt}\int_{B_{\delta_0}^c(z)\cap(\vartheta_1,\vartheta_2)}e^{-ikv(y_0)t}\partial_{y_0} \left( \frac{v'(z) - v'(y_0)}{v'(y_0)}V_{k,\ep}^\pm(z,y_0) \varphi_{k,\ep}^\pm(z,y_0) \right)  \d y_0 \\
&=  \frac{1}{ikt} \left( -\widetilde\cI_{\ep,3,2,1}^\pm + \widetilde\cI_{\ep,3,2,2}^\pm + \widetilde\cI_{\ep,3,2,3}^\pm  \right)
\end{align*}
We argue for each term.

\diampar{Estimates for $\widetilde\cI_{\ep,3,2,1}^\pm$} If $y_0=\vartheta_1$ or $y_0=\vartheta_2$ and  $y_0\in B_\delta^c(z)$ we have $|v(z) - v(y_0)\pm i\ep|^{-1}\lesssim (k\delta)^{-1}$ and thus from Lemma \ref{lemma:vanishingpsiIE} and the dominated convergence theorem we have
\begin{align*}
\lim_{\ep\rightarrow 0} \int_y^{\vartheta_2} e^{ikv(z) t} \frac{\widetilde\cI_{\ep,3,2,1}^-(z) - \widetilde\cI_{\ep,3,2,1}^+(z)}{v'(z)} \d z =0.
\end{align*}
On the other hand, for $y_0\in \partial B_{\delta}(y)\subset  (\vartheta_1,\vartheta_2)$, we have $k|z-y_0|=\delta$ and we appeal to the local bounds of Proposition \ref{prop:Xkvarphi} to have
\begin{align*}
\left| V_{k,\ep}^\pm(z,y_0)\varphi_{k,\ep}^\pm(z,y_0) (v'(z) - v'(y_0)) \right| \lesssim k^{-\frac32}\left( 1 + \P(z) (k\delta)^{-\frac12-\mu(z)} \right) \cS_{k,1}
\end{align*}
which then yields $|\widetilde\cI_{\ep,3,2,1}^\pm | \lesssim k^{-\frac52}\left( 1 + \P(z) (k\delta)^{-\frac12-\mu(z)} \right) \cS_{k,2}$, uniformly for all $0< \ep< \ep_*$.

\diampar{Estimates for $\widetilde\cI_{\ep,3,2,2}$} For $y_0\in B_{\delta}^c(z)\cap B_{\delta_0}(z)\cap(\vartheta_1,\vartheta_2)$ we use Propositions \ref{prop:Xkvarphi} and \ref{prop:Xkpynotvarphi} to obtain
\begin{align*}
&\left| \partial_{y_0} \left( \frac{v'(z) - v'(y_0)}{v'(y_0)}V_{k,\ep}^\pm(z,y_0) \varphi_{k,\ep}^\pm(z,y_0) \right)  \right| \\
&\qquad\lesssim k^{-\frac12}\left( 1 + |v''(z)|(k|z-y_0|)^{-\frac12-\mu(z)} +\P(z)(k|z-y_0|)^{-\frac32-\mu(z)} \right) \cS_{k,2}
\end{align*}
and thus
\begin{align*}
|\widetilde\cI_{\ep,3,2,2}^\pm| &\lesssim k^{-\frac12}\cS_{k,2}\int_{B_\delta^c(z)\cap B_{\delta_0}(z)}\left( 1 + |v''(z)|(k|z-y_0|)^{-\frac12-\mu(z)} + \P(z)(k|z-y_0|)^{-\frac32-\mu(z)} \right) \d y_0 \\
&\lesssim k^{-\frac32}\left( 1 + (k\delta)^{\frac12-\mu(z)} + \P(z)(k\delta)^{-\frac12-\mu(z)} \right) \cS_{k,2}
\end{align*}
since $\frac{|v''(z)|}{\frac12-\mu(z)}\lesssim 1$ due to the compact support of $v''$.

\diampar{Estimates for $\widetilde\cI_{\ep,3,2,3}^\pm$} Now $y_0\in B_{\delta_0}^c(z)\cap (\vartheta_1,\vartheta_2)$ and we can use the global pointwise estimates of Propositions \ref{prop:Xkvarphi} and \ref{prop:LinfpynotcV} to obtain
\begin{align*}
\left| \partial_{y_0} \left( \frac{v'(z) - v'(y_0)}{v'(y_0)}V_{k,\ep}^\pm(z,y_0) \varphi_{k,\ep}^\pm(z,y_0) \right)  \right| \lesssim k^{-\frac12}\cS_{k,2}.
\end{align*}

\diampar{Estimates for $\cI_{\ep,3}$} Hence, we conclude that $|\widetilde\cI_{\ep,3,2}| \lesssim k^{-\frac12} \left( 1 + \P(z)(k\delta)^{-\frac12-\mu(z)} \right) \cS_{k,2}$. Together with the bound for $\widetilde\cI_{\ep,3,1}$, we obtain
\begin{align*}
|\widetilde\cI_{\ep,3}| &\lesssim k^{-\frac32}\left( (k\delta) + (k\delta)^{\frac12-\mu(z)} \right) \cS_{k,2} +k^{-\frac12}(kt)^{-1} \left( 1 + \P(z) (k\delta)^{-\frac12-\mu(z)} \right) \cS_{k,2}.
\end{align*}
Next, to obtain bounds on $\cI_{\ep,3}$. Optimizing in $k\delta$, we find that $k\delta = \frac{\beta}{t}$ we have $\delta\leq \frac{\delta_0}{2}$ for $t\geq 1$ and 
\begin{align*}
|\widetilde\cI_{\ep,3}| &\lesssim k^{-\frac32} \left( t^{-1} + t^{-\frac12+\mu(z)} \right) \cS_{k,2}
\end{align*}
so that
\begin{align*}
\left| \cI_{\ep,3} \right| \lesssim k^{-\frac32}(\vartheta_2-y) (1+t^{-1}) \cS_{k,2}
\end{align*}
In fact, we can recover $t^{-\frac12+\mu(y)}$ at the expenses of a logarithmic loss. Indeed,
\begin{align*}
t^{-\frac12+\mu(z)} \lesssim t^{-\frac12 + \mu(y)} + |z-y|\log(t)
\end{align*}
for all $t\geq 1$, thus concluding that
\begin{align*}
\left| \cI_{\ep,3} \right| \lesssim k^{-\frac32}(\vartheta_2-y) \left(  t^{-1} + t^{-\frac12+\mu(y)} + (\vartheta_2-y)\log(t) \right) \cS_{k,2}.
\end{align*}

\bullpar{Estimates for $\cI_{\ep,4}$}
As before we write
\begin{align*}
\cI_{\ep,4} = \int_{y}^{\vartheta_2}e^{ikv(z)t} \widetilde\cI_{\ep,4}(z) \d z
\end{align*}
where 
\begin{align*}
\widetilde\cI_{\ep,4} &= \int_{\vartheta_1}^{\vartheta_2} e^{-ikv(y_0)t} \partial_z \left( \frac{v'(z) - v'(y_0)}{v'(z)} \right) \partial_z\left( \varphi_{k,\ep}^-(z,y_0) - \varphi_{k,\ep}^+(z,y_0) \right) \d y_0  \\
&=\int_{B_\delta(z)\cap(\vartheta_1,\vartheta_2)} e^{-ikv(y_0)t} \partial_z \left( \frac{v'(z) - v'(y_0)}{v'(z)} \right) \partial_z\left( \varphi_{k,\ep}^-(z,y_0) - \varphi_{k,\ep}^+(z,y_0) \right) \d y_0 \\
&\quad + \int_{B_\delta^c(z)\cap(\vartheta_1,\vartheta_2)} e^{-ikv(y_0)t} \partial_z \left( \frac{v'(z) - v'(y_0)}{v'(z)} \right) \partial_z\left( \varphi_{k,\ep}^-(z,y_0) - \varphi_{k,\ep}^+(z,y_0) \right) \d y_0  \\
&=\widetilde\cI_{\ep,4,1} + \widetilde\cI_{\ep,4,2}.
\end{align*}
Next, Proposition \ref{prop:Xkvarphi} yields
\begin{align*}
|\widetilde\cI_{\ep,4,1}| &\lesssim k^{-\frac12}\cS_{k,1} \int_{z-\delta}^{z+\delta} \left( 1 +  (k|z-y_0|)^{-\frac12+\mu(z)} + (1-2\mu(z))(k|z-y_0|)^{-\frac12-\mu(z)} \right) \d z \\
&\lesssim k^{-\frac32} \left( (k\delta) + (k\delta)^{\frac12-\mu(z)} \right)\cS_{k,1}.
\end{align*}
For $\widetilde\cI_{\ep,4,2}$ we integrate by parts again, obtaining
\begin{align*}
\widetilde\cI_{\ep,4,2} &= -\frac{1}{ikt} \frac{e^{-ikv(y_0)t}}{v'(y_0)}\partial_z \left( \frac{v'(z) - v'(y_0)}{v'(z)} \right) \partial_z\left( \varphi_{k,\ep}^-(z,y_0) - \varphi_{k,\ep}^+(z,y_0) \right) \Big|_{y_0\in \partial B_\delta^c(z)\cap (\vartheta_1,\vartheta_2)} \\
&\quad +\frac{1}{ikt} \int_{B_\delta^c(z)\cap B_{\delta_0}(z)\cap(\vartheta_1,\vartheta_2)}{e^{-ikv(y_0)t}} \partial_{y_0} \left( \partial_z \left( \frac{v'(z) - v'(y_0)}{v'(z)} \right) \frac{\partial_z\left( \varphi_{k,\ep}^-(z,y_0) - \varphi_{k,\ep}^+(z,y_0) \right)}{v'(y_0)} \right) \d y_0 \\
&\quad +\frac{1}{ikt} \int_{B_{\delta_0}^c(z)\cap(\vartheta_1,\vartheta_2)}{e^{-ikv(y_0)t}} \partial_{y_0} \left( \partial_z \left( \frac{v'(z) - v'(y_0)}{v'(z)} \right) \frac{\partial_z\left( \varphi_{k,\ep}^-(z,y_0) - \varphi_{k,\ep}^+(z,y_0) \right)}{v'(y_0)} \right) \d y_0 \\
&= \frac{1}{ikt}\left( -\widetilde\cI_{\ep,4,2,1} + \widetilde\cI_{\ep,4,2,2} + \widetilde\cI_{\ep,4,2,3} \right)
\end{align*}
We see from the local and global pointwise estimates of Proposition \ref{prop:Xkvarphi} that
\begin{align*}
\left|\widetilde\cI_{\ep,4,2,1} \right| \lesssim k^{-\frac12} \left( 1 + (k\delta)^{-\frac12-\mu(z)} \right) \cS_{k,1}.
\end{align*}
Likewise, the local decomposition and bounds of Propositions \ref{prop:Xkvarphi} and \ref{prop:Xkpynotpynotvarphi} yield
\begin{align*}
\left|\widetilde\cI_{\ep,4,2,1} \right| \lesssim k^{-\frac12}\cS_{k,2}\int_{B_\delta^c(z)\cap B_{\delta_0}(z)} (k|z-y_0|)^{-\frac32-\mu(z)} \d y_0 \lesssim k^{-\frac32}\left( 1 + (k\delta)^{-\frac12-\mu(z)}\right)\cS_{k,2}.
\end{align*}
Finally, the global pointwise bounds of Propositions \ref{prop:Xkvarphi} and \ref{prop:Xkpynotpynotvarphi} give
\begin{align*}
\left|\widetilde\cI_{\ep,4,2,1} \right| \lesssim k^{-\frac12}\cS_{k,2}.
\end{align*}
Hence we conclude that $\left| \widetilde\cI_{\ep,4,2} \right| \lesssim k^{-\frac12}(kt)^{-1}\left( 1 + (k\delta)^{-\frac12-\mu(z)} \right) \cS_{k,2}$. Together with $\widetilde\cI_{\ep,4,1}$ we reach
\begin{align*}
\left| \widetilde\cI_{\ep,4} \right| &\lesssim k^{-\frac32} \left( (k\delta) + (k\delta)^{\frac12-\mu(z)}\right) \cS_{k,1} + k^{-\frac12}(kt)^{-1}\left( 1 + (k\delta)^{-\frac12-\mu(z)} \right) \cS_{k,2} \\
&\lesssim k^{-\frac32} \left( t^{-1} + t^{-\frac12+\mu(z)} \right) \cS_{k,2}
\end{align*}
and we obtain the same estimate for $\cI_{\ep,4}$ as those for $\cI_{\ep,3}$ setting $\delta = \frac{\delta_0}{4t}$.

\bullpar{End of proof} Gathering the bounds for $\cI_{\ep,j}$, for $j=1,2,3,4$, we conclude that 
\begin{align*}
\left| \psi_{1,k,\ep}(t,y) \right| \lesssim k^{-\frac52}\left( t^{-2} + (\vartheta_2-y)t^{-1} \right)\cS_{k,2} 
\end{align*}
thus finishing the proof.
\end{proof}

The ideas presented in the proof of the above Lemma also yield the next result.

\begin{lemma}\label{lemma:secondintegralIDpsiweak}
Let $k\geq 1$ and $y\in (\varpi_{2,2}, \vartheta_2)$. Then,
\begin{align*}
\left| \psi_{2,k,\ep}(t,y) \right| \lesssim k^{-\frac32}\left( t^{-2} + (\vartheta_2-y)t^{-1} \right)\cS_{k,2} 
\end{align*}
uniformly for all $0<\ep<\ep_*$.
\end{lemma}

We remark here that the loss of one full power of $k$ decay in the estimate is due to the bounds of $\varphi_{1,k,\ep}^\pm$, which are indeed one power of $k$ worse than those for $\varphi_{k,\ep}^\pm$ as can be seen from Propositions \ref{prop:Xkvarphi} and \ref{prop:Xkvarphi1}. This loss can be avoided if one is willing to obtain less vanishing in $(\vartheta_2-y)_+$. More precisely, using Minkowski inequality in the global estimates for $\varphi_{1,k,\ep}$ we are able to show that

\begin{corollary}\label{cor:secondintegralIDpsiweak}
Let $k\geq 1$ and $y\in (\varpi_{2,2}, \vartheta_2)$. Then,
\begin{align*}
\left| \psi_{2,k,\ep}(t,y) \right| \lesssim k^{-2}\left( t^{-2} + (\vartheta_2-y)^\frac12 t^{-1} \right)\cS_{k,2} 
\end{align*}
uniformly for all $0<\ep<\ep_*$.
\end{corollary}

\subsection{Decay estimates for $\partial_y\psi_k(t,y)$}\label{sec:timedecaypypsi} In this section we comment on the proof for the inviscid damping for $\partial_y\psi_k(t,y)$. 

\begin{proposition}\label{prop:IDpypsinonstrat}
Let $k\geq 1$ and $y\in [0,\vartheta_1]\cup [\vartheta_2,2]$. Then,
\begin{align*}
\left| \partial_y\psi_k(t,y) \right| \lesssim k^{-\frac12} t^{-1}\cS_{k,1}
\end{align*} 
for all $t\geq 1$.
\end{proposition}

\begin{proposition}\label{prop:IDpypsiweak}
Let $k\geq 1$ and $y\in (\vartheta_1,\varpi_{1,2})\cup (\varpi_{2,2}\vartheta_2)$. Then,
\begin{align*}
\left| \partial_y\psi_k(t,y) \right| \lesssim k^{-\frac12}t^{-\frac12+\mu(y)}\cS_{k,1}
\end{align*} 
for all $t\geq 1$.
\end{proposition}

\begin{proposition}\label{prop:IDpypsimild}
Let $k\geq 1$ and $y\in [\varpi_{1,1}, \varpi_{1,2}]\cup [\varpi_{2,1}, \varpi_{2,2}]$. Then,
\begin{align*}
\left| \partial_y\psi_k(t,y) \right| \lesssim k^{-\frac12} t^{-\frac12+\mu(y)} \left( 1 + \log(t) \right) \cS_{k,1}
\end{align*} 
for all $t\geq 1$.
\end{proposition}

\begin{proposition}\label{prop:IDpypsistrong}
Let $k\geq 1$ and $y\in (\varpi_{1,2}, \varpi_{2,1})$. Then,
\begin{align*}
\left| \partial_y\psi_k(t,y) \right| \lesssim k^{-\frac12} t^{-\frac12}\cS_{k,1}
\end{align*} 
for all $t\geq 1$.
\end{proposition}

To prove the above decay rates, we see from Lemma \ref{lemma:vanishingpsiIE} that
\begin{align*}
\partial_y\psi_k(t,y) = \frac{1}{2\pi i} \lim_{\ep\rightarrow 0} \int_{\vartheta_1}^{\vartheta_2} e^{-ikv(y_0)t}\left( \partial_y\varphi_{k,\ep}^-(y,y_0) - \partial_y\varphi_{k,\ep}^+(y,y_0) \right)  \d y_0.
\end{align*}
and we proceed as usual, considering the integral over $B_\delta$ and over $B_\delta^c$. On $\B_\delta$ we can use the local estimates of Propositions \ref{prop:Xkvarphi} and \ref{prop:LXkvarphi} to gain smallness from $(k\delta)^{\frac12-\mu(y)}$  for $y\in I_S\cup I_W$ and $(k\delta)^{\frac12-\mu(y)}\left( 1 - \log(k\delta)\right)$ for $y\in I_M$. On $B_{\delta}^c$, we integrate by parts using the oscillatory factor and we bound the resulting contributions with the local and global estimates of Propositions \ref{prop:Xkpypynotvarphi} and \ref{prop:LXkpypynotvarphi}, roughly collecting $(kt)^{-1}(k\delta)^{-\frac12-\mu(y)}$ and $(kt)^{-1}(k\delta)^{-\frac12-\mu(y)}\left( 1 + \log(k\delta)\right)$ for $y\in I_S\cup I_W$ and $y\in I_M$, respectively. Optimizing for $\delta\sim \frac{1}{kt}$ gives the desired result. While we omit the routine details to the basic decay rates, we do show how to obtain the improved localised bounds for $\partial_y\psi_k(t,y)$.

\begin{corollary}\label{cor:improvedIDpypsiweak}
Let $k\geq 1$ and $t\geq 1$. Then,
\begin{align*}
\left| \partial_y \psi_k(t,y) \right| \lesssim k^{-\frac12} \min \left( t^{-\frac12+\mu(y)}, t^{-1} + (\vartheta_2-y) \right) \cS_{k,2},
\end{align*}
for all $y\in(\varpi_{2,2},\vartheta_2)$.
\end{corollary}

\begin{proof}
We have
\begin{align*}
\partial_y \psi_{k,\ep}(t,y) &= -ikv'(y)t \psi_{k,\ep}(t,y) + \frac{v'(y)}{2\pi i}\int_{\vartheta_1}^{\vartheta_2} e^{-ikv(y_0)t}\left( \varphi_{1,k,\ep}^-(y,y_0) - \varphi_{1,k,\ep}^+(y,y_0) \right) \d y_0 \\
&\quad -\frac{1}{2\pi i} \int_{\vartheta_1}^{\vartheta_2} e^{-ikv(y_0)t}\left( \partial_y\varphi_{k,\ep}^-(y,y_0) - \partial_y\varphi_{k,\ep}^+(y,y_0) \right)(v'(y) - v'(y_0)) \d y_0 \\
\end{align*}
Using Propositions \ref{prop:Xkvarphi1} and \ref{prop:Xkvarphi2} and the ideas presented in the proof of Lemma \ref{lemma:firstintegralIDpsiweak}, it is easy to see that 
\begin{align*}
\left| \int_{\vartheta_1}^{\vartheta_2} e^{-ikv(y_0)t}\left( \varphi_{1,k,\ep}^-(y,y_0) - \varphi_{1,k,\ep}^+(y,y_0) \right) \d y_0 \right| \lesssim (kt)^{-1}k^{-\frac12}\cS_{k,2},
\end{align*}
while integration by parts once and using Propositions \ref{prop:Xkvarphi} and \ref{prop:Xkpypynotvarphi} as in the proof of the estimates of $\widetilde\cI_{\ep,4}$ in Lemma \ref{lemma:firstintegralIDpsiweak} yield 
\begin{align*}
\left| \int_{\vartheta_1}^{\vartheta_2} e^{-ikv(y_0)t}\left( \partial_y\varphi_{k,\ep}^-(y,y_0) - \partial_y\varphi_{k,\ep}^+(y,y_0) \right)(v'(y) - v'(y_0)) \d y_0 \right| \lesssim (kt)^{-1}k^{-\frac32}\cS_{k,2}
\end{align*}
Therefore,
\begin{align*}
\left| \partial_y\psi_{k,\ep}(t,y) \right| &\lesssim (kt) k^{-\frac32}\min \left( t^{-\frac32+\mu(y)}, t^{-2} + (\vartheta_2-y)t^{-1} \right)\cS_{k,2} + (kt)^{-1} k^{-\frac12}\cS_{k,2} \\
&\lesssim k^{-\frac12}\min \left( t^{-\frac12+\mu(y)}, t^{-1} + (\vartheta_2-y) \right) \cS_{k,2}
\end{align*}
and the proof is concluded.
\end{proof}

\subsection{Time-decay estimates for $\rho_k(t,y)$}
Lastly, we address the inviscid damping experienced by the density perturbation. We recall that
\begin{align}\label{eq:rhooscint}
\rho_k(t,y) = \frac{1}{2\pi i}\lim_{\ep\rightarrow 0}\int_0^2 e^{-ikv(y_0)t} \left( \rho_{k,\ep}^-(y,y_0) - \rho_{k,\ep}^+(y,y_0) \right) v'(y_0) \d y_0,
\end{align}
where we have
\begin{align*}
\rho_{k,\ep}^\pm(y,y_0) = \mathrm{P}(y)\frac{\varphi_{k,\ep}^\pm(y,y_0)}{v(y) - v(y_0) \pm i\ep}.
\end{align*}
From here, we readily see that $\rho_k(t,y)$ is compactly supported in $(\vartheta_1,\vartheta_2)$, namely
\begin{align*}
\rho_k(t,y) = 0, \quad \text{ for all }y\in [0,2]\setminus (\vartheta_1,\vartheta_2).
\end{align*}
Moreover, for all $y\in (\vartheta_1,\vartheta_2)$, we further observe that Lemma \ref{lemma:vanishingpsiIE} and the dominated convergence theorem give
\begin{align*}
\rho_k(t,y) &= \frac{1}{2\pi i}\lim_{\ep\rightarrow 0} \mathrm{P}(y)\int_{\vartheta_1}^{\vartheta_2} e^{-ikv(y_0)t} \left( \frac{\varphi_{k,\ep}^-(y,y_0)}{v(y) - v(y_0) - i\ep} - \frac{\varphi_{k,\ep}^+(y,y_0)}{v(y) - v(y_0) + i\ep} \right) v'(y_0) \d y_0.
\end{align*}
To ease notation, we set
\begin{align}\label{eq:defrhokep}
\rho_{k,\ep}(t,y) = \mathrm{P}(y)\int_{\vartheta_1}^{\vartheta_2} e^{-ikv(y_0)t} \left( \frac{\varphi_{k,\ep}^-(y,y_0)}{v(y) - v(y_0) - i\ep} - \frac{\varphi_{k,\ep}^+(y,y_0)}{v(y) - v(y_0) + i\ep} \right) v'(y_0) \d y_0,
\end{align}
so that now $\rho_k(t,y) = \frac{1}{2\pi i}\lim_{\ep\rightarrow 0} \rho_{k,\ep}(t,y)$. The inviscid damping in the weakly stratified regime reads

\begin{proposition}\label{prop:IDrhoweak}
Let $k\geq 1$ and $t\geq 1$. Then,
\begin{align*}
\left| \rho_k(t,y) \right| \lesssim k^{-\frac12}t^{-\frac12+\mu(y)}\cS_{k,1}
\end{align*}
for all $y\in (\vartheta_1,\varpi_{1,1})\cup (\varpi_{2,2}, \vartheta_2)$.
\end{proposition}

\begin{proof}
As usual, let $\delta_0=\frac{3\beta}{k}$ and $\delta_0\in (0,\frac{\delta_0}{2})$. We have
\begin{align*}
\rho_{k,\ep}(t,y) &= \mathrm{P}(y)\int_{B_\delta(y)}e^{-ikv(y_0)t}\left( \frac{\varphi_{k,\ep}^-(y,y_0)}{v(y) - v(y_0) - i\ep} - \frac{\varphi_{k,\ep}^+(y,y_0)}{v(y) - v(y_0) + i\ep} \right) v'(y_0) \d y_0 \\
&\quad -\frac{\mathrm{P}(y)}{ikt} {e^{-ikv(y_0)t}}\left( \frac{\varphi_{k,\ep}^-(y,y_0)}{v(y) - v(y_0) - i\ep} - \frac{\varphi_{k,\ep}^+(y,y_0)}{v(y) - v(y_0) + i\ep} \right) \Big|_{y_0\in \partial B_\delta(z) \cap (\vartheta_1,\vartheta_2)} \\
&\quad + \frac{\mathrm{P}(y)}{ikt}\int_{B_\delta^c(y)\cap B_{\delta_0}(y) \cap (\vartheta_1,\vartheta_2)}{e^{-ikv(y_0)t}}  \partial_{y_0} \left(  \frac{\varphi_{k,\ep}^-(y,y_0)}{v(y) - v(y_0) - i\ep} - \frac{\varphi_{k,\ep}^+(y,y_0)}{v(y) - v(y_0) + i\ep}\right) \d y_0 \\
&\quad + \frac{\mathrm{P}(y)}{ikt}\int_{B_{\delta_0}^c(y) \cap (\vartheta_1,\vartheta_2)}{e^{-ikv(y_0)t}}  \partial_{y_0} \left(  \frac{\varphi_{k,\ep}^-(y,y_0)}{v(y) - v(y_0) - i\ep} - \frac{\varphi_{k,\ep}^+(y,y_0)}{v(y) - v(y_0) + i\ep} \right) \d y_0  \\
&= \cI_{\ep,1} + \frac{1}{ikt}\left( -\cI_{\ep,2} + \cI_{\ep,3} + \cI_{\ep,4} \right).
\end{align*}
Appealing to the local decomposition and estimates of Proposition \ref{prop:Xkvarphi} we get
\begin{align*}
\left| \cI_{\ep,1} \right| \lesssim k^{-\frac12} (k\delta)^{\frac12-\mu(y)}\cS_{k,0}
\end{align*}
as we recall that $\frac{\mathrm{P}(y)}{\frac12-\mu(y)}\lesssim 1$. Likewise, Proposition \ref{prop:Xkvarphi} also gives
\begin{align*}
\left| \lim_{\ep\rightarrow 0}\cI_{\ep,2} \right| \lesssim k^{\frac12} (k\delta)^{-\frac12-\mu(y)}\cS_{k,0}
\end{align*}
because the contributions of $y_0=\vartheta_1$ and $y_0=\vartheta_2$ vanish in the limit thanks to Lemma \ref{lemma:vanishingpsiIE}. Next, the local decomposition and estimates of Proposition \ref{prop:Xkpynotvarphi} yield
\begin{align*}
\left| \cI_{\ep,3} \right| \lesssim k^{-\frac12} (k\delta)^{-\frac12-\mu(y)}\cS_{k,1}
\end{align*}
while the global pointwise bounds of the same Proposition \ref{prop:Xkpynotvarphi} provide
\begin{align*}
\left| \cI_{\ep,4} \right| \lesssim k^{\frac12} \cS_{k,1}
\end{align*}
gathering all the above estimates and optimizing for $\delta= \frac{\beta}{kt}$ we get the desired result.
\end{proof}

As before, we can in fact localize the behaviour of $\rho_{k}(t,y)$, now at the cost of a logarithmic loss.

\begin{corollary}\label{cor:improvedIDrhoweak}
Let $k\geq 1$ and $t\geq 1$. Then,
\begin{align*}
\left| \rho_k(t,y) \right| \lesssim k^{-\frac12}\min \left( t^{-\frac12+\mu(y)}, \mathrm{P}(y) (1+\log(t)) \right)\cS_{k,1}
\end{align*}
for all $y\in (\vartheta_1,\varpi_{1,1})\cup (\varpi_{2,2}, \vartheta_2)$.
\end{corollary}

\begin{proof}
We proceed as in the proof of Proposition \ref{prop:IDrhoweak}, but now we treat
\begin{align*}
\cI_{\ep,1}&= \mathrm{P}(y)\int_{B_\delta(y) \cap (\vartheta_1,\vartheta_2)}e^{-ikv(y_0)t}\left( \frac{\varphi_{k,\ep}^-(y,y_0)}{v(y) - v(y_0) - i\ep} - \frac{\varphi_{k,\ep}^+(y,y_0)}{v(y) - v(y_0) + i\ep} \right) v'(y_0) \d y_0\\
&=:\cI_{\ep,1}^- - \cI_{\ep,1}^+
\end{align*}
and
\begin{align*}
\cI_{\ep,2} &= \mathrm{P}(y)\int_{B_\delta^c(y)\cap B_{\delta_0}(y) \cap (\vartheta_1,\vartheta_2)}e^{-ikv(y_0)t}\left( \frac{\varphi_{k,\ep}^-(y,y_0)}{v(y) - v(y_0) - i\ep} - \frac{\varphi_{k,\ep}^+(y,y_0)}{v(y) - v(y_0) + i\ep} \right) v'(y_0) \d y_0 \\
&\quad +\mathrm{P}(y)\int_{B_{\delta_0}^c(y) \cap (\vartheta_1,\vartheta_2)}e^{-ikv(y_0)t}\left( \frac{\varphi_{k,\ep}^-(y,y_0)}{v(y) - v(y_0) - i\ep} - \frac{\varphi_{k,\ep}^+(y,y_0)}{v(y) - v(y_0) + i\ep} \right) v'(y_0) \d y_0 \\
&= \mathrm{P}(y) \left( \cI_{\ep,2,1} + \cI_{\ep,2,2} \right)
\end{align*}
differently. Since $\frac{v'(y_0)}{v(y) - v(y_0) \pm i\ep} = -\partial_{y_0} \log (k(v(y) - v(y_0) \pm i\ep))$, we now integrate by parts to obtain
\begin{align*}
\cI_{\ep,1}^\pm &= -\mathrm{P}(y)e^{-ikv(y_0)t} {\varphi_{k,\ep}^\pm(y,y_0)} \log (k({v(y) - v(y_0) \pm i\ep}))  \Big|_{y_0\in B_\delta(y)\cap (\vartheta_1,\vartheta_2)} \\
&\quad -ikt \mathrm{P}(y)\int_{B_\delta(y)\cap (\vartheta_1,\vartheta_2)}  e^{-ikv(y_0)t}{\varphi_{k,\ep}^\pm(y,y_0)} \log (k({v(y) - v(y_0) \pm i\ep})) v'(y_0) \d y_0 \\
&\quad +\mathrm{P}(y) \int_{B_\delta(y)\cap (\vartheta_1,\vartheta_2)}  e^{-ikv(y_0)t}\partial_{y_0}{\varphi_{k,\ep}^\pm(y,y_0)} \log (k({v(y) - v(y_0) \pm i\ep}))   \d y_0 \\
&= \mathrm{P}(y) \left( -\cI_{\ep,1,1}^\pm - ikt\cI_{\ep,1,2}^\pm + \cI_{\ep,1,3}^\pm \right).
\end{align*}
The local decomposition of Proposition \ref{prop:Xkvarphi} gives
\begin{align*}
\left| \lim_{\ep\rightarrow 0} \cI_{\ep,1,1}^\pm \right| \lesssim k^{-\frac12} (k\delta)^{\frac12-\mu(y)} \log(k\delta)\cS_{k,0}.
\end{align*}
Likewise, we also have
\begin{align*}
\left| \lim_{\ep\rightarrow 0} \cI_{\ep,1,2}^\pm \right| \lesssim k^{-\frac32} (k\delta)^{\frac32-\mu(y)} \left( 1 + \log(k\delta) \right)\cS_{k,0}
\end{align*}
while we lastly use Proposition \ref{prop:Xkpynotvarphi} to deduce that
\begin{align*}
\left| \lim_{\ep\rightarrow 0} \cI_{\ep,1,3}^\pm \right| \lesssim k^{-\frac32} (k\delta)^{\frac12-\mu(y)} \left( 1 + \log(k\delta) \right)\cS_{k,1}.
\end{align*}
We remark here that $\left( \partial_{y_0}\varphi_{k,\ep}^\pm \right)_{\s} \lesssim (1-2\mu(y))$. We next address $\cI_{\ep,2,1}$, where we use the local decomposition of Proposition \ref{prop:Xkvarphi} to write
\begin{align*}
\left| \lim_{\ep\rightarrow} \cI_{\ep,2,1} \right| \lesssim k^{\frac12}\cS_{k,0}\int_{y+\delta}^{y+\delta_0} (k|y-y_0|)^{-1} \d y_0 \lesssim k^{-\frac12}\log\left(\frac{\delta_0}{\delta} \right)\cS_{k,0}.
\end{align*}
Finally, the global estimates of Proposition \ref{prop:Xkvarphi} provide
\begin{align*}
\left| \lim_{\ep\rightarrow} \cI_{\ep,2,2} \right| \lesssim k^\frac12\cS_{k,0}.
\end{align*}
Combining together the estimates and setting $\delta=\frac{\delta_0}{4t}$ we obtain the desired result.
\end{proof}

Using the ideas of Propositions \ref{prop:IDpsistrong}, \ref{prop:IDpsimild} and proceeding as in Proposition~\ref{prop:IDrhoweak} we also have

 \begin{proposition}\label{prop:IDrhomild}
Let $k\geq 1$ and $t\geq 1$. Then,
\begin{align*}
\left| \rho_k(t,y) \right| \lesssim k^{-\frac12}t^{-\frac12+\mu(y)}(1+\log(t))\cS_{k,1}
\end{align*}
for all $y\in [\varpi_{1,1},\varpi_{1,2}]\cup [\varpi_{2,1}, \varpi_{2,2}]$.
\end{proposition}

\begin{proposition}\label{prop:IDrhostrong}
Let $k\geq 1$ and $t\geq 1$. Then,
\begin{align*}
\left| \rho_k(t,y) \right| \lesssim k^{-\frac12}t^{-\frac12}\cS_{k,1}
\end{align*}
for all $y\in (\varpi_{1,2},\varpi_{2,1})$.
\end{proposition}

\subsection{Proof of Theorem \ref{thm:mainID}}

The proof of Theorem \ref{thm:mainID} follows from Parseval identity,
\begin{align*}
\Vert v^x(t,x,y) \Vert_{L^2_x} &= \left( \sum_{k\in \Z} |v^x_k(t,y)|^2 \right)^\frac12 = \left( 2\sum_{k\geq 1} |\partial_y \psi_k(t,y)|^2 \right)^\frac12,  \\
\Vert v^y(t,x,y) \Vert_{L^2_x} &= \left( \sum_{k\in \Z} |v^y_k(t,y)|^2 \right)^\frac12 = \left( 2\sum_{k\geq 1} k^2|\psi_k(t,y)|^2 \right)^\frac12, \\
\Vert \rho(t,x,y) \Vert_{L^2_x} &= \left( \sum_{k\in \Z} |\rho_k(t,y)|^2 \right)^\frac12 = \left( 2\sum_{k\geq 1} |\rho_k(t,y)|^2 \right)^\frac12,
\end{align*}
the decay estimates from Propositions \ref{prop:IDpsistrong} - \ref{prop:IDrhoweak} 
and the observations that
\begin{align*}
\left( \sum_{k>0} k^{-3}\cS_{k,2}^2 \right) & \lesssim \Vert \omega^0 \Vert_{H^{1/2}_x H^3_y} + \Vert \varrho ^0 \Vert_{H^{1/2}_x H^4_y}, 
\end{align*}
together with
\begin{align*}
\left( \sum_{k>0} k^{-1}\cS_{k,2}^2 \right) &\lesssim\Vert \omega^0 \Vert_{H^{3/2}_x H^3_y} + \Vert \varrho^0 \Vert_{H^{3/2}_x H^4_y}, 
\end{align*}
and
\begin{align*}
\left( \sum_{k>0} k^{-1}\cS_{k,1}^2 \right) &\lesssim \Vert \omega^0 \Vert_{H^{1/2}_x H^2_y } + \Vert \varrho^0 \Vert_{H^{1/2}_x H^3_y}.
\end{align*}

\section{Growth of vorticity and density gradients}\label{sec:growth}
In this last section we address Theorem \ref{thm:growth}. 

\subsection{Growth of vorticity}
Since $\omega_k(t,y) = \Delta_k\psi_k(t,y)$, we have from \eqref{eq:psioscint} that
\begin{align}\label{eq:omegaoscint}
\omega_k(t,y) = \frac{1}{2\pi i}\lim_{\ep\rightarrow 0} \int_{\vartheta_1}^{\vartheta_2} e^{-ikv(y_0)t}\left( \Delta_k\varphi_{k,\ep}^-(y,y_0) - \Delta_k\varphi_{k,\ep}^+(y,y_0) \right) v'(y_0) \d y_0
\end{align}

\begin{proposition}\label{prop:growthomegaweak}
Let $k\geq 1$ and $y\in (\vartheta_1,\varpi_{1,1})\cup (\varpi_{2,2},\vartheta_2)$. Then,
\begin{align*}
|\omega_k(t,y)| \lesssim k^{-\frac12}t^{\frac12+\mu(y)}\cS_{k,2},
\end{align*}
for all $t\geq 1$.
\end{proposition}

\begin{proof}
We observe that 
\begin{align*}
\Delta_k\varphi_{k,\ep}^\pm(y,y_0) &= \partial_y^2\varphi_{k,\ep}^\pm - k^2\varphi_{k,\ep}^\pm \\
&= \partial_{y_0}^2\varphi_{k,\ep}^\pm + 2\partial_y\varphi_{1,k,\ep}^\pm(y,y_0) - \varphi_{2,k,\ep}^\pm - k^2\varphi_{k,\ep}^\pm.
\end{align*}
Moreover, it is an exercise to see that, thanks to the local and global pointwise bounds of Propositions \ref{prop:Xkvarphi}, \ref{prop:Xkvarphi1} and \ref{prop:Xkvarphi2}, there holds
\begin{align*}
\left| \frac{1}{2\pi i}\lim_{\ep\rightarrow 0} \int_{\vartheta_1}^{\vartheta_2} e^{-ikv(y_0)t}\left( 2\partial_y\varphi_{1,k,\ep}^\pm(y,y_0) - \varphi_{2,k,\ep}^\pm - k^2\varphi_{k,\ep}^\pm \right) v'(y_0) \d y_0 \right| \lesssim k^{-\frac12}\cS_{k,2}
\end{align*}
Hence, in view of \eqref{eq:omegaoscint}, we shall consider 
\begin{align*}
\omega_{1,k}(t,y) = \frac{1}{2\pi i}\lim_{\ep\rightarrow 0} \int_{\vartheta_1}^{\vartheta_2} e^{-ikv(y_0)t}\left( \partial_{y_0}^2\varphi_{k,\ep}^-(y,y_0) - \partial_{y_0}^2\varphi_{k,\ep}^+(y,y_0) \right) v'(y_0) \d y_0.
\end{align*}
Assume $y\in (\varpi_{2,2},\vartheta_2)$, let $\delta_0=\min\left( \frac{3\beta}{k}, \frac{\varpi_{2,2}-\varpi_2}{2}\right) $ and $\delta\in (0,\frac{\delta_0}{2})$. We shall assume that $\delta_0=\frac{3\beta}{k}$ since otherwise the proof follows with obvious modifications. Given the singular behaviour of $\partial_{y_0}^2\varphi_{k,\ep}^\pm$ described in Proposition \ref{prop:Xkpynotpynotvarphi} near the critical layer, set
\begin{align*}
\cI_{\ep,1} = \int_{B_\delta(y)\cap(\vartheta_1,\vartheta_2)} e^{-ikv(y_0)t}\left( \partial_{y_0}^2\varphi_{k,\ep}^-(y,y_0) - \partial_{y_0}^2\varphi_{k,\ep}^+(y,y_0) \right) v'(y_0) \d y_0
\end{align*}
and
\begin{align*}
\cI_{\ep,2} = \int_{B_\delta^c(y)\cap(\vartheta_1,\vartheta_2)} e^{-ikv(y_0)t}\left( \partial_{y_0}^2\varphi_{k,\ep}^-(y,y_0) - \partial_{y_0}^2\varphi_{k,\ep}^+(y,y_0) \right) v'(y_0) \d y_0.
\end{align*}
To estimate these two contributions, we shall use the ideas presented in the proof of Proposition \ref{prop:IDpsiweak}. 

\bullpar{Estimates for $\cI_{\ep,1}$} For $\cI_{\ep,1}$ we now integrate by parts the singular $\partial_{y_0}^2\varphi_{k,\ep}^\pm$ term, yielding
\begin{align*}
\cI_{\ep,1} &= e^{-ikv(y_0)t} \left( \partial_{y_0}\varphi_{k,\ep}^-(y,y_0) - \partial_{y_0}\varphi_{k,\ep}^+(y,y_0) \right) v'(y_0) \Big|_{y_0\in \partial B_\delta(y)\cap (\vartheta_1,\vartheta_2)} \\
& \quad + ikt\int_{B_\delta(y)\cap(\vartheta_1, \vartheta_2)} e^{-ikv(y_0)t} \left( \partial_{y_0}\varphi_{k,\ep}^-(y,y_0) - \partial_{y_0}\varphi_{k,\ep}^+(y,y_0) \right) (v'(y_0))^2 \d y_0 \\
&\quad - \int_{B_\delta(y)\cap(\vartheta_1, \vartheta_2)} e^{-ikv(y_0)t} \left( \partial_{y_0}\varphi_{k,\ep}^-(y,y_0) - \partial_{y_0}\varphi_{k,\ep}^+(y,y_0) \right) v''(y_0) \d y_0 \\
&= \cI_{\ep,1,1} + ikt \cI_{\ep,1,2} - \cI_{\ep,1,3}.
\end{align*}
It follows from the local estimates of Proposition \ref{prop:Xkpynotvarphi} that
\begin{align*}
\left| \cI_{\ep,1,2} \right| + \left| \cI_{\ep,1,3} \right| \lesssim k^{-\frac32}(k\delta)^{\frac12-\mu(y)}\cS_{k,1}.
\end{align*}
Proposition \ref{prop:Xkpynotvarphi} also shows that for $|y_0-y|=\delta$ with $y_0>\vartheta_2$ there holds
\begin{align*}
\left| \cI_{\ep,1,1} \right| \lesssim k^{-\frac12}(k\delta)^{-\frac12-\mu(y)} \cS_{k,1}
\end{align*}
while if $y_0\rightarrow \vartheta_2\in B_\delta(y)$ then
\begin{align*}
\left| \cI_{\ep,1,1}\right| \lesssim k^{-\frac12} \left( (k|y-\vartheta_2|)^{-\frac12+\mu(y)} + (\tfrac12 -\mu(y)) (k|y-\vartheta_2|)^{-\frac12-\mu(y)} \right) \cS_{k,1} \lesssim k^{-\frac12}\cS_{k,1}
\end{align*}
because
\begin{align*}
(k|y-\vartheta_2|)^{-\frac12+\mu(y)} = e^{(\mu(y) - \mu(\vartheta_2))\log(k|y-\vartheta_2|)} \lesssim e^{ck^{-1}} \lesssim 1
\end{align*}
for some $c>0$ bounded and
\begin{align*}
\left| (\tfrac12 -\mu(y)) (k|y-\vartheta_2|)^{-\frac12-\mu(y)} \right| = \left| (\mu(\vartheta_2) -\mu(y)) (k|y-\vartheta_2|)^{-\frac12-\mu(y)} \right| \lesssim k^{-1} (k|y-\vartheta_2|)^{\frac12-\mu(y)}\lesssim 1.
\end{align*}
Hence, we conclude that
\begin{align*}
\left| \cI_{\ep,1} \right| \lesssim k^{-\frac12}(k\delta)^{\frac12-\mu(y)} \left( t + (k\delta)^{-1} \right) \cS_{k,1}.
\end{align*}

\bullpar{Estimates for $\cI_{\ep, 2}$} We further define
\begin{align*}
\cI_{\ep,2,1} &= \int_{B_\delta^c(y)\cap B_{\delta_0}(y)\cap(\vartheta_1,\vartheta_2)} e^{-ikv(y_0)t}\left( \partial_{y_0}^2\varphi_{k,\ep}^-(y,y_0) - \partial_{y_0}^2\varphi_{k,\ep}^+(y,y_0) \right) v'(y_0) \d y_0 \\
\cI_{\ep,2,2} &= \int_{ B_{\delta_0}^c(y)\cap(\vartheta_1,\vartheta_2)} e^{-ikv(y_0)t}\left( \partial_{y_0}^2\varphi_{k,\ep}^-(y,y_0) - \partial_{y_0}^2\varphi_{k,\ep}^+(y,y_0) \right) v'(y_0) \d y_0.
\end{align*}
so that $\cI_{\ep,2} = \cI_{\ep,2,1} + \cI_{\ep,2,2}$. The global pointwise bounds of Proposition \ref{prop:Xkpynotpynotvarphi} readily give 
\begin{align*}
\left| \cI_{\ep,2,2} \right| \lesssim k^{-\frac12}\cS_{k,2}.
\end{align*}
Similarly, we use the local estimates on $\partial_{y_0}^2\varphi_{k,\ep}^\pm$ from Proposition \ref{prop:Xkpynotpynotvarphi} to integrate the singularities and find that
\begin{align*}
\left| \cI_{\ep,2,1} \right| \lesssim k^{-\frac32}(k\delta)^{-\frac12-\mu(y)}\cS_{k,2}.
\end{align*}
Therefore, we obtain $\left| \cI_{\ep,2} \right| \lesssim k^{-\frac12}(k\delta)^{-\frac12-\mu(y)} \cS_{k,2}$

\bullpar{End of proof} Gathering the estimate of $\cI_{\ep,1}$ and $\cI_{\ep,2}$, we deduce that
\begin{align*}
\left| \omega_{1,k}(t,y) \right| \lesssim k^{-\frac12}(k\delta)^{\frac12-\mu(y)}( t + (k\delta)^{-1} )\cS_{k,2}
\end{align*}
Optimizing in $\delta>0$, we find that $\delta= \frac{\beta}{kt}$ gives the desired result.
\end{proof}

As before, we can further localise the estimates near the fragile regime, losing some time-decay. 

\begin{proposition}\label{prop:improvedgrowthomegaweak}
Let $k\geq 1$ and $y\in (\varpi_{2,2},2]$. Then
\begin{align*}
\left| \omega_k(t,y) \right| \lesssim k^{\frac12}\min( 1+(\vartheta_2-y)_+ t, t^{\frac12+\mu(y)} ) \cS_{k,2}
\end{align*}
for all $t\geq 1$.
\end{proposition}

\begin{proof}
Similar to the arguments that lead to Proposition \ref{prop:improvedIDpsiweak}, we now compute
\begin{align*}
\partial_y \left( e^{ikv(y)t}\omega_{k,\ep}(t,y) \right) &= \left( ikv'(y) t \omega_{k,\ep}(t,y) + \partial_y \omega_{k,\ep}(t,y) \right) e^{ikv(y) t} \\
&=ikv'(y)t e^{ikv(y) t}\int_{\vartheta_1}^{\vartheta_2}e^{-ikv(y_0)t}\left( \Delta_k\varphi_{k,\ep}^-(y,y_0) - \Delta_k\varphi_{k,\ep}^+(y,y_0) \right) v'(y_0) \d y_0 \\
&\quad+ e^{ikv(y)t}\int_{\vartheta_1}^{\vartheta_2}e^{-ikv(y_0)t}\left( \Delta_k\varphi_{1,k,\ep}^-(y,y_0) - \Delta_k\varphi_{1,k,\ep}^+(y,y_0) \right) v'(y_0) \d y_0 \\
&\quad - e^{ikv(y)t}\int_{\vartheta_1}^{\vartheta_2}e^{-ikv(y_0)t} \partial_{y_0}\left( \Delta_k\varphi_{k,\ep}^-(y,y_0) - \Delta_k\varphi_{k,\ep}^+(y,y_0) \right) v'(y_0) \d y_0.
\end{align*}
Integrating by parts, we obtain
\begin{align*}
- \int_{\vartheta_1}^{\vartheta_2}e^{-ikv(y_0)t} &\partial_{y_0}\left( \Delta_k\varphi_{k,\ep}^-(y,y_0) - \Delta_k\varphi_{k,\ep}^+(y,y_0) \right) v'(y_0) \d y_0 \\
&= - e^{-ikv(y_0)t} \left( \Delta_k\varphi_{k,\ep}^-(y,y_0) - \Delta_k\varphi_{k,\ep}^+(y,y_0) \right) v'(y_0) \Big|_{y_0=\vartheta_1}^{y_0=\vartheta_2} \\
&\quad -ikt \int_{\vartheta_1}^{\vartheta_2}e^{-ikv(y_0)t} \left( \Delta_k\varphi_{k,\ep}^-(y,y_0) - \Delta_k\varphi_{k,\ep}^+(y,y_0) \right) (v'(y_0))^2 \d y_0 \\
&\quad + \int_{\vartheta_1}^{\vartheta_2}e^{-ikv(y_0)t} \left( \Delta_k\varphi_{k,\ep}^-(y,y_0) - \Delta_k\varphi_{k,\ep}^+(y,y_0) \right) v''(y_0) \d y_0.
\end{align*}
Moreover, due to the support assumptions on the initial data, $v''(y)$ and $\P(y)$ we further have from Lemma \ref{lemma:vanishingpsiIE} that
\begin{align*}
\lim_{\ep\rightarrow 0}\int_{y}^{\vartheta_2}e^{-ik(v(y_0)-v(y))t} \left( \Delta_k\varphi_{k,\ep}^-(y,y_0) - \Delta_k\varphi_{k,\ep}^+(y,y_0) \right) v'(y_0) \Big|_{y_0=\vartheta_1}^{y_0=\vartheta_2} \d y = 0.
\end{align*}
Hence, since $\omega_k(t,\vartheta_2)=0$, see Proposition \ref{prop:growthomeganonstrat} below, we can write
\begin{align*}
\lim_{\ep\rightarrow 0}\omega_{k}(t,y) = -\lim_{\ep\rightarrow 0} \left( ikt \omega_{1,k,\ep}(t,y) + \omega_{2,k,\ep}(t,y) + \omega_{3,k,\ep}(t,y) \right)
\end{align*}
where
\begin{align*}
\omega_{1,k,\ep}(t,y) &:=\int_y^{\vartheta_2} \int_{\vartheta_1}^{\vartheta_2}e^{-ik(v(y_0) - v(z)) t} \left( \Delta_k\varphi_{k,\ep}^-(z,y_0) - \Delta_k\varphi_{k,\ep}^+(z,y_0) \right) (v'(z) - v'(y_0))v'(y_0) \d y_0 \d z, \\
\omega_{2,k,\ep}(t,y) &:= \int_{y}^{\vartheta_2} \int_{\vartheta_1}^{\vartheta_2}e^{-ik(v(y_0)-v(z))t}\left( \Delta_k\varphi_{1,k,\ep}^-(z,y_0) - \Delta_k\varphi_{1,k,\ep}^+(z,y_0) \right) v'(y_0) \d y_0 \d z \\
\omega_{3,k,\ep}(t,y) &:= \int_y^{\vartheta_2} \int_{\vartheta_1}^{\vartheta_2}e^{-ik(v(y_0)-v(z))t} \left( \Delta_k\varphi_{k,\ep}^-(z,y_0) - \Delta_k\varphi_{k,\ep}^+(z,y_0) \right) v''(y_0) \d y_0 \d z.
\end{align*}
Appealing to Proposition \ref{prop:growthomegaweak}, it is not difficult to see that 
\begin{align*}
\left| \omega_{3,k,\ep}(t,y) \right| \lesssim k^{-\frac12}\cS_{k,2}\int_y^{\vartheta_2} t^{\frac12+\mu(z)} \d z \lesssim k^{-\frac12}(\vartheta_2-y)t \cS_{k,2},
\end{align*}
while we can argue similarly, with the help of  \eqref{eq:TGeqvarphiSobolevReg}, Proposition \ref{prop:Xkvarphi} and $\P(y)\lesssim (1-2\mu(y))$ to deduce that
\begin{align*}
\left| \omega_{1,k,\ep}(t,y) \right| \lesssim \int_y^{\vartheta_2}k^{-\frac12} \cS_{k,2} \d z \lesssim k^{-\frac12}(\vartheta_2-y)\cS_{k,2}.
\end{align*}
Now, for $\omega_{2,k,\ep}(t,y)$, let $\delta_0=\frac{3\beta}{k}$, we write
\begin{align*}
\omega_{2,k,\ep}(t,y) = \int_y^{\vartheta_2} \left( \cI_{\ep,1}(t,z) + \cI_{\ep,2}(t,z) \right) \d z
\end{align*}
with
\begin{align*}
\cI_{\ep,1}(t,z) &:= \int_{B_{\delta_0}(z)\cap (\vartheta_1,\vartheta_2)} e^{-ik(v(y_0)-v(z))t}\left( \Delta_k\varphi_{1,k,\ep}^-(z,y_0) - \Delta_k\varphi_{1,k,\ep}^+(z,y_0) \right) v'(y_0) \d y_0 \\
\cI_{\ep,2}(t,z) &:= \int_{B_{\delta_0}^c(z)\cap (\vartheta_1,\vartheta_2)} e^{-ik(v(y_0)-v(z))t}\left( \Delta_k\varphi_{1,k,\ep}^-(z,y_0) - \Delta_k\varphi_{1,k,\ep}^+(z,y_0) \right) v'(y_0) \d y_0 \\
\end{align*}
We can easily use the global bounds of Proposition \ref{prop:Xkvarphi1} to have
\begin{align*}
\left| \int_y^{\vartheta_2}\cI_{\ep,2}(t,z) \d z  \right| \lesssim k^\frac12(\vartheta_2-y)\cS_{k,2}.
\end{align*}
Instead, for $\cI_{\ep,1}$ we note that
\begin{align*}
\Delta_k\varphi_{1,k,\ep}^\pm(z,y_0) = \partial_y\varphi_{2,k,\ep}^\pm(z,y_0) - \partial_{y_0}\partial_y\varphi_{1,k,\ep}^\pm(z,y_0) - k^2\varphi_{1,k,\ep}^\pm(z,y_0)
\end{align*}
so that
\begin{align*}
\cI_{\ep,1} &= \int_{B_{\delta_0}(z)\cap (\vartheta_1,\vartheta_2)} e^{-ik(v(y_0)-v(z))t}\left( \partial_y\varphi_{2,k,\ep}^-(z,y_0) - \partial_y\varphi_{2,k,\ep}^+(z,y_0) \right) v'(y_0) \d y_0 \\ 
&\quad - \int_{B_{\delta_0}(z)\cap (\vartheta_1,\vartheta_2)} e^{-ik(v(y_0)-v(z))t}\left( \partial_{y_0}\partial_y\varphi_{1,k,\ep}^-(z,y_0) - \partial_{y_0}\partial_y\varphi_{1,k,\ep}^+(z,y_0) \right) v'(y_0) \d y_0 \\
&\quad - k^2\int_{B_{\delta_0}(z)\cap (\vartheta_1,\vartheta_2)} e^{-ik(v(y_0)-v(z))t}\left( \varphi_{1,k,\ep}^-(z,y_0) - \varphi_{1,k,\ep}^+(z,y_0) \right) v'(y_0) \d y_0 \\
&= \cI_{\ep,1,1} - \cI_{\ep,1,2} - \cI_{\ep,1,3}
\end{align*}
Using the local estimates of Propositions \ref{prop:Xkvarphi1} and \ref{prop:Xkvarphi2} we obtain
\begin{align*}
\left| \cI_{\ep,1,1} \right| + \left| \cI_{\ep,1,3} \right| \lesssim k^{-\frac12}\cS_{k,2}.
\end{align*}
For $\cI_{\ep,1,2}$ we note that $|\partial_{y_0}\partial_y\varphi_{1,k,\ep}^\pm |$ is too singular to be integrated directly. Hence, we integrate the $\partial_{y_0}$ derivative, 
\begin{align*}
\cI_{\ep,1,2} &=  e^{-ik(v(y_0)-v(z))t}\left( \partial_y\varphi_{1,k,\ep}^-(z,y_0) - \partial_y\varphi_{1,k,\ep}^+(z,y_0) \right) v'(y_0) \Big|_{y_0\in \partial B_{\delta_0}(z)\cap (\vartheta_1,\vartheta_2)} \\
&\quad + ikt \int_{B_{\delta_0}(z)\cap (\vartheta_1,\vartheta_2)} e^{-ik(v(y_0)-v(z))t}\left( \partial_y\varphi_{1,k,\ep}^-(z,y_0) - \partial_y\varphi_{1,k,\ep}^+(z,y_0) \right) (v'(y_0))^2 \d y_0 \\
&\quad - \int_{B_{\delta_0}(z)\cap (\vartheta_1,\vartheta_2)} e^{-ik(v(y_0)-v(z))t}\left( \partial_y\varphi_{1,k,\ep}^-(z,y_0) - \partial_y\varphi_{1,k,\ep}^+(z,y_0) \right) v''(y_0) \d y_0 \\
&= \cI_{\ep,1,2,1} + ikt \cI_{\ep,1,2,2} - \cI_{\ep,1,2,3}
\end{align*}
According to the local estimates of Proposition \ref{prop:Xkvarphi1} we have
\begin{align*}
\left| \cI_{\ep,1,2,1} \right|  \lesssim k^\frac12\cS_{k,1}, \quad \sum_{j=2}^3\left| \cI_{\ep,1,2,j} \right| \lesssim k^{-\frac12}\cS_{k,1}
\end{align*}
from which we conclude that $\left| \cI_{\ep,1,2} \right| \lesssim k^{-\frac12}t \cS_{k,2}$, for all $t\geq 1$. Therefore, $\left| \cI_{\ep,1} \right| \lesssim k^{-\frac12} t \cS_{k,2}$ and thus
\begin{align*}
\left| \int_y^{\vartheta_2} \cI_{\ep,1}(t,z) \d z \right| \lesssim k^{-\frac12}(\vartheta_2-y) t \cS_{k,2}.
\end{align*}
With this, the proposition follows.
\end{proof}

Following the ideas of Proposition \ref{prop:growthomegaweak}, it is an exercise to prove

\begin{proposition}\label{prop:growthomegamild}
Let $k\geq 1$ and $y\in [\varpi_{1,1},\varpi_{1,2}]\cup [\varpi_{2,1},\varpi_{2,2}]$. Then,
\begin{align*}
|\omega_k(t,y)| \lesssim k^{-\frac12}t^{\frac12+\mu(y)}(1+\log(t))\cS_{k,2},
\end{align*}
for all $t\geq 1$.
\end{proposition}
We also have
\begin{proposition}\label{prop:growthomegastrong}
Let $k\geq 1$ and $y\in (\varpi_{1,2}, \varpi_{2,1})$. Then,
\begin{align*}
|\omega_k(t,y)| \lesssim k^{-\frac12}t^{\frac12}\cS_{k,2},
\end{align*}
for all $t\geq 1$.
\end{proposition}

Moreover, due to the compact support of the initial data and $v''(y)$ and $\P(y)$, we have from \eqref{eq:omegaoscint} and \eqref{eq:TGeqvarphiSobolevReg} the following.
\begin{proposition}\label{prop:growthomeganonstrat}
Let $k\geq 1$. Then, $\omega_k(t,y) =0$, for all $y\in [0,\vartheta_1] \cup [\vartheta_2,2]$ and $t\geq 1$.
\end{proposition}

\subsection{Growth of density gradients}
We now show how to obtain the growth bounds on the density gradient. We shall prove them for $\partial_y\rho_k(t,y)$. We now have

\begin{proposition}\label{prop:growthpyrhoweak}
Let $k\geq 1$ and $y\in (\vartheta_0,\varpi_{1,1})\cup (\varpi_{2,2},\vartheta_2)$. Then,
\begin{align*}
|\partial_y\rho_k(t,y)| \lesssim k^{-\frac12}\left( \mathrm{P}'(y)(1+\log(t)) + t\min(t^{-\frac12+\mu(y)}, \mathrm{P}(y)(1+\log(t)) \right)\cS_{k,2}
\end{align*}
for all $t\geq 1$.
\end{proposition}

\begin{proof}
In view of \eqref{eq:rhooscint} and \eqref{eq:defrhokep} we have
\begin{align*}
\partial_y\rho_{k,\ep}(t,y) &= \frac{\mathrm{P'}(y)}{\mathrm{P}(y)}\rho_{k,\ep}(t,y) \\
&\quad + \mathrm{P}(y)\int_{\vartheta_1}^{\vartheta_2} e^{-ikv(y_0)t} \left( \partial_y + \partial_{y_0} \right) \left( \frac{\varphi_{k,\ep}^-(y,y_0)}{v(y) - v(y_0) - i\ep} - \frac{\varphi_{k,\ep}^+(y,y_0)}{v(y) - v(y_0) + i\ep} \right) v'(y_0) \d y_0 \\
&\quad - \mathrm{P}(y)\int_{\vartheta_1}^{\vartheta_2} e^{-ikv(y_0)t}  \partial_{y_0} \left( \frac{\varphi_{k,\ep}^-(y,y_0)}{v(y) - v(y_0) - i\ep} - \frac{\varphi_{k,\ep}^+(y,y_0)}{v(y) - v(y_0) + i\ep} \right) v'(y_0) \d y_0 \\
&= \frac{\mathrm{P'}(y)}{\mathrm{P}(y)}\rho_{k,\ep}(t,y) + \rho_{1,k,\ep}(t,y) - \rho_{2,k,\ep}(t,y).
\end{align*}
It is immediate from Corollary \ref{cor:improvedIDrhoweak} that
\begin{align*}
\left|  \frac{\mathrm{P'}(y)}{\mathrm{P}(y)}\rho_{k,\ep}(t,y) \right| &\lesssim  \frac{\mathrm{P'}(y)}{\mathrm{P}(y)}k^{-\frac12}\min( t^{-\frac12+\mu(y)} , \mathrm{P}(y)(1+\log(t)) )\cS_{k,1} \\
&\lesssim \mathrm{P}'(y)k^{-\frac12}(1+\log(t)) \cS_{k,1}.
\end{align*}
We treat $\rho_{1,k,\ep}(t,y)$ and $\rho_{2,k,\ep}(t,y)$ separately.

\bullpar{Estimates for $\rho_{1,k,\ep}(t,y)$} We crucially observe that
\begin{align*}
\left( \partial_y + \partial_{y_0} \right) &\left( \frac{\varphi_{k,\ep}^-(y,y_0)}{v(y) - v(y_0) - i\ep} - \frac{\varphi_{k,\ep}^+(y,y_0)}{v(y) - v(y_0) + i\ep} \right) \\
&= \frac{\varphi_{1,k,\ep}^-(y,y_0)}{v(y) - v(y_0) - i\ep} - \frac{\varphi_{1,k,\ep}^+(y,y_0)}{v(y) - v(y_0) + i\ep}  \\
&\quad - \left( \frac{\varphi_{k,\ep}^-(y,y_0)}{(v(y) - v(y_0) - i\ep)^2} - \frac{\varphi_{k,\ep}^+(y,y_0)}{(v(y) - v(y_0) + i\ep)^2} \right)(v'(y) - v'(y_0)).
\end{align*}
Hence, we can argue as in the proof of Corollary \ref{cor:improvedIDrhoweak} to deduce that
\begin{align*}
\left| \rho_{1,k,\ep}(t,y) \right| \lesssim k^{-\frac12}\min (t^{-\frac12+\mu(y)}, \mathrm{P}(y)(1+\log(t)))\cS_{k,2}
\end{align*}
for all $t\geq 1$. We remark here that in the integration by parts argument of Corollary \ref{cor:improvedIDrhoweak} performed to keep $\mathrm{P}(y)$, we shall now obtain a term of the form
\begin{align*}
v''(y_0)\frac{\varphi_{k,\ep}^\pm(y,y_0)}{v(y) - v(y_0) \pm i\ep}\log(v(y) - v(y_0) \pm i\ep) = \frac{v''(y_0)}{2v'(y_0)}\varphi_{k,\ep}^\pm(y,y_0) \partial_{y_0} \log^2 (v(y) - v(y_0) \pm i\ep)
\end{align*}
which can be further integrated by parts to obtain the desired regularity, which in turn is translated to the appropriate time decay, we omit the details.

\bullpar{Estimates for $\rho_{2,k,\ep}(t,y)$} We now integrate by parts the $\partial_{y_0}$ derivative to obtain
\begin{align*}
\rho_{2,k,\ep}(t,y) &=  \mathrm{P}(y)e^{-ikv(y_0)t} \left( \frac{\varphi_{k,\ep}^-(y,y_0)}{v(y) - v(y_0) - i\ep} - \frac{\varphi_{k,\ep}^+(y,y_0)}{v(y) - v(y_0) + i\ep} \right) v'(y_0) \Big|_{y_0=\vartheta_1}^{y_0=\vartheta_2} \\
&\quad + ikt \mathrm{P}(y)\int_{\vartheta_1}^{\vartheta_2} e^{-ikv(y_0)t} \left( \frac{v'(y_0)\varphi_{k,\ep}^-(y,y_0)}{v(y) - v(y_0) - i\ep} - \frac{v'(y_0)\varphi_{k,\ep}^+(y,y_0)}{v(y) - v(y_0) + i\ep} \right) v'(y_0) \d y_0 \\
&\quad - \mathrm{P}(y)\int_{\vartheta_1}^{\vartheta_2} e^{-ikv(y_0)t} \left( \frac{\varphi_{k,\ep}^-(y,y_0)}{v(y) - v(y_0) - i\ep} - \frac{\varphi_{k,\ep}^+(y,y_0)}{v(y) - v(y_0) + i\ep} \right) v''(y_0) \d y_0 \\
&= \rho_{3,k,\ep}(t,y) + ikt \rho_{4,k,\ep}(t,y) - \rho_{5,k,\ep}(t,y).
\end{align*}
For $y\in (\vartheta_1,\vartheta_2)$ it is immediate from Lemma \ref{lemma:vanishingpsiIE} that
\begin{align*}
\lim_{\ep\rightarrow 0} \rho_{3,k,\ep}(t,y) = 0.
\end{align*}
Moreover, we can directly use the ideas of Corollary \ref{cor:improvedIDrhoweak} to deduce that
\begin{align*}
\left| \rho_{4,k,\ep}(t,y) \right| + \left| \rho_{5,k,\ep}(t,y) \right| \lesssim k^{-\frac12}\min ( t^{-\frac12+\mu(y)} , \mathrm{P}(y)(1+\log(t)) ) \cS_{k,1}
\end{align*}
for all $t\geq 1$.

\bullpar{End of proof} Combining the estimates for $\rho_{k,\ep}(t,y)$, with those for $\rho_{1,k,\ep}(t,y)$ and $\rho_{2,k,\ep}(t,y)$ we obtain the claimed result.
\end{proof}

For the other regimes we do not need to consider the localised bounds, so that the ideas of Proposition \ref{prop:growthpyrhoweak} provide the next two results.

\begin{proposition}\label{prop:growthpyrhomild}
Let $k\geq 1$ and $y\in [\varpi_{1,1} , \varpi_{1,2}] \cup  [\varpi_{2,1}, \varpi_{2,2}]$. Then,
\begin{align*}
|\partial_y\rho_k(t,y)| \lesssim k^{-\frac12} t^{\frac12+\mu(y)}(1+\log(t))\cS_{k,2}
\end{align*}
for all $t\geq 1$.
\end{proposition}
For the strong regime, we have
\begin{proposition}\label{prop:growthpyrhostrong}
Let $k\geq 1$ and $y\in (\varpi_{1,2}, \varpi_{2,1})$. Then,
\begin{align*}
|\partial_y\rho_k(t,y)| \lesssim k^{-\frac12} t^{\frac12}\cS_{k,2}
\end{align*}
for all $t\geq 1$.
\end{proposition}

\subsection{Proof of Theorem \ref{thm:growth}}
Just as for Theorem \ref{thm:mainID}, the growth estimates from Theorem \ref{thm:growth} and Theorem \ref{thm:growthlocalweak} are a consequence of Propositions \ref{prop:growthomegaweak} - \ref{prop:growthpyrhostrong}, Parseval identity, and the observations that
\begin{align*}
\left( \sum_{k>0} k^{-1}\cS_{k,2}^2 \right) &\lesssim\Vert \omega^0 \Vert_{H^{3/2}_x H^3_y} + \Vert \widetilde q^0 \Vert_{H^{3/2}_x H^4_y}, 
\end{align*}
and
\begin{align*}
\left( \sum_{k>0} k\cS_{k,2}^2 \right) &\lesssim\Vert \omega^0 \Vert_{H^{5/2}_x H^3_y} + \Vert \widetilde q^0 \Vert_{H^{5/2}_x H^4_y}, 
\end{align*}
we omit the details.

\begin{remark}\label{rmk:growthlineq}
As we commented in Section \ref{sec:sketch}, we can use \eqref{eq:linEBomegarho} and Propositions \ref{prop:IDpsistrong}-\ref{prop:IDpsinonstrat} and Propositions \ref{prop:IDrhoweak}-\ref{prop:IDrhostrong} to prove the growth bounds of Theorem \ref{thm:growth}. To do that, we note that 
\begin{align*}
e^{ikv(y)t}\rho_k(t,y) = \rho_k^0(y) + ik\mathrm{P}(y)\int_{0}^t\psi_k(s,y) \d s
\end{align*}
for all $t\geq 0$. Hence,
\begin{align*}
e^{ikv(y)t}\partial_y\rho_k(t,y) &= -ikv(y)t\rho_k(t,y) + \partial_y\rho_k^0(y) \\
&\quad + ik\mathrm{P}(y)\int_0^te^{ikv(y)s} \partial_y\psi_k(s,y) \d s + ik\mathrm{P}'(y)\int_0^t e^{ikv(y)s} \psi_k(s,y) \d s
\end{align*}
from which the growth rate for $\partial_y\rho_k(t,y)$ follows from the decay estimates on $\rho_k(t,y)$, $\psi_k(t,y)$ and $\partial_y\psi_k(t,y)$. On the other hand, from the vorticity equation we note that
\begin{align*}
e^{ikv(y)t}\omega_k(t,y) = \omega_k^0(y) +ikv''(y) \int_{0}^t e^{ikv(y)s}\psi_k(s,y) \d s - ik\g \int_0^t e^{ikv(y)s}\rho_k(s,y) \d s
\end{align*}
and thus the growth bounds for $\omega_k(t,y)$ are obtained once we integrate the decay estimates for $\psi_k(t,y)$ and $\rho_k(t,y)$.
\end{remark}

\section*{Acknowledgements}

We would like to thank Augusto Del Zotto for insightful discussions on the problem. This work has received funding from the European Research Council (ERC) under the European Union's Horizon 2020 research and innovation programme through the grant agreement~862342. It is also partially supported by the MCIN/AEI grants CEX2023-001347-S, RED2022-134301-T and PID2022-136795NB-I00.

\appendix

\section{Logarithmic approximations}
In this section we state and prove two simple identities involving limiting regimes that give rise to logarithmic corrections.

\begin{lemma}\label{lemma:firstgammalog}
Let $\gamma\in \C$, $\Re(\gamma)\geq 0$ and $|\gamma|\leq 1$. Let $\zeta\in \C$. Then,
\begin{align*}
\frac{\zeta^{2\gamma} - 1}{2\gamma} = \log(\zeta) \Q_{\gamma}(\zeta), \quad \Q_{\gamma}(\zeta) := \int_0^1 e^{2\gamma s \log(\zeta)} \d s
\end{align*}
and we further have that $\sup_{|\gamma|\leq 1} \Vert \Q_{\gamma} \Vert_{L^\infty(B_R(0))}\lesssim_R 1$.
\end{lemma}

\begin{proof}
The fundamental theorem of calculus shows that
\begin{align*}
\frac{\zeta^{2\gamma} - 1}{2\gamma} = \frac{1}{2\gamma}\int_0^1 \frac{\d}{\d s}e^{2\gamma s \log(\zeta)} \d s= \log(\zeta) \int_0^1 e^{2\gamma s \log(\zeta)} \d s.
\end{align*}
Since $\log(\zeta) = \log(|\zeta|) + i \text{Arg}(\zeta)$, the  bound follows from the fact that $\Re(2\gamma s \log(\zeta))\lesssim 1 + \log R$. 
\end{proof}

\begin{lemma}\label{lemma:secondgammalog}
Let $\gamma\in \C$, $\Re(\gamma)\geq 0$ and $|\gamma|\leq \tfrac{1}{16}$. Let $\zeta\in \C$. Then,
\begin{align*}
\frac{\zeta^{2\gamma} - \zeta^{-2\gamma} }{2\gamma} = 2 \log(\zeta) + \mathrm{Q}_{\gamma}(\zeta)
\end{align*}
and we further have $\sup_{|\gamma|\leq 1} \Vert \gamma^{-2} \zeta^\frac12 \mathrm{Q}_\gamma (\zeta) \Vert_{L^\infty(B_R(0))}\lesssim 1$.
\end{lemma}

\begin{proof}
We note that 
\begin{align*}
\mathrm{Q}_{\gamma}(\zeta) &= \frac{\zeta^{2\gamma} - \zeta^{-2\gamma} }{2\gamma} - 2 \log(\zeta) \\
&= 2\log(\zeta) \int_0^1  e^{2\gamma(2u-1)\log(\zeta)}\d u - 2\log(\zeta) \\
&= 2\log(\zeta)\int_0^1 \left( e^{2\gamma(2u-1)\log(\zeta)} -1\right) \d u \\
&= 2\log(\zeta)\int_0^1 2\gamma(2u-1)\log(\zeta) \int_0^1 e^{2\gamma(2u-1) r \log(\zeta)} \d r \d u \\
&= 8\gamma^2 \log^3(\zeta) \int_0^1(2u-1)^2\int_0^1 r \int_0^1 e^{2\gamma(2u-1)rs\log(\zeta)}\d s \d r \d u
\end{align*}
because $\int_0^1(2u-1)\d u =0$. Hence, we see that 
\begin{align*}
\left| \frac{\zeta^\frac12}{4\gamma^2} \mathrm{Q}_{\gamma}(\zeta) \right| \lesssim |\zeta|^\frac14|\log(\zeta)|^3\lesssim 1
\end{align*}
since $\Re(\tfrac14-2\gamma r s)>\tfrac18$, for all $r,s\in(0,1)$ and all $|\gamma|\leq \tfrac{1}{16}$.  In particular, $\left| \frac{\mathrm{Q}_{\gamma}(\zeta)}{4\gamma^2}  \right|$ is uniformly bounded whenever $|\zeta|$ is uniformly bounded from below. 
\end{proof}

\section{Properties of the Whittaker functions}\label{app:Whittaker}
Here we state and prove some properties of the Whittaker functions that are used throughout the paper, we refer to \cite{NIST} for a complete description of the Whittaker functions. For $\gamma, \zeta\in \C$, the Whittaker function $M_{0,\gamma}(\zeta)$ is given by
\begin{equation*}
M_{0,\gamma}(\zeta)=e^{-\frac12\zeta}\zeta^{\frac12 + \gamma}M\l( \tfrac12 + \gamma, 1+2\gamma, \zeta\r), \quad M(a,b,\zeta) = \sum_{s=0}^\infty \frac{(a)_s}{(b)_s s!}\zeta^s,
\end{equation*}
where $(a)_s=a(a+1)(a+2)\dots (a+s-1)$ denotes the Pochhammer symbol. For $\gamma\neq 0$, we also introduce the Whittaker function 
\begin{equation*}
W_{0,\gamma}(\zeta) = \sqrt{\frac{\zeta}{\pi}}K_{\gamma}(\zeta/2)
\end{equation*}
where $K_\gamma$ denotes the second modified Bessel function. It is such that
\begin{align*}
K_\gamma(\zeta) = \frac{\pi}{2}\frac{I_{-\gamma}(\zeta) - I_{\gamma}(\zeta)}{\sin(\gamma\pi)}
\end{align*}
where $I_\gamma$ stands for the first modified Bessel function and it is given by the series representation
\begin{align*}
I_\gamma(\zeta) = \left( \frac{\zeta}{2}\right)^\gamma \sum_{n=0}^\infty \frac{(\zeta/2)^n}{\Gamma (\gamma + n + 1)}.
\end{align*}
For $\gamma=0$, we instead define
\begin{equation*}
W_{0,0}(\zeta) :=e^{-\frac12\zeta}\sqrt{\frac{\zeta}{\pi}}\sum_{s=0}^\infty \frac{\l( \tfrac12 \r)_s}{(s!)^2}\zeta^s \l( 2\frac{\Gamma'(1+s)}{\Gamma(1+s)} - \frac{\Gamma'(\tfrac12+s)}{\Gamma(\tfrac12+s)} - \log(\zeta)\r),
\end{equation*}
where $\Gamma(x)$ denotes the Gamma function. The Whittaker functions $M_{0,\gamma}$ and $W_{0,\gamma}$ are solutions to the Taylor-Goldstein equation
\begin{align*}
\partial_\zeta^2\phi(\zeta) + \left( - \frac{1}{4} + \frac{\frac14 - \gamma^2}{\zeta^2} \right)\phi(\zeta)=0.
\end{align*}
For $\gamma\in \C$ we next set $\gamma=\mu+i\nu$. We begin by recording some basic properties regarding complex conjugation for $M_{0,\gamma}(\zeta)$, which can be deduce from the series definition of $M_{0,\gamma}$.
\begin{lemma}
We have the following
\begin{itemize}
\item For $\b^2>1/4$, then $M_{0,i\nu}(\zeta)=\overline{M_{0,-i\nu}\l(\overline{\zeta}\r)}$.
\item For $\b^2\leq 1/4$, then $M_{0,\mu}(\zeta) = \overline{M_{0,\mu}\l(\overline{\zeta}\r)}$. Additionally, for $x\in\R$ then $M_{0,\mu}(x),W_{0,0}(x)\in \R$. 
\end{itemize}
\end{lemma}

We next state an analytic continuation property, which is key in studying the Wronskian of the Green's function and is directly determined by the analytic continuation of the non-entire terms of $M_{0,\gamma}(\zeta)$ and $W_{0,\gamma}$. 

\begin{lemma}[\cite{NIST}]\label{lemma:analyticcontinuation}
Let $\b^2>0$. Then
\begin{equation*}
M_{0,\gamma}(\zeta \e^{\pm \pi i}) = \pm i \e^{\pm \gamma \pi i}M_{0,\gamma}(\zeta), \quad W_{0,\gamma}(\zeta e^{\pm i\pi})=\frac{\Gamma(\tfrac12 + \gamma)}{\Gamma(1+2\gamma)} M_{0,\gamma}(\zeta) \pm i e^{\mp \gamma\pi i}W_{0,\gamma}(\zeta)
\end{equation*}
for all $\zeta\in\C$.
\end{lemma}

The next result gives a precise description of the asymptotic expansion of $M_\pm(\zeta)$ and its derivatives, for $\zeta$ in a bounded domain. 
\begin{lemma}\label{lemma:asymptoticexpansionM}
Let $\gamma=\mu + i\nu$ with $0\leq \mu < \frac12$ and $0\leq \nu$. Let $\zeta\in\C$. Let $B_R\subset\C$ denote the closed unit ball of radius $R>0$ centered in the origin. Then,
\begin{equation*}
\begin{aligned}
M_{0,\pm\gamma}(\zeta)&=\zeta^{\frac12 \pm\gamma}\mathcal{E}_{0,\pm\gamma}(\zeta),
\end{aligned}
\end{equation*}
where $\mathcal{E}_{0,\pm\gamma}\in C^\infty(B_R)$, with $\E_{0,\pm\gamma}(0) =1$,  $\E_{0,\pm\gamma}'(0)=0$ and $\Vert \mathcal{E}_{0,\pm\gamma}\Vert_{C^2(B_R)} \lesssim_{\gamma,R} 1$.
\end{lemma}

\begin{proof}
From \cite{NIST} we know that
\begin{align*}
M_{0,\pm \gamma}(\zeta) = 2^{\frac12\pm 2\gamma}\Gamma(1\pm \gamma) \sqrt{\tfrac{\zeta}{2}} I_{\pm \gamma} (\tfrac{\zeta}{2}) &= \zeta^{\frac12\pm \gamma} \sum_{j\geq 0} \frac{\Gamma(1\pm \gamma)}{j! \Gamma(j+1 \pm \gamma)}\left( \tfrac{\zeta}{4}\right)^{2j} \\
&= \zeta^{\frac12\pm \gamma} \E_{0,\pm \gamma}(\zeta)
\end{align*}
where $\E_{0,\pm\gamma}$ is an entire function in $\C$ with $\E_{0,\pm\gamma}(0) = 1$ and $\E_{0,\pm\gamma}'(0)=0$. In particular,  $\Vert \mathcal{E}_{0,\pm\gamma}\Vert_{C^2(B_R)} \lesssim_{R} C_\gamma$, with $C_\gamma>0$ uniformly bounded for $|\gamma|$ bounded and $|\mu|\leq\frac12$. 
\end{proof}

\begin{lemma}\label{lemma:asymptoticexpansionW}
Let $\zeta\in\C$. Let $B_R\subset\C$ denote the closed unit ball of radius $R>0$ centered in the origin. Then,
\begin{align*}
W_{0,\gamma}(\zeta) &= \zeta^{\frac12+\gamma} \E_{1,\gamma}(\zeta) - \zeta^{\frac12-\gamma}\log (\zeta)\mathcal{Q}_\gamma(\zeta) \E_{2,\gamma}(\zeta),
\end{align*}
where we recall
\begin{align*}
\mathcal{Q}_\gamma(\zeta) = \int_0^1 e^{2\gamma s \log(\zeta)} \d s
\end{align*}
and $\E_{j,\gamma}(\zeta)$ are entire functions in $\C$, $\E_{j,\gamma}'(0)=0$ and $\Vert \E_{j,\gamma}\Vert_{C^2(B_R)}\lesssim 1$, for $j=1,2$ uniformly for $|\gamma|\leq\tfrac14$. 
\end{lemma}

\begin{proof}
We begin by noting that 
\begin{align*}
W_{0,\gamma}(\zeta) &= \sqrt{\frac{\zeta}{\pi}}K_{\gamma}(\tfrac{\zeta}{2}) \\
&= -\frac12 \sqrt{\pi\zeta}\frac{I_\gamma(\tfrac{\zeta}{2})- I_{-\gamma}(\tfrac{\zeta}{2})}{\sin(\gamma\pi)} \\
&= -\frac12 \frac{\zeta^\frac12}{\sin(\gamma\pi)} \left( \frac{\zeta^\gamma}{4^\gamma \Gamma\left( \frac12 + \gamma \right)} \int_0^\pi e^{\frac{\zeta}{2} \cos\theta } (\sin \theta)^{2\gamma} \d \theta -  \frac{\zeta^{-\gamma}}{4^{-\gamma} \Gamma\left( \frac12 - \gamma \right)} \int_0^\pi e^{\frac{\zeta}{2} \cos\theta } (\sin \theta)^{-2\gamma} \d \theta \right).
\end{align*}
For 
\begin{align*}
f_1(\gamma)&=\frac{1}{4^{\gamma}\Gamma(\tfrac12+\gamma)}, \quad f_2(\gamma,\zeta) =\int_0^\pi e^{\frac{\zeta}{2} \cos\theta } (\sin \theta)^{2\gamma} \d \theta,  \quad f_3(\gamma) = \zeta^\gamma,
\end{align*}
together with
\begin{align*}
f_1(\gamma)f_2(\gamma)f_3(\gamma) -  f_1(-\gamma) f_2(-\gamma)f_3(-\gamma) &= \left( f_1(\gamma) - f_1(-\gamma) \right) f_2(\gamma) f_3(\gamma) \\
&\quad + f_1(-\gamma) \left( f_2(\gamma) - f_2(-\gamma) \right) f_3(\gamma) \\
&\quad + f_1(-\gamma)f_2(-\gamma) \left( f_3(\gamma) - f_3(-\gamma) \right).
\end{align*}
we now have
\begin{align*}
W_{0,\gamma}(\zeta) = \zeta^{\frac12+\gamma} \E_{1,\gamma}(\zeta) + \zeta^{-\frac12-\gamma}\log(\zeta)\Q_\gamma(\zeta) \E_{2,\gamma}(\zeta)
\end{align*}
where we define
\begin{align*}
\E_{1,\gamma}(\zeta) &:= -\frac12 \left( f_2(\gamma,\zeta) \frac{f_1(\gamma) - f_1(-\gamma)}{\sin(\gamma\pi)} + f_1(-\gamma) \frac{f_2(\gamma,\zeta) - f_2(-\gamma,\zeta)}{\sin(\gamma\pi)} \right), \\
\E_{2,\gamma}(\zeta) &:= -\frac{\gamma}{\sin(\gamma\pi)} f_1(-\gamma)f_2(-\gamma,\zeta).
\end{align*}
Firstly, we note that $\partial_\zeta f_2(\gamma,0)=0$ for all $0<|\gamma|\leq\tfrac14$. Hence, $\E_{j,\gamma}'(0) = 0$, for all $0<|\gamma|\leq\tfrac14$. Secondly, both $f_1(\gamma)$ and $\partial_\zeta^n f_2(\gamma,\zeta)$ are uniformly bounded for $|\gamma|\leq \tfrac14$ and $|\zeta|\leq R$, for $n=0,1,2$. With this, we see that $\Vert \E_{2,\gamma}\Vert_{C^2(B_R)}\lesssim 1$ uniformly in $|\gamma|\leq\tfrac14$.

On the other hand,  there holds $f_1(\gamma) - f_1(-\gamma) = \gamma f_{1,\infty}(\gamma)$, with $\Vert f_{1,\infty}\Vert_{L^\infty(B_\frac14(0))}\lesssim 1$. Similarly,
\begin{align*}
f_2(\gamma,\zeta) - f_2(-\gamma,\zeta) = 4\gamma\int_0^\pi e^{\frac{\zeta}{2}\cos\theta} (\sin\theta)^{-2\gamma}  \log(\sin\theta) Q_{2\gamma}(\sin\theta) \d \theta
\end{align*}
Since $|\sin(\theta)|\leq 1$ and $x^{-\frac12}\log x$ is integrable, we see that $|\Q_{2\gamma}(\sin\theta)|\lesssim1$ and
\begin{align*}
\Vert f_2(\gamma,\cdot) - f_2(-\gamma,\cdot)\Vert_{C^2(B_R)}\lesssim \gamma.
\end{align*}
for all $|\gamma|\leq \tfrac14$. With this, we conclude that $\Vert \E_{1,\gamma}\Vert_{C^2(B_R)}\lesssim 1$ uniformly in $|\gamma|\leq\tfrac14$ and the proof is complete.
\end{proof}

In the future, for $f$ smooth and compactly supported, we shall need
\begin{align*}
\int \frac{f(z)}{z}\log(z) \mathcal{Q}_\gamma(z) \d z &= \int \log(z) f(z) \int_0^1 \partial_z \left( \frac{e^{2\gamma u \log(z)} - 1}{2\gamma u} \right) \d u \d z \\
&= - \int f(z) \int_0^1 \frac12\partial_z \left(  \frac{e^{2\gamma u \log(z)} -1 }{2\gamma u} \right)^2 \d u \d z \\
&\quad - \int f'(z) \log^2(z) \int_0^1 \int_0^1 e^{2\gamma u r \log(z)} \d r \d u \d z \\
&= \frac12\int f'(z) \log^2(z)\int_0^1 \left( \int_0^1 e^{2\gamma u r \log(z)} \d r \right)^2 \d u \d z
\\
&\quad - \int f'(z) \log^2(z) \int_0^1 \int_0^1 e^{2\gamma u r \log(z)} \d r \d u \d z
\end{align*}
which is uniformly bounded in $\gamma$ for $|\gamma|\leq \frac14$, say, by $\Vert f \Vert_{W^{1,\infty}}$.

\section{The Green's function for the mild stratified regime}
Here we prove the main estimates for the Green's function $\G_{k,\ep}^\pm(y,y_0,z)$ for $y_0\in I_M$.
Firstly, we note that for $\gamma\in \C$, 
\begin{align*}
\Gamma_\gamma := \frac{\Gamma(1+2\gamma)}{\Gamma(1/2+\gamma)} = \frac{1}{\sqrt\pi} + O(|\gamma|)
\end{align*}
as $\gamma\rightarrow 0$. We recall that $\gamma_0=\gamma(y_0)$ and 
\begin{equation*}
\W_{k,\ep}^\pm(y_0):= -2k\Gamma_{\gamma_0} \Big( M_{\sr}( - y_0\pm i\ep_0)W_{\sr}(2- y_0\pm i\ep_0) - W_{\sr}( - y_0\pm i\ep_0)M_{\sr}(2-y_0\pm i\ep_0) \Big).
\end{equation*} 
Using the analytic continuation properties of $M_{0,\gamma}$ and $W_{0,\gamma}$, we get
\begin{align*}
\W_{k,\ep}^+(y_0) &= 2k M_{\sr}(2-y_0+i\ep_0) M_{\sr}(y_0-i\ep_0)\\
&\quad-2ki \Gamma_{\gamma_0}\Big( e^{{\gamma_0}\pi i} M_{\sr}( y_0- i\ep_0)W_{\sr}(2-y_0+ i\ep_0) - e^{-{\gamma_0}\pi i} M_{\sr}(  2 -y_0 - i\ep_0)W_{\sr}(y_0- i\ep_0) \Big) 
\end{align*}

\begin{lemma}\label{lemma:Wronskianlowerboundmildstrat}
Let $y_0\in I_S$. There exists $C>0$, $\ep_*>0$ and $\gamma_*>0$ such that 
\begin{align*}
|\W_{k,\ep}^+(y_0)| \geq Ck  |M_{\sr}(2-y_0+i\ep_0)||M_{\sr}(y_0-i\ep_0)|
\end{align*}
for all $0<\ep<\ep_*$ and all $0<|{\gamma_0}|<\gamma_*$.
\end{lemma}

\begin{proof}
It is useful to write 
\begin{align*}
\W_{k,\ep}^+(y_0) &= 2k M_{\sr}(2-y_0+i\ep_0)M_{\sr}(y_0-i\ep_0) \\
&\qquad\times\left( 1 + i\Gamma_{\gamma_0} \left( e^{-{\gamma_0}\pi i}\frac{W_{\sr}(y_0-i\ep_0)}{M_{\sr}(y_0-i\ep_0)} - e^{{\gamma_0} \pi i}\frac{W_{\sr}(2-y_0+i\ep_0)}{M_{\sr}(2-y_0+i\ep_0)} \right) \right).
\end{align*}
We treat two cases according to the size of the wave-number $k$. Let $N_0>0$ be given by Lemma \ref{lemma:smallnulargearg}. 

\bullpar{Case 1} Assume $k< \vartheta^{-1}\frac{N}{2}$. Since $y_0$ and $2-y_0$ are bounded away from zero for all $y_0\in I_M$, we observe that for ${\gamma_0} = i\nu_0\in i \R_+$ there holds
\begin{align*}
&\left| 1 + i\Gamma_{\nu_0} \left( e^{{\nu_0}\pi}\frac{W_{\sr}(y_0-i\ep_0)}{M_{\sr}(y_0-i\ep_0)} - e^{-{\nu_0} \pi}\frac{W_{\sr}(2-y_0+i\ep_0)}{M_{\sr}(2-y_0+i\ep_0)} \right) \right| \\
&\qquad\geq 1 - |\Gamma_{\nu_0}|e^{{\nu_0}\pi} \left| \Im \left(\frac{W_{\sr}(y_0-i\ep_0)}{M_{\sr}(y_0-i\ep_0)} \right) \right| - |\Gamma_{\nu_0}|e^{-{\nu_0}\pi } \left| \Im \left( \frac{W_{\sr}(2-y_0+i\ep_0)}{M_{\sr}(2-y_0+i\ep_0)} \right) \right|
\end{align*}
The desired bound follows from Lemma \ref{lemma:smallnuboundedarg}. Likewise, for $\gamma_0 = \mu_0>0$ we now observe that
\begin{align*}
& \left| 1 +i\Gamma_{\sr} \left( e^{-{\mu_0} \pi i} \frac{W_{\sr}(y_0)}{M_{\sr}(y_0)} - e^{{\mu_0}\pi i} \frac{W_{\sr}(2-y_0)}{M_{\sr}(2-y_0)} \right) \right| \\
&\qquad \geq 1 - \Gamma_0\sin({\mu_0}\pi) \left( \frac{W_{\sr}(y_0)}{M_{\sr}(y_0)} + \frac{W_{\sr}(2-y_0)}{M_{\sr}(2-y_0)} \right)  + O({\mu_0}).
\end{align*}
We recall that $M_{\sr}(\cdot)=M_{0,\mu_0}(2k\cdot)$ is real-valued for real arguments, continuous, uniformly bounded above and also uniformly bounded away from 0 for ${\mu_0}\in \left(0,\frac12 \right)$, and that $W_{\sr}(\cdot)=W_{0,\mu_0}(\cdot)$ is also real-valued for real arguments, continuous and bounded uniformly for ${\mu_0}\in \left(0,\frac12 \right)$ and for its argument in compact sets away from 0, see \ref{lemma:smallgammaasymptotic} and Lemma \ref{lemma:smallnuboundedarg}. Hence, for $\ep>0$ small enough and ${\mu_0}>0$ small enough, the desired bound follows.

\bullpar{Case 2} Assume $k\geq \vartheta^{-1}\frac{N}{2}$. Since both $y_0\geq \vartheta$ and $2-y_0\geq \vartheta$, we now have $2ky_0\geq N$ and also $2k(2-y_0)\geq N$. Hence, for $\gamma_0 = \mu_0 + i \nu_0$, since 
\begin{align*}
&\left| 1 + i\Gamma_{\gamma_0} \left( e^{-{\gamma_0}\pi i}\frac{W_{\sr}(y_0-i\ep_0)}{M_{\sr}(y_0-i\ep_0)} - e^{{\gamma_0} \pi i}\frac{W_{\sr}(2-y_0+i\ep_0)}{M_{\sr}(2-y_0+i\ep_0)} \right) \right| \\
&\qquad\geq 1 - |\Gamma_{\gamma_0}|e^{\nu_0\pi} \left|   \frac{W_{\sr}(y_0-i\ep_0)}{M_{\sr}(y_0-i\ep_0)}  \right| - |\Gamma_{\gamma_0}|e^{-\nu_0\pi } \left| \frac{W_{\sr}(2-y_0+i\ep_0)}{M_{\sr}(2-y_0+i\ep_0)} \right|,
\end{align*}
The result follows once we use Lemma \ref{lemma:smallnulargearg}.
\end{proof}

\begin{lemma}\label{lemma:smallgammaasymptotic}
Let $\zeta\in \C$ and $\gamma\in \C$. For any compact set $K\subset\C$, there holds
\begin{align*}
M(1/2+\gamma,1+2\gamma,\zeta) &=  M(1/2,1,\zeta) + O(\gamma).
\end{align*}
uniformly for all $\zeta\in K$. Moreover, if $K\subset \lbrace \Re(\zeta)>0 \rbrace$ and it is uniformly bounded away from $0$, then
\begin{align*}
U(1/2+\gamma,1+2\gamma,\zeta) = U(1/2,1,\zeta) + O(\gamma)
\end{align*}
uniformly for all $\zeta\in K$. 
\end{lemma}

\begin{proof}
From \cite{NIST}, we note that 
\begin{align*}
M(1/2+\gamma,1+2\gamma,\zeta) &= \Gamma (1+2\gamma)\mathbf{M}(1/2+\gamma,1+2\gamma,\zeta) \\
&= \frac{\Gamma(1+2\gamma)}{\Gamma(1/2+\gamma)^2}\int_0^1 e^{\zeta s} \left( s(1-s) \right)^{-\frac12 + i \gamma} \d s \\
&= \left(\frac{1}{\pi} + O(|\gamma|) \right)\int_0^1 e^{\zeta s} \left( s(1-s) \right)^{-\frac12}\d s \\
&\quad +  \gamma\left(\frac{1}{\pi} + O(|\gamma|) \right)\int_0^1 e^{\zeta s} \left( s(1-s) \right)^{-\frac12}\left( \log(s(1-s)) \right)\int_0^1 (s(1-s))^{\gamma u} \d u \d s
\end{align*}
so that for $\zeta$ bounded, we reach
\begin{align*}
M(1/2+\gamma,1+2\gamma,\zeta) &= \frac{1}{\pi}\int_0^1 e^{\zeta s} \left( s(1-s) \right)^{-\frac12}\d s + O(|\gamma|) = M(1/2,1,\zeta) + O(|\gamma|).
\end{align*}
Similarly, for $\Re(\zeta)>0$, 
\begin{align*}
U(1/2+\gamma,1+2\gamma,\zeta) &= \frac{1}{\Gamma(1/2+\gamma)}\int_0^\infty e^{-\zeta s} (s(1-s))^{-\frac12+\gamma}\d s \\
&=\left( \frac{1}{\sqrt\pi} + O(|\gamma|) \right)\int_0^\infty e^{-\zeta s} (s(1-s))^{-\frac12}\d s \\
&\quad +\gamma \left( \frac{1}{\sqrt\pi} + O(|\gamma|) \right) \int_0^\infty e^{-\zeta s} (s(1-s))^{-\frac12} \log (s(1-s))\int_0^1 (s(1-s))^{\gamma u}\d u \d s.
\end{align*}
and the lemma follows for $|\zeta|$ uniformly bounded from below.
\end{proof}

\begin{lemma}\label{lemma:smallnulargearg}
Let $\gamma\in \C$. There exists $N_0>0$ and $\gamma_*>0$ such that 
\begin{align*}
\left| \frac{W_{0,\gamma}(\zeta)}{M_{0,\gamma}(\zeta)} \right| \lesssim e^{-\Re(\zeta)},
\end{align*}
for all $\Re(\zeta)>N$ and all $|\gamma|<\gamma_*$.
\end{lemma}

\begin{proof}
Let $a=\frac12+\gamma$. We recall from \cite{NIST} that for $a\in \C$ and $\Re(\zeta)>0$,
\begin{align*}
M(a,2a,\zeta) = i\frac{\Gamma(2a)}{\Gamma(a)} \left( e^{\gamma\pi i}U(a,2a,\zeta) - e^\zeta e^{-\gamma\pi i}U(a,2a,e^{-i\pi}\zeta) \right). 
\end{align*}
Therefore,
\begin{align*}
\frac{W_{0,\gamma}(\zeta)}{M_{0,\gamma}(\zeta)} = i\frac{\Gamma(a)}{\Gamma(2a)}\frac{1}{e^{-\gamma\pi i} e^\zeta \frac{U(a,2a,e^{-i\pi}\zeta)}{U(a,2a,\zeta)} - e^{\gamma\pi i}}.
\end{align*}
Moreover, one can see from 13.7.5 in \cite{NIST} that
\begin{align*}
U(a,2a,\zeta) = \zeta^{-a}\left( 1 + O\left( |\zeta|^{-1}\right) \right)
\end{align*}
for all $|\zeta|>1$ uniformly in $a$, for bounded $|a|$. The lemma follows taking $\Re(\zeta)>0$ large enough.
\end{proof}

\begin{lemma}\label{lemma:smallnuboundedarg}
Let $\vartheta = \min(\vartheta_1, 2-\vartheta_2)$ and $y_0\in [\vartheta_1, \vartheta_2]$  such that $2k y_0\leq N_0$. Then, there exists $\ep_*>0$ and $\gamma_*>0$ such that 
\begin{align}
\left| \frac{W_{0,\gamma}(y_0-i\ep)}{M_{0,\gamma}(y_0-i\ep)} - \frac{W_{0,0}(y_0)}{M_{0,}(y_0)} \right| \leq \epsilon,
\end{align}
for all $\ep\leq \ep_*$ and all $0<|\gamma|\leq \gamma_*$. In particular, 
\begin{align*}
\left| \Im \left( \frac{W_{0,\gamma}(y_0-i\ep)}{M_{0,\gamma}(y_0-i\ep)} \right) \right| \leq \epsilon.
\end{align*}
\end{lemma}
\begin{proof}
We observe that
\begin{align*}
\frac{W_{0,\gamma}(y_0-i\ep)}{M_{0,\gamma}(y_0-i\ep)} = \frac{U(1/2+\gamma,1+2\gamma,2k(y_0-i\ep))}{M(1/2+\gamma,1+2\gamma,2k(y_0-i\ep))}.
\end{align*}
Since $M(1/2,1,\cdot)$ is real-valued, uniformly continuous, and non-zero, and $U(1/2,1,\cdot)$ is real-valued, uniformly continuous and bounded away from 0,  then the Lemma follows from the asymptotic expansions in Lemma \ref{lemma:smallgammaasymptotic}.
\end{proof}

\section{The Green's function for the fragile stratification}
Here we show the main modifications required to obtain bounds on the Wronskian of the Green's function associated with the fragile stratification regime. We begin with the following.

\begin{lemma}\label{lemma:growthboundMfragile}
Let $\gamma =\mu \in \left( \frac14, \frac12 \right)$. Then, 
\begin{equation*}
\lim_{\Re(\zeta)\rightarrow +\infty}\frac{M_{0,-\mu}(\zeta)}{M_{0,\mu}(\zeta)} = 2^{-4\mu}\frac{\Gamma(1-\mu)}{\Gamma(1+\mu)} 
\end{equation*} 
\end{lemma}

\begin{proof}
Let $\zeta\in\C$, for $a_\pm = \frac12\pm \mu$ and $b_\pm =2a_\pm$ we recall that
\begin{equation*}
\begin{aligned}
M_{0,\pm\mu}(\zeta)&=\e^{-\frac12\zeta}\zeta^{a_\pm}\frac{\Gamma(b_\pm)}{\Gamma(a_\pm)}\l( \e^{-i\pi a_\pm}U(a_\pm,b_\pm,\zeta) + \e^{i\pi a_\pm}\e^\zeta U(a_\pm,b_\pm,\e^{i\pi} \zeta)\r).
\end{aligned}
\end{equation*}
Moreover, we have that $U(a_\pm,b_\pm,\zeta) = \zeta^{-a_\pm}+ \mathcal{E}_{\pm}(\zeta)$, where further
\begin{equation*}
|\zeta^{a_\pm}\E_\pm(\zeta)|\leq \frac{2\b^2}{|\zeta|}\e^{\frac{2\b^2}{|\zeta|}}.
\end{equation*}
with $\beta^2 = \sqrt{\frac14 - \mu^2}$. Therefore, 
\begin{equation*}
M_{0,\pm\mu}(\zeta)=\frac{\Gamma(b_\pm)}{\Gamma(a_\pm)}\e^{\frac12\zeta}\l(  \l[1 + (\e^{i\pi}\zeta)^{a_\pm}\E_\pm(\e^{i\pi}\zeta)\r] + \e^{-\zeta}\e^{-i\pi a_\pm}\l[1+ \zeta^{a_\pm}\E_\pm(\zeta)\r]\r)
\end{equation*}
Now, since $b_\pm=2a_\pm$, we have from \cite{NIST} that
\begin{equation*}
\frac{\Gamma(b_\pm)}{\Gamma(a_\pm)}=\pi^{-\frac12}2^{2a_\pm-1}\Gamma\l( a_\pm +\frac12\r),
\end{equation*}
and thus
\begin{equation*}
\frac{\displaystyle {\Gamma(b_-)}/{\Gamma(a_-)}}{\displaystyle {\Gamma(b_+)}/{\Gamma(a_+)}}=2^{-4\mu}\frac{\Gamma(1-\mu)}{\Gamma(1+\mu)}.
\end{equation*}
Hence,
\begin{equation*}
\frac{M_{0,-\mu}(\zeta)}{M_{0,\mu}(\zeta)}=2^{-4\mu}\frac{\Gamma(1-\mu)}{\Gamma(1+\mu)}\frac{ 1 + (\e^{i\pi}\zeta)^{a_-}\E_-(\e^{i\pi}\zeta) + \e^{-\zeta}\e^{-i\pi a_-}\l[1+ \zeta^{a_-}\E_-(\zeta)\r]}{  1 + (\e^{i\pi}\zeta)^{a_+}\E_+(\e^{i\pi}\zeta) + \e^{-\zeta}\e^{-i\pi a_+}\l[1+ \zeta^{a_+}\E_+(\zeta)\r]}.
\end{equation*}
Since $|\zeta^{a_\pm}\E_\pm(\zeta)|\lesssim \frac{1}{|\zeta|}$, the lemma follows.
\end{proof}

\begin{lemma}\label{lemma:fragileMlimit}
Let $\gamma =\mu \in \left( \frac14, \frac12 \right)$. Let $\zeta = x+iy$, with $\zeta\in B_R$ and $x\geq c_0>0$, for some $c_0>0$. Then,
\begin{align*}
\left| M_{0,-\mu}(x+iy) - M_{0,-\mu}(x) \right| \lesssim \frac{|y|}{x}
\end{align*}
\end{lemma}

\begin{proof}
Recall that $M_{0,-\mu}(\zeta)=\zeta^{\frac12-\mu}\mathcal{E}_{2,-\mu}(\zeta)$ with
\begin{align*}
(x+iy)^{\frac12-\mu} - x^{\frac12-\mu} &= (x+iy)^{\frac12-\mu} - 1 - \left( x^{\frac12-\mu} - 1 \right)  \\
&= \left(\tfrac12-\mu \right) \left( \log(x+iy)\int_0^1 e^{(\frac12-\mu)u \log(x+iy)} \d u  - \log(x)\int_0^1 e^{(\frac12-\mu)u \log(x)} \d u \right) \\
&= \left(\tfrac12 - \mu\right) \int_0^1 \frac{y}{x+iys} \left( \int_0^1 \left( 1 + u\left(\tfrac12-\mu \right)\log(x+iys) \right) e^{u \left(\frac12-\mu\right) \log(x+iys)} \d u \right) \d s .
\end{align*}
Since $|\zeta|\leq R$ and $x\geq c_0>0$ we conclude that
\begin{align*}
\left| (x+iy)^{\frac12-\mu} - x^{\frac12-\mu} \right| \lesssim \frac{|y|}{x}.
\end{align*}
and as $\mathcal{E}_{2,-\mu}(\zeta)$ is entire in $\zeta$ uniformly for $ \mu \in \left( \frac14, \frac12 \right)$, the lemma follows.
\end{proof}

We next obtain suitable lower bounds for the Wronskian of the Green's function, which we recall takes the form
\begin{equation*}
\W_{k,\ep}^\pm(y_0):= -4k \mu_0\Big( M_\sr( - y_0\pm i\ep_0)M_\s(2 - y_0\pm i\ep_0) - M_\s( - y_0\pm i\ep_0)M_\sr(2-y_0\pm i\ep_0) \Big).
\end{equation*} 
The analytic continuation properties of $M_\sr$ and $M_\s$ then give
\begin{align*}
\W_{k,\ep}^+(y_0)&=-4ki\mu_0\left( \e^{i\mu_0\pi}M_\sr(y_0-i\ep_0)M_\s(2-y_0+i\ep_0)-\e^{-i\mu_0\pi}M_\s(y_0-i\ep_0)M_\sr(2-y_0+i\ep_0)\right).
\end{align*}
We are now in position to prove the main lower bounds for $\W_{k,\ep}^+(y_0)$.

\begin{lemma}\label{lemma:Wronskianlowerboundweakstrat}
Let $\gamma_0=\mu_0\in \left( \frac14, \frac12 \right)$. There exists $N>0$ such that
\begin{align*}
\left| \W_{k,\ep}^+(y_0) \right| \geq 4k\mu_0\sin(\mu_0\pi)|M_\s(y_0-i\ep_0)||M_\sr(2-y_0+i\ep_0)|,
\end{align*}
for all $k \geq \vartheta^{-1}\frac{N}{2}$. Likewise, for all $k < \vartheta^{-1}\frac{N}{2}$ there exists $\varepsilon_*>0$, independent of $\mu$ such that 
\begin{align*}
\left| \W_{k,\ep}^+(y_0)\right| \geq 2 k \mu_0\sin(\mu_0\pi) \left( M_\sr(y_0) M_\s(2-y_0) - M_\s(y_0) M_\sr(2-y_0) \right) > 0,
\end{align*}
for all $0< \ep < \ep_*$.
\end{lemma}

\begin{proof}
To prove the first statement, recall that $y_0\in (\vartheta_1,\vartheta_2)$ and $\vartheta=\min(\vartheta_1,2-\vartheta_2)$ so that $2k y_0, \, 2k(2-y_0) \geq 2k\vartheta \geq N$, for all $k \geq \vartheta^{-1}\frac{N}{2}$. Moreover, note that 
\begin{align*}
\W_{k,\ep}^+(y_0)&=-4ki\mu_0 M_\sr(y_0-i\ep_0)M_\s(2-y_0+i\ep_0) \left( e^{i\mu\pi}-e^{-i\mu\pi}\frac{M_\s(y_0-i\ep_0)M_\sr(2-y_0+i\ep_0)}{M_\sr(y_0-i\ep_0)M_\s(2-y_0+i\ep_0)} \right)
\end{align*}
with further 
\begin{align*}
\frac{M_\s(y_0-i\ep_0)M_\sr(2-y_0+i\ep_0)}{M_\sr(y_0-i\ep_0)M_\s(2-y_0+i\ep_0)} = 1 + O(N^{-1})
\end{align*}
due to Lemma \ref{lemma:growthboundMfragile}. The first statement follows taking $N$ large enough. For the second statement, since $y_0,\, 2-y_0\geq \vartheta>0$ and $2k|y_0-i\ep_0|\leq \vartheta^{-1}(N+1)$ and $2k|1-y_0+i\ep_0|\leq \vartheta^{-1}(N+1)$ for $\ep_0$ small enough, we have from Lemma \ref{lemma:fragileMlimit} that
\begin{align*}
\W_{k,\ep}^+(y_0)&=-4ki\mu_0\left( \e^{i\mu_0\pi}M_\sr(y_0)M_\s(2-y_0)-\e^{-i\mu_0\pi}M_\s(y_0)M_\sr(2-y_0) + O(\ep)  \right) 
\end{align*}
so that, in particular,
\begin{align*}
\left| \W_{k,\ep}^+(y_0) \right| &\geq 4k\mu_0\sin(\mu_0\pi)\left( M_\sr(y_0)M_\s(2-y_0)+M_\s(y_0)M_\sr(2-y_0) \right) + kO(\ep)  
\end{align*}
The Lemma follows for $\ep$ sufficiently small, since $M_\sr(y_0)M_\s(2-y_0)+M_\s(y_0)M_\sr(2-y_0)$ is strictly positive bounded away from zero, uniformly for $\mu \in \left( \frac14,\frac12 \right)$ and for $\vartheta_1 \leq y_0 \leq \vartheta_2$. 
\end{proof}

\section{High-order operator estimates for weak and strong stratifications}\label{sec:highorderweakstrong}
In this section we state and prove further several mapping properties of the maps $R_{j,k,\ep}^\pm$, for $j=0,1,2,3,4$ that are needed to obtain Sobolev regularity on the spectral density function. To avoid repetition, we shall only prove the local $X_k^1$ bounds of the operator norms, since the global $H_k^1(I_3(y_0))$ estimates follow form the usual entanglement inequality. Moreover, to avoid repetition, we shall restrict ourselves to the $X^0_k$ bounds, since the strategy to prove the full $X_k^1$ bounds is the same. 

\subsection{Second order operator estimates}
We begin the section showing the mapping properties of $R_{2,k,\ep}^\pm$.
\begin{lemma}\label{lemma:R2mapsXktoXk}
Let $k\geq 1$, $y_0\in I_S\cup I_W$ and $f(y,y_0)\in X_k$. There holds
\begin{align*}
\Vert (R_{2,k,\ep}^\pm f)(\cdot,y_0) \Vert_{X_k} \lesssim  \Vert f \Vert_{X_k} 
\end{align*}
uniformly for all $0<\ep < \ep_*$, and all $y_0\in I_S\cup I_W$.
\end{lemma}

\begin{proof}
As usual we define $g_{k,\ep}^\pm(y,y_0):=(R_{2,k,\ep}^\pm f)(y,y_0)$, where now
\begin{align*}
g_{k,\ep}^\pm (y,y_0) &= \int_{I_3(y_0)}  \G_{k,\ep}^\pm (y,y_0,z) \frac{f(z,y_0)}{(v(z)-v(y_0)\pm i\ep)^2} \d z + \int_{I_3^c(y_0)}  \G_{k,\ep}^\pm (y,y_0,z) \frac{f(z,y_0)}{(v(z)-v(y_0)\pm i\ep)^2} \d z \\
&= g_{1,k,\ep}^\pm(y,y_0) + g_{2,k,\ep}^\pm(y,y_0).
\end{align*}
Firstly, Proposition \ref{prop:sobolevregdecomG} and Corollary \ref{cor:sobolevregpartialdecomG} show first $\Vert g_{2,k,\ep}^\pm\Vert_{X_k^1}\lesssim k^\frac12\Vert f \Vert_{L^2(I_3^c(y_0))} \lesssim \Vert f \Vert_{X_k}$ and then $\Vert g_{2,k,\ep}^\pm \Vert_{X_k}\lesssim \Vert f \Vert_{X_k}$ by means of the entanglement inequality. Secondly, we integrate by parts to obtain
\begin{align*}
g_{1,k,\ep}^\pm(y,y_0) &= - \frac{\G_{k,\ep}^\pm(y,y_0,z)}{v'(z)}\frac{f(z,y_0)}{v(z) - v(y_0) \pm i\ep}\Big|_{\partial I_3(y_0)} + \int_{I_3(y_0)} \frac{1}{v'(z)} \frac{ \partial_z \left( \G_{k,\ep}^\pm(y,y_0,z)f(z,y_0)	\right)  }{v(z) - v(y_0) \pm i\ep} \d z\\
&\quad +\int_{I_3(y_0)} \G_{k,\ep}^\pm(y,y_0,z) \frac{f(z,y_0)}{v(z) - v(y_0) \pm i\ep}\left( \frac{1}{v'(z)} \right)' \d z \\
&= g_{3,k,\ep}^\pm(y,y_0) + g_{4,k,\ep}^\pm(y,y_0) + g_{5,k,\ep}^\pm(y,y_0) \d z.
\end{align*}
Since $\left( \tfrac{1}{v'(z)} \right)' \in C^1$ we appeal to Corollary \ref{cor:R1mapsCXktoXk} to get $\Vert g_{5,k,\ep}^\pm \Vert_{X_k^1} \lesssim k^{-1}\Vert f \Vert_{X_k}$, while Proposition \ref{prop:sobolevregdecomG}, Corollary \ref{cor:sobolevregpartialdecomG}  and $f\in X_k$ shows that $\Vert g_{3,k,\ep}^\pm\Vert_{X_k^1}\lesssim \Vert f \Vert_{X_k}$. As for $g_{4,k,\ep}^\pm$, we recall
\begin{align*}
\frac{1}{v(y) - v(y_0) \pm i\ep} = \frac{(v'(y_0))^{-1}}{y-y_0\pm i\ep_0} - \mathrm{V}_{1,\ep}^\pm(y,y_0),
\end{align*}
with $\mathrm{V}_{1,\ep}^\pm(y,y_0)\in L_y^\infty(0,2)$ uniformly for all $y_0\in [0,2]$ and all $\ep>0$. Hence,
\begin{align*}
g_{4,k,\ep}^\pm(y,y_0) &= \frac{2k}{v'(y_0)}\int_{I_3(y_0)} \frac{1}{v'(z)}\frac{\partial_z \left(\G_{k,\ep}^\pm(y,y_0,z)f(z,y_0)\right)}{\xi} \d z \\
&\quad- \frac{1}{v'(y_0)}\int_{I_3(y_0)} \frac{1}{v'(z)}\partial_z \left(\G_{k,\ep}^\pm(y,y_0,z)f(z,y_0)\right) V_{1,\ep}^\pm(z,y_0) \d z\\
&= g_{6,k,\ep}^\pm(y,y_0) + g_{7,k,\ep}^\pm(y,y_0)
\end{align*}
We first address $g_{6,k,\ep}^\pm$. We have
\begin{align*}
g_{6,k,\ep}^\pm &= \frac{4k^2}{v'(y_0)} \int_{I_3(y_0)} \frac{1+2\gamma_0}{v'(z)}  \left(\G_{k,\ep}^\pm \right)_\sr (y,y_0,z) f_\sr(z,y_0) \xi^{-1+2\gamma_0} \d z \\
&\quad +  \frac{2k}{v'(y_0)} \int_{I_3(y_0)} \frac{1}{v'(z)}  \partial_z \left( \left(\G_{k,\ep}^\pm \right)_\sr (y,y_0,z) f_\sr(z,y_0) \right) \xi^{2\gamma_0} \d z \\
&\quad+ \frac{4k^2}{v'(y_0)} \int_{I_3(y_0)} \frac{1}{v'(z)} \left( \left(\G_{k,\ep}^\pm \right)_\sr (y,y_0,z) f_\s(z,y_0)  +  \left(\G_{k,\ep}^\pm \right)_\s (y,y_0,z) f_\sr(z,y_0) \right)  \xi^{-1} \d z \\
&\quad + \frac{2k}{v'(y_0)} \int_{I_3(y_0)} \frac{1}{v'(z)} \partial_z \left( \left(\G_{k,\ep}^\pm \right)_\sr (y,y_0,z) f_\s(z,y_0)  +  \left(\G_{k,\ep}^\pm \right)_\s (y,y_0,z) f_\sr(z,y_0) \right)  \d z \\
&\quad + \frac{4k^2}{v'(y_0)} \int_{I_3(y_0)} \frac{1-2\gamma_0}{v'(z)}  \left(\G_{k,\ep}^\pm \right)_\s (y,y_0,z) f_\s(z,y_0) \xi^{-1-2\gamma_0} \d z \\
&\quad +  \frac{2k}{v'(y_0)} \int_{I_3(y_0)} \frac{1}{v'(z)}  \partial_z \left( \left(\G_{k,\ep}^\pm \right)_\s (y,y_0,z) f_\s(z,y_0) \right) \xi^{-2\gamma_0} \d z 
\end{align*}
and we argue for the last two integral, since they are the most singular. Firstly, integrating by parts once more we have 
\begin{align*}
\frac{4k^2}{v'(y_0)}& \int_{I_3(y_0)} \frac{1-2\gamma_0}{v'(z)}  \left(\G_{k,\ep}^\pm \right)_\s (y,y_0,z) f_\s(z,y_0) \xi^{-1-2\gamma_0} \d z \\
&= -\frac{2k}{v'(y_0)}\frac{1-2\gamma_0}{2\gamma_0 v'(z)}\left(\G_{k,\ep}^\pm \right)_\s (y,y_0,z) f_\s(z,y_0) \xi^{-2\gamma_0} \Big|_{\partial I_3(y_0)} \\
&\quad + \frac{2k}{v'(y_0)}\frac{1-2\gamma_0}{2\gamma_0}\int_{I_3(y_0)} \partial_z \left( \frac{\left(\G_{k,\ep}^\pm \right)_\s (y,y_0,z) f_\s(z,y_0)}{v'(z)} \right) \xi^{-2\gamma_0} \d z
\end{align*}
Using Proposition \ref{prop:sobolevregdecomG} and Corollary \ref{cor:sobolevregpartialdecomG}, the boundary term is clearly bounded by $k^{-1}\Vert f \Vert_{X_k^1}$. As for the integral, we observe that 
\begin{align*}
\left|\frac{2k}{v'(y_0)}\frac{1-2\gamma_0}{2\gamma_0}\int_{I_3(y_0)} \partial_z \left( \frac{\left(\G_{k,\ep}^\pm \right)_\s (y,y_0,z) f_\s(z,y_0)}{v'(z)} \right) \xi^{-2\gamma_0} \d z \right| &\lesssim 2k\Vert f \Vert_{X_k^1}(1-2\mu_0)\int_{I_3(y_0)}|\xi|^{-2\mu_0} \d z\\
&\lesssim \Vert f \Vert_{X_k^1}
\end{align*}
from which the $X^0_k$ bounds follows swiftly. The complete $X_k^1$ bound is deduced similarly. Finally, note that
\begin{align*}
\frac{2k}{v'(y_0)}& \int_{I_3(y_0)} \frac{1}{v'(z)}  \partial_z \left( \left(\G_{k,\ep}^\pm \right)_\s (y,y_0,z) f_\s(z,y_0) \right) \xi^{-2\gamma_0} \d z  \\
&= \frac{2k}{v'(y_0)} \int_{I_3(y_0)} \frac{1}{v'(z)}  \partial_z \left( \left(\G_{k,\ep}^\pm \right)_\s (y,y_0,z)\right) f_\s(z,y_0) \xi^{-2\gamma_0} \d z \\
&\quad +\frac{2k}{v'(y_0)} \int_{I_3(y_0)} \frac{1}{v'(z)}   \left(\G_{k,\ep}^\pm \right)_\s (y,y_0,z) \partial_z \left( f_\s(z,y_0) \right) \xi^{-2\gamma_0} \d z
\end{align*}
and since 
\begin{align*}
\left\Vert \partial_z \left(\G_{k,\ep}^\pm \right)_\s (\cdot,y_0,z) \right\Vert_{X_k^1} \lesssim |\xi|, \quad \left| \partial_z \left( f_\s(z,y_0) \right) \right| \lesssim k |\xi|^\frac12 \Vert f \Vert_{X_k^1}
\end{align*}
we deduce from Proposition \ref{prop:sobolevregdecomG}, Proposition \ref{prop:sobolevregpartialdecomG} and Corollary \ref{cor:sobolevregpartialdecomG} that we can bound
\begin{align*}
\left| \frac{2k}{v'(y_0)} \int_{I_3(y_0)} \frac{1}{v'(z)}  \partial_z \left( \left(\G_{k,\ep}^\pm \right)_\s (y,y_0,z) f_\s(z,y_0) \right) \xi^{-2\gamma_0} \d z \right| &\lesssim 2k \Vert f \Vert_{X_k^1} \int_{I_3(y_0)} |\xi|^{\frac12-2\mu_0}\d z \\
&\lesssim \Vert f \Vert_{X_k}
\end{align*}
and similarly for the $\partial_y$ derivative, so that we deduce
\begin{align*}
\left\Vert \frac{2k}{v'(y_0)} \int_{I_3(y_0)} \frac{1}{v'(z)}  \partial_z \left( \left(\G_{k,\ep}^\pm \right)_\s (y,y_0,z) f_\s(z,y_0) \right) \xi^{-2\gamma_0} \d z \right\Vert_{X_k^1}\lesssim \Vert f \Vert_{X_K^1}.
\end{align*}
The other four integrals are bounded following the same ideas, namely integrating by parts to reduce the singularity in $\xi$ and using Propositions \ref{prop:sobolevregdecomG} and \ref{prop:sobolevregpartialdecomG}, we omit the routine details. Likewise, note that for $g_{7,k,\ep}^\pm$ we shall only focus on two contributions. The first one is
\begin{align*}
\frac{1}{v'(y_0)}\int_{I_3(y_0)} \frac{1}{v'(z)} \partial_z \left( \left( \G_{k,\ep}^\pm \right)_\s (y,y_0,z) f_\s(z,y_0) \right) \xi^{1-2\gamma_0}V_{1,\ep}^\pm(z,y_0) \d z
\end{align*}
which is bounded in $X_k^1$ by $k^{-1}\Vert f\Vert_{X_k}$ because of Propositions \ref{prop:sobolevregdecomG} and \ref{prop:sobolevregpartialdecomG} and $\left| \xi^{1-2\gamma_0} \right| \lesssim 1 + |\log(\xi) |$
is integrable for $|\xi|$ bounded. The second one is
\begin{align*}
\frac{2k}{v'(y_0)}\int_{I_3(y_0)}\frac{1-2\gamma_0}{v'(z)}  \left(\G_{k,\ep}^\pm \right)_\s (y,y_0,z) f_\s(z,y_0) \xi^{-2\gamma_0}\d z,
\end{align*}
which is again bounded by Propositions \ref{prop:sobolevregdecomG} and \ref{prop:sobolevregpartialdecomG} once we realise
\begin{align*}
2k(1-2\gamma_0)\xi^{-2\gamma_0} = \partial_z \left( \xi^{1-2\gamma_0} \right)
\end{align*}
and we integrate by parts, we omit the details. Hence,
\begin{align*}
\Vert g_{7,k,\ep}^\pm\Vert_{X_K^1}\lesssim \Vert f \Vert_{X_k}.
\end{align*} 
Concerning the $H_k^1(I_3(y_0))$ estimate,  we have that $g_{k,\ep}^\pm(y,y_0)$ satisfies
\begin{align*}
\left( \D_k - k^2 + \frac{\cJ(y_0)}{(y-y_0\pm i\ep)^2} \right) g_{k,\ep}^\pm(y,y_0) = \frac{f(y)}{(v(y)-v(y_0)\pm i\ep)^2}
\end{align*}
and thus the entanglement inequality provides
\begin{align*}
\Vert g_{k,\ep}^\pm \Vert_{H_k^1(I_3^c(y_0))} &\lesssim \frac{1}{k^2}\left \Vert \frac{f(z,y_0)}{(v(z)-v(y_0)\pm i\ep)^2}\right\Vert_{L^2(I_2^c(y_0)\cap I_3(y_0))} + \frac{1}{k^2}\left \Vert \frac{f(z,y_0)}{(v(z)-v(y_0)\pm i\ep)^2}\right\Vert_{L^2(I_3^c(y_0))} \\
&\quad+  \Vert g_{k,\ep}^\pm \Vert_{L^2(I_2^c(y_0)\cap I_3(y_0))} \\
&\lesssim  k^{-\frac12}\Vert f \Vert_{X_k^1} +  \Vert f \Vert_{L^2(I_3^c(y_0))} 
\end{align*}
Here we have used that $\int_a^b |x^{-\frac32\pm \gamma_0}|^2 \d x$ is uniformly bounded in $\gamma_0$ for $0<\mu_0 \leq \frac12$ for all $0 < a < b < +\infty$. With this we finish the proof.
\end{proof}

We can combine Lemma \ref{lemma:R2mapsC1XktoXk} and Lemma \ref{lemma:R2mapsXktoXk} to obtain the following

\begin{corollary}\label{cor:R2mapsCXktoXk}
Let $k\geq 1$ and $y_0\in I_S\cup I_W$. Let $f\in X_k$ and $h(z,y_0)\in C^2$ uniformly in $y_0$. Then,
\begin{align*}
\Vert R_{2,k,\ep}^\pm hf \Vert_{X_k}\lesssim \Vert f \Vert_{X_k},
\end{align*}
uniformly for all $0<\ep<\ep_*$ and all $y_0\in I_S\cup I_W$. 
\end{corollary}

\begin{proof}
It follows from Lemma \ref{lemma:R2mapsC1XktoXk} and Lemma \ref{lemma:R2mapsXktoXk} once we write
\begin{align*}
g_{k,\ep}^\pm = R_{2,k,\ep}^\pm h f(y,y_0) = h(y_0) R_{2,k,\ep}^\pm f(y,y_0) + R_{2,k,\ep}^\pm h_1f(y,y_0),
\end{align*}
where we have defined $h_1(y,y_0) = h(y,y_0) - h(y_0,y_0)$, with $h_1(y_0,y_0)=0$ and $h_1(\cdot,y_0)\in C^2$ uniformly in $y_0\in [0,2]$.
\end{proof}

We can dispense with the assumption that $f\in X_k$ by instead assuming further regularity on $f$.

\begin{lemma}\label{lemma:R2mapsC1H1toXk}
Let $k\geq 1$ and $y_0\in I_S\cup I_W$. Let $f(y)\in H^1_k$ and $h(y,y_0)\in C^2_y$ uniformly in $y_0$ with $h(y_0,y_0) = 0$. Then,
\begin{align*}
\Vert R_{2,k,\ep}^\pm hf \Vert_{X_k}\lesssim k^{-\frac12} \Vert f \Vert_{H^1_k},
\end{align*}
uniformly for all $0<\ep<\ep_*$ and all $y_0\in I_S\cup I_W$. 
\end{lemma}

\begin{proof}
As usual, let $g_{k,\ep}^\pm=R_{2,k,\ep}^\pm f(y,y_0)$, let 
\begin{align*}
g_{1,k,\ep}^\pm(y,y_0) := \int_{I_3(y_0)} \G_{k,\ep}^\pm(y,y_0,z) \frac{h(z,y_0)}{(v(z) - v(y_0) \pm i\ep)^2}f(z) \d z
\end{align*}
and $g_{2,k,\ep}^\pm := g_{k,\ep}^\pm - g_{1,k,\ep}^\pm$. Appealing to Proposition \ref{prop:sobolevregdecomG} and Corollary \ref{cor:sobolevregpartialdecomG}, we have $\Vert g_{2,k,\ep}^\pm \Vert \lesssim k^{-\frac12}\Vert f \Vert_{L^2}$. Integrating by parts,
\begin{align*}
g_{1,k,\ep}^\pm(y,y_0) &= - \frac{\G_{k,\ep}^\pm(y,y_0,z)}{v'(z)}\frac{h(z,y_0)}{v(z) - v(y_0) \pm i\ep}f(z) \Big|_{z\in \partial I_3(y_0)} \\
&\quad + \int_{I_3(y_0)} \frac{\G_{k,\ep}^\pm(y,y_0,z)}{v(z) - v(y_0) \pm i\ep}\left( \partial_z f(z)\frac{h(z,y_0)}{v'(z)}  + f(z)\partial_z  \left(\frac{h(z,y_0)}{v'(z)} \right) \right)  \d z \\
&\quad +\int_{I_3(y_0)}\frac{h(z,y_0)}{v'(z)}\partial_z  \G_{k,\ep}^\pm(y,y_0,z) \frac{f(z)}{v(z) - v(y_0) \pm i\ep} \d z \\
&=g_{3,k,\ep}^\pm(y,y_0) + g_{4,k,\ep}^\pm(y,y_0) + g_{5,k,\ep}^\pm(y,y_0).
\end{align*}  
As before, Lemma \ref{lemma:LinfH1bound}, Proposition \ref{prop:sobolevregdecomG}, Corollary \ref{cor:sobolevregpartialdecomG} and the vanishing of $h(y_0,y_0)$ provide 
\begin{align*}
\Vert g_{3,k,\ep}^\pm \Vert_{X_k^1} + \Vert g_{4,k,\ep}^\pm \Vert_{X_k^1} \lesssim k^{-\frac12}\Vert f \Vert_{H_k^1}.
\end{align*}
We note here that for $g_{4,k,\ep}^\pm$ we can argue as in Lemma \ref{lemma:R0mapsL2toXk} and Lemma \ref{lemma:R1mapsH1toXk}. Regarding $g_{5,k,\ep}^\pm$, we shall only study
\begin{align*}
g_{6,k,\ep}^\pm(y,y_0) &:=2k\int_{I_3(y_0)} \frac{h(z,y_0)}{v'(z)}\left( \G_{k,\ep}^\pm \right)_\s(y,y_0,z) \frac{\xi^{-\frac12-\gamma_0}}{v(z) - v(y_0) \pm i\ep} f(z) \d z \\
&= (2k)^2\int_{I_3(y_0)} \frac{h(z,y_0)}{v'(y_0)v'(z)}\left( \G_{k,\ep}^\pm \right)_\s(y,y_0,z) f(z)\xi^{-\frac32-\gamma_0} \d z \\
&\quad -2k \int_{I_3(y_0)} \frac{h(z,y_0)}{v'(z)}\left( \G_{k,\ep}^\pm \right)_\s(y,y_0,z) f(z)\xi^{-\frac12-\gamma_0}\mathrm{V}_{1,\ep}^\pm(z,y_0) \d z \\
&= g_{7,k,\ep}^\pm(y,y_0) - g_{8,k,\ep}^\pm.
\end{align*}
Since $V_{1,\ep}^\pm\in L^\infty_z$ uniformly in $y_0$ and $\ep>0$, we readily have $\Vert g_{8,k,\ep}^\pm \Vert_{X_k^1}\lesssim k^{-\frac12}\Vert f \Vert_{L^2}$ thanks to the vanishing properties of $h(\cdot, y_0)$. On the other hand, integrating by parts in $g_{7,k,\ep}^\pm$ we reach
\begin{align*}
g_{7,k,\ep}^\pm(y,y_0) &= -\frac{4k}{1+2\gamma_0} \frac{h(z,y_0)}{v'(y_0)v'(z)}\left( \G_{k,\ep}^\pm \right)_\s(y,y_0,z) f(z)\xi^{-\frac12-\gamma_0} \Big|_{z\in \partial I_3(y_0)} \\
&\quad +\frac{4k}{1+2\gamma_0}\int_{I_3(y_0)} \frac{h(z,y_0)}{v'(y_0)v'(z)}f(z) \partial_z \left( \G_{k,\ep}^\pm \right)_\s(y,y_0,z) \xi^{-\frac12-\gamma_0} \d z \\
&\quad + \frac{4k}{1+2\gamma_0}\int_{I_3(y_0)} \left( \G_{k,\ep}^\pm \right)_\s(y,y_0,z) \partial_z \left( f(z)  \frac{h(z,y_0)}{v'(y_0)v'(z)} \right)\xi^{-\frac12-\gamma_0} \d z.
\end{align*}
The solid term is bounded as usual. For the first integral, from Corollary \ref{cor:sobolevregpartialdecomG} we note that 
\begin{align*}
\left|4k \frac{h(z,y_0)}{v'(y_0)v'(z)}f(z) \partial_z \left( \G_{k,\ep}^\pm \right)_\s(y,y_0,z) \xi^{-\frac12-\gamma_0} \right| \lesssim |f(z)||\xi|^{\frac32-\gamma_0}
\end{align*}
which is clearly integrable. For the second integral,
\begin{align*}
\frac{4k}{1+2\gamma_0}&\int_{I_3(y_0)} \left( \G_{k,\ep}^\pm \right)_\s(y,y_0,z) \partial_z \left( f(z)  \frac{h(z,y_0)}{v'(y_0)v'(z)} \right)\xi^{-\frac12-\gamma_0} \d z \\
&= \frac{4k}{1+2\gamma_0}\int_{I_3(y_0)} \left( \G_{k,\ep}^\pm \right)_\s(y,y_0,z)   \frac{h(z,y_0)}{v'(y_0)v'(z)} \partial_z f(z)\xi^{-\frac12-\gamma_0} \d z \\
&\quad + \frac{4k}{1+2\gamma_0}\int_{I_3(y_0)} \left( \G_{k,\ep}^\pm \right)_\s(y,y_0,z) f(z)   \partial_z \left( \frac{h(z,y_0)}{v'(y_0)v'(z)} \right)\xi^{-\frac12-\gamma_0} \d z
\end{align*}
with 
\begin{align*}
\left| 4k h(z,y_0) \xi^{-\frac12-\gamma_0} \right| \lesssim |\xi|^{\frac12-\mu_0}\lesssim 1 + |\log|\xi| |\in L^2(I_3(y_0))
\end{align*}
while we further integrate by parts to get
\begin{align*}
\frac{4k}{1+2\gamma_0}&\int_{I_3(y_0)} \left( \G_{k,\ep}^\pm \right)_\s(y,y_0,z) f(z)   \partial_z \left( \frac{h(z,y_0)}{v'(y_0)v'(z)} \right)\xi^{-\frac12-\gamma_0} \d z \\
&= \frac{2}{1+2\gamma_0}\left( \G_{k,\ep}^\pm \right)_\s(y,y_0,z) f(z)   \partial_z \left( \frac{h(z,y_0)}{v'(y_0)v'(z)} \right) \frac{\xi^{\frac12-\gamma_0}-1}{1-2\gamma_0}\Big|_{z\in \partial I_3(y_0)} \\
&\quad - \frac{2}{1+2\gamma_0}\int_{I_3(y_0)} \partial_z \left( \left( \G_{k,\ep}^\pm \right)_\s(y,y_0,z) f(z)   \partial_z \left( \frac{h(z,y_0)}{v'(y_0)v'(z)} \right) \right) \frac{\xi^{\frac12-\gamma_0}-1}{1-2\gamma_0}\d z \\
\end{align*}
Since
\begin{align*}
\left| \frac{\xi^{1-2\gamma_0}-1}{(1-2\gamma_0)} \right| \lesssim |\log|\xi||\in L^2(I_3(y_0)),
\end{align*}
confer Lemma \ref{lemma:firstgammalog}, the solid term and the integral can then be bounded in $X_k^1$ using Proposition \ref{prop:sobolevregdecomG}, Corollary \ref{cor:sobolevregpartialdecomG} and $f\in H^1_k$, we omit the details.
\end{proof}

\subsection{Third order operator estimates}
We define
\begin{align*}
R_{3,k,\ep}^\pm f(y,y_0) = \int_0^2 \G_{k,\ep}^\pm(y,y_0,z) \frac{f(z)}{(v(z) - v(y_0) \pm i\ep)^3} \d z.
\end{align*}
The first result is the following.
\begin{lemma}\label{lemma:R3mapsC1XktoXk}
Let $k\geq 1$ and $y_0\in I_S\cup I_W$. Let $f(y,y_0)\in X_k$ and $h(y,y_0)\in C^2_y$ uniformly in $y_0$ with $h(y_0,y_0) = 0$. Then,
\begin{align*}
\Vert R_{3,k,\ep}^\pm hf \Vert_{X_k}\lesssim  \Vert f \Vert_{X_k},
\end{align*}
uniformly for all $0<\ep<\ep_*$ and all $y_0\in I_S\cup I_W$. 
\end{lemma}

\begin{proof}
We shall only focus on the most singular contribution of $g_{k,\ep}^\pm(y,y_0) = R_{3,k,\ep}^\pm(y,y_0)hf(y,y_0)$, namely
\begin{align*}
g_{1,k,\ep}^\pm &:= \int_{I_3(y_0)} \G_{k,\ep}^\pm(y,y_0,z) \frac{h(z,y_0)f(z,y_0)}{(v(z) - v(y_0) \pm i\ep)^3} \d z \\
&= -\frac{1}{2}  \frac{\G_{k,\ep}^\pm(y,y_0,z)}{v'(z)} \frac{h(z,y_0)f(z,y_0)}{(v(z) - v(y_0) \pm i\ep)^2} \Big|_{z\in \partial I_3(y_0)} \\
&\quad +\frac12 \int_{I_3(y_0)} \G_{k,\ep}^\pm(y,y_0,z) \frac{f(z,y_0)}{(v(z) - v(y_0) \pm i\ep)^2}\partial_z \left( \frac{h(z,y_0)}{v'(z)} \right) \d z \\
&\quad +\frac12 \int_{I_3(y_0)} \frac{h(z,y_0)}{v'(z)} \frac{\partial_z \left( \G_{k,\ep}^\pm(y,y_0,z) f(z,y_0) \right)}{(v(z) - v(y_0) \pm i\ep)^2} \d z \\
&= g_{2,k,\ep}^\pm(y,y_0) + g_{3,k,\ep}^\pm(y,y_0) + g_{4,k,\ep}^\pm(y,y_0).
\end{align*}
As usual, Proposition \ref{prop:sobolevregdecomG} and Corollary \ref{cor:sobolevregpartialdecomG} show $\Vert g_{2,k,\ep}^\pm \Vert_{X_k^1} \lesssim k^{-1}\Vert f \Vert_{X_k^1}$. Regarding $g_{3,k,\ep}^\pm(y,y_0)$, we appeal to Corollary \ref{cor:R2mapsCXktoXk} to obtain $\Vert g_{3,k,\ep}^\pm\Vert_{X_k^1}\lesssim \Vert f \Vert _{X_k}$. For $g_{4,k,\ep}^\pm$ both $\G_{k,\ep}^\pm(y,y_0,\cdot)\in X_k^1$ and $f(\cdot,y_0)\in X_k^1$, we use \eqref{eq:partialzGf} and we argue for the most singular contributions
\begin{align*}
g_{5,k,\ep}^\pm(y,y_0) &= \frac12 \int_{I_3(y_0)} \frac{h(z,y_0)}{v'(z)} \frac{\partial_z \left( \left(\G_{k,\ep}^\pm\right)_\s(y,y_0,z) f_\s(z,y_0) \right)}{(v(z) - v(y_0) \pm i\ep)^2}\xi^{1-2\gamma_0} \d z \\
&\quad + k(1-2\gamma_0) \int_{I_3(y_0)} \frac{h(z,y_0)}{v'(z)} \frac{ \left( \G_{k,\ep}^\pm\right)_\s (y,y_0,z) f_\s(z,y_0) }{(v(z) - v(y_0) \pm i\ep)^2}\xi^{-2\gamma_0} \d z  \\
&= g_{6,k,\ep}^\pm(y,y_0) + g_{7,k,\ep}^\pm(y,y_0).
\end{align*}
Given the vanishing of $h(y_0,y_0)$, Proposition \ref{prop:sobolevregdecomG}, the bounds $\left\Vert \partial_z \left(\G_{k,\ep}^\pm\right)_\s (\cdot,y_0,z) \right\Vert_{X_k^1} \lesssim |\xi|^\frac12$ from  Corollary \ref{cor:sobolevregpartialdecomG} and $\left| \partial_z f_\s(z,y_0) \right| \lesssim |\xi|^\frac12 k\Vert f \Vert_{X_k^1}$, we have 
\begin{align*}
\Vert g_{6,k,\ep}^\pm \Vert_{X_k^1}\lesssim k\Vert f\Vert_{X_k^1}\left(\int_{I_3(y_0)} |\xi|^{\frac12-2\mu_0} \d z \right) \lesssim \Vert f\Vert_{X_k^1}
\end{align*}
since $x^{\frac12-2\mu_0}=x^{-\frac12}x^{1-2\mu_0}= x^{-\frac12}\left( 1 +(1-2\mu_0) \log(x) \mathcal{Q}_{\mu_0}(x)\right)$ is locally integrable. We now address $g_{7,k,\ep}^\pm$. Integrating by parts once more, 
\begin{align*}
g_{7,k,\ep}^\pm(y,y_0) &= -k(1-2\gamma_0) \frac{h(z,y_0)}{(v'(z))^2} \frac{ \left( \G_{k,\ep}^\pm\right)_\s (y,y_0,z) f_\s(z,y_0) }{v(z) - v(y_0) \pm i\ep}\xi^{-2\gamma_0}\Big|_{z\in\partial I_3(y_0)} \\
&\quad + k(1-2\gamma_0) \int_{I_3(y_0)}\frac{h(z,y_0)}{(v'(z))^2} \frac{ \partial_z \left( \left( \G_{k,\ep}^\pm \right)_\s(y,y_0,z) f_\s(z,y_0) \right)}{v(z) - v(y_0) \pm i\ep}\xi^{-2\gamma_0} \d z \\
&\quad + k(1-2\gamma_0) \int_{I_3(y_0)} \partial_z \left( \frac{h(z,y_0)}{(v'(z))^2} \right) \frac{  \left( \G_{k,\ep}^\pm \right)_\s(y,y_0,z) f_\s(z,y_0) }{v(z) - v(y_0) \pm i\ep}\xi^{-2\gamma_0} \d z \\ 
&\quad -4k^2\gamma_0(1-2\gamma_0)\int_{I_3(y_0)} \frac{h(z,y_0)}{(v'(z))^2} \frac{  \left( \G_{k,\ep}^\pm \right)_\s(y,y_0,z) f_\s(z,y_0) }{v(z) - v(y_0) \pm i\ep}\xi^{-1-2\gamma_0} \d z \\ 
&= \sum_{j=8}^{11} g_{j,k\ep}^\pm(y,y_0)
\end{align*}
We can estimate the boundary term $g_{8,k,\ep}^\pm$ in $X_k^1$ as usual. The term $g_{9,k,\ep}^\pm$ is bounded following the same strategy as for $g_{6,k,\ep}^\pm$, taking advantage of the vanishing of $h(z,y_0)$, of $\left\Vert \partial_z \left(\G_{k,\ep}^\pm\right)_\s (\cdot,y_0,z) \right\Vert_{X_k^1} \lesssim |\xi|^\frac12$ from  Corollary \ref{cor:sobolevregpartialdecomG} and $\left| \partial_z f_\s(z,y_0) \right| \lesssim |\xi|^\frac12 k\Vert f \Vert_{X_k^1}$. For $g_{10,k,\ep}^\pm$, we write
\begin{align*}
g_{10,k,\ep}^\pm(y,y_0) &= \frac{2k^2(1-2\gamma_0)}{v'(y_0)} \int_{I_3(y_0)} \partial_z \left( \frac{h(z,y_0)}{(v'(z))^2} \right)  \left( \G_{k,\ep}^\pm \right)_\s(y,y_0,z) f_\s(z,y_0) \xi^{-1-2\gamma_0} \d z \\ 
&\quad - \frac{k(1-2\gamma_0)}{v'(y_0)}\int_{I_3(y_0)} \partial_z \left( \frac{h(z,y_0)}{(v'(z))^2} \right)  \left( \G_{k,\ep}^\pm \right)_\s(y,y_0,z) f_\s(z,y_0) \mathrm{V}_{1,\ep}^\pm(z,y_0)\xi^{-2\gamma_0} \d z \\ 
&= g_{12,k,\ep}^\pm(y,y_0) + g_{13,k,\ep}^\pm(y,y_0).
\end{align*}
Since $(1-2\mu_0)\int_0^{3\beta^2}x^{-2\mu_0}\d x \lesssim 1$, we readily have $\Vert g_{13,k,\ep}^\pm\Vert_{X_k^1}\lesssim \Vert f \Vert_{X_k^1}$. For $g_{12,k,\ep}^\pm$, we integrate once more to get
\begin{align*}
g_{12,k,\ep}^\pm &= -\frac{k(1-2\gamma_0)}{v'(y_0)2\gamma_0} \partial_z \left( \frac{h(z,y_0)}{(v'(z))^2} \right)  \left( \G_{k,\ep}^\pm \right)_\s(y,y_0,z) f_\s(z,y_0) \xi^{-2\gamma_0} \Big|_{z\in \partial I_3(y_0)} \\
&\quad + \frac{k(1-2\gamma_0)}{2\gamma_0 v'(y_0)}\int_{I_3(y_0)} \partial_z \left( \frac{h(z,y_0)}{(v'(z))^2} \right) \partial_z \left( \left( \G_{k,\ep}^\pm \right)_\s(y,y_0,z) f_\s(z,y_0) \right) \xi^{-2\gamma_0} \d z \\
&\quad + \frac{k(1-2\gamma_0)}{2\gamma_0 v'(y_0)}\int_{I_3(y_0)} \partial_z^2 \left( \frac{h(z,y_0)}{(v'(z))^2} \right)  \left( \G_{k,\ep}^\pm \right)_\s(y,y_0,z) f_\s(z,y_0)  \xi^{-2\gamma_0} \d z 
\end{align*}
where now each of them is bounded in $X_k^1$; we estimate the first integral as in $g_{6,k,\ep}^\pm$ and the second integral as in $g_{13,k,\ep}^\pm$. The solid term is estimated as usual.

Finally, for $g_{11,k,\ep}^\pm$ we proceed as for $g_{10,k,\ep}^\pm$ and we focus on 
\begin{align*}
g_{14,k,\ep}^\pm(y,y_0) &= -8k^3\gamma_0(1-2\gamma_0)\int_{I_3(y_0)} \frac{h(z,y_0)}{(v'(z))^2}  \left(\G_{k,\ep}^\pm \right)_\s(y,y_0,z) f_\s(z,y_0) \xi^{-2-2\gamma_0} \d z \\ 
&= \frac{4k^2\gamma_0(1-2\gamma_0)}{1+2\gamma_0}\frac{h(z,y_0)}{(v'(z))^2}  \left(\G_{k,\ep}^\pm \right)_\s(y,y_0,z) f_\s(z,y_0) \xi^{-1-2\gamma_0}\Big|_{z\in \partial I_3(y_0)} \\
&\quad - \frac{4k^2\gamma_0(1-2\gamma_0)}{1+2\gamma_0}\int_{I_3(y_0)} \frac{h(z,y_0)}{(v'(z))^2}  \partial_z \left( \left(\G_{k,\ep}^\pm \right)_\s(y,y_0,z) f_\s(z,y_0) \right) \xi^{-1-2\gamma_0} \d z \\
&\quad + k\gamma_0\frac{1-2\gamma_0}{1+2\gamma_0} \partial_z \left(\frac{h(z,y_0)}{(v'(z))^2} \right)  \left(\G_{k,\ep}^\pm \right)_\s(y,y_0,z) f_\s(z,y_0) \xi^{-2\gamma_0}\Big|_{z\in \partial I_3(y_0)} \\
&\quad - k\gamma_0\frac{1-2\gamma_0}{1+2\gamma_0}\int_{I_3(y_0)} \partial_z \left(  \left(\G_{k,\ep}^\pm \right)_\s(y,y_0,z) f_\s(z,y_0) \partial_z \left(\frac{h(z,y_0)}{(v'(z))^2} \right)  \right) \xi^{-2\gamma_0}
\end{align*}
where we have integrated by parts twice. The vanishing and regularity of $h(\cdot,y_0)$, Proposition \ref{prop:sobolevregdecomG} and Corollary \ref{cor:sobolevregpartialdecomG} provide the $X_k^1$ estimates for the boundary terms, of order 1 in $k$. The first integral is estimated with the same reasoning, combining the ideas of $g_{6,k,\ep}^\pm$ and $g_{13,k,\ep}^\pm$, while for the second integral we have that 
\begin{align*}
k(1-2\mu_0)&\left\Vert \int_{I_3(y_0)}\partial_z \left(  \left(\G_{k,\ep}^\pm \right)_\s(y,y_0,z) f_\s(z,y_0) \partial_z \left(\frac{h(z,y_0)}{(v'(z))^2} \right)  \right) \xi^{-2\gamma_0}\d z \right\Vert_{X_k^1}\\
&\lesssim \Vert f \Vert_{X_k^1}(1-2\mu_0)\int_0^{3\beta} x^{-2\mu_0}\d x \lesssim \Vert f \Vert_{X_k},
\end{align*}
 we proceed as in $g_{13,k,\ep}^\pm$. With this, the proof is finished.
\end{proof}

As a corollary of Lemma \ref{lemma:R3mapsC1XktoXk} noting that 
\begin{align*}
k^2\Vert h(\cdot,y_0) \Vert_{L^\infty(I_3(y_0))} + k\Vert \partial_y h(\cdot, y_0) \Vert_{L^\infty(I_3(y_0))} \lesssim 1.
\end{align*}
for all $h(y,y_0)\in C^3_y$ uniformly in $y_0$ with $h(y_0,y_0) = \partial_y h(y,y_0)=0$, we have the following result.

\begin{lemma}\label{lemma:R3mapsC2XktoXk}
Let $k\geq 1$ and $y_0\in I_S\cup I_W$. Let $f(y,y_0)\in X_k$ and $h(y,y_0)\in C^3_y$ uniformly in $y_0$ with $h(y_0,y_0) = \partial_y h(y,y_0)=0$. Then,
\begin{align*}
\Vert R_{3,k,\ep}^\pm hf \Vert_{X_k}\lesssim k^{-1} \Vert f \Vert_{X_k},
\end{align*}
uniformly for all $0<\ep<\ep_*$ and all $y_0\in I_S\cup I_W$. 
\end{lemma}

We finish the subsection studying how $R_{3,k,\ep}^\pm$ acts on $H_k^1$ functions weighted with a smooth quadratically vanishing pre-factor.

\begin{lemma}\label{lemma:R3mapsC2H1toXk}
Let $k\geq 1$ and $y_0\in I_S\cup I_W$. Let $f(y,y_0)\in H_k^1$ and $h(y,y_0)\in C^3_y$ uniformly in $y_0$ with $h(y_0,y_0) = \partial_y h(y_0,y_0)=0$. Then,
\begin{align*}
\Vert R_{3,k,\ep}^\pm hf \Vert_{X_k}\lesssim  \Vert f \Vert_{H_k^1},
\end{align*}
uniformly for all $0<\ep<\ep_*$ and all $y_0\in I_S\cup I_W$. 
\end{lemma}

\begin{proof}
We integrate by parts to have
\begin{align*}
g_{1,k,\ep}^\pm &:= \int_{I_3(y_0)}\G_{k,\ep}^\pm(y,y_0,z)\frac{h(z,y_0)f(z,y_0)}{(v(z) - v(y_0) \pm i\ep)^3} \d z \\
&= -\frac12 \G_{k,\ep}^\pm(y,y_0,z)\frac{h(z,y_0)f(z,y_0)}{(v(z) - v(y_0) \pm i\ep)^2}\Big|_{z\in \partial I_3(y_0)} \\
&\quad + \frac12 \int_{I_3(y_0)} \partial_z \left( \frac{h(z,y_0)}{v'(z)}\right)\G_{k,\ep}^\pm (y,y_0,z)\frac{f(z,y_0)}{(v(z) - v(y_0) \pm i\ep)^2}\d z \\
&\quad + \frac12 \int_{I_3(y_0)} \frac{h(z,y_0)}{v'(z)} \G_{k,\ep}^\pm(y,y_0,z)\frac{\partial_z f(z,y_0)}{(v(z) - v(y_0) \pm i\ep)^2} \d z\\
&\quad +\frac12 \int_{I_3(y_0)}\frac{h(z,y_0)}{v'(z)} \frac{\partial_z \G_{k,\ep}^\pm(y,y_0,z)}{(v(z) - v(y_0) \pm i\ep)^2}f(z,y_0) \d z.
\end{align*}
The quadratic vanishing of $h(\cdot, y_0)$ and Lemma \ref{lemma:LinfH1bound} ensures
\begin{align*}
\left \Vert \G_{k,\ep}^\pm(y,y_0,z)\frac{h(z,y_0)f(z,y_0)}{(v(z) - v(y_0) \pm i\ep)^2}\Big|_{z\in \partial I_3(y_0)} \right\Vert_{X_k^1}\lesssim k^{-\frac12}\Vert f \Vert_{H_k^1},
\end{align*}
while we can argue as in Lemma \ref{lemma:R2mapsC1H1toXk} to bound
\begin{align*}
\left \Vert  \int_{I_3(y_0)} \partial_z \left( \frac{h(z,y_0)}{v'(z)}\right)\G_{k,\ep}^\pm (y,y_0,z)\frac{f(z,y_0)}{(v(z) - v(y_0) \pm i\ep)^2}\d z \right\Vert_{X_k^1} \lesssim k^{-\frac12}\Vert f \Vert_{H_k^1}.
\end{align*}
Similarly, we can use the smooth quadratic vanishing of $h(\cdot,y_0)$ to conclude 
\begin{align*}
\left \Vert  \int_{I_3(y_0)} \frac{h(z,y_0)}{v'(z)} \G_{k,\ep}^\pm(y,y_0,z)\frac{\partial_z f(z,y_0)}{(v(z) - v(y_0) \pm i\ep)^2} \d z \right\Vert_{X_k^1} \lesssim k^{-\frac12}\Vert f \Vert_{H_k^1}.
\end{align*}
Regarding $\partial_z\G_{k,\ep}^\pm(z,y_0)$, we shall consider its most singular term, and we observe that
\begin{align*}
&\left \Vert \int_{I_3(y_0)} \frac{h(z,y_0)}{v'(z)} \frac{\partial_z \left(\G_{k,\ep}^\pm\right)_\s(y,y_0,z) \xi^{\frac12-\gamma_0} + 2k(\tfrac12-\gamma_0)\left( \G_{k,\ep}^\pm\right)_\s(y,y_0,z) \xi^{-\frac12-\gamma_0}}{(v(z) - v(y_0) \pm i\ep)^2}f(z,y_0) \d z \right\Vert_{X_k^1} \\
&\qquad\qquad \lesssim k^{-\frac12}\Vert f\Vert_{H_k^1}
\end{align*}
because of the smooth quadratic vanishing of $h(\cdot, y_0)$, the embedding $\Vert f \Vert_{L^\infty}\lesssim k^\frac12\Vert f \Vert_{H_k^1}$ coming from Lemma \ref{lemma:LinfH1bound},  Propositions \ref{prop:sobolevregdecomG}, \ref{prop:sobolevregpartialdecomG} and Corollary \ref{cor:sobolevregpartialdecomG}.  This completes the proof.
\end{proof}

\subsection{Fourth order operator estimates}
We next define the most singular operator
\begin{align*}
R_{4,k,\ep}^\pm f(y,y_0) = \int_0^2 \G_{k,\ep}^\pm(y,y_0,z) \frac{f(z)}{(v(z) - v(y_0) \pm i\ep)^4} \d z.
\end{align*}
The next result shows that $R_{4,k,\ep}^\pm$ maps $X_k$ functions weighted with smooth quadratically vanishing pre-factors into $X_k$.
\begin{lemma}\label{lemma:R4mapsC1XktoXk}
Let $k\geq 1$ and $y_0\in I_S\cup I_W$. Let $f(y,y_0)\in X_k$ and $h(y,y_0)\in C^3_y$ uniformly in $y_0$ with $h(y_0,y_0) = \partial_y h(y_0,y_0) = 0$. Then,
\begin{align*}
\Vert R_{4,k,\ep}^\pm hf \Vert_{X_k}\lesssim  \Vert f \Vert_{X_k},
\end{align*}
uniformly for all $0<\ep<\ep_*$ and all $y_0\in I_S\cup I_W$. 
\end{lemma}

\begin{proof}
For $g_{k,\ep}^\pm(y,y_0):= R_{4,k,\ep}^\pm hf(y,y_0)$ we write
\begin{align*}
g_{k,\ep}^\pm(y,y_0) &= \int_{I_3(y_0)}\G_{k,\ep}^\pm(y,y_0,z)\frac{h(z,y_0)f(z,y_0)}{(v(z) - v(y_0) \pm i\ep)^4} \d z + \int_{I_3^c(y_0)}\G_{k,\ep}^\pm(y,y_0,z)\frac{h(z,y_0)f(z,y_0)}{(v(z) - v(y_0) \pm i\ep)^4} \d z \\
&= g_{1,k,\ep}^\pm(y,y_0) + g_{2,k,\ep}^\pm(y,y_0). 
\end{align*}
As usual, we have $\Vert g_{2,k,\ep}^\pm \Vert_{X_k^1} \lesssim \Vert f \Vert_{X_k}$. For the local contribution, we integrate by parts once,
\begin{align*}
g_{1,k,\ep}^\pm(y,y_0) &= -\frac13 \frac{h(z,y_0)}{v'(z)}\frac{\G_{k,\ep}^\pm(y,y_0,z)f(z,y_0)}{(v(z) - v(y_0) \pm i\ep)^3} \Big|_{z\in \partial I_3(y_0)} \\
& \quad + \frac13 \int_{I_3(y_0)} \partial_z \left( \frac{h(z,y_0)}{v'(z)} \right) \frac{\G_{k,\ep}^\pm(y,y_0,z)f(z,y_0)}{(v(z) - v(y_0) \pm i\ep)^3} \d z \\
&\quad +\frac13 \int_{I_3(y_0)} \frac{h(z,y_0)}{v'(z)}  \frac{\partial_z \left(\G_{k,\ep}^\pm(y,y_0,z)f(z,y_0)\right)}{(v(z) - v(y_0) \pm i\ep)^3} \d z.
\end{align*}
Proposition \ref{prop:sobolevregdecomG}, Corollayr \ref{cor:sobolevregpartialdecomG} and the smooth quadratic vanishing of $h(\cdot,y_0)$ yield
\begin{align*}
\left \Vert \frac{h(z,y_0)}{v'(z)}\frac{\G_{k,\ep}^\pm(y,y_0,z)f(z,y_0)}{(v(z) - v(y_0) \pm i\ep)^3} \Big|_{z\in \partial I_3(y_0)} \right\Vert_{X_k^1} \lesssim \Vert f \Vert_{X_k^1},
\end{align*}
while the arguments of Lemma \ref{lemma:R3mapsC1XktoXk} show that
\begin{align*}
\left\Vert \int_{I_3(y_0)} \partial_z \left( \frac{h(z,y_0)}{v'(z)} \right) \frac{\G_{k,\ep}^\pm(y,y_0,z)f(z,y_0)}{(v(z) - v(y_0) \pm i\ep)^3} \d z \right\Vert_{X_k^1}\lesssim \Vert f \Vert_{X_k^1}.
\end{align*}
Next, since $f\in X_k$, from \eqref{eq:partialzGf} we consider the most singular contributions of $\partial_z \left( \G_{k,\ep}^\pm(y,y_0,z) f(z,y_0)\right)$, namely
\begin{align*}
g_{3,k,\ep}^\pm(y,y_0) &:= \frac13\int_{I_3(y_0)} \frac{h(z,y_0)}{v'(z)}  \frac{\partial_z \left( \left(\G_{k,\ep}^\pm\right)_\s(y,y_0,z)f_\s(z,y_0)\right)}{(v(z) - v(y_0) \pm i\ep)^3}\xi^{1-2\gamma_0} \d z, \\
g_{4,k,\ep}^\pm(y,y_0) &:= \frac{2k(1-2\gamma_0)}{3}\int_{I_3(y_0)} \frac{h(z,y_0)}{v'(z)}  \frac{\left(\G_{k,\ep}^\pm\right)_\s(y,y_0,z)f_\s(z,y_0)}{(v(z) - v(y_0) \pm i\ep)^3}\xi^{-2\gamma_0} \d z.
\end{align*}
For $g_{3,k,\ep}^\pm$, thanks to the vanishing of $\partial_z \left( \left(\G_{k,\ep}^\pm\right)_\s(y,y_0,z)f_\s(z,y_0)\right)$ coming from Proposition \ref{prop:sobolevregpartialdecomG}, Corollary \ref{cor:sobolevregpartialdecomG} and $f\in X_k^1$, and the smooth quadratic vanishing of $h(\cdot,y_0)$ and  we argue as in $g_{6,k,\ep}^\pm$ in the proof of Lemma \ref{lemma:R3mapsC1XktoXk} to obtain $\Vert g_{3,k,\ep}^\pm \Vert_{X_k^1}\lesssim \Vert f \Vert_{X_k^1}$. We now address $g_{4,k,\ep}^\pm$. Since 
\begin{align*}
\frac{1}{v(y) - v(y_0) \pm i\ep} = \frac{(v'(y_0))^{-1}}{y-y_0\pm i\ep_0} -\mathrm{V}_{1,\ep}^\pm(y,y_0)
\end{align*}
with $\mathrm{V}_{1,\ep}^\pm(y,y_0)\in L^\infty(0,2)$ uniformly for all $y_0\in [0,2]$ and all $0<\ep$, we have
\begin{align*}
\frac{1}{(v(y) - v(y_0) \pm i\ep)^3} = \frac{(v'(y_0))^{-3}}{(y-y_0\pm i\ep_0)^3} + \frac{1}{(y-y_0\pm i\ep)^2}\mathrm{V}_{3,\ep}^\pm(y,y_0),
\end{align*}
with $\mathrm{V}_{3,\ep}^\pm(y,y_0)\in L^\infty(0,2)$ uniformly for all $y_0\in [0,2]$ and all $0<\ep$. Therefore,
\begin{align*}
g_{4,k,\ep}^\pm(y,y_0) &= (2k)^4\frac{1-2\gamma_0}{3(v'(y_0))^3}\int_{I_3(y_0)}\frac{h(z,y_0)}{v'(z)} \left(\G_{k,\ep}^\pm\right)_\s(y,y_0,z)f_\s(z,y_0) \xi^{-3-2\gamma_0} \d z \\
&\quad + (2k)^3\frac{1-2\gamma_0}{3(v'(y_0))^2} \int_{I_3(y_0)} \frac{h(z,y_0)}{v'(z)} \left(\G_{k,\ep}^\pm\right)_\s(y,y_0,z)f_\s(z,y_0)\mathrm{V}_{3,\ep}^\pm(z,y_0) \xi^{-2-2\gamma_0} \d z \\
&= g_{5,k,\ep}^\pm(y,y_0) + g_{6,k,\ep}^\pm(y,y_0).
\end{align*}
The quadratic vanishing of $h(\cdot, y_0)$,  $(1-2\mu_0)\int_0^{3\beta}x^{-2\mu_0}\d x \lesssim 1$ and Proposition \ref{prop:sobolevregdecomG} provide the estimate $\left \Vert g_{6,k,\ep}^\pm \right\Vert_{X_k^1}\lesssim \Vert f \Vert_{X_k^1}$. Integrating by parts whenever $\partial_z$ does not land on 
$$\mathrm{f}_{\s,k,\ep}^\pm(y,y_0,z):= \left(\G_{k,\ep}^\pm\right)_\s(y,y_0,z)f_\s(z,y_0),$$ 
we have
\begin{align*}
g_{5,k,\ep}^\pm(y,y_0) &= - \frac{(2k)^3}{3(v'(y_0))^3}\frac{1-2\gamma_0}{2+2\gamma_0} \frac{h(z,y_0)}{v'(z)} \mathrm{f}_{\s,k,\ep}^\pm(y,y_0,z) \xi^{-2-2\gamma_0} \Big|_{z\in \partial I_3(y_0)} \\
&\quad +\frac{(2k)^3}{3(v'(y_0))^3}\frac{1-2\gamma_0}{2+2\gamma_0}\int_{I_3(y_0)} \frac{h(z,y_0)}{v'(z)} \partial_z \mathrm{f}_{\s,k,\ep}^\pm(y,y_0,z) \xi^{-2-2\gamma_0} \d z \\
&\quad  - \frac{(2k)^2}{3(v'(y_0))^3}\frac{1-2\gamma_0}{(2+2\gamma_0)(1+2\gamma_0)} \partial_z \left(\frac{h(z,y_0)}{v'(z)}\right) \mathrm{f}_{\s,k,\ep}^\pm(y,y_0,z) \xi^{-1-2\gamma_0} \Big|_{z\in \partial I_3(y_0)} \\
&\quad + \frac{(2k)^2}{3(v'(y_0))^3}\frac{1-2\gamma_0}{(2+2\gamma_0)(1+2\gamma_0)} \int_{I_3(y_0)} \partial_z \left(\frac{h(z,y_0)}{v'(z)}\right)\mathrm{f}_{\s,k,\ep}^\pm(y,y_0,z) \xi^{-1-2\gamma_0} \d z \\
&\quad - \frac{2k}{3(v'(y_0))^3}\frac{1-2\gamma_0}{(2+2\gamma_0)(1+2\gamma_0)2\gamma_0} \partial_z^2 \left(\frac{h(z,y_0)}{v'(z)}\right) \mathrm{f}_{\s,k,\ep}^\pm(y,y_0,z) \xi^{-2\gamma_0} \Big|_{z\in \partial I_3(y_0)} \\
&\quad +  \frac{2k}{3(v'(y_0))^3}\frac{1-2\gamma_0}{(2+2\gamma_0)(1+2\gamma_0)2\gamma_0} \int_{I_3(y_0)}\partial_z \left( \partial_z^2 \left(\frac{h(z,y_0)}{v'(z)}\right)\mathrm{f}_{\s,k,\ep}^\pm(y,y_0,z) \right) \xi^{-2\gamma_0} \d z.
\end{align*}
The application of Proposition \ref{prop:sobolevregdecomG}, Propistion \ref{prop:sobolevregpartialdecomG} and Corollary \ref{cor:sobolevregpartialdecomG}, together with the smooth quadratic vanishing of $h(\cdot, y_0)$ and the observation that $(1-2\mu_0)\int_0^{3\beta}x^{-2\mu_0}\d x \lesssim 1$ give the desired bounds, we omit the routine details.
\end{proof}

\section{High-order operator estimates for mild stratifications}\label{sec:highordermild}
In this subsection we address the analogue results of Section \ref{sec:highorderweakstrong} for the mildly stratified regime $y_0\in I_M$. 

\subsection{Second order operator estimates}
We begin the section showing the mapping properties of $R_{2,k,\ep}^\pm$.
\begin{lemma}\label{lemma:R2mapsLXktoLXk}
Let $k\geq 1$, $y_0\in I_M$ and $f(y,y_0)\in LX_k$. There holds
\begin{align*}
\Vert (R_{2,k,\ep}^\pm f)(\cdot,y_0) \Vert_{LX_k} \lesssim  \Vert f \Vert_{LX_k} 
\end{align*}
uniformly in $0<\ep<\ep_*$, and in $y_0\in I_M$.
\end{lemma}

\begin{proof}
Let $g_{k,\ep}^\pm(y,y_0) = R_{2,k,\ep}^\pm f(y,y_0)$. We have
\begin{align*}
g_{k,\ep}^\pm(y,y_0) &= \frac{(2k)^2}{(v'(y_0))^2}\int_{I_3(y_0)} \G_{k,\ep}^\pm(y,y_0,z) \frac{f(z,y_0)}{\xi^2} \d z + \int_0^2 \G_{k,\ep}^\pm (y,y_0,z) \frac{\mathrm{V}_{2,\ep}^\pm(z,y_0) f(z,y_0)}{v(z) - v(y_0) \pm i\ep} \d z \\
&\quad + \frac{(2k)^2}{(v'(y_0))^2}\int_{I_3^c(y_0)} \G_{k,\ep}^\pm(y,y_0,z) \frac{f(z,y_0)}{\xi^2} \d z \\
&= g_{1,k,\ep}^\pm(y,y_0) + g_{2,k,\ep}^\pm(y,y_0) + g_{3,k,\ep}^\pm(y,y_0),
\end{align*}
where we have used that
\begin{align*}
\frac{1}{(v(z) - v(y_0) \pm i\ep)^2} = \frac{(v'(y_0))^{-2}}{(y-y_0 \pm i\ep_0)^2} + \frac{\mathrm{V}_{2,\ep}^\pm(y,y_0)}{v(y) - v(y_0) \pm i\ep},
\end{align*}
with $\mathrm{V}_{2,\ep}^\pm (\cdot,y_0)\in L^\infty$ uniformly in $y_0\in (0,2)$ and in $0<\ep<\ep_*$. Hence, we can use Lemma \ref{lemma:R1mapsLinfL2toLXk} to estimate $g_{2,k,\ep}$. The bounds for $g_{3,k,\ep}^\pm$ are deduced from Proposition \ref{prop:logsobolevregdecomG} and Corollary \ref{cor:logsobolevregpartialdecomG}. Using the regularity structures of $\G_{k,\ep}^\pm$ and of $f\in LX_k$, we have
\begin{align*}
(v'(y_0))^2 g_{1,k,\ep}^\pm(y,y_0) = g_{4,k,\ep}^\pm(y,y_0) + g_{5,k,\ep}^\pm(y,y_0) + g_{6,k,\ep}^\pm(y,y_0),
\end{align*}
with
\begin{align*}
g_{4,k,\ep}^\pm(y,y_0) &= (2k)^2\int_{I_3(y_0)} \left( \G_{k,\ep}^\pm \right)_\sr f_\sr(z,y_0) \xi^{-1+2\gamma_0}\d z, \\
g_{5,k,\ep}^\pm(y,y_0) &= (2k)^2\int_{I_3(y_0)} \left(  \left( \G_{k,\ep}^\pm \right)_\sr f_\s(z,y_0) + \left( \G_{k,\ep}^\pm \right)_\s f_\sr(z,y_0) \right) \xi^{-1}\log(\xi)\Q_{\gamma_0}(\xi)\d z, \\
g_{6,k,\ep}^\pm(y,y_0) &= (2k)^2\int_{I_3(y_0)} \left( \G_{k,\ep}^\pm \right)_\s f_\s(z,y_0) \xi^{-1-2\gamma_0}\left( \log(\xi) \Q_{\gamma_0}(\xi) \right)^2 \d z.
\end{align*}
To bound the three contributions, we recall Lemma \ref{lemma:firstgammalog}, namely 
\begin{align*}
\frac{\xi^{2\gamma_0}-1}{2\gamma_0} = \log(\xi) \Q_{\gamma_0}(\xi),
\end{align*}
with $|\Q_{\gamma_0}(\xi)|\lesssim 1$ for $|\xi|$ bounded. To estimate $g_{4,k,\ep}^\pm$ in $LX_k^1$, we note that
\begin{align*}
g_{4,k,\ep}^\pm(y,y_0) &= 2k \left( \G_{k,\ep}^\pm \right)_\sr f_\sr(z,y_0) \log(\xi) \Q_{\gamma_0}(\xi) \Big|_{z\in \partial I_3(y_0)} \\
&\quad - 2k\int_{I_3(y_0)} \partial_z \left( \left( \G_{k,\ep}^\pm \right)_\sr f_\sr(z,y_0)\right) \log(\xi) \Q_{\gamma_0}(\xi) \d z
\end{align*}
and the bounds follow from Propositions \ref{prop:logsobolevregdecomG}, \ref{prop:logsobolevregpartialdecomG} and Corollary \ref{cor:logsobolevregpartialdecomG}. Next, for $g_{5,k,\ep}^\pm$ we note that for
\begin{equation}\label{eq:defQ2gamma}
\Q_{2,\gamma_0}(\xi) := \int_0^1 \frac{e^{2\gamma_0 u \log(\xi)} - 1}{2\gamma_0 u} \d u  
\end{equation}
we have $2k\zeta^{-1}\Q_{\gamma_0}(\zeta) = \partial_z\Q_{2,\gamma_0}(\xi)$. Likewise, for 
\begin{equation}\label{eq:defQ3gamma}
\Q_{3,\gamma_0}(\xi) := \int_0^1 \left(\frac{e^{2\gamma_0 u \log(\xi)} - 1}{2\gamma_0 u}\right)^2 \d u  
\end{equation}
we also have $2k\zeta^{-1}\Q_{2,\gamma_0}(\xi)   = \frac12\partial_z \Q_{3,\gamma_0}(\xi)$. Furthermore,  
\begin{align*}
\left| \Q_{j,\gamma_0}(\xi)  \right| \lesssim |\log(\xi)|^{j-1},
\end{align*}
for $j=2,3$. Therefore, for $\mathrm{f}_{k,\ep}^\pm(y,y_0,z) := \left( \G_{k,\ep}^\pm \right)_\sr f_\s(z,y_0) + \left( \G_{k,\ep}^\pm \right)_\s f_\sr(z,y_0)$, we have
\begin{align*}
g_{5,k,\ep}^\pm &= 2k \,\mathrm{f}_{k,\ep}^\pm(y,y_0,z) \log(\xi) \Q_{2,\gamma_0}(\xi) \Big|_{z\in \partial I_3(y_0)} - 2k\int_{I_3(y_0)} \partial_z \mathrm{f}_{k,\ep}^\pm(y,y_0,z) \log(\xi) \Q_{2,\gamma_0}(\xi)\d z \\
&\quad - \mathrm{f}_{k,\ep}^\pm(y,y_0,z) \Q_{3,\gamma_0}(\xi) \Big|_{z\in \partial I_3(y_0)}  + \int_{I_3(y_0)} \partial_z \mathrm{f}_{k,\ep}^\pm(y,y_0,z) \Q_{3,\gamma_0}(\xi)\d z.
\end{align*}
Then, the $LX_k^1$ estimates are obtained as usual from Propositions \ref{prop:logsobolevregdecomG}, \ref{prop:logsobolevregpartialdecomG} and Corollary \ref{cor:logsobolevregpartialdecomG}. Finally, for $g_{6,k,\ep}^\pm$ we first note from Lemma \ref{lemma:secondgammalog} that
\begin{align*}
2k\xi^{-1-2\gamma_0}\left( \log(\xi) \Q_{\gamma_0}(\xi) \right)^2 &= \frac{2k}{4\gamma_0^2} \left( \xi^{-1+2\gamma_0} -2\xi^{-1} + \xi^{-1-2\gamma_0} \right) \\
&=  \frac{1}{4\gamma_0^2}\partial_z \left( \frac{\xi^{2\gamma_0}-1}{2\gamma_0} - 2\log(\xi) + \frac{1 - \xi^{-2\gamma_0}}{2\gamma_0} \right)\\
&=\frac{1}{4\gamma_0^2}\partial_z \mathrm{Q}_{\gamma_0}(\xi)
\end{align*}
Thus, integrating by parts we reach
\begin{align*}
g_{6,k,\ep}^\pm(y,y_0) &= \frac{2k}{4\gamma_0^2} \left( \G_{k,\ep}^\pm \right)_\s f_\s(z,y_0) \mathrm{Q}_{\gamma_0}(\xi) \Big|_{z\in \partial I_3(y_0)} \\
&\quad - \frac{2k}{4\gamma_0^2}\int_{I_3(y_0)}\partial_z \left( \left( \G_{k,\ep}^\pm \right)_\s f_\s(z,y_0) \right) \mathrm{Q}_{\gamma_0}(\xi) \d z.
\end{align*}
Together with Lemma~\ref{lemma:secondgammalog}, Propositions~\ref{prop:logsobolevregdecomG}, \ref{prop:logsobolevregpartialdecomG} and Corollary \ref{cor:logsobolevregpartialdecomG}, we obtain the estimates for $g_{6,k,\ep}^\pm$ and the proof is finished.
\end{proof}

We can combine Lemma \ref{lemma:R2mapsLXktoLXk} and Lemma \ref{lemma:R1mapsLinfL2toLXk} to obtain the following

\begin{corollary}\label{cor:R2mapsCLXktoLXk}
Let $k\geq 1$ and $y_0\in I_SM$. Let $f\in X_k$ and $h(z,y_0)\in C^2$ uniformly in $y_0$. Then,
\begin{align*}
\Vert R_{2,k,\ep}^\pm hf \Vert_{X_k}\lesssim \Vert f \Vert_{X_k},
\end{align*}
uniformly for all $0<\ep<\ep_*$ and all $y_0\in I_M$. 
\end{corollary}

\begin{proof}
It follows from Lemma \ref{lemma:R2mapsLXktoLXk} and Lemma \ref{lemma:R1mapsLinfL2toLXk} once we write
\begin{align*}
g_{k,\ep}^\pm(y,y_0) = R_{2,k,\ep}^\pm h f(y,y_0) = h(y_0) R_{2,k,\ep}^\pm f(y,y_0) + \left(R_{1,k,\ep}^\pm \frac{h_1(\cdot,y_0)}{v(\cdot) - v(y_0) \pm i\ep}f(\cdot,y_0) \right)(y),
\end{align*}
where we have defined $h_1(y,y_0) = h(y,y_0) - h(y_0,y_0)$, with $h_1(y_0,y_0)=0$ and $h_1(\cdot,y_0)\in C^2$ uniformly in $y_0\in [0,2]$.
\end{proof}

\subsection{Third order operator estimates}

We next address the mapping properties of $R_{3,k,\ep}^\pm$ in the regime $y_0\in I_M$.

\begin{lemma}\label{lemma:R3mapsC1LXktoLXk}
Let $k\geq 1$, $y_0\in I_M$ and $f(y,y_0)\in LX_k$. Let $h(y,y_0)$ be such that $h(y_0,y_0)=0$ and $h(\cdot,y_0)\in C^3$ uniformly in $y_0$. There holds
\begin{align*}
\Vert (R_{3,k,\ep}^\pm hf)(\cdot,y_0) \Vert_{LX_k} \lesssim  \Vert f \Vert_{LX_k} 
\end{align*}
uniformly in $0<\ep<\ep_*$, and in $y_0\in I_M$.
\end{lemma}

\begin{proof}
We argue as in the proof of Lemma \ref{lemma:R2mapsLXktoLXk}, we shall focus on
\begin{align*}
g_{1,k,\ep}^\pm &= (2k)^3\int_{I_3(y_0)} \G_{k,\ep}^\pm(y,y_0,z) h(z,y_0) f(z,y_0) \xi^{-3}\d z \\
&= (2k)^3 \int_{I_3(y_0)} \left( \G_{k,\ep}^\pm \right)_\sr(y,y_0,z) h(z,y_0) f_\sr(z, y_0)\xi^{-2+2\gamma_0} \d z \\
&\quad + (2k)^3\int_{I_3(y_0)} \mathrm{f}_{k,\ep}^\pm(y,y_0,z) h(z,y_0)\xi^{-2}\frac{\xi^{2\gamma_0}-1}{2\gamma_0} \d z \\
&\quad + (2k)^3 \int_{I_3(y_0)} \left( \G_{k,\ep}^\pm \right)_\s (y,y_0,z) h(z,y_0) f_\s(z, y_0)\xi^{-2-2\gamma_0} \left( \frac{\xi^{2\gamma_0}-1}{2\gamma_0} \right)^2 \d z \\
&= g_{2,k,\ep}^\pm(y,y_0) + g_{3,k,\ep}^\pm(y,y_0) + g_{4,k,\ep}^\pm(y,y_0).
\end{align*}
Integrating by parts, we have
\begin{align*}
g_{2,k,\ep}^\pm(y,y_0) &= -\frac{(2k)^2}{1-2\gamma_0}  \left( \G_{k,\ep}^\pm \right)_\sr(y,y_0,z) h(z,y_0) f_\sr(z, y_0)\xi^{-1+2\gamma_0} \Big|_{z\in \partial I_3(y_0} \\
&\quad + \frac{(2k)^2}{1-2\gamma_0}\int_{I_3(y_0)}   h(z,y_0) \partial_z \left(  \left( \G_{k,\ep}^\pm \right)_\sr(y,y_0,z)  f_\sr(z, y_0) \right) \xi^{-1+2\gamma_0} \d z \\
&\quad + \frac{2k}{1-2\gamma_0}\partial_z h(z,y_0) \left( \G_{k,\ep}^\pm \right)_\sr(y,y_0,z)  f_\sr(z, y_0) \log(\xi) \Q_{\gamma_0}(\xi) \Big|_{z\in \partial I_3(y_0)} \\
&\quad - \frac{2k}{1-2\gamma_0} \int_{I_3(y_0)} \partial_z \left( \partial_z h(z,y_0) \left( \G_{k,\ep}^\pm \right)_\sr(y,y_0,z)  f_\sr(z, y_0) \right) \log(\xi) \Q_{\gamma_0}(\xi)  \d z
\end{align*}
and the bounds for $g_{2,k,\ep}^\pm$ follow from the usual applications of Propositions \ref{prop:logsobolevregdecomG}, \ref{prop:logsobolevregpartialdecomG} and Corollary \ref{cor:logsobolevregpartialdecomG}. For $g_{3,k,\ep}^\pm$, we have
\begin{align*}
g_{3,k,\ep}^\pm(y,y_0) &= (2k)^2 \mathrm{f}_{k,\ep}^\pm(y,y_0,z) h(z,y_0) \frac{\xi^{-1}}{2\gamma_0-1}\left( 1 + \log(\xi)\Q_{\gamma_0}(\xi) \right) \Big|_{z\in \partial I_3(y_0)} \\
&\quad + \frac{(2k)^2}{1-2\gamma_0}\int_{I_3(y_0)} \partial_z \left( \mathrm{f}_{k,\ep}^\pm(y,y_0,z) h(z,y_0) \right) \xi^{-1}(1+\log(\xi) \Q_{\gamma_0}(\xi) ) \d z
\end{align*}
and thus we see that the estimates follow from the arguments in the proof of Lemma \ref{lemma:R2mapsLXktoLXk} for the term involving $\mathrm{f}_{k,\ep}^\pm(y,y_0,z) \partial_z h(z,y_0)$ and from the vanishing of $h(z,y_0)$ at $z=y_0$ for the contribution containing $\partial_z \mathrm{f}_{k,\ep}^\pm(y,y_0,z) h(z,y_0)$, we omit the details. Finally, for $g_{4,k,\ep}^\pm$ we observe that
\begin{align*}
2k\xi^{-2-2\gamma_0}\left( \frac{\xi^{2\gamma_0}-1}{2\gamma_0} \right)^2 &= \frac{1}{4\gamma_0^2} \partial_z \left( \frac{\xi^{-1+2\gamma_0}}{-1+2\gamma_0} +2\xi^{-1} -\frac{\xi^{-1-2\gamma_0}}{1+2\gamma_0} \right),
\end{align*}
with 
\begin{align}
2k& \left( \frac{\xi^{-1+2\gamma_0}}{-1+2\gamma_0} +2\xi^{-1} -\frac{\xi^{-1-2\gamma_0}}{1+2\gamma_0} \right) \\
&= \xi^{-1}\frac{2k}{4\gamma_0^2 -1}\left( 2\gamma_0 \left( \xi^{2\gamma_0} - \xi^{-2\gamma_0} \right) + \xi^{-2\gamma_0}\left( \xi^{2\gamma_0} -1 \right)^2 + 8\gamma_0^2 \right) \label{eq:R3Qpartialz} \\
&= \frac{1}{4\gamma_0^2-1} \partial_z \left( (1+2\gamma_0) \frac{\xi^{2\gamma_0}-1}{2\gamma_0} + 2(4\gamma_0^2 -1)\log(\xi) - (2\gamma_0-1)\frac{1-\xi^{-2\gamma_0}}{2\gamma_0} \right). \\
&= \frac{1}{4\gamma_0^2-1} \partial_z \left( \mathrm{Q}_{\gamma_0}(\xi)  + \xi^{-2\gamma_0}\left( \xi^{2\gamma_0}-1 \right)^2  + 8\gamma_0^2\log(\xi)  \right)\label{eq:R3Qpartialpartialz}
\end{align}
Hence, for $\mathrm{f}_{\s,k,\ep}^\pm(y,y_0,z):=\left( \G_{k,\ep}^\pm \right)_\s(y,y_0,z) f_\s(z,y_0)$, we have
\begin{align*}
g_{4,k,\ep}^\pm(y,y_0) &= \frac{(2k)^2}{4\gamma_0^2}\mathrm{f}_{\s,k,\ep}^\pm(y,y_0,z) h(z,y_0) \left( \frac{\xi^{-1+2\gamma_0}}{-1+2\gamma_0} +2\xi^{-1} -\frac{\xi^{-1-2\gamma_0}}{1+2\gamma_0} \right) \Big|_{z\in \partial I_3(y_0)} \\
&\quad - \frac{(2k)^2}{4\gamma_0^2}\int_{I_3(y_0)} \frac{h(z,y_0)}{4\gamma_0^2 -1} \partial_z \mathrm{f}_{\s,k,\ep}^\pm(y,y_0,z)  \left( \frac{\xi^{-1+2\gamma_0}}{-1+2\gamma_0} +2\xi^{-1} -\frac{\xi^{-1-2\gamma_0}}{1+2\gamma_0} \right) \d z \\
&\quad -\frac{2k}{4\gamma_0^2}\frac{\partial_z h(z,y_0)}{4\gamma_0^2-1} \mathrm{f}_{\s,k,\ep}^\pm(y,y_0,z)\left( \mathrm{Q}_{\gamma_0}(\xi)  + \xi^{-2\gamma_0}\left( \xi^{2\gamma_0}-1 \right)^2  + 8\gamma_0^2\log(\xi)  \right) \Big|_{z\in \partial I_3(y_0)} \\
&\quad + \frac{2k}{4\gamma_0^2}\int_{I_3(y_0)}\partial_z \left( \frac{\partial_z h(z,y_0)}{4\gamma_0^2-1}\mathrm{f}_{\s,k,\ep}^\pm(y,y_0,z)\right) \left( \mathrm{Q}_{\gamma_0}(\xi)  + \xi^{-2\gamma_0}\left( \xi^{2\gamma_0}-1 \right)^2  + 8\gamma_0^2\log(\xi)  \right) \d z
\end{align*}
and the $LX_k^1$ estimates for $g_{4,k,\ep}^\pm$ follow from \eqref{eq:R3Qpartialz}, \eqref{eq:R3Qpartialpartialz}, the observations that
\begin{align*}
\left| 2\gamma_0 \left( \xi^{2\gamma_0} - \xi^{-2\gamma_0} \right) + \xi^{-2\gamma_0}\left( \xi^{2\gamma_0} -1 \right)^2 + 8\gamma_0^2 \right| \lesssim \gamma_0^2\left( 1 + |\log(\xi)|^3 \right)
\end{align*}
and
\begin{align*}
\left| \mathrm{Q}_{\gamma_0}(\xi)  + \xi^{-2\gamma_0}\left( \xi^{2\gamma_0}-1 \right)^2  + 8\gamma_0^2\log(\xi)  \right| \lesssim |\xi|^{-\frac12}\gamma_0^2 \left( 1 + |\log(\xi)| \right) 
\end{align*}
and the usual usage of Propositions \ref{prop:logsobolevregdecomG}, \ref{prop:logsobolevregpartialdecomG} and Corollary \ref{cor:logsobolevregpartialdecomG}.
\end{proof}

\subsection{Fourth order operator estimates}
We finish the section addressing the bounds for the most singular operator $R_{4,k,\ep}^\pm$ in the mildly stratified regime.

\begin{lemma}\label{lemma:R4mapsC2LXktoC2LXk}
Let $k\geq 1$ and $y_0\in I_M$. Let $f(y,y_0)\in X_k$ and $h(y,y_0)\in C^3_y$ uniformly in $y_0$ with $h(y_0,y_0) = \partial_y h(y_0,y_0) = 0$. Then,
\begin{align*}
\Vert R_{4,k,\ep}^\pm hf \Vert_{LX_k}\lesssim  \Vert f \Vert_{LX_k},
\end{align*}
uniformly for all $0<\ep<\ep_*$ and all $y_0\in I_M$. 
\end{lemma}

\begin{proof}
We shall argue as in the proofs of Lemma \ref{lemma:R2mapsLXktoLXk} and Lemma \ref{lemma:R3mapsC1LXktoLXk}, writing
\begin{align*}
g_{1,k,\ep}^\pm(y,y_0) = g_{2,k,\ep}^\pm(y,y_0) + g_{3,k,\ep}^\pm(y,y_0) + g_{4,k,\ep}^\pm(y,y_0),
\end{align*}
where now
\begin{align*}
g_{2,k,\ep}^\pm(y,y_0) &= (2k)^4\int_{I_3(y_0)}h(z,y_0) \mathrm{f}_{\sr,k,\ep}^\pm(z,y_0,y) \xi^{-3+2\gamma_0} \d z, \\
g_{3,k,\ep}^\pm(y,y_0) &= (2k)^4\int_{I_3(y_0)}h(z,y_0) \mathrm{f}_{k,\ep}^\pm(z,y_0,y) \xi^{-3}\frac{\xi^{2\gamma_0}-1}{2\gamma_0} \d z, \\
g_{3,k,\ep}^\pm(y,y_0) &= (2k)^4\int_{I_3(y_0)}h(z,y_0) \mathrm{f}_{\s,k,\ep}^\pm(z,y_0,y) \xi^{-3-2\gamma_0}\left(\frac{\xi^{2\gamma_0}-1}{2\gamma_0}\right)^2 \d z.
\end{align*}
For $g_{2,k,\ep}^\pm$ we note that
\begin{align*}
(2k)^3 \xi^{-3+2\gamma_0} = \frac{(2k)^2}{2\gamma_0-1}\partial_z \xi^{-2+2\gamma_0} = \frac{2k}{(2\gamma_0-2)(2\gamma_0-1)}\partial_z^2 \xi^{-1+2\gamma_0}.
\end{align*}
Hence,
\begin{align*}
g_{2,k,\ep}^\pm(y,y_0) &= \frac{(2k)^3}{2\gamma_0-2} h(z,y_0) \mathrm{f}_{\sr,k,\ep}^\pm(y,y_0,z) \xi^{-2+2\gamma_0} \Big|_{z\in \partial I_3(y_0)} \\
&\quad - \frac{(2k)^3}{2\gamma_0-2}\int_{I_3(y_0)}h(z,y_0) \partial_z \mathrm{f}_{\sr,k,\ep}^\pm(y,y_0,z) \xi^{-2+2\gamma_0} \d z \\
&\quad - \frac{(2k)^2}{(2\gamma_0-2)(2\gamma_0-1)}\partial_z h(z,y_0) \mathrm{f}_{\sr,k,\ep}^\pm(y,y_0,z) \xi^{-1+2\gamma_0} \Big|_{z\in \partial I_3(y_0)} \\
&\quad + \frac{(2k)^2}{(2\gamma_0-2)(2\gamma_0-1)}\int_{I_3(y_0)} \partial_z \left( \partial_z h(z,y_0) \mathrm{f}_{\sr,k,\ep}^\pm(y,y_0,z) \right) \xi^{-1+2\gamma_0} \d z,
\end{align*}
from which the $LX_k^1$ estimates follow, after using the smooth quadratic vanishing properties of $h(z,y_0)$ at $z=y_0$. For $g_{3,k,\ep}^\pm$, we now argue 
\begin{align*}
2k\xi^{-3}\left( \xi^{2\gamma_0}-1\right) = \partial_z \left( \frac{\xi^{2\gamma_0-2}}{-2+2\gamma_0} +\frac{\xi^{-2}}{2} \right) = \partial_z \left( \frac{\xi^{-2}}{2\gamma_0-2} \left(  \xi^{2\gamma_0} -1 + \gamma_0 \right) \right)
\end{align*}
with further
\begin{align*}
2k \left( \frac{\xi^{2\gamma_0-2}}{-2+2\gamma_0} +\frac{\xi^{-2}}{2} \right) &= \partial_z \left( \frac{\xi^{2\gamma_0-1}}{(2\gamma_0-2)(2\gamma_0-1)} - \frac{\xi^{-1}}{2} \right) \\
&= \frac{\xi^{-1}}{(2\gamma_0-2)(2\gamma_0-1)}\left( \xi^{2\gamma_0} - 1 +2\gamma_0^2-3\gamma_0 \right)
\end{align*}
and finally
\begin{align*}
2k \left( \frac{\xi^{2\gamma_0-1}}{(2\gamma_0-2)(2\gamma_0-1)} - \frac{\xi^{-1}}{2} \right) &= \frac{1}{(2\gamma_0-2)(2\gamma_0-1)} \partial_z \left( \frac{\xi^{2\gamma_0}-1}{2\gamma_0} - (\gamma_0-1)(2\gamma_0-1)\log(\xi) \right) \\
&\quad= \frac{1}{4\gamma_0^2-6\gamma_0+2}\left( \frac{\xi^{2\gamma_0}-1}{2\gamma_0} - \log(\xi) - (2\gamma_0^2-3\gamma_0) \log(\xi) \right)
\end{align*}
so that a repeated integration by parts on the integral terms of the form $\mathrm{f}_{k,\ep}^\pm(y,y_0,z) \partial_z^n h(z,y_0)$ for $n\geq 0$, together with the quadratic vanishing of $h(z,y_0)$ at $z=y_0$ and the estimates of $\partial_z^n\mathrm{f}_{k,\ep}^\pm$ for $n=0,1$ yields the desired result, we omit the details. Finally, for $g_{4,k,\ep}^\pm$ we observe that
\begin{align*}
(2k)\left( \xi^{-3+2\gamma_0} - 2\xi^{-3} + \xi^{-3+2\gamma_0} \right) &= \partial_z \left( \frac{\xi^{-2+2\gamma_0}}{2\gamma_0-2} + \xi^{-2} - \frac{\xi^{-2-2\gamma_0}}{2\gamma_0+2} \right) \\
&=\partial_z \left(\frac{\xi^{-2}}{4\gamma_0^2-4} \left( 2\gamma_0 \left( \xi^{2\gamma_0}- \xi^{-2\gamma_0} \right) + 2\xi^{-2\gamma_0}\left( \xi^{2\gamma_0}-1\right)^2 +4\gamma_0^2 \right) \right)
\end{align*}
with further
\begin{align*}
2k\left( \frac{\xi^{-2+2\gamma_0}}{2\gamma_0-2} + \xi^{-2} - \frac{\xi^{-2-2\gamma_0}}{2\gamma_0+2} \right) &= \partial_z \left( \frac{\xi^{-1+2\gamma_0}}{(2\gamma_0-2)(2\gamma_0-1)}  - \xi^{-1}  + \frac{\xi^{-1-2\gamma_0}}{(2\gamma_0+2)(2\gamma_0+1)} \right) \\
&= \partial_z \left(\frac{2\gamma_0\xi^{-1}}{(4\gamma_0^2-4)(4\gamma_0^2-1)}\left( 2\gamma_0 \left(  \xi^{2\gamma_0} + \xi^{-2\gamma_0} \right)  + 3\left( \xi^{2\gamma_0}-\xi^{-2\gamma_0} \right) \right) \right) \\
&\quad +\partial_z \left( \frac{\xi^{-1}}{(2\gamma_0^2-2)(4\gamma_0^2-1)}\left( \xi^{-2\gamma_0} \left( \xi^{2\gamma_0} -1 \right)^2 - 2\gamma_0^2\left( 4\gamma_0^2 -4 \right) \right) \right)
\end{align*}
and also
\begin{align*}
2k& \left( \frac{\xi^{-1+2\gamma_0}}{(2\gamma_0-2)(2\gamma_0-1)}  - \xi^{-1}  + \frac{\xi^{-1-2\gamma_0}}{(2\gamma_0+2)(2\gamma_0+1)} \right) \\
&= \frac{\partial_z \left( (4\gamma_0^2 + 6\gamma_0 +2)\frac{\xi^{2\gamma_0}-1}{2\gamma_0} - (4\gamma_0^2 -4)(4\gamma_0^2-1)\log(\xi) + (4\gamma_0^2 -6\gamma_0 + 2)\frac{1-\xi^{-2\gamma_0}}{2\gamma_0} \right)}{(4\gamma_0^2-4)(4\gamma_0^2-1)} \\
&= \frac{ \partial_z \left( 2\gamma_0\left( \xi^{2\gamma_0}-\xi^{-2\gamma_0} \right) +3 \left( \xi^{2\gamma_0} +\xi^{-2\gamma_0} -2 \right) +2 \mathrm{Q}_{\gamma_0}(\xi) \right) }{(4\gamma_0^2-4)(4\gamma_0^2 - 1)}
\end{align*}
With the above formulae, the estimate follows by repeatedly integrating by parts the integral terms of the form $\mathrm{f}_{\s,k,\ep}^\pm(y,y_0,z) \partial_z^n h(z,y_0)$ for $n\geq 0$, we omit the routine details.
\end{proof}

\section{Refined regularity of the spectral density function}
Let $\varphi_{k,\ep}^\pm$ be the unique solution to \eqref{eq:TGeqvarphiSobolevReg}. In this section we obtain improved bounds on
\begin{align*}
\cV_{k,\ep}^\pm(y,y_0) := (v'(y) - v'(y_0))\varphi_{k,\ep}^\pm(y,y_0).
\end{align*}
Firstly, we note that
\begin{equation}\label{eq:TGcV}
\textsc{TG}_{k,\ep}^\pm \cV_{k,\ep}^\pm(y,y_0) = v'''(z)\varphi_{k,\ep}^\pm + 2 v''(z) \partial_y\varphi_{k,\ep}^\pm(z,y_0) + \frac{(v'(y) - v'(y_0)) w_k^0(y)}{v(y) - v(y_0) \pm i\ep} + (v'(y) - v'(y_0)) q_k^0(y)
\end{equation}
together with $\cV_{k,\ep}^\pm(0,y_0)=\cV_{k,\ep}^\pm(2,y_0)=0$. 

\begin{proposition}
We have
\begin{align*}
\Vert \partial_y^n \cV_{k,\ep}^\pm \Vert_{H_k^1(I_3^c(y_0))} \lesssim k^{n-2}\cS_{k,0}, \quad \Vert \partial_y^n \cV_{k,\ep}^\pm \Vert_{L^\infty(I_3^c(y_0))} \lesssim k^{n-\frac32}\cS_{k,0},
\end{align*}
for $n=0,1$, uniformly for all $y_0\in (\vartheta_1,\vartheta_2)$.
\end{proposition}

\begin{proof}
We first show the $H_k^1(I_3^c(y_0))$ estimate for $n=0$. From \eqref{eq:TGcV} we can actually write
\begin{align*}
\textsc{RTG}_{k,\ep}^\pm \cV_{k,\ep}^\pm(y,y_0) &= - \textsc{E}_{k,\ep}^\pm \cV_{k,\ep}^\pm(y,y_0) + v'''(y)\varphi_{k,\ep}^\pm(y,y_0) + 2 v''(y) \partial_y\varphi_{k,\ep}^\pm(y,y_0) \\
&\quad+ \frac{(v'(y) - v'(y_0)) w_k^0(y)}{v(y) - v(y_0) \pm i\ep} + (v'(y) - v'(y_0)) q_k^0(y).
\end{align*}
Appealing to the entanglement inequality Lemma \ref{lemma:entanglementRTG}, we obtain
\begin{align*}
 \Vert \cV_{k,\ep}^\pm \Vert_{H_k^1(I_3^c(y_0))} &\lesssim \Vert \cV_{k,\ep}^\pm \Vert_{L^2(I_2^c(y_0)\cap I_3(y_0))} + \frac{1}{k^2}\Vert \textsc{E}_{k,\ep}^\pm \cV_{k,\ep}^\pm \Vert_{L^2(I_2^c(y_0))}\\
 &\quad + \frac{1}{k^2}\Vert v'''\varphi_{k,\ep}^\pm + 2 v'' \partial_y\varphi_{k,\ep}^\pm \Vert_{L^2(I_2^c(y_0))} \\
 &\quad + \frac{1}{k^2}\left \Vert \frac{(v'(y) - v'(y_0)) w_k^0}{v(y) - v(y_0) \pm i\ep} + (v'(y) - v'(y_0)) q_k^0\right\Vert_{L^2(I_2^c(y_0))}.
\end{align*}
Using Proposition \ref{prop:Xkvarphi} we easily get $\Vert \cV_{k,\ep}^\pm \Vert_{L^2(I_2^c(y_0)\cap I_3(y_0))} \lesssim k^{-2}\cS_{k,0}$. On the other hand, we now have
\begin{align*}
\left| \textsc{E}_{k,\ep}^\pm(y,y_0)(v'(y) - v'(y_0) ) \right| \lesssim 1
\end{align*}
uniformly for all $y, y_0\in[0,2]$, so that Proposition \ref{prop:Xkvarphi} yields $\Vert \textsc{E}_{k,\ep}^\pm \cV_{k,\ep}^\pm \Vert_{L^2(I_2^c(y_0))} \lesssim k^{-1}\cS_{k,0}$. Similarly, we obtain $\Vert v'''\varphi_{k,\ep}^\pm + 2 v'' \partial_y\varphi_{k,\ep}^\pm \Vert_{L^2(I_2^c(y_0))} \lesssim \cS_{k,0}$. Lastly, 
$$\left \Vert \frac{(v'(y) - v'(y_0)) w_k^0}{v(y) - v(y_0) \pm i\ep} + (v'(y) - v'(y_0)) q_k^0\right\Vert_{L^2(I_2^c(y_0))} \lesssim \cS_{k,0}.$$ 
Combining the estimates we obtain the desired $H_{k}^1(I_3^c(y_0))$ bound for $\cV_{k,\ep}^\pm$. Those for $\partial_y\cV_{k,\ep}^\pm$ follow then from \eqref{eq:TGcV}. For the $L^\infty(I_3^c(y_0))$ estimate on $\partial_y^n\cV_{k,\ep}^\pm$ we proceed as in the proof of Proposition \ref{prop:Xkvarphi}, we omit the details.
\end{proof}

\begin{proposition}\label{prop:LinfpynotcV}
We have
\begin{align*}
\left \Vert \partial_{y_0}^n\left( (v'(z) - v'(y_0)) \varphi_{k,\ep}^\pm(z,y_0) \right) \right\Vert_{L^\infty(y_0\in I_3^c(z))}  \lesssim k^{-\frac32}\cS_{k,n}
\end{align*}
for $n=0,1$.
\end{proposition}

\begin{proof}
We argue as in the proof of Proposition \ref{prop:Xkvarphi}. 

\diampar{Case $y\in I_6(y_0)$} Assume further that $y>y_0$. Then, for $y_3=y_0+\frac{3\beta}{k}$ we have
\begin{align*}
\partial_{y_0}\cV_{k,\ep}^\pm(y,y_0) = \partial_{y_0}\cV_{k,\ep}^\pm(y_3,y_0) + \int_{y_3}^y \partial_y\partial_{y_0} \cV_{k,\ep}^\pm.
\end{align*}
Since $y_3\in I_3(y_0)$, we have from Propositions \ref{prop:Xkvarphi} and \ref{prop:Xkpynotvarphi} that
\begin{align*}
\left| \partial_{y_0}\cV_{k,\ep}^\pm(y_3,y_0) \right| \lesssim k^{-\frac32}\cS_{k,1}.
\end{align*}
On the other hand, since $y\in I_6(y_0)$, there holds
\begin{align*}
\int_{y_3}^y \partial_{y}\partial_{y_0}\cV_{k,\ep}^\pm(z,y_0) \d z &\lesssim \int_{y_3}^y \left( k^{-1}|\partial_{y,y_0}^2\varphi_{k,\ep}^\pm(z,y_0)| + |\partial_y\varphi_{k,\ep}^\pm(z,y_0)| + |\partial_y\varphi_{k,\ep}^\pm(z,y_0)| \right) \d z \\
&\lesssim k^{-\frac32}\cS_{k,1}.
\end{align*}

\diampar{Case $y\in I_6^c(y_0)$} Then $I_3(y)\cap I_3(y_0) = \emptyset$, namely $I_3(y)\subset I_3^c(y_0)$. Moreover, from Lemma \ref{lemma:LinfH1bound} and Propositions \ref{prop:Xkvarphi} and \ref{prop:Xkpynotvarphi} there holds
\begin{align*}
\Vert \partial_{y_0}\cV_{k,\ep}^\pm \Vert_{L^\infty(I_3(y))} &\lesssim k^\frac12 \Vert \partial_{y_0}\cV_{k,\ep}^\pm \Vert_{H_k^1(I_3(y))} \\
&\lesssim k^\frac12 \Vert \partial_{y_0} (v'(\cdot) - v'(y_0))\varphi_{k,\ep}^\pm \Vert_{L^2(I_3(y))} \\
&\quad+ k^{-\frac12}\Vert \partial_y\partial_{y_0} (v'(\cdot) - v'(y_0))\varphi_{k,\ep}^\pm \Vert_{L^2(I_3(y))}  \\
&\lesssim k^{-\frac32}\cS_{k,1}
\end{align*}
since we used that $|v'(z)-v'(y_0)|\lesssim k^{-1}$ for $z\in I_3(y)$.

\end{proof}

\bibliographystyle{abbrv}
\bibliography{Refs}
\end{document}